\documentclass[a4paper,11pt]{amsart}
\usepackage{amssymb}
\usepackage{amscd}
\usepackage{comment}
\usepackage{amsmath,amsthm}
\usepackage[linktocpage,linkcolor=red,colorlinks]{hyperref}
\usepackage{enumerate}
\usepackage{booktabs,multirow}		
\usepackage{tikz}
\usepackage{rotating}
\usetikzlibrary{patterns}
\usetikzlibrary{decorations.pathreplacing}
\usetikzlibrary{calc,through}

\allowdisplaybreaks[1]
\makeatletter
\@namedef{subjclassname@2020}{%
  \textup{2020} Mathematics Subject Classification}
\makeatother
\makeatletter
\def\l@subsection{\@tocline{2}{0pt}{2.5pc}{5pc}{}}
\makeatother

\numberwithin{figure}{section}
\numberwithin{equation}{section}

\title[Multi-colored dimer models in one-dimension]{Multi-colored dimer models in one-dimension: \\
lattice paths and generalized Rogers--Ramanujan identities}
\author[K.~Shigechi]{Keiichi~Shigechi}
\email{k1.shigechi AT gmail.com}
\date{\today}

\newcommand\tikzpic[2]{
\raisebox{#1\totalheight}{
\begin{tikzpicture}
#2
\end{tikzpicture}
}}
\DeclareMathOperator*{\prodd}{{\prod}^\prime}
\newtheorem{theorem}[figure]{Theorem}
\newtheorem{example}[figure]{Example}
\newtheorem{lemma}[figure]{Lemma}
\newtheorem{defn}[figure]{Definition}
\newtheorem{prop}[figure]{Proposition}
\newtheorem{cor}[figure]{Corollary}

\newtheorem{remark}[figure]{Remark}

\begin{document}

\begin{abstract}
We define and study multi-colored dimer models on a segment and on a circle.
The multivariate generating functions for the dimer models satisfy 
the recurrence relations similar to the one for Fibonacci numbers.
We give closed formulae for the generating functions.
We show that, in the large size limit with specializations of the formal 
variables, the generating functions exhibit the summations appearing in 
generalized Rogers--Ramanujan identities. 
Further, the generating functions of the dimer models have infinite product formulae for 
general values of formal variables in the large size limit.
These formulae are generalizations of Rogers--Ramanujan identities for multi variables.
We also give other several specializations which exhibit 
simple combinatorial formulae.
The analysis of the correlation functions, which we call emptiness
formation probabilities and moments, leads to the application of 
the formal power series associated to the Dyck, Motzkin and Schr\"oder paths to the generating 
functions for the dimer models.
We give descriptions of the generating 
functions of finite size in terms of these combinatorial objects, 
Dyck and Motzkin paths with statistics.
We have three additional results.
First, the convoluted generating functions for Fibonacci, Catalan and 
Motzkin numbers are shown to be expressed as generating functions of Fibonacci, 
Dyck and Motzkin words with the weights given by binomial coefficients.
The second one is a weight preserving correspondence between a Motzkin path and 
a set of Dyck paths. 
The third one is a connection of the generating functions for the dimer models 
to the generating functions of independent sets of special classes of graphs.
\end{abstract}

\maketitle

\tableofcontents

\section{Introduction}
\subsection{Background}
Fibonacci numbers $F_{n}$ is one of the most well-known combinatorial numbers 
which satisfy the recurrence relation $F_{n}=F_{n-1}+F_{n-2}$ with initial 
conditions $F_{-1}=F_{0}=1$ \cite[A000045]{Slo}.
An introduction for this well-studied sequence can be found 
in the books \cite{Kos01,Mol12}.
More generally, generalized Fibonacci polynomials ($q$-analogue of Fibonacci 
numbers) are studied, for example, in \cite{AmdCheMolSag13,HogLon74,Ric95}. 
One of the methods to obtain the generating function for Fibonacci numbers, 
$\sum_{0\le n}F_{n}x^{n}$, 
is the description by another combinatorial object, dimer configurations 
on a segment of size $n$.

The theory of heaps of dimers has vast applications such as Fibonacci polynomials, Catalan numbers,
percolation, lattice animals and polyominoes (see \cite{BouMelRec02,GarGan20} and references 
therein).
The connection of Fibonacci polynomials to the theory of heaps of pieces by Viennot \cite{Vie83,Vie86}
is discussed in \cite{GarGan20}.
In this paper, we consider two simpler dimer models than the theory of heaps studied 
in \cite{Vie83,Vie86}.
The dimer models considered in this paper are the ones in one-dimension, and 
defined as follows.

Suppose that a segment of size $n$ consists of $n+1$ vertices $v_0,\ldots,v_{n}$.
A dimer at position $i$ on a segment is an edge connecting two vertices $v_{i-1}$ and 
$v_{i}$.
A dimer configuration satisfies the following two conditions (see, e.g., \cite[I.6. in page 27]{FlaSed09}):
\begin{enumerate}[(C1)]
\item There exists at most one dimer at position $i$,
\item If there exists a dimer at position $i$, then there are no dimers 
at position $i-1$ and $i+1$.
\end{enumerate}
Then, one can easily show that the total number of dimer configurations 
on a segment of size $n$ is equal to the $n$-th Fibonacci number by 
use of the sequence construction of the generating function \cite[I.6. in page 27]{FlaSed09}.

We generalize the dimer model by relaxing the second condition 
(C2) with the introduction of colors of a dimer.
Let $d(i)$ be a dimer at position $i$ and $c:=c(d(i))$ be its color.
The second condition can be replaced by the following condition:
\begin{enumerate}
\item[(C$2'$)]
If there exists a dimer at position $i$ with color $c$,
the color of a dimer adjacent to the dimer $d(i)$ is different 
from the color $c$.
\end{enumerate} 
Note that if we have only single color, the condition (C$2'$) is equivalent 
to the second condition (C2).
The condition (C$2'$) implies that if the number $m$ of colors satisfies 
$m\ge2$, the maximum number of dimers on a segment of size $n$ is $n$.

Another generalization of the dimer model for Fibonacci numbers is to 
replace a segment by a circle.
The generating function on a circle still satisfies the recurrence relation 
similar to the one for Fibonacci numbers.
The leading term of the generating function on a circle is given by 
the generating function on a segment.
The difference from the case of a segment reflects the effect of 
the boundary conditions.
As we will see later, the difference of the boundary conditions causes 
the different behaviors in the large size limit.

In this paper, we introduce and study the generating functions of multi-colored dimer configurations
in one-dimension, that is, on a segment and on a circle.
We consider the following three statistics to a dimer configuration: 
the number $m$ of colors of a dimer, the number of dimers, and the position of a dimer
from the left end.
For each statistics, we assign a formal variable $m$, $r$ and $s$.
Thus, the exponents of $r$ and $s$ enumerate the number of dimers and the sum of 
the positions of dimers.
By introducing the formal variables $(m,r,s)$, we consider the multivariate 
generating functions for dimer configurations. 	

Before stating the main results, we introduce several combinatorial objects 
appearing in the rest of the paper.
The first combinatorial object is Dyck paths, whose total number of size $n$
is given by a combinatorial number, the $n$-th Catalan number.
There are many combinatorial objects whose total number is given by 
Catalan number (see, e.g., \cite{Sta97,Sta972,Sta15} and references therein). 
The generating function for Catalan number can be derived by use of 
algebraic languages \cite{DelVie84}.
Further, connections of Dyck paths and Fibonacci polynomials are 
studied in \cite{BarDLunFezzPin97,BirGilWei14}.
The second combinatorial object is Motzkin paths, whose total number of size $n$ is 
given by the $n$-th Motzkin number and studied in \cite{Aig98,BirGilWei14,DonSha77,OstderJeu15}.
The generating function for Motzkin paths with weights in a strip is studied 
in \cite{Kra01,OstderJeu15,Vie85} with the applications of \cite{Vie83,Vie85} 
(see also the Appendix in \cite{Kra01} for an overview and proofs).
We remark that the theory of continued fraction was studied in view of 
Motzkin paths in \cite{Fla80}.
The third combinatorial object is Schr\"oder paths, whose total number of size $n$ is given 
by the large Schr\"oder numbers \cite{Sch70}. 
The fourth one is Narayana numbers. This number can be interpreted as 
a fine refinement of a Catalan number.
A Narayana number is the total number of Dyck paths of size $n$ 
and $k$ peaks \cite{KreMos86,Sta972}.

One can generalize the generating functions of these combinatorial numbers 
by considering several statistics on them.
In the case of Catalan number, one can consider 
statistics and weights on Dyck paths, which are, for example, 
the number and positions of peaks or valleys, the area below a Dyck path, 
bounce statistics and patterns. 
They are studied in, for example, 
\cite{BacBerFerCunOinWes14,BlaPet14,DenSim95,Deu99,Dra09,Eli20,FloRam20,Kre86,KreMos86,KrePou86,
Man02,ManSun08,MerRogSprVer97,MerSprVer02,SapTasTsi07}.  
Similarly, one can consider the statistics on Motzkin paths such as 
the color of horizontal steps not at the ground level, peaks, 
levels, descents, and area.
They are studied in, for example, \cite{Aig98,DonSha77,Vie85}. 
The relations among Dyck, Mozkin and Schr\"oder numbers are studied in 
\cite{BonShaSim93,ChePan17,PerPin99,Sul98}

The Rogers--Ramanujan identities are first established by Rogers \cite{Rog1894},
and were rediscovered by Ramanujan \cite{Mac1918} and Schur \cite{Sch1917}.
Slater established many identities of single-fold and infinite products in \cite{Sla52}.
On the other hand, Andrews presented multi-sum Rogers--Ramanujan type identites 
in \cite{And84a}.
The reader will find an overview and historical remarks in, for example, \cite{And86,Sil18}.
Rogers--Ramanujan identities play an important role 
in the theory of partitions (see, e.g., \cite{And84}), and reappear in the many context 
such as number theory \cite{And84}, combinatorics \cite{Bres79,Gor61}, 
and the hard-hexagon model in statistical mechanics \cite{And81,AndBaxFor84,Bax82}.
The Rogers--Ramanujan identities naturally appear in the large $n$ limit 
of the $q$-analogue of the recurrence relation for Fibonacci numbers \cite{Cig04}.
In many cases, we have the recurrence relation for Fibonacci numbers 
behind the appearance of Rogers--Ramanujan identities.
As we will see later, we have distinct Rogers--Ramanujan type identities
by studying the large $n$ limit of the generating functions on a segment
and on a circle. 

\subsection{Main results}
The first main result is a description of a convoluted generating 
function for Fibonacci, Catalan, or Motzkin numbers in terms 
of binomial coefficients.
Since these three combinatorial numbers can be interpreted 
as total numbers of combinatorial objects such as 
dimer configurations on a segment, Dyck paths, and Motzkin paths respectively.
The three combinatorial numbers can be interpreted as 
the generating function of the combinatorial objects with 
the weight one.
By assigning a weight to each combinatorial object, 
we show that the convoluted generating function can 
be expressed as a generating function of  combinatorial 
objects with this new weight.
In all three cases, the weights can be expressed by use 
of binomial coefficients.

The second result is the explicit expressions of the two generating functions 
$G_{n}^{(s)}(m,r,s)$ and $G_{n}^{(c)}(m,r,s)$ 
for the dimer models on a segment and on a circle respectively. 
We first show that the generating functions satisfy 
the recurrence relations similar to the one for Fibonacci 
numbers. 
More precisely, the generating function $G_{n}^{(s)}(m,r,s)$ for the dimer model on a 
segment of size $n$ satisfies the recurrence relation (see Theorem \ref{thrm:FibGfin})
\begin{align*}
G_{n}^{(s)}(m,r,s)=f(rs^{n})G_{n-1}^{(s)}(m,r,s)+rs^{n}G_{n-2}^{(s)}(m,r,s).
\end{align*}
Note that if we specialize $f(x)=1$ and $(r,s)=(1,1)$, the recurrence relation 
above is nothing but the recurrence relation for Fibonacci numbers.
We give an explicit formula for $G_{n}^{(s)}$ in terms of 
the polynomial $f(x)$.
This formula is valid for general $f(x)$.
The dimer model on a segment corresponds to the case 
$f(x)=1+x(m-1)$.
Then, by specializing the formal variables $(m,r,s)$
properly, we obtain the closed formula for the generating 
functions $G_{n}^{(s)}(m,r,s)$ with respect to $(m,r,s)$ 
(Theorem \ref{thrm:Gsgeneric} and Corollary \ref{cor:Gsgenetype2}).
The generating function $G_{n}^{(c)}(m,r,s)$ for the multi-colored dimer 
model on a circle satisfies recurrence relations involving the generating 
function $G_{n}^{(s)}(m,r,s)$ (Corollary \ref{cor:Gcrr2}): 
\begin{align*}
G_{n}^{(c)}(m,r,s)&=f(rs^{n})G_{n-1}^{(c)}(m,r,s)+rs^{n}G_{n-2}^{(c)}(m,r,s) \\
&\quad+(s-1) \left(G_{n-1}^{(c)}(m,r,s)-G_{n-1}^{(s)}(m,r,s)\right)+(-1)^{n}r^{n}s^{n(n+1)/2}m(m-1).
\end{align*}
We have two more recurrence relations for $G_{n}^{(c)}(m,r,s)$ 
(Theorem \ref{thrm:Gcrr1} and Proposition \ref{prop:GcinG1G2}). 
As in the case of $G_{n}^{(s)}(m,r,s)$, similar results hold for the generating function 
$G_{n}^{(c)}(m,r,s)$ of the dimer model on a circle for general $(m,r,s)$ 
(Theorem \ref{thrm:Gcmrs1} and Corollary \ref{cor:Gcmrs1}).
We also give derivations of $G_{n}^{(s)}(m,r,s)$ and $G_{n}^{(c)}(m,r,s)$
by the transfer matrix method (Section \ref{sec:tmmforGs} and 
Section \ref{sec:tmmforGc}). 
To describe $G_{n}^{(s)}(m,r,s)$ and $G_{n}^{(c)}(m,r,s)$, 
the transfer matrix method naturally requires the four types of 
generating functions on a segment.
Then, roughly speaking, both $G_{n}^{(s)}(m,r,s)$ and $G_{n}^{(c)}(m,r,s)$ are 
expressed as the sum of two generating functions chosen from four (Proposition \ref{prop:tmmGs} 
and Proposition \ref{prop:tmmGc}).
We also give a determinant expression for $G_{n}^{(s)}(m,r,s)$, which yields 
the symmetry for the change of variables $(r,s)\leftrightarrow(rs^{n+1},s^{-1})$.
Due to this symmetry, we obtain the second recurrence relation for $G_{n}^{(s)}(m,r,s)$ 
(Theorem \ref{thrm:Gsrrflip}):
\begin{align}
\label{eqn:recrelGs2}
G^{(s)}_{n}(m,r,s)=f(rs)G_{n-1}^{(s)}(m,rs,s)+rsG_{n-2}^{(s)}(m,rs^{2},s).
\end{align}
Note that this recurrence relation has a shift of a formal variable.
The recurrence relation with a shift of a variable plays a central 
role in the study of the correlation functions.
Recall that the generating functions $G_{n}^{(s)}(m,r,s)$ and $G_{n}^{(c)}(m,r,s)$ satisfy 
the several recurrence relations similar to the one for 
Fibonacci numbers (Theorem \ref{thrm:FibGfin}, Theorem \ref{thrm:Gcrr1},
Corollary \ref{cor:Gcrr2} and Proposition \ref{prop:GcinG1G2}).
Since the two dimer models, on a segment and on a circle, have similar 
recurrence relations, the leading contribution of $G_{n}^{(c)}(m,r,s)$ is given 
by the generating function $G_{n'}^{(s)}(m,r,s)$ where $n'$ is either $n$ or $n-1$.
The effect of the boundary conditions can be captured by the difference 
between $G_{n}^{(c)}(m,r,s)$ and $G_{n'}^{(s)}(m,r,s)$.

The third result is the study of the large $n$ limit of 
the generating functions of dimers on a segment and on a circle.
As already mentioned above, the Rogers--Ramanujan identities appear
when one considers the large $n$ limit of the $q$-analogue of the 
recurrence relation of Fibonacci type.
Since the generating functions for the dimer models with multi colors 
on a segment and on a circle still have properties which come from 
the recurrence relation of Fibonacci type, 
the appearances of the Rogers--Ramanujan type identities in the large
$n$ limit seem to be natural.
Due to the effect of the boundary conditions 
(the difference between a segment and a circle), 
we have distinct Rogers--Ramanujan type identities in the cases of 
a segment and a circle.
More precisely, we obtain two sums appearing in 
Rogers--Ramanujan identities \cite{Mac1918,Rog1894,Sch1917}:
\begin{align}
\label{eqn:RR14}
\sum_{0\le n}\genfrac{}{}{}{}{s^{n^2}}{\prod_{j=1}^{n}(1-s^{j})}&=\prod_{0\le n}\genfrac{}{}{}{}{1}{(1-s^{5n+1})(1-s^{5n+4})}, \\
\label{eqn:RR23}
\sum_{0\le n}\genfrac{}{}{}{}{s^{n(n+1)}}{\prod_{j=1}^{n}(1-s^{j})}&=\prod_{0\le n}\genfrac{}{}{}{}{1}{(1-s^{5n+2})(1-s^{5n+3})},
\end{align}
which requires the specialization of the formal variables $(m,r)=(1,1)$ in the 
large $n$ limit.
The Rogers--Ramanujan type sum (\ref{eqn:RR14}) appears in the case of 
$G_{n}^{(s)}(m,r,s)$ (Proposition \ref{prop:Gsm1}).
On the other hand, the sum (\ref{eqn:RR23}) appears in the case of 
the generating function 
$G_{n}^{(2)}(m,r,s):=G_{n}^{(c)}(m,r,s)-G_{n-1}^{(s)}(m,r,s)$ (Proposition \ref{prop:G2m1}).
Under this specialization, the generating functions of size $n$ are 
polynomials in $s$.
If we consider the reversed generating functions, which are obtained by replacing 
$s$ by $s^{-1}$, we obtain four distinguished sums in Rogers--Ramanujan type 
identities (see \cite{And84a,And86,Sil18} and references therein):
\begin{align}
\label{eqn:RR1416}
\sum_{0\le n}\genfrac{}{}{}{}{s^{n^2}}{\prod_{j=1}^{2n}(1-s^{j})}
&=\prod_{0\le n}\genfrac{}{}{}{}{1}{(1-s^{2n+1})(1-s^{20n+4})(1-s^{20n+16})}, \\
\label{eqn:RR14162}
\sum_{0\le n}\genfrac{}{}{}{}{s^{n(n+1)}}{\prod_{j=1}^{2n+1}(1-s^{j})}
&=\prod_{0\le n}\prod_{j\neq0,\pm3,\pm4,\pm7,10(\bmod{20})}\genfrac{}{}{}{}{1}{(1-s^{2n+j})}, \\
\label{eqn:RR1812}
\sum_{1\le n}\genfrac{}{}{}{}{s^{n^2-1}}{\prod_{j=1}^{2n-1}(1-s^{j})}
&=\prod_{0\le n}\genfrac{}{}{}{}{1}{(1-s^{2n+1})(1-s^{20n+8})(1-s^{20n+12})}, \\
\label{eqn:RR18122}
\sum_{0\le n}\genfrac{}{}{}{}{s^{n(n+1)}}{\prod_{j=1}^{2n}(1-s^{j})}
&=\prod_{0\le n}\prod_{j\neq0,\pm1,\pm8,\pm9,10(\bmod{20})}\genfrac{}{}{}{}{1}{(1-s^{20n+j})}. 
\end{align} 
The first two summations arise from the generating function $G_{n}^{(s)}(m=1,r=1,1/s)$ 
(see Proposition \ref{prop:Gsinvs}) and the third and fourth summations in the identities come 
from $G_{n}^{(c)}(m=1,r=1,1/s)$ and $G_{n}^{(2)}(m=1,r=1,1/s)$ 
(see Propositions \ref{prop:GcinftyRR} and \ref{prop:G2invs}).
\begin{table}[ht]
\begin{center}
\begin{tabular}{cc|c|c|c|c}
  &  & \multicolumn{2}{c|}{ordinary} & \multicolumn{2}{c}{reversed}   \\ 
  &  & $G_{n}^{(s)}$ & $G_{n}^{(2)}$ & $G_{n}^{(s)}$ & $G_{n}^{(2)}$ \\ \hline 
\multirow{4}{*}{$m=1$} & \multirow{2}{*}{odd} & \multirow{4}{*}{\shortstack{Eq. (\ref{eqn:RR14}) \\ Proposition \ref{prop:Gsm1}}}  
& \multirow{4}{*}{\shortstack{Eq. (\ref{eqn:RR23}) \\ Proposition \ref{prop:G2m1}}}
& \multirow{2}{*}{\shortstack{Eq. (\ref{eqn:RR1416}) \\ Proposition \ref{prop:Gsinvs}}}  
& \multirow{2}{*}{\shortstack{Eq. (\ref{eqn:RR1812}) \\ Proposition \ref{prop:G2invs}}} \\ 
 & & & & &   \\  \cline{2-2} \cline{5-6}
 & \multirow{2}{*}{even} & & & \multirow{2}{*}{\shortstack{Eq. (\ref{eqn:RR14162}) \\ Proposition \ref{prop:Gsinvs}}}  
 &\multirow{2}{*}{\shortstack{Eq. (\ref{eqn:RR18122}) \\ Proposition \ref{prop:G2invs}}} \\
 & & & & & \\ \hline
\multicolumn{2}{c|}{\multirow{2}{*}{$m=2$}}  & \multirow{2}{*}{\shortstack{modulus $6$ \\ Proposition \ref{prop:Gsm2}}}  
& \multirow{2}{*}{\shortstack{modulus $12$ \\ Proposition \ref{prop:G2m2}}} 
& \multirow{2}{*}{\shortstack{modulus $4$ \\ Theorem \ref{thrm:revGnsm2}}} 
& \multirow{2}{*}{\shortstack{modulus $16$ \\ Theorem \ref{thrm:revGnm2}}}\\ 
& & & & &  \\ 	
\end{tabular}
\end{center}
\vspace*{10pt}
\caption{Rogers--Ramanujan type identities and generating functions with specializations.}
\label{table:RR}
\end{table}
Apart from $m=1$, we have simple infinite product expressions for the generating 
functions with $m=2$.
These expressions are Rogers--Ramanujan type identities of modulus $6$ (Proposition \ref{prop:Gsm2})
and of modulus $12$ (Proposition \ref{prop:G2m2}).
Then, as for the reversed generating functions, we obtain Rogers--Ramanujan type identities of modulus $4$ (Theorem \ref{thrm:revGnsm2})
and modulus $16$ (Theorem \ref{thrm:revGnm2}).
We summarize the appearances of Rogers--Ramanujan type identities for the ordinary and reversed 
generating functions with specializations $m=1$ or $2$ in Table \ref{table:RR}.
Besides these generalized Rogers--Ramanujan identities, we give simple combinatorial formulae 
for the generating functions with other specializations of the formal variables.
For example, if we consider the specialization $s=1$, then the ratio of generating functions of 
size $n$ and $n-1$ gives the generating function of Narayana and Catalan numbers with respect 
to $r$ and $m$ in the large $n$ limit (Proposition \ref{prop:Gs11m}).
The appearances of the Narayana and Catalan structure behind the dimer models 
are due to the introduction of the formal variable $m$.

The fourth result is the analysis of the correlation functions of the 
dimers. 
We first introduce the formal power series $\chi(m,r,s)$ which is characterized 
by Dyck paths (Proposition \ref{prop:chiinDyck}).
The power series $\chi(m,r,s)$ satisfies the recurrence relation similar to 
the one for Catalan numbers studied in \cite{CarRio64}.
We remark that the power series $\chi(m,r,s)$ can also be interpreted as a 
generating function of Motzkin paths with an appropriate statistics (Proposition \ref{prop:chiinMotzkin}).
Recall that the generating functions $G_{n}^{(s)}$ and $G_{n}^{(c)}$ have 
expressions in terms of $f(x)$.
Since the recurrence relation for $\chi(m,r,s)$ contains $f(x)$, 
the power series $\chi(m,r,s)$ has also an expression in terms of $f(x)$ as in the 
cases of $G_{n}^{(s)}$ and $G_{n}^{(c)}$.
Then, we study correlation functions, which we call the emptiness formation probabilities (EFP)
of dimers.
The EFP in the large $n$ limit can be expressed by finite products of $\chi(m,r,s)$ (Theorem \ref{thrm:Zinf}).
In the case of finite size, the EFP is expressed in terms of the power series $\chi(m,r,s)$ 
and another power series $B_{n}(m,r,s)$ (Proposition \ref{prop:Zni}).
The power series $B_{n}(m,r,s)$ is also defined in terms of the power series $\chi(m,r,s)$.
The terms involving $B_{n}(m,r,s)$ detect the finiteness of the system.
Then, the expression of the EFP in terms of $\chi(m,r,s)$ also implies the expressions 
of the generating functions for dimers in terms of the power series $\chi(m,r,s)$.
Actually, for general $(m,r,s)$, the generating function $G_{n}^{(s)}(m,r,s)$ in the large $n$ limit can 
be expressed as an infinite product of $\chi(m,r,s)$ (Theorem \ref{thrm:Ginfchi}):
\begin{align*}
G_{\infty}^{(s)}(m,r,s)&=\prod_{0\le n}\genfrac{}{}{}{}{1}{\chi(m,rs^{n},s)}.
\end{align*}
Similarly, the generating function $G_{n}^{(2)}(m,r,s)$ behaves in the large $n$ limit as 
\begin{align*}
\lim_{n\rightarrow\infty}(rs^{n})^{-1}G_{n}^{(2)}(m,r,s)&=\prod_{0\le n}\genfrac{}{}{}{}{1}{\xi(m,rs^{n},s)},
\end{align*}
as shown in Proposition \ref{prop:G2infxi}.
These two expressions can be viewed as generalizations of 
the Rogers--Ramanujan identities (\ref{eqn:RR14}) and (\ref{eqn:RR23}) to general $(m,r,s)$ 
respectively.
The formal power series $\xi(m,r,s)$ is characterized by Schr\"oder paths.
For reversed generating functions for the dimer model on a segment, 
we obtain a similar expression as an infinite product 
(Theorem \ref{thrm:revGsineta}):
\begin{align*}
\lim_{n\rightarrow\infty}r^{n}s^{n(n+1)/2}G_{n}^{(s)}(m,r^{-1},s^{-1})
=\genfrac{}{}{}{}{m}{m-1}\prod_{0\le n}\genfrac{}{}{}{}{1}{\eta(m,rs^{n},s)},
\end{align*}
where $\eta(m,r,s)$ is a formal power series with respect to $r$ and $s$ which 
has poles at $m=1$.
This expression can be viewed as a generalization of 
the Rogers--Ramanujan identities (\ref{eqn:RR1416}) and (\ref{eqn:RR14162}).
Similarly, for reversed generating functions for the dimer model on 
a circle, we have (Theorem \ref{thrm:revGsinnu})
\begin{align*}
\lim_{n\rightarrow\infty}\left(mr^{n}s^{n(n+1)/2}G_{n}^{(2)}(m,r^{-1},s^{-1})
+(-1)^{n-1}(m-1)\right)
=\prod_{0\le n}\genfrac{}{}{}{}{1}{\nu(m,rs^{n},s)},
\end{align*}
where $G_{n}^{(2)}(m,r,s)$ is a generating function appearing in the study of 
$G_{n}^{(c)}(m,r,s)$ and $\nu(m,r,s)$ is a power series with respect to $r$ and $s$ 
which has poles at $m=1$.
The power series $\eta(m,r,s)$ and $\nu(m,r,s)$ are characterized by 
Schr\"oder paths and Dyck paths respectively.
These expressions can be viewed as a generalization of 
the Rogers--Ramanujan identities (\ref{eqn:RR1812}) and (\ref{eqn:RR18122}).
\begin{table}[ht]
\begin{center}
\begin{tabular}{c|c|c} 
 & ordinary & reversed \\ \hline
 $G_{n}^{(s)}(m,r,s)$ & Dyck and Motzkin paths & Schr\"oder paths \\ \hline
 $G_{n}^{(2)}(m,r,s)$ & Schr\"oder paths & Dyck and Motzkin paths
\end{tabular}
\end{center}
\vspace*{10pt}
\caption{Combinatorial objects regarding the generating functions.}
\label{table:combobj}
\end{table}
In Table \ref{table:combobj}, we summarize the combinatorial 
objects behind these infinite product expressions of the generating functions 
in the large $n$ limit.
Note that the roles of Dyck and Motzkin paths, and the ones of Schr\"oder paths are exchanged 
in the case of ordinary and reversed generating functions.
\begin{table}[ht]
\begin{center}
\begin{tabular}{c|c||c|c|c} 
 & Dimers & Dyck paths & Motzkin paths & Schr\"oder paths \\ \hline
GF & $G_{n}^{(\ast)}$ & $\chi, \nu$ & $\chi, \nu$ & $\xi, \eta$ \\ \hline
\multirow{2}{*}{$m$} & \multirow{2}{*}{number of colors} & \multirow{2}{*}{number of peaks} 
& \multirow{2}{*}{\shortstack{number of up \\ and horizontal steps}} & \multirow{2}{*}{\shortstack{number of \\ horizontal steps}}  \\ 
 & & & &  \\ \hline
$r$ & number of dimers & size  & size & size  \\ \hline
$s$ & positions of dimers & area  & area  & sizes of arcs  \\
\end{tabular}
\end{center}
\vspace*{10pt}
\caption{The roles of the formal variables $(m,r,s)$.  
GF stands for generating function and the symbol $\ast$ stands for $s$, $c$ and $2$.}
\label{table:roles}
\end{table}
In Table \ref{table:roles}, we summarize the roles of the formal variables 
$(m,r,s)$ for dimer models and combinatorial lattice paths.

Recall that $\chi(m,r,s)$ is itself a formal power series in the formal 
variables $(m,r,s)$.
We also obtain the finite analogue of the above results regarding the 
generating function $G_{n}^{(s)}(m,r,s)$ (Theorem \ref{thrm:Gsfinchi}).
The generating function $G_{n}^{(s)}(m,r,s)$ with finite $n$ has an expression in terms of $\chi(m,r,s)$
and the polynomial $f(x)$. 
Since the product of $\chi(m,r,s)$ gives the leading contribution to the generating 
function of finite size, the terms involving $\chi(m,r,s)$ and $f(x)$ detect the finiteness of the system size.
Especially, note that the generating function $G_{n}^{(s)}(m,r,s)$ of dimers of size $n$ is a polynomial with respect to 
$(m,r,s)$, and the power series $\chi(m,r,s)$ are ``infinite" power series. 
Therefore, the correction by the terms involving $\chi(m,r,s)$ and $f(x)$ reduces 
the infinite power series in terms of $\chi(m,r,s)$ to a finite polynomials in $(m,r,s)$.
As already mentioned above, since the power series $\chi(m,r,s)$ has an interpretation 
by Dyck paths or Motzkin paths, one can expect that the generating functions 
$G_{n}^{(s)}(m,r,s)$ and $G_{n}^{(c)}(m,r,s)$ with finite 
$n$ have expressions in terms of Dyck paths and Motzkin paths.
In fact, we obtain such expressions by Dyck paths and Motzkin paths with statistics 
(Theorems \ref{thrm:GsninDyck} ,\ref{thrm:GsninMot}, \ref{thrm:GcninDyck} and \ref{thrm:GcninMot}).
As already mentioned above, the appearances of Narayana and Catalan structures in the large $n$ limit 
of the generating function are due to the formal variable $m$. 
On the other hand, the appearances of Dyck and Motzkin paths in the analysis of the generating 
functions are due to similarity of the recurrence relation (\ref{eqn:recrelGs2}) of Fibonacci type  
and the ones of Dyck and Motzkin paths.
We also study another correlation function, which we call moments.
A moment is defined as partial derivatives of generating functions with respect to 
the formal variables $r$ and $s$.
The first moment gives the average number of dimers in dimer configurations.
A general moment also has an expression in terms of two formal power series 
$\chi(m,r,s)$ and $B_{n}(m,r,s)$ (Corollary \ref{cor:moments}). 

Finally, we have two applications which appear in the study of the dimer 
models (Section \ref{sec:app}). 
The first one is a correspondence between a Motzkin path and 
a set of Dyck paths. Since the power series $\chi(m,r,s)$ is 
expressed in terms of the two different objects, Dyck paths and Motzkin paths, 
it is natural to ask if these two objects have a correspondence. 
We construct a weight preserving one-to-many correspondence between a Motzkin path and Dyck paths.
The second application is a connection between 
a dimer configuration and an independent set of a particular graph.
For a dimer configuration on a segment, we assign a graph which is 
the Cartesian product of a segment and a complete graph.
Similarly, for the dimer model on a circle, we assign a graph which is the Cartesian 
product of a circle and a complete graph.
Then, the generating functions of the dimer models with $s=1$ can be interpreted 
as the generating functions of independent sets of the graphs.

\subsection{Plan of the paper}
In Section \ref{sec:comcov}, we introduce five combinatorial objects, 
which are Fibonacci words, Dyck paths, Motzkin paths, Schr\"oder paths and Narayana numbers.
Then, we study the convolutions of the generating functions for Fibonacci,
Catalan, and Motzkin numbers.
Section \ref{sec:dimersseg}, we define dimer configurations on a segment, 
study its generating function and its specializations of formal variables.
Section \ref{sec:dimerscirc} is devoted to the analysis of the generating 
function of dimers on a circle.
We study the large $n$ limit of the generating functions in Section \ref{sec:gfinfty}.
Section \ref{sec:cf} is devoted to the analysis of correlation functions.
In Section \ref{sec:app}, we study the correspondence between a Motzkin path and a 
set of Dyck paths, 
and the relation between a dimer configuration and an independent set of a graph.

\subsection{Notation}
\label{sec:notation}
We first summarize the notations used in this paper.

Given a formal power series $F(x,y,z,\ldots):=\sum_{0\le k}F_{k}(y,z,\ldots)x^{k}$ 
in the formal variables $x,y,z,\ldots$,  
we denote the coefficient of $x^{k}$ by 
\begin{align*}
[x^{k}]F(x,y,z\ldots):=F_{k}(y,z,\ldots).
\end{align*}

We introduce the quantum integer $[n]_{q}:=\sum_{i=0}^{n-1}q^{i}$,
quantum factorial $[n]_{q}!:=\prod_{i=1}^{n}[i]_{q}$, 
and the $q$-analogue of the binomial coefficients:
\begin{align*}
\genfrac{[}{]}{0pt}{}{n}{m}_{q}:=\genfrac{}{}{1pt}{}{[n]_{q}!}{[m]_q![n-m]_{q}!}.
\end{align*}
When $q=1$, we denote a binomial coefficient 
by $\genfrac{(}{)}{0pt}{}{n}{m}:=\genfrac{[}{]}{0pt}{}{n}{m}_{q=1}$.
More generally, we denote the multi-nomial coefficient by 
\begin{align*}
\genfrac{(}{)}{0pt}{}{n}{k_1,k_2,\ldots,k_m}:=\genfrac{}{}{}{}{n!}{k_1!k_2!\cdots k_{m}!}.
\end{align*}

Let $\mathbf{a}:=\{a_1,\ldots,a_{p}\}$ and $\mathbf{b}:=\{b_1,\ldots,b_{q}\}$ 
be two sequences of length $p$ and $q$.
We define a generalized hypergeometric function $_{p}\mathcal{F}_{q}$ as 
\begin{align*}
_{p}\mathcal{F}_{q}\left[\mathbf{a,b},z\right]
:=\sum_{0\le n}^{\infty}
\genfrac{}{}{1pt}{}{(a_1)_{n}(a_2)_{n}\cdots(a_p)_{n}}{(b_1)_{n}(b_2)_{n}\cdots(b_q)_{n}}z^{n},
\end{align*}
where $(x)_{n}:=\prod_{k=0}^{n-1}(x+k)$ with $(x)_{0}:=1$ is the Pochhammer symbol.

We first define $q$-Pochhammer symbol $(a;q)_{k}$ is given by
\begin{align*}
(a;q)_{k}:=\begin{cases}
\prod_{j=0}^{k-1}(1-aq^{j}), & \text{if } k>0, \\
1, & \text{if } k=0, \\
\prod_{j=1}^{k}(1-aq^{-j})^{-1}, & \text{if } k<0.
\end{cases}
\end{align*}

We define Ramanujan's general theta functions $f^{R}(a,b)$ and $f^{R}(a)$ by 
\begin{align*}
f^{R}(a,b)&:=\sum_{n=-\infty}^{\infty}a^{n(n+1)/2}b^{n(n-1)/2}, \\
f^{R}(a)&:=f^{R}(a,-a^{2}),	
\end{align*}
We also define $\psi^{R}(a)$, $\varphi^{R}(a)$ and $\chi^{R}(a)$ by
\begin{align*}
\varphi^{R}(a)&:=f^{R}(a,a)=\sum_{n=-\infty}^{\infty}a^{n^{2}}, \\
\psi^{R}(a)&:=f^{R}(a,a^3)=\sum_{0\le n}a^{n(n+1)/2}, \\
\chi^{R}(a)&:=\genfrac{}{}{}{}{f^{R}(a)}{f^{R}(-a^2)}=(-a;a^{2})_{\infty}.
\end{align*}
Ramanujan's general theta function with a specialization 
$(a,b)=(-q^{\alpha},-q^{\beta})$ can be expressed in 
terms of $q$-Pochhammer symbols:
\begin{align*}
f^{R}(-q^{\alpha},-q^{\beta})
=(q^{\alpha};q^{\alpha+\beta})_{\infty}(q^{\beta};q^{\alpha+\beta})_{\infty}
(q^{\alpha+\beta};q^{\alpha+\beta})_{\infty}.
\end{align*}

We introduce two Rogers--Ramanujan identities $f^{RR}_1(s)$ and 
$f^{RR}_{2}(s)$ by
\begin{align*}
f^{RR}_{1}(s)
&:=1+\sum_{1\le n}\genfrac{}{}{}{}{s^{k^2}}{\prod_{j=1}^{n}(1-s^j)}, \\
&=\genfrac{}{}{}{}{f^{R}(-s^2,-s^3)}{f^{R}(-s)}=\prod_{0\le n}\prod_{j\in\{1,4\}}(1-s^{5n+j})^{-1}, \\
&=1+s+s^2+s^3+2s^4+2s^5+3s^6+3s^7+4s^8+5s^9+6s^{10}+\cdots, \\
f^{RR}_{2}(s)
&:=1+\sum_{1\le n}\genfrac{}{}{}{}{s^{k(k+1)}}{\prod_{j=1}^{n}(1-s^j)}, \\
&=\genfrac{}{}{}{}{f^{R}(-s,-s^4)}{f^{R}(-s)}=\prod_{0\le n}\prod_{j\in\{2,3\}}(1-s^{5n+j})^{-1},  \\
&=1+s^2+s^3+s^4+s^5+2s^6+2s^7+3s^8+3s^9+4s^{10}+\cdots.
\end{align*}
More generally, for $3\le i$, we define 
\begin{align}
\label{eqn:deffRRi}
\begin{split}
f_{i}^{RR}(s)&:=s^{-(i-2)}\left(f_{i-2}^{RR}(s)-f_{i-1}^{RR}(s)\right), \\
&=1+\sum_{1\le n}\genfrac{}{}{}{}{s^{n(n+i-1)}}{\prod_{j=1}^{n}(1-s^{j})}.
\end{split}
\end{align}

\section{Combinatorial numbers and convolutions}
\label{sec:comcov}
In this section, we introduce five combinatorial numbers which appear in the rest of
the paper. They are Fibonacci, Catalan, Motzkin, large Schr\"oder and Narayana numbers.
For the first three combinatorial numbers, we also consider convolutions of the generating 
function. 
With the notions of Fibonacci words, Dyck words or paths, and Motzkin paths, we introduce 
new weights for the words or paths, and show that the generating functions as the sum of their weights 
are equal to the convoluted generating functions.

\subsection{Fibonacci numbers}
Fibonacci numbers $F_{n}$, $n\ge0$, satisfy the well-known recurrence 
relation
\begin{align*}
F_{n}=F_{n-1}+F_{n-2},
\end{align*} 
with the initial conditions $F_{-1}=1$ and $F_{0}=1$ \cite[A000045]{Slo}.
The generating function of Fibonacci numbers is 
\begin{align*}
\mathtt{Fib}(x)&:=\sum_{0\le n}F_{n}x^{n}, \\
&=\genfrac{}{}{}{}{1+x}{1-x-x^2}, \\
&=1+2x+3x^2+5x^3+8x^4+13x^5+21x^6+34x^7+55x^8+\cdots.
\end{align*}
More generally, the generating function of Fibonacci numbers 
with the initial conditions $F_{-1}=a$ and $F_{0}=b$ is given 
by
\begin{align*}
\mathtt{Fib}(x;a,b)&:=\genfrac{}{}{}{}{a+bx}{1-x-x^2}.
\end{align*}
It may be interesting to consider the following convoluted generating function:
\begin{align*}
\mathtt{Fib}^{(r)}(x;a,b)&:=\left(\genfrac{}{}{}{}{a+bx}{1-x-x^2}\right)^{r}, \\
&:=\sum_{0\le k}\mathcal{F}(r,k)x^{k}.
\end{align*} 
The convolution of the generating function is studied in \cite{HogBicJoh77,Liu02,Mor04,Rio68}.

The first few values of $\mathcal{F}(r,k)$ with $(a,b)=(1,1)$ are 
\begin{align*}
\begin{array}{c|ccccccc}
r\backslash k & 0 & 1 & 2 & 3 & 4 & 5 & 6\\ \hline
1 & 1 & 2 & 3 & 5 & 8 & 13 & 21  \\
2 & 1 & 4 & 10 & 22 & 45 & 88 & 167 \\
3 & 1 & 6 & 21 & 59 & 147 & 339 & 741 \\
4 & 1 & 8 & 36 & 124 & 366 & 976 & 2422
\end{array}
\end{align*}

The first few values of $\mathcal{F}(r,k)$ with respect to $r$ with 
the initial conditions $(a,b)=(1,1)$ are given by
\begin{align*}
&\mathcal{F}(r,1)=2r,\qquad\mathcal{F}(r,2)=2r^2+r, \\
&\mathcal{F}(r,3)=\genfrac{}{}{}{}{1}{3}(4r^3+6r^2+5r), \\
&\mathcal{F}(r,4)=\genfrac{}{}{}{}{1}{6}(4r^4+12r^3+23r^2+9r), \\
&\mathcal{F}(r,5)=\genfrac{}{}{}{}{1}{15}(4r^5+20r^4+65r^3+70r^2+36r).
\end{align*}

\begin{prop}
\label{prop:reccalFrk}
The coefficients $\mathcal{F}(r,k)$ satisfy the recurrence relation:
\begin{align}
\label{eqn:Fibr}
\mathcal{F}(r,k)=\mathcal{F}(r,k-1)+\mathcal{F}(r,k-2)+a\mathcal{F}(r-1,k-1)+b\mathcal{F}(r-1,k).
\end{align}
\end{prop}
\begin{proof}
Since we have $1=x+x^{2}+(ax+b)(\mathtt{Fib}^{(1)}(x;a,b))^{-1}$, we obtain Eq. (\ref{eqn:Fibr}).
\end{proof}

We give another expression to compute $\mathcal{F}(r,k)$ in terms of Fibonacci numbers $\mathcal{F}(1,k)$.
Let $\mathbf{k}:=(k_1,\ldots,k_{r})$ be an integer sequence of length $r$ 
such that $k_i\ge0$ and $\sum_{i=1}^{r}k_{i}=k$.
Let $\mathcal{K}(r,k)$ be the set of such sequences $\mathbf{k}$.

Similarly, let $\mathbf{t}:=(t_0,\ldots,t_k)$ be an integer sequence of length $k+1$ such that 
$0\le t_{i}\le r$ and $\sum_{i=0}^{k}i\cdot t_{i}=k$.
Let $\mathcal{T}(r,k)$ be the set of such sequences $\mathbf{t}$.

\begin{prop}
\label{prop:Fibrk}
We have 
\begin{align*}
\mathcal{F}(r,k)&=\sum_{\mathbf{k}\in\mathcal{K}(r,k)}
\prod_{i=1}^{r}\mathcal{F}(1,k_i), \\
&=\sum_{\mathbf{t}\in\mathcal{T}(r,k)}\genfrac{(}{)}{0pt}{}{r}{t_0,t_1,\ldots t_{r}}
\prod_{i=0}^{k}\mathcal{F}(1,i)^{t_{i}}.
\end{align*}
\end{prop}
\begin{proof}
The expression follows from $\mathtt{Fib}^{(r)}(x;a,b)=(\mathtt{Fib}(x;a,b))^{r}$ 
by expanding the right hand side to a sum of products of $F_{n}$.
\end{proof}

To compute $\mathcal{F}(r,k)$ by use of Proposition \ref{prop:Fibrk} requires 
the values $\mathcal{F}(1,i)$.
Below, we give two other descriptions of $\mathcal{F}(r,k)$.
The basic idea is that we introduce a combinatorial set, whose total number of 
elements is given by the $k$-th Fibonacci number.
Then, by assigning a weight to an element of the set, we show that the generating 
function of the set is equal to the value $\mathcal{F}(r,k)$.
We give two different such descriptions: Proposition \ref{prop:Fibrk1} 
and Theorem \ref{thrm:FibConv2}.
Especially, in the second description, the weights are given in terms of binomial
coefficients.

First, we introduce the following set as in \cite[A000045]{Slo}:
Let $S^{\mathrm{Fib}}(k+2)$ be the set of subsets $S$ in $[1,k+2]$ such that 
\begin{enumerate}
\item each integer in $[1,k+2]$ appears in $S$ at most once,
\item the number of elements in $S$ is equal to the minimum element in $S$.
\end{enumerate}
For example, when $k=4$, we have eight elements in $S^{\mathrm{Fib}}(6)$:
\begin{align*}
S^{\mathrm{Fib}}(6)=\{\{1\}, \{2,3\},\{2,4\},\{2,5\}\{2,6\}, 
\{3,4,5\}, \{3,4,6\}, \{3,5,6\} \}.
\end{align*}  
Note that the number of elements in $S^{\mathrm{Fib}}(6)$ is given by
the fourth Fibonacci number $F_{4}=8$.
We have the following bijection between $S^{\mathrm{Fib}}(k+2)$ and 
$S^{\mathrm{Fib}}(k+1)\cup S^{\mathrm{Fib}}(k)$ to show $|S^{\mathrm{Fib}}(k+2)|=F_{k}$.
It is obvious that  $S^{\mathrm{Fib}}(k+1)\subset S^{\mathrm{Fib}}(k+2)$.
Therefore, we need to construct a bijection between $S^{\mathrm{Fib}}(k)$ and 
$S^{\mathrm{Fib}}(k+2)\setminus S^{\mathrm{Fib}}(k+1)$.
Given $S\in S^{\mathrm{Fib}}(k)$, we increase all elements by one and append 
$k+2$ to the newly obtained set. Then, it is easy to see that we have the bijection.

We define a polynomial $a_{k}(r)$ of $r$ satisfying the following 
two properties:
\begin{enumerate}
\item the polynomials $\{a_{k}(r): 0\le k\}$ satisfy 
\begin{align}
\label{eqn:rrakr}
a_{k}(r)-a_{k}(r-1)=a_{k-1}(r)+a_{k-1}(r-1)+a_{k-2}(r).
\end{align}
\item $a_{-2}(r):=0$, $a_{-1}(r):=0$ and $a_{k}(1)=1$ for $0\le k$. 
\end{enumerate}
First few values of $a_{k}(r)$ are given by 
\begin{align*}
&a_{0}(r)=1, \quad a_{1}(r)=2r-1, \\
&a_{2}(r)=2r^2-r, \\
&a_{3}(r)=\genfrac{}{}{}{}{1}{3}(4r^3+2r-3).
\end{align*}

\begin{prop}
\label{prop:Fibrk1}
We have 
\begin{align}
\label{eqn:calFrk}
\mathcal{F}(r,k)=\sum_{S\in S^{\mathrm{Fib}}(k+2)}a_{k+2-\max(S)}(r),
\end{align}
where $\max(S)$ is the maximum elements in $S$ for $S\neq\{1\}$ and 
$\max(S):=2$ for $S=\{1\}$.
\end{prop} 

For example, we have five elements in $S^{\mathrm{Fib}}(5)$:
\begin{align*}
S^{\mathrm{Fib}}(5)=\{\{1\}, \{2,3\}, \{2,4\}, \{2,5\}, \{3,4,5\}\},
\end{align*}
and the sum of the weights is 
\begin{align*}
a_{3}(r)+a_2(r)+a_{1}(r)+2a_{0}(r)&=\genfrac{}{}{}{}{r}{3}(4r^{2}+6r+5), \\
&=\mathcal{F}(r,3).
\end{align*}

\begin{proof}[Proof of Proposition \ref{prop:Fibrk1}]
Since the polynomials $\{a_{k}(r):0\le k\}$ satisfy 
the recurrence relation (\ref{eqn:rrakr}) and the condition that $a_{k}(1)=1$, 
it is obvious that the right hand side of (\ref{eqn:calFrk}) satisfies 
the recurrence relation (\ref{eqn:Fibr}) in Proposition \ref{prop:reccalFrk}.
Further, it satisfies the initial conditions $\mathcal{F}(r,0)=1$ and 
$\mathcal{F}(r,1)=2r$.
Thus, we have Eq. (\ref{eqn:calFrk}). 
\end{proof}

We introduce the notion of Fibonacci words to give another description of 
$\mathcal{F}(r,k)$ in terms of binomial coefficients.
Later in Section \ref{sec:dimersseg}, we will introduce the notion of dimers on a 
segment, which is equivalent to the Fibonacci words.
We consider a binary word $w\in\{0,1\}^{\ast}$ consisting of two letters $0$ and $1$.
The length of $w$ is defined as the number of letters in $w$.
A {\it Fibonacci word} $w$ of size $n$ is a binary word of size $n$ 
such that there is no successive $1$'s in $w$. 
We denote by $\mathrm{Fib}(n)$ the set of Fibonacci words of size $n$.
For example, we have 
\begin{align*}
&\mathrm{Fib}(0)=\emptyset, \qquad \mathrm{Fib}(1)=\{0,1\}, \\
&\mathrm{Fib}(2)=\{00, 01, 10\}, \\
&\mathrm{Fib}(3)=\{000, 001, 010, 100, 101\}.
\end{align*}
The cardinality of $\mathrm{Fib}(n)$ is equal to the $n$-th Fibonacci number $F_{n}$.

Similarly, given a Fibonacci word $w:=w_1\ldots w_{k}$ of size $k$ with $w_{1}=0$,
we define $a(w)$ as the number of successive $0$'s starting from $w_1$.
For example, we have $a(00010010)=3$. 

Given a Fibonacci word $w:=w_1w_2\ldots w_{k}$ of size $k$ with $(w_1,w_2)=(1,0)$,
we define the number $b(w)$ as follows.
We replace the pattern $10$ in $w$ by $\ast$, and replace $w_k$ by $\ast$ if  
$w_{k}=1$.
Then, we have a word consisting of $0$'s and $\ast$'s.
The number $b(w)$ is defined as the number of blocks consisting of successive $0$'s.
For example, we have $b(1000100100)=3$ since 
\begin{align*}
1000100100\rightarrow *00*0*0,
\end{align*}
and we have three blocks consisting of $0$'s.

We define the weight for a Fibonacci word $w:=w_1w_2\ldots w_{k}$ recursively as follows.
\begin{enumerate}
\item If $(w_1,w_2)=(1,0)$, we define the weight $\mathrm{wt}^{(r)}(w)$ in terms of binomial coefficients as 
\begin{align*}
\mathrm{wt}^{(r)}(w):=\genfrac{(}{)}{0pt}{}{r+b(w)-1}{b(w)}.
\end{align*}
\item Suppose $w_1=0$. We denote by $w'$ the Fibonacci word obtained from $w$ 
by deleting first $a(w)$ zeros.
Then, the weight $\mathrm{wt}^{(r)}(w)$ of $w$ is given by 
\begin{align*}
\mathrm{wt}^{(r)}(w)=\genfrac{(}{)}{0pt}{}{2r+a(w)-2}{a(w)}\mathrm{wt}^{(r)}(w'),
\end{align*}
where the word $w'$ starts from $1$.
\end{enumerate}
Note that the weight given to a Fibonacci word for $r=1$ is one.

For example, we have 
\begin{align*}
\mathrm{wt}^{(r)}(000100100101)&=\genfrac{(}{)}{0pt}{}{2r+1}{3}\mathrm{wt}^{(r)}(100100101), \\
&=\genfrac{(}{)}{0pt}{}{2r+1}{3}\genfrac{(}{)}{0pt}{}{r+1}{2}.
\end{align*}

The values $\mathcal{F}(r,k)$ are expressed in terms of the weights given above.
\begin{theorem}
\label{thrm:FibConv2}
Let $(a,b)=(1,1)$. 
The coefficient $\mathcal{F}(r,k)$ is given by
\begin{align}
\label{eqn:FibConv2}
\mathcal{F}(r,k)=\sum_{w\in\mathrm{Fib}(k)}\mathrm{wt}^{(r)}(w).
\end{align}
\end{theorem}

We introduce several definitions and lemmas for the proof of Theorem \ref{thrm:FibConv2}.
A Fibonacci word $w$ is expressed as a concatenation of two Fibonacci words $0^{j}$ 
and $w'$ where the first letter of $w'$ is $1$. Namely, $w=0^{j}w'$.
We define $w_1:=0^{j-1}w'$ and $w_2:=0^{j-2}w'$.
We define the weight $\mathcal{G}^{*}(w)$ for a Fibonacci word $w$ as 
\begin{align*}
\mathcal{G}^{*}(w):=\mathrm{wt}^{(r)}(w)-\mathrm{wt}^{(r-1)}(w)-\mathrm{wt}^{(r)}(w_1)-\mathrm{wt}^{(r-1)}(w_1)
-\mathrm{wt}^{(r)}(w_2),
\end{align*}
where we impose the conditions $\mathrm{wt}^{(r)}(0^{j}w_1)=0$ for $j\le-1$.

We denote by $v\circ w$ the concatenation of two Fibonacci words $v$ and $w$.
\begin{defn}
Given a Fibonacci word $w=0^{j}w'$, we define a generating function $\mathcal{G}^{\uparrow}(w)$ as
\begin{align*}
\mathcal{G}^{\uparrow}(w):=\sum_{v\in\mathrm{Fib}(j)}\mathcal{G}^{*}(v\circ w').
\end{align*}
\end{defn}

\begin{lemma}
\label{lemma:Gup1}
Suppose that $w$ is written as $w=0^{j}(100)^{k}$.
Then, the generating function is give by
\begin{align*}
\mathcal{G}^{\uparrow}(w)=\genfrac{(}{)}{0pt}{}{2r+j-3}{j}\genfrac{(}{)}{0pt}{}{r+k-2}{k-1}.
\end{align*}
\end{lemma}
\begin{proof}
We prove the lemma by induction the size of $w$.
By definition of $\mathcal{G}^{\uparrow}(w)$, we have 
\begin{align*}
\mathcal{G}^{\uparrow}(w)&=\mathcal{G}^{*}(w)
+\sum_{i=0}^{j-3}\sum_{v\in\mathrm{Fib}(i)}\mathcal{G}^{*}(v\circ 010^{j-i-2}(100)^{k})
+\mathcal{G}^{*}(10^{j-1}(100)^{k}), \\
&=\mathcal{G}^{*}(w)
+\sum_{i=0}^{j-2}\mathcal{G}^{\uparrow}(0^{i}10^{j-i-1}(100)^{k}).
\end{align*}
From the definition of the weight $\mathrm{wt}^{(r)}(w)$, we have 
\begin{align}
\label{eqn:calG1}
\begin{split}
\mathcal{G}^{*}(w)&=\genfrac{(}{)}{0pt}{}{2r+j-2}{j}\genfrac{(}{)}{0pt}{}{r+k-1}{k}
-\genfrac{(}{)}{0pt}{}{2r+j-4}{j}\genfrac{(}{)}{0pt}{}{r+k-2}{k} \\
&\quad-\genfrac{(}{)}{0pt}{}{2r+j-3}{j-1}\genfrac{(}{)}{0pt}{}{r+k-1}{k}
-\genfrac{(}{)}{0pt}{}{2r+j-5}{j-1}\genfrac{(}{)}{0pt}{}{r+k-2}{k} \\
&\quad-\genfrac{(}{)}{0pt}{}{2r+j-4}{j-2}\genfrac{(}{)}{0pt}{}{r+k-1}{k},
\end{split}
\end{align}
and 
\begin{align}
\label{eqn:calG2}
\mathcal{G}^{\uparrow}(0^{i}10^{j-i-1}(100)^{k})
=
\begin{cases}
\mathcal{G}^{\uparrow}(0^{j-2}(100)^{k}), & i=j-2,   \\
\mathcal{G}^{\uparrow}(0^{i}10^{2}(100)^{k}), & i\le j-3.
\end{cases}
\end{align}
since we have $\mathrm{wt}^{(r)}(v\circ 1010\circ w)=\mathrm{wt}^{(r)}(v\circ 10\circ w)$ and 
$\mathrm{wt}^{(r)}(v\circ 1000^{p}10\circ w)=\mathrm{wt}^{(r)}(v\circ 10010\circ w)$ for 
a non-negative integer $p$.
From expressions (\ref{eqn:calG1}) and (\ref{eqn:calG2}), we have 
\begin{align*}
\mathcal{G}^{\uparrow}(w)&=\mathcal{G}^{*}(w)+\sum_{i=0}^{j-3}\genfrac{(}{)}{0pt}{}{2r+i-3}{i}\genfrac{(}{)}{0pt}{}{r+k-1}{k}
+\genfrac{(}{)}{0pt}{}{2r+j-5}{j-2}\genfrac{(}{)}{0pt}{}{r+k-2}{k-1}, \\
&=\genfrac{(}{)}{0pt}{}{2r+j-3}{j}\genfrac{(}{)}{0pt}{}{r+k-2}{k-1},
\end{align*}
where we have used the summation formula
\begin{align*}
\sum_{i=0}^{j-3}\genfrac{(}{)}{0pt}{}{2r+i-3}{i}=\genfrac{(}{)}{0pt}{}{2r+j-5}{j-3}.
\end{align*}
This completes the proof.
\end{proof}

Let $w'$ be the Fibonacci word of the form $w'=(10)^{p}$, $w'=(10)^{p}1$ or $w'=\emptyset$ with 
some non-negative integer $p$.
An excited Fibonacci word $w:=\mathrm{ex}(w')$ of size $k$ is a Fibonacci word
such that $w$ is size $k$ and can be written as a concatenation of Fibonacci words $w=v\circ 00w'$ 
with $v\in\mathrm{Fib}(k-|w'|-2)$.
\begin{defn}
Let $w'$ be the Fibonacci word as above.
We define $\mathrm{Ex}_{k}(w')$ as the set of excited Fibonacci words $\mathrm{ex}(w')$ of size $k$.
\end{defn}

\begin{lemma}
\label{lemma:Fibex}
Fix a Fibonacci word $w$ of size $k$.
Then, there exists a unique $w'$ such that $w\in\mathrm{Ex}_{k}(w')$.
\end{lemma}
\begin{proof}
Given a Fibonacci word $w$, let $i$ be the maximum integer such that $(w_{i-2},w_{i-1})=(0,0)$ if 
such $i$ exists and $k$ otherwise.
Then, the subword $w'=w_{i}w_{i+1}\ldots w_{k}$ is of the form $(10)^{p}$, $(10)^{p}1$, or $\emptyset$ for 
some positive integer $p$.
This implies that $w$ is an element in $\mathrm{Ex}_{k}(w')$.
By construction, we have a unique $w'$, which completes the proof.
\end{proof}

\begin{proof}[Proof of Theorem \ref{thrm:FibConv2}]
From the definition of the weight $\mathcal{G}^{*}(w)$, 
Eq. (\ref{eqn:FibConv2}) is equivalent to 
\begin{align}
\label{eqn:FibcalGcond}
\sum_{w\in\mathrm{Fib}(k)}\mathcal{G}^{*}(w)=0.
\end{align}
Let $\mathrm{Fib}^{*}(k)$ is the set of Fibonacci words $w'$ 
satisfying the following three properties: 1) the size of $w'$ is less than or equal to $k$,
2) the first letter of $w'$ is $1$, and 3) there exists no consecutive zeroes in $w'$.

From Lemma \ref{lemma:Fibex}, any Fibonacci word $w$ of size $k$ is uniquely written 
as $w\in\mathrm{Ex}_{k}(w')$ with some $w'\in\mathrm{Fib}^{*}(k)$.
Thus, Eq. (\ref{eqn:FibcalGcond}) is rewritten as 
\begin{align*}
\sum_{w'\in\mathrm{Fib}^{*}(k)}\sum_{w\in\mathrm{Ex}_{k}(w')}\mathcal{G}^{\ast}(w)=0.
\end{align*}
Below, we will show that 
\begin{align}
\label{eqn:calGcond2}
\sum_{w\in\mathrm{Ex}_{k}(w')}\mathcal{G}^{\ast}(w)=0.
\end{align}
Let $k'$ be the size of $w'\in\mathrm{Fib}^{\ast}(k)$.
From the definitions of $\mathrm{Ex}_{k}(w')$ and $\mathcal{G}^{\uparrow}(w')$, 
we have
\begin{align}
\label{eqn:FibGsum}
\sum_{w\in\mathrm{Ex}_{k}(w')}\mathcal{G}^{\ast}(w)
=\mathcal{G}^{\ast}(0^{k-k'}w')
+\sum_{j=0}^{k-k'-3}\mathcal{G}^{\uparrow}(0^{j}10^{k-k'-j-1}w').
\end{align}
Since there exists no consecutive zeroes in $w'$, we have, by a simple calculation,  
\begin{align}
\label{eqn:FibGast}
\mathcal{G}^{\ast}(0^{k-k'}w')=-\genfrac{(}{)}{0pt}{}{2r+k-k'-5}{k-k'-3}.
\end{align}
From Lemma \ref{lemma:Gup1}, we have 
\begin{align}
\label{eqn:FibGup}
\begin{split}
\sum_{j=0}^{k-k'-3}\mathcal{G}^{\uparrow}(0^{j}10^{k-k'-j-1}w')
&=\sum_{j=0}^{k-k'-3}\genfrac{(}{)}{0pt}{}{2r+j-3}{j}\genfrac{(}{)}{0pt}{}{r-1}{0}, \\
&=\genfrac{(}{)}{0pt}{}{2r+k-k'-5}{k-k'-3}.
\end{split}
\end{align}
Thus, by substituting Eqs.(\ref{eqn:FibGast}) and (\ref{eqn:FibGup}) into 
Eq. (\ref{eqn:FibGsum}), we have Eq. (\ref{eqn:calGcond2}), which completes the proof.
\end{proof}

\subsection{Catalan numbers}
\label{sec:Catalan}
Catalan numbers often appear as the total number of combinatorial 
objects such as Dyck paths, binary trees, rooted plane trees and polygon triangulations (see, {\it e.g.},\cite{Sta15}).

In this subsection, we introduce the generating function of Catalan 
numbers, and study its convolutions.

The generating function for Catalan numbers $C_{k}$, $0\le k$, \cite[A000108]{Slo} is given by 
\begin{align*}
\mathtt{Cat}(x)&:=\sum_{0\le n}C_{n}x^{n}=\sum_{0\le n}\genfrac{}{}{}{}{1}{n+1}\genfrac{(}{)}{0pt}{}{2n}{n}x^{n}, \\
&=\genfrac{}{}{}{}{1-\sqrt{1-4x}}{2x}, \\
&=1+x+2x^2+5x^3+14x^4+42x^5+132x^6+429x^7+1430 x^8+\cdots.
\end{align*}
Catalan numbers $C_{n}$ satisfy the recurrence relation
\begin{align*}
C_{n+1}=\sum_{i=0}^{n}C_{i}C_{n-i},
\end{align*}
with the initial conditions $C_{0}=1$.

We consider the convolution of Catalan numbers. 
Namely, we define 
\begin{align*}
\sum_{0\le n}\mathtt{Cat}(r,n)x^{n}&:=\genfrac{(}{)}{}{}{1-\sqrt{1-4x}}{2x}^{r}, \\
&=\sum_{0\le n}\genfrac{}{}{}{}{r}{n+r}\genfrac{(}{)}{0pt}{}{2n+r-1}{n}x^{n}.
\end{align*}

As in the case of Fibonacci numbers, the convolutions of Catalan numbers
also satisfy a recurrence relation.
\begin{prop}
\label{prop:ConvCatlan1}
The coefficients $\mathtt{Cat}(r,n)$ satisfy the recurrence relation
\begin{align}
\label{eqn:ConvCatlan1}
\mathtt{Cat}(r,n)=\mathtt{Cat}(r-1,n)+\mathtt{Cat}(r+1,n-1).
\end{align}
\end{prop}
\begin{proof}
Since we have $1=\mathtt{Cat}(x)^{-1}+x\mathtt{Cat}(x)$, we obtain Eq. (\ref{eqn:ConvCatlan1}).
\end{proof}

First few values of $\mathtt{Cat}(r,n)$ are 
\begin{align*}
\begin{array}{c|cccccccc}
r\backslash n & 0 & 1 & 2 & 3 & 4 & 5 & 6 & 7 \\ \hline
1 & 1 & 1 & 2 & 5 & 14 & 42 & 132 & 429  \\
2 & 1 & 2 & 5 & 14 & 42 & 132 & 429 & 1430 \\
3 & 1 & 3 & 9 & 28 & 90 & 297 & 1001 & 3432 \\
4 & 1 & 4 & 14 & 48 & 165 & 572 & 2002 & 7072
\end{array}
\end{align*}

As in the case of Fibonacci numbers, we give an expression of $\mathtt{Cat}(r,k)$ 
in terms of binomial coefficients.
For this purpose, we introduce the notion of Dyck words, whose total number 
of size $n$ is given by the $n$-th Catalan number $\mathtt{Cat}(1,n)$.

A {\it Dyck word} $w_{D}:=(w_1,\ldots,w_{2n})\in\{0,1\}^{\ast}$ of size $n$ is a binary 
word such that subsequences $w_{D}(i):=(w_1,\ldots,w_{i})$  for $1\le i\le 2n$ 
satisfy 
\begin{align*}
|w_{D}(i)|_{0}\ge |w_{D}(i)|_{1},
\end{align*}
for all $1\le i\le n$.
Here $|w_{D}(i)|_{x}$ with $x\in\{0,1\}$ is the number of $x$ in the word $w_{D}(i)$.
For example, we have five Dyck words: $000111$, $001011$, $010011$, $001101$, and $010101$.

Given a Dyck word $w_{D}$, let $p(w_{D})$ be the smallest integer such that 
$|w_{D}(p(w_{D}))|_{0}=|w_{D}(p(w_{D})|_{1}$.
Since the first letter in any Dyck path is $0$ and the total numbers of $0$'s
and $1$'s are equal, we have at least one $p$ satisfying $|w_{D}(p)|_{0}=|w_{D}(p)|_{1}$. 
We define $\mathtt{Dyck}(n)$ as the set of Dyck words of size $n$.

We give a weight $\mathrm{wt}(w)$ to a Dyck word $w$ in terms of binomial coefficients.
\begin{defn}
\label{defn:ConvDyckwt}
Let $w$ be a Dyck word of size $n$, and $p:=p(w)$ be the integer defined as above.
We define a Dyck word $w'$ of size $n-1$ by deleting the first and $p$-th entries 
from $w$.
Then, the weight of $w$ is recursively defined as 
\begin{align*}
\mathrm{wt}^{(r)}(w):=
\begin{cases}
\displaystyle\genfrac{(}{)}{0pt}{}{r+n-1}{n}, & \text{for } w=(01)^{n}, \\[12pt]
\mathrm{wt}^{(r)}(w'), & \text{otherwise}.
\end{cases}
\end{align*}
\end{defn}	

For example, when $n=3$, we have five Dyck paths and their weights are given by 
\begin{align*}
&\mathrm{wt}^{(r)}(000111)=r, \qquad \mathrm{wt}^{(r)}(001011)=\genfrac{(}{)}{0pt}{}{r+1}{2}, \\[11pt]
&\mathrm{wt}^{(r)}(010011)=r, \qquad \mathrm{wt}^{(r)}(001101)=\genfrac{(}{)}{0pt}{}{r+1}{2}, 
\qquad \mathrm{wt}^{(r)}(010101)=\genfrac{(}{)}{0pt}{}{r+2}{3}.
\end{align*}

The convoluted generating function $\mathtt{Cat}(r,n)$ is expressed 
in terms of the weights defined in Definition \ref{defn:ConvDyckwt}.
\begin{theorem}
\label{thrm:ConvDyck}
We have
\begin{align}
\label{eqn:ConvDyckgf}
\mathtt{Cat}(r,n)
=\sum_{w\in\mathtt{Dyck}(n)}\mathrm{wt}^{(r)}(w).
\end{align}
\end{theorem}

Before proceeding to the proof of Theorem \ref{thrm:ConvDyck}, 
we introduce several definitions.

By Definition \ref{defn:ConvDyckwt}, any Dyck word has a weight of the 
form $\genfrac{(}{)}{0pt}{}{r+n-1}{n}$ for some non-negative integer $n$.
Since $n$ is unique if we fix $w$, we denote this $n$ by $\mathfrak{n}(w)$ and 
call it the structure number of $w$.

Since a Dyck word $w$ consists of zeroes and ones, it is written as a concatenation 
of $0^{p_{1}}1^{q_{1}}0^{p_{2}}\ldots 1^{q_{l}}$ with some positive integer $l$.
We call $1^{q_{i}}$ a $1$-block in $w$. Note that the number of $1$-blocks in $w$ 
is less than or equal to the structure number $\mathfrak{n}(w)$.

Given a Dyck word $w$ of size $k$, we will define the set $\mathrm{Ex}(w)$ of Dyck paths of 
size $k+1$, which we call the set of excited Dyck words of $w$.
We define $w_{0}:=w\circ(01)$ as a concatenation of $w$ and the Dyck word $01$.
Then, we define $\mathfrak{n}(w)$ Dyck words $w_{i}$, $1\le i\le \mathfrak{n}(w)$, 
of size $k+1$ by the following procedure.
\begin{enumerate}
\item Define $w'_{i}:=w\circ(1)$ as a concatenation of $w$ and $(1)$.
\item $w_{i}$ is obtained from $w'_{i}$ by inserting $0$ left to the 
$i$-th $1$-block from right in $w'_{i}$.
\end{enumerate}

\begin{defn}
Let $w_{i}$, $0\le i\le \mathfrak{n}(w)$, be the Dyck words of size $k+1$ constructed from 
$w$ as above.
We define the set of Dyck paths 
\begin{align*}
\mathrm{Ex}(w):=\{w_{i}| 0\le i\le \mathfrak{n}(w)\}.
\end{align*}
We call the set $\mathrm{Ex}(w)$ the set of excited Dyck words.
\end{defn}
For example, we have 
\begin{align*}
\mathrm{Ex}(001011)&=\{00101101,00100111,00010111\}, \\
\mathrm{Ex}(010101)&=\{01010101,01010011,01001011,00101011\}.
\end{align*}

The next lemma is obvious from the definition of excited Dyck words.
\begin{lemma}
Given a Dyck word $w'$ of size $k+1$, the exists a unique Dyck word $w$ of size $k$
such that $w'\in\mathrm{Ex}(w)$.
\end{lemma}

\begin{lemma}
\label{lemma:strnumEx}
Fix a Dyck word $w\in\mathtt{Dyck}(k)$.
The structure numbers $\mathfrak{n}(w')$ for Dyck words $w'\in\mathrm{Ex}(w)$ are 
all distinct and in $[1,\mathfrak{n}(w)+1]$.
\end{lemma}
\begin{proof}
We first prove the statement for $w=(01)^{k}$.
Then, $w'\in\mathrm{Ex}(w)$ has the form 
\begin{align*}
w'=(01)^{i}0(01)^{k-i}1,
\end{align*}
where $0\le i\le k$.
Then, it is obvious that $\mathfrak{n}(w')=k-i$ for $0\le i\le k-1$, and 
$\mathfrak{n}(w')=k+1$ for $i=k$.
Thus, statement hods for $w=(01)^{k}$.

For general $w$, by using Definition \ref{defn:ConvDyckwt}, 
we have $\mathrm{wt}^{(r)}=\mathrm{wt}^{(r)}(w')$ for some $w'$ when $w\neq(01)^{k}$.
By repeating this procedure, we arrive at $w'=(01)^{k'}$ with some $k'\ge 1$.
Then, this case can be reduced to the case of $w'=(01)^{k'}$.
Note that we construct $\mathfrak{n}(w)+1$ excited Dyck words from $w$, and 
it is obvious that the structure numbers for excited Dyck words are all distinct 
and in $[1,\mathfrak{n}(w)+1]$.
\end{proof}

\begin{defn}
We define a weight $\mathcal{G}^{\ast}(w)$ for a Dyck word $w$ 
as 
\begin{align}
\label{eqn:calGDyckast}
\begin{split}
\mathcal{G}^{\ast}(w)&:=\mathrm{wt}^{(r)}(w)-\mathrm{wt}^{(r-1)}(w), \\
&=\genfrac{(}{)}{0pt}{}{r+\mathfrak{n}(w)-2}{\mathfrak{n}(w)-1}.
\end{split}
\end{align}
Then, we define 
\begin{align*}
\mathcal{G}^{\uparrow}(w)
:=\sum_{w'\in\mathrm{Ex}(w)}\mathcal{G}^{*}(w').
\end{align*}
\end{defn}

\begin{proof}[Proof of Theorem \ref{thrm:ConvDyck}]
Let $w$ be a Dyck word of size $k$ and $\mathfrak{n}(w)$ 
be the structure number.
From above observations and the recurrence relation (\ref{eqn:ConvCatlan1})
in Proposition \ref{prop:ConvCatlan1}, Eq. (\ref{eqn:ConvDyckgf}) 
is equivalent to 
\begin{align}
\label{eqn:ConvCatwtcalG}
\mathrm{wt}^{(r+1)}(w)=\mathcal{G}^{\uparrow}(w).
\end{align}
Note that the left hand side involves the Dyck path $w$ of size $k$ and 
the right hand side involves  the set of Dyck paths $w'\in\mathrm{Ex}(w)$ of 
size $k+1$.

We rewrite the weight $\mathrm{wt}^{(r+1)}(w)$ as
\begin{align}
\label{eqn:ConvCatwtr1}
\begin{split}
\mathrm{wt}^{(r+1)}(w)&=\genfrac{(}{)}{0pt}{}{r+\mathfrak{n}(w)}{\mathfrak{n}(w)}, \\
&=\sum_{j=0}^{\mathfrak{n}(w)}\genfrac{(}{)}{0pt}{}{r+j-1}{j}.
\end{split}
\end{align}
By comparing the expression (\ref{eqn:calGDyckast}) with Eq. (\ref{eqn:ConvCatwtr1}),
Eq. (\ref{eqn:ConvCatwtcalG}) is satisfied 
if the structure numbers $\mathfrak{n}(w')$ for $w'\in\mathrm{Ex}(w)$  are all 
distinct and in $[1,\mathfrak{n}(w)+1]$.
This condition is satisfied by Lemma \ref{lemma:strnumEx}, and 
we have Eq. (\ref{eqn:ConvCatwtcalG}), which completes the proof.
\end{proof}

For example, when $n=3$, the sum of the weights is given by 
\begin{align*}
\mathtt{Cat}(r,3)&=2r+2\genfrac{(}{)}{0pt}{}{r+1}{2}+\genfrac{(}{)}{0pt}{}{r+2}{3},\\
&=\genfrac{}{}{}{}{1}{6}r(r+4)(r+5),
\end{align*}
which are the entries in the fourth column of the table above.

For later purpose, we introduce the notion of Dyck paths.
A {\it Dyck path} of size $n$ is a lattice path consisting 
of the up steps $(1,1)$ and the down steps $(1,-1)$ which starts 
from $(0,0)$ to $(2n,0)$ such that it never goes below the horizontal 
line $y=0$.
In other words, given a Dyck word of size $n$, 
we simply replace $0$ by the up step $(1,1)$ and 
$1$ by the down step $(1,-1)$.
The highest path corresponds to the Dyck word $0^{n}1^{n}$, 
and the lowest path corresponds to $(01)^{n}$.
We denote by $l(\lambda):=n$ the size of a Dyck path $\lambda$.

For example, the set of Dyck paths of size $3$ is depicted 
in Figure \ref{fig:Dyckpath}.
\begin{figure}[ht]
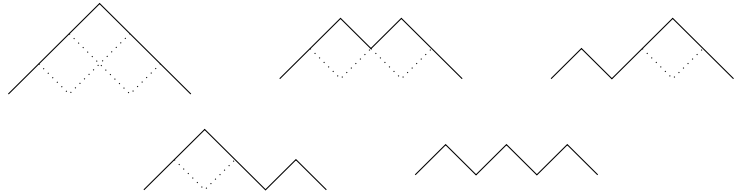

\tikzpic{-0.5}{[scale=0.4]
\draw(0,0)--(3,3)--(6,0);
\draw[dotted](1,1)--(2,0)--(4,2)(2,2)--(4,0)--(5,1);
}\qquad
\tikzpic{-0.5}{[scale=0.4]
\draw(0,0)--(2,2)--(3,1)--(4,2)--(6,0);
\draw[dotted](1,1)--(2,0)--(3,1)(3,1)--(4,0)--(5,1);
}\qquad
\tikzpic{-0.5}{[scale=0.4]
\draw(0,0)--(1,1)--(2,0)--(4,2)--(6,0);
\draw[dotted](3,1)--(4,0)--(5,1);
} \\[12pt]
\tikzpic{-0.5}{[scale=0.4]
\draw(0,0)--(2,2)--(4,0)--(5,1)--(6,0);
\draw[dotted](1,1)--(2,0)--(3,1);
}\qquad
\tikzpic{-0.5}{[scale=0.4]
\draw(0,0)--(1,1)--(2,0)--(3,1)--(4,0)--(5,1)--(6,0);
}
\caption{Dyck paths of size $3$.}
\label{fig:Dyckpath}
\end{figure}
The highest path is the leftmost path in the first row, 
and the lowest path is the rightmost path in the second row.

Let $\mu_{0}$ (resp. $\mu_1$) be the highest (resp. lowest) Dcyk path of size $n$ and $\mu$ be a Dyck 
path of size $n$.
Dyck paths $\mu_{0}$ and $\mu_1$ are $U^nD^n$ and $(UD)^{n}$ respectively,
where $U$ and $D$ represent an up and down step.

Since we identify Dyck words with Dyck paths, we denote 
by $\mathtt{Dyck}(n)$ the set of Dyck paths of size $n$ 
by abuse of notation.

A Dyck path $\mu$ is {\it prime} if $\mu$ does not 
touch the horizontal line $y=0$ except the starting and 
ending points.
We denote by $\mathtt{Dyck^{(\mathrm{pr})}}(n)$ the set of 
prime Dyck paths of size $n$.
Note that $|\mathtt{Dyck^{(\mathrm{pr})}}(n)|=|\mathtt{Dyck}(n-1)|$ 
since the deletion of the first and last steps in $\mathtt{Dyck^{(\mathrm{pr})}}(n)$
gives a Dyck path in $\mathtt{Dyck}(n-1)$.

A Dyck path $\mu$ is called {\it fundamental} if $\mu$ has a unique peak
and no valleys, {\it i.e.}, $\mu=U^{N}D^{N}$.
We denote by $\mathtt{Dyck^{(f)}}(n)$ be the set of Dyck paths $\mu$ of size $n$
such that $\mu$ is written as a concatenation of fundamental Dyck paths.
For example, in the case of $n=3$, we have
\begin{align*}
\mathtt{Dyck^{(f)}}(3)=\{(UD)^3, UDUUDD, UUDDUD, U^{3}D^{3}\}.
\end{align*}
Note that $UUDUDD$ is not in $\mathtt{Dyck^{(f)}}(3)$ since we have a valley 
at height one.

\subsection{Motzkin numbers}
In this subsection, we introduce the generating function for 
Motzkin paths, and study its convolution.
As in the case of Fibonacci and Catalan numbers, 
the convoluted generating function is expressed in terms of 
binomial coefficients.

A {\it Motzkin path} of size $n$ is a lattice path from $(0,0)$ to $(n,0)$ 
which consists of up steps $U:=(1,1)$, down steps $D:=(1,-1)$ and 
horizontal step $H:=(1,0)$ and does not go below the horizontal line $y=0$.
For examples, we have nine Motzkin paths for $n=4$. 
They are 
\begin{align}
\label{eqn:Motpaths}
\begin{split}
&HHHH, \quad UHHD, \quad UHDH, \quad UDHH, \quad HUHD \\
&\qquad HUDH, \quad HHUD, \quad UUDD, \quad UDUD.
\end{split}
\end{align}
We denote by $\mathtt{Mot}(n)$ the set of Motzkin paths of size $n$.
The generating function for $|\mathtt{Mot}(n)|$ is 
given by 
\begin{align*}
\sum_{0\le n}|\mathtt{Mot}(n)|x^{n}&=
\genfrac{}{}{}{}{1-x-\sqrt{1-2x-3x^2}}{2x^2}, \\
&=1+x+2x^2+4x^3+9x^4+21x^5+51x^6+127x^7+323x^8+\cdots.
\end{align*}
The numbers $|\mathtt{Mot}(n)|$ appear in the sequence A001006 in OEIS \cite{Slo}.
Then, the cardinality $M_{n}:=|\mathtt{Mot}(n)|$ satisfies the 
recurrence relation 
\begin{align*}
M_{n}=M_{n-1}+\sum_{i=0}^{n-2}M_{i}M_{n-2-i},
\end{align*}
with the initial conditions $M_{0}=1$ and $M_{1}=1$.

Let $\lambda_{0}^{n}$ be the lowest Motzkin path of size $n$, 
{\it i.e.}, $\lambda_{0}^{n}=H^{n}$.
We denote by $Y(\lambda_{0}^{n}/\lambda)$ the region above $\lambda_{0}^{n}$ and 
below $\lambda$.
\begin{defn}
Given a Motzkin path $\lambda\in\mathtt{Mot}(n)$, 
we define $|Y(\lambda_{0}^{n}/\lambda)|$ as 
the area of the skew shape $Y(\lambda_{0}^{n}/\lambda)$.
\end{defn}

For example, $|Y(\lambda_{0}^{n})/\lambda|$ for $n=4$ is given 
by 
\begin{align*}
0, 3, 2, 1, 2, 1, 1, 4, 2,
\end{align*} 
where $\lambda$ is listed in Eq. (\ref{eqn:Motpaths}) from left to 
right and from top to bottom.

Suppose that $\lambda_1,\ldots,\lambda_{p}$ be Motzkin paths 
such that each $\lambda_{i}$, $1\le i\le p$, can not 
be decomposed into a concatenation of Motzkin paths of samller lengths.
Namely, $\lambda_{i}$ does not touch the horizontal line $y=0$ except 
the starting and ending points.
We call such a path $\lambda_{i}$ a prime path.
Any Motzkin path can be uniquely written as a concatenation of 
prime Motzkin paths.
We denote by $\mathtt{Mot}^{(\mathrm{pr})}(n)$ the set of 
prime Motzkin paths of size $n$.

For example, we have two and four prime Motzkin paths for $n=4$ and $n=5$:
\begin{align*}
\mathtt{Mot}^{(\mathrm{pr})}(4)&=\{ UHHD, UUDD\}, \\
\mathtt{Mot}^{(\mathrm{pr})}(5)&=\{ UHHHD, UHUDD, UUDHD, UUHDD \}.
\end{align*}
Note that $|\mathtt{Mot}^{(\mathrm{pr})}(n)|=|\mathtt{Mot}(n-2)|$ and 
we have a natural bijection between them.
This bijection is realized by deleting the first step $U$ and last step $D$ 
in $\mathtt{Mot}^{(\mathtt{pr})}(n)$, and we obtain a Motzkin path in 
$\mathtt{Mot}(n-2)$.

A Motzkin path is called {\it fundamental} if it is of the form $U^{p}H^{q}D^{p}$ 
with $p\ge0$ and $q$ is either $0$ or $1$.
We denote by $\mathtt{Mot}^{(f)}(n)$ the set of Motzkin paths $\lambda$ of size $n$ 
such that $\lambda$ is written as a concatenation of fundamental Motzkin paths.
For example, in the case of $n=4$, we have eight paths in $\mathtt{Mot}^{(f)}(4)$:
\begin{align*}
\mathtt{Mot}^{(f)}(4)=\mathtt{Mot}(4)\setminus\{UHHD\}.
\end{align*}

Below, we consider the convolution of Motzkin numbers. 
We define 
\begin{align}
\label{eqn:convMot}
\sum_{0\le n}\mathtt{Mot}(r,n)x^{n}:=\left(\genfrac{}{}{}{}{1-x-\sqrt{1-2x-3x^2}}{2x^2}\right)^{r},
\end{align}
where the right hand side should be interpreted as a formal power series of $x$.
We denote by $\mathcal{M}(r)$ the right hand side of Eq. (\ref{eqn:convMot}).
First few values of $\mathtt{Mot}(r,n)$ are as follow:
\begin{align*}
\begin{array}{c|cccccccc}
r\backslash n & 0 & 1 & 2 & 3 & 4 & 5 & 6 & 7 \\ \hline
1 & 1 & 1 & 2 & 4 & 9 & 21 & 51 & 127  \\
2 & 1 & 2 & 5 & 12 & 30 & 76 & 196 & 512 \\
3 & 1 & 3 & 9 & 25 & 69 & 189 & 518 & 1422 \\
4 & 1 & 4 & 14 & 44 & 133 & 382 & 1140 & 3288
\end{array}
\end{align*}
We also have 
\begin{align*}
&\mathtt{Mot}(r,1)=r, \qquad\mathtt{Mot}(r,2)=\genfrac{}{}{}{}{1}{2}r(r+3), \\
&\mathtt{Mot}(r,3)=\genfrac{}{}{}{}{1}{6}r(r+2)(r+7), \\
&\mathtt{Mot}(r,4)=\genfrac{}{}{}{}{1}{24}(r^{4}+18r^3+83r^2+114r).
\end{align*}

\begin{prop}
\label{prop:convMotrec}
The polynomials $\mathtt{Mot}(r,n)$ satisfy the recurrence relation
\begin{align}
\label{eqn:convMotrec}
\mathtt{Mot}(r,n)=\mathtt{Mot}(r-1,n)+\mathtt{Mor}(r,n-1)+\mathtt{Mot}(r+1,n-2).
\end{align}
\end{prop}
\begin{proof}
From the definition of $\mathcal{M}(r)$, we have 
$1=\mathcal{M}(-1)+x+x^{2}\mathcal{M}(1)$.
By applying this to Eq. (\ref{eqn:convMot}), we obtain Eq. (\ref{eqn:convMotrec}).
\end{proof}

Let $\lambda$ be a Motzkin path of size $n$.
We decompose the path $\lambda$ into prime Motzkin paths $\lambda_1,\ldots,\lambda_{k}$ 
with $k\ge1$. This means that $\lambda$ can be expressed as a concatenation of prime paths $\lambda_{i}$, 
that is, $\lambda=\lambda_{1}\circ\cdots\circ\lambda_{k}$.
Then, we define the weight $\mathrm{wt}^{(r)}(\lambda)$ as 
\begin{align*}
\mathrm{wt}^{(r)}(\lambda)=\genfrac{(}{)}{0pt}{}{r+k-1}{k}.
\end{align*}

\begin{theorem}
The value $\mathtt{Mot}(r,n)$ is given as a generating function of Motzkin paths:
\begin{align}
\label{eqn:convMotgf}
\mathtt{Mot}(r,n)=\sum_{\lambda\in\mathtt{Mot}(n)}\mathrm{wt}^{(r)}(\lambda).
\end{align}

\end{theorem}
\begin{proof}
Let $\lambda=\lambda_1\circ\cdots\circ\lambda_{k}$ be a decomposition of Motzkin 
path $\lambda$ of size $n$ in terms of prime Motzkin paths $\lambda_{i}$, $1\le i\le k$.
We have two cases: 1) $\lambda_1=H$ and 2) $\lambda_1$ is a prime 
Motzkin path of size larger than one.
Note that the first step in $\lambda_1$ is $U$ in the second case.

\paragraph{Case 1)}
Since the first step in $\lambda_1$ is $H$, we have a Motzkin path 
of size $n-1$ obtained from $\lambda$ by deleting the first step $H$.
We denote by $\lambda'$ the newly obtained Motzkin path.
Then, by a simple calculation, we have
\begin{align*}
\mathrm{wt}^{(r)}(\lambda)=\mathrm{wt}^{(r-1)}(\lambda)+\mathrm{wt}^{(r)}(\lambda').
\end{align*}

\paragraph{Case 2)}
Suppose $\lambda_1=UD$. 
Then, we construct $k$ Motzkin paths of size $n$ starting from 
$\lambda^{(0)}:=\lambda$.
The paths $\lambda^{(i)}$, $1\le i\le k-1$, are defined by 
\begin{align*}
\lambda^{(i)}:=U\lambda_{2}\ldots\lambda_{i+1}D\lambda_{i+2}\ldots\lambda_{k}.
\end{align*}
The number of prime Motzkin paths in $\lambda^{(i)}$ is $k-i$.
Let $\lambda'$ be a Motzkin path of size $n-2$ obtained from $\lambda$ 
by deleting $\lambda_1$.
By a simple calculation, we have 
\begin{align*}
\sum_{i=0}^{k-1}\mathrm{wt}^{(r)}(\lambda^{(i)})=\mathrm{wt}^{(r+1)}(\lambda').
\end{align*}

A Motzkin path of size $n$ appears exactly once in either Case 1) or Case 2).
Further, in Case 2), we have one-to-$k$ correspondence between a Motzkin path
of size $n-2$ and Motzkin paths $\lambda^{(i)}$, $0\le i\le k-1$.
From these observations, it is easy to see that the generating function (\ref{eqn:convMotgf}) 
satisfies the recurrence relation in Proposition \ref{prop:convMotrec}.
This completes the proof.
\end{proof}

\subsection{Large Schr{\"o}der numbers}
In this subsection, we introduce large Schr\"oder numbers or equivalently 
Schr\"oder paths.
A {\it Schr\"oder path } of size $n$ is a lattice path from $(0,0)$ to
$(2n,0)$ such that each step is either up $U:=(1,1)$, down $D=(1,-1)$ or 
horizontal $H=(2,0)$ and does not go below the horizontal line $y=0$.

For example, we have six Schur\"oder paths of size $2$:
\begin{align*}
HH, \quad HUD, \quad UDH, \quad UHD, \quad UDUD, \quad UUDD.
\end{align*}
We denote by $\mathtt{Sch}(n)$ the set of Schr\"oder paths of size $n$.
The generating function for $\mathtt{Sch}(n)$ is given by 
\begin{align*}
\sum_{0\le n}|\mathtt{Sch}(n)|x^{n}
&=\genfrac{}{}{}{}{1-x-\sqrt{x^2-6x+1}}{2x}, \\
&=1+2x+6x^2+22x^3+90x^4+394x^5+1806x^6+\cdots. 
\end{align*}
This integer sequence appears as A006318 in OEIS \cite{Slo}.
The cardinality of $S_{n}:=|\mathtt{Sch}(n)|$ satisfies the 
recurrence relation 
\begin{align*}
S_{n}=3S_{n-1}+\sum_{k=1}^{n-2}S_{k}S_{n-k-1},
\end{align*}
and initial conditions $S_{0}=1$ and $S_{1}=2$.

\subsection{Narayana numbers}
In this subsection, we introduce Narayana numbers.
As we will see below, Narayana numbers can be regarded 
fine refinement of Catalan numbers.	
Actually, Narayana numbers count the number of Dyck  words  of  
length $2n$  having $k$  peaks (see, e.g., \cite{Kre70}). 
$q$-Analogues of Narayana Numbers appear in the study of 
$q$-analogues of Catalan numbers in the view point of $q$-Lagrange 
inversion theories \cite{FurHof85}.
$q$-Analogues of Narayana numbers can also be obtained by counting the level steps in 
another combinatorial object, Schr\"oder paths \cite{BonShaSim93}. 
For further applications of the Narayana polynomials to the related topics, see, for example, 
\cite{Sul00,Sul02} and the references therein.

The Narayana numbers $\mathtt{Nar}(n,k)$ are defined by 
\begin{align}
\label{eqn:defNara}
\mathtt{Nar}(n,k):=\genfrac{}{}{}{}{1}{k}\genfrac{(}{)}{0pt}{}{n-1}{k-1}\genfrac{(}{)}{0pt}{}{n}{k-1},
\end{align}
with $1\le n$ and $1\le k\le n$.
The generating function is given by 
\begin{align*}
\sum_{1\le n,k}\mathtt{Nar}(n,k)x^{n}y^{k}
&=\genfrac{}{}{}{}{1}{2x}\left(1-x(1+y)-\sqrt{(1-x(1+y))^2-4x^2y}\right), \\
&=xy+x^2(y+y^2)+x^3(y+3y^2+y^3)+x^4(y+6y^2+6y^3+y^4) \\
&\quad+x^5(y+10y^2+20y^3+10y^4+y^5)+\cdots.
\end{align*}
The integer sequence appears as A001263 in OEIS \cite{Slo}.
Note that if we set $y=1$, we obtain 
\begin{align*}
\sum_{1\le n}\left(\sum_{1\le k}\mathtt{Nar}(n,k)\right)x^{n}
=\mathtt{Cat}(x)-1.
\end{align*}
Equivalently, we have 
\begin{align*}
\sum_{k=1}^{n}\mathtt{Nar}(n,k)=C_{n},
\end{align*}
where $C_{n}$ is the $n$-th Catalan number.
The relation between Narayana numbers and Catalan numbers are 
studied in, e.g., \cite{Cok03,ManSun09,DGRog81a,DGRog81b,Sul00}.

\section{Dimers on a segment}
\label{sec:dimersseg}
In this section, we introduce dimers with multi colors 
on a segment.
We define and study the generating function for dimers as a 
formal power series.
Then, we obtain several formulae by specializing 
the formal variables.

\subsection{Dimers on a segment: definition}
We introduce dimer configurations with multi colors on a segment.
Let $n$ be the length of segment and $[0,n]:=\{0,1,2,\ldots,n\}$.
{\it A dimer} at position $i$, $1\le i\le n$, is an edge connecting the 
vertices labeled by $i-1$ and $i$.
We say that a dimer has a color $c\in[1,m]$ if it has a label $c$.
A {\it dimer configuration} is a configuration of dimers on the segment of 
length $n$ satisfying
\begin{enumerate}
\item There is at most one dimer at the position $i$.
\item If there exists a dimer of color $c$ at position $i$, 
there is no dimers of color $c$ at position $i-1$ and $i+1$. 
\end{enumerate}
The second condition implies that if there is a pair of 
dimers at position $i$ and $i+1$, 
the colors of the two dimers are different.

A {\it connected component} of a dimer configuration consists 
of dimers of positions $i_1,\ldots,i_p$ such that 
$i_{j+1}=i_{j}+1$ for $1\le j\le p-1$ and $p$ is maximal.
The size of a connected component is defined as $p$.

For example, we have four dimer configurations of length $2$ for $m\ge2$ 
as in Figure \ref{fig:d1seg}.
If $m=1$, the fourth configuration is not allowed.
The numbers of connected components are $0,1,1$ and $1$ from left to right.
\begin{figure}[ht]
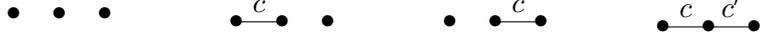

\tikzpic{-0.5}{[scale=0.6]
\draw (0,0)node{$\bullet$}(1,0)node{$\bullet$}(2,0)node{$\bullet$};
}
\qquad
\tikzpic{-0.5}{[scale=0.6]
\draw (0,0)node{$\bullet$}(1,0)node{$\bullet$}(2,0)node{$\bullet$};
\draw(0,0)--(1,0);
\draw(0.5,0)node[anchor=south]{$c$};
}\qquad
\tikzpic{-0.5}{[scale=0.6]
\draw (0,0)node{$\bullet$}(1,0)node{$\bullet$}(2,0)node{$\bullet$};
\draw(1,0)--(2,0);
\draw(1.5,0)node[anchor=south]{$c$};
}\qquad
\tikzpic{-0.5}{[scale=0.6]
\draw (0,0)node{$\bullet$}(1,0)node{$\bullet$}(2,0)node{$\bullet$};
\draw(0,0)--(2,0);
\draw(0.5,0)node[anchor=south]{$c$};
\draw(1.5,0)node[anchor=south]{$c'$};
}
\caption{Dimer configurations of length $2$.
The color $c'$ is different from $c$.
}
\label{fig:d1seg}
\end{figure}

We introduce four parameters $(t,m,r,s)$.
The parameter $t$ is for the length of a segment, $m$ is for the number of colors, 
$r$ is for the total number of dimers, and $s$ is for the sum of the positions of dimers from left end.
Let $\mathcal{D}_{n}$ be the set of dimer configurations of length $n$.
Fix a dimer configuration $d\in\mathcal{D}_{n}$.
Let $a$ be the sum of the positions of $d$, $b$ be the number of dimers,
$f$ be the number of connected components.
Then, we define the weight of a dimer configuration as
\begin{align}
\label{eqn:wtdimerconf}
\mathrm{wt}(d):=s^{a}r^{b}m^{f}(m-1)^{b-f}.
\end{align}

We are ready to define the generating function for dimers on a segment.
\begin{defn}
\label{defn:Gtmrs}
We define 
\begin{align*}
G^{(s)}(t,m,r,s)&:=\sum_{0\le n}\left(\sum_{d\in\mathcal{D}_{n}}\mathrm{wt}(d)\right)t^{n}, \\
G_{n}^{(s)}(m,r,s)&:=[t^{n}]G^{(s)}(t,m,r,s),
\end{align*}
with the initial condition $[t^{0}]G^{(s)}(t,m,r,s):=1$. 
\end{defn}
For example, in the case $n=2$, we have 
\begin{align*}
[t^{2}]G^{(s)}(t,m,r,s)=1+mr(s+s^2)+m(m-1)r^{2}s^{3}.
\end{align*}

To compute $G^{(s)}(t,m,r,s)$, we introduce two power series $z(t,m,r,s)$ 
and $z_{c}(t,m,r,s)$.

The series $z(t,m,r,s)$ is the generating function for dimers 
such that there is no dimer at the position $1$.
Contrarily, the series $z_{c}(t,m,r,s)$ is the generating 
function with a dimer at the position $1$ with color $c$.

The next proposition is obvious by definitions of $G^{(s)},z$, and $z_{c}$.
\begin{prop}
The generating functions satisfy 
\begin{align}
\label{eqn:relGz}
G^{(s)}(t,m,r,s)&=z(t,m,r,s)+mz_{c}(t,m,r,s), \\
\label{eqn:relzG}
z(t,m,r,s)&=1+t\cdot G^{(s)}(t,m,rs,s), \\
\label{eqn:relzcG}
z_{c}(t,m,r,s)&=tsrG^{(s)}(t,m,rs,s)-tsrz_{c}(t,m,rs,s),
\end{align}
\end{prop}

The system of difference equations from Eq. (\ref{eqn:relGz}) to Eq. (\ref{eqn:relzcG})
can be reduced to the single difference equation for the generating function $G_{n}^{(s)}(m,r,s)$.
\begin{prop}
\label{prop:Gsgfrec}
The generating function $G(t,m,r,s)$ satisfies 
the following relation:
\begin{align}
\label{eqn:1dgf}
G^{(s)}(t,m,r,s)=1+str+t(1+sr(m-1))G^{(s)}(t,m,rs,s)+t^2srG^{(s)}(t,m,rs^2,s).
\end{align}
\end{prop}
\begin{proof}
We abbreviate $G(r):=G^{(s)}(t,m,r,s)$, $z(r):=z(t,m,r,s)$ and $z_{c}(r):=z(t,m,r,s)$.
From Eqs. (\ref{eqn:relzcG}) and (\ref{eqn:relGz}), 
we have 
\begin{align*}
mz_{c}(r)&=mtsrG(r)-mtsrz_{c}(rs), \\
&=mtsrG(r)-tsr(G(rs)-z(rs)),
\end{align*}
and $z(rs)=1+tG(rs^2)$ from Eq. (\ref{eqn:relzG}).
Substituting these into Eq. (\ref{eqn:relGz}) and arranging terms, 
we obtain Eq. (\ref{eqn:1dgf}).
\end{proof}

\subsection{Generating function}
In this subsection, we solve the difference equation (\ref{eqn:1dgf}) in 
Proposition \ref{prop:Gsgfrec}.
We rewrite Eq. (\ref{eqn:1dgf}) by the function:
\begin{align*} 
g(r):=G^{(s)}(t,m,r,s)-\overline{\alpha}(r)G^{(s)}(t,m,rs,s)+\beta(r),
\end{align*}
where $\overline{\alpha}(r)$ and $\beta(r)$ are formal power series.
One can easily show that $g(r)$, $\overline{\alpha}(r)$ and $\beta(r)$ 
satisfy the following system of equations by a straightforward 
computation.
\begin{prop}
The equation (\ref{eqn:1dgf}) is equivalent to the following 
system of equations:
\begin{align}
\label{eqn:gfg}
g(r)=\gamma g(rs),
\end{align}
with 
\begin{align}
\label{eqn:alpha1}
&\overline{\alpha}(r)+\gamma =t(1+rs(m-1)), \\
\label{eqn:alpha2}
&-\gamma \overline{\alpha}(rs)=t^2rs, \\
\label{eqn:beta}
&\beta(r)=-1-srt+\gamma \beta(rs).
\end{align}
\end{prop}

We define two power series as follows.
\begin{defn}
\label{def:f}
We define $f(x):=(1+x(m-1))$ and $\alpha(r):=t \overline{\alpha}(r)$.
\end{defn}
Below, we solve the system of equations from Eq. (\ref{eqn:alpha1}) 
to Eq. (\ref{eqn:beta}) in terms of $f(x)$.
The power series $f(x)$ and $\alpha(r)$ play a central role in the 
calculation of the generating function $G^{(s)}(t,m,r,s)$.
To calculate $G^{(s)}(t,m,r,s)$, we first calculate $\overline{\alpha}(r)$,
secondly $\beta(r)$ in terms of $\alpha(r)$ and $f(x)$.

First, from Eqs. (\ref{eqn:alpha1}) and (\ref{eqn:alpha2}),
one can show that $\alpha(r)$ and $f(x)$ satisfy the following 
difference relation: 
\begin{align}
\label{eqn:alpha3}
\alpha(r)=f(rs)+\genfrac{}{}{1pt}{}{1}{(rs)^{-1}\alpha(rs)}.
\end{align}

To obtain an analytic expression of $\alpha(r)$, we solve the 
recurrence relation (\ref{eqn:alpha3}) by use of a continued 
fraction (see, {\it e.g.},\cite{Fla80}).
We start with the definition of a continued fraction.
\begin{defn}
We define a continued fraction by 
\begin{align*}
\left[a_0,a_1,a_2,a_3\ldots\right]:=
a_0+\cfrac{1}{a_1+\cfrac{1}{a_2+\cfrac{1}{a_3+\cdots}}}.
\end{align*}
\end{defn}

From Eq. (\ref{eqn:alpha3}), $\alpha(r)$ has 
an expression in terms of a continued fraction 
$\alpha(r):=[\alpha_0,\alpha_1,\ldots]$ 
with 
\begin{align*}
\alpha_{2n}&=f(rs^{2n+1})s^{-n}, \\
\alpha_{2n-1}&=f(rs^{2n})r^{-1}s^{-n}.
\end{align*}
We define two polynomials $p_{n}$ and $q_{n}$ with two variables $r$ and $s$
as 
\begin{align*}
\genfrac{}{}{1pt}{}{p_{n}}{q_{n}}:=\left[\alpha_0,\alpha_{1},\ldots,\alpha_{2n-1}\right].
\end{align*}
Note that the fraction $p_{n}/q_{n}$ for $1\le n$ converges to $\alpha(r)$ in the large $n$ limit.

By a straightforward calculation, the polynomials $p_n$ and $q_{n}$ can be obtained 
from $\alpha_{i}$ by the following form.
\begin{prop}
\label{prop:cfpq}
Given a sequence $\alpha_0,\alpha_1,\alpha_2,\ldots$, 
we have 
\begin{align*}
\begin{pmatrix}
\alpha_0 & 1 \\
1 & 0
\end{pmatrix}
\begin{pmatrix}
\alpha_1 & 1 \\
1 & 0
\end{pmatrix}
\begin{pmatrix}
\alpha_2 & 1 \\
1 & 0
\end{pmatrix}
\cdots
\begin{pmatrix}
\alpha_{2n-1} & 1 \\
1 & 0
\end{pmatrix}
=
\begin{pmatrix}
p_n & p'_{n-1} \\
q_n & q'_{n-1}
\end{pmatrix}
\quad for\ n=0,1,2,\ldots,
\end{align*}
with the initial conditions $p_{0}=q'_{-1}=1$ and $p'_{-1}=q_{0}=0$.
\end{prop}

The two polynomials $p_{n}$ and $q_{n}$ with $n\ge0$
can be expressed in terms of $f(rs^{k})$ with $k\ge1$
as follows.
\begin{prop}
Let $I^{n}(l,k)$ be the set of indices: 
\begin{align}
\label{eqn:Indef}
I^{n}(l,k):=\left\{
(i_1,i_2,\ldots,i_{2k}) \Bigg|  
\begin{array}{c}
1\le i_1<i_2<\ldots<i_{2k}\le 2n \\
i_{2m}=i_{2m-1}+1,\quad 1\le m\le k \\
i_1+i_2+\cdots+i_{2k}=k+2l 
\end{array}
\right\}.
\end{align}
Then, we have 
\begin{align}
\label{eqn:pdef2}
p_{n}=r^{-n}s^{-n^2}\left(\prod_{k=1}^{2n}f(rs^{k})\right)\cdot
\left(\sum_{0\le k\le n}\sum_{k^2\le l\le k(2n-k)}P^{n}_{l,k}s^{l}r^{k}\right),
\end{align}
where 
\begin{align*}
P^{n}_{l,k}=\sum_{(i_1,\ldots,i_{2k})\in I^{n}(l,k)}\prod_{j=1}^{2k}f(rs^{i_j})^{-1}.
\end{align*}
Similarly, let $J^{n}(l,k)$ be the set of indices:
\begin{align*}
J^{n}(l,k):=\left\{
(i_1,i_2,\ldots,i_{2k}) \Bigg|  
\begin{array}{c}
2\le i_1<i_2<\ldots<i_{2k}\le 2n \\
i_{2m}=i_{2m-1}+1,\quad 1\le m\le k \\
i_1+i_2+\cdots+i_{2k}=k+2l 
\end{array}
\right\}.
\end{align*}
Then, we have 
\begin{align*}
q_{n}=r^{-n}s^{-n^2}\left(\prod_{k=2}^{2n}f(rs^{k})\right)\cdot
\left(
\sum_{0\le k\le n-1}\sum_{k(k+1)\le l\le k(2n-k)}
Q^{n}_{l,k}s^{l}r^{k}
\right),
\end{align*}
where 
\begin{align*}
Q^{n}_{l,k}=\sum_{(i_1,i_2,\ldots,i_{2k})\in J^{n}(l,k)}
\prod_{j=1}^{2k}f(rs^{i_j})^{-1}.
\end{align*}
\end{prop}
\begin{proof}
Below, we only show that the proposition is true for $p_n$ since the proof for $q_{n}$
is similar to the one for $p_{n}$.

From Proposition \ref{prop:cfpq}, we have 
\begin{align*}
\begin{pmatrix}
p_{n} & p'_{n-1} \\
q_{n} & q'_{n-1}
\end{pmatrix}
&=
\begin{pmatrix}
p_{n-1} & p'_{n-2} \\
q_{n-1} & q'_{n-2}
\end{pmatrix}
\begin{pmatrix}
\alpha_{2n-2} & 1 \\
1 & 0
\end{pmatrix}
\begin{pmatrix}
\alpha_{2n-1} & 1 \\
1 & 0
\end{pmatrix}.
\end{align*}
From this, we obtain 
\begin{align*}
\begin{split}
&p_{n}=(1+r^{-1}s^{-(2n-1)}f(rs^{2n-1})f(rs^{2n}))p_{n-1}+r^{-1}s^{-n}f(rs^{2n})p'_{n-2}, \\
&p'_{n-1}=s^{-(n-1)}f(rs^{2n-1})p_{n-1}+p'_{n-2}.
\end{split}
\end{align*}
and $p_n$ satisfies 
\begin{align}
\label{eqn:prec3}
p_{n}=p_{n-1}+\sum_{k=1}^{n-1}s^{-(n+k)}r^{-1}f(rs^{2k+1})f(rs^{2n})p_{k}.
\end{align}
We will show that Eq. (\ref{eqn:pdef2}) satisfies the recurrence relation above.
The expression (\ref{eqn:pdef2}) is equivalent to the following expression:
\begin{align}
\label{eqn:pkinn}
p_{k}=r^{-n}s^{-n^2}\left(\prod_{l=1}^{2n}f(rs^{l})\right) 
\left(\sum_{0\le u\le k}\sum_{u^2\le v\le u(2k-u)}\left(\prod_{l=2k+1}^{2n}f(rs^{l})^{-1}\right)
P^{k}_{v,u}s^{n^2-k^2+v}t^{n-k+u}\right),
\end{align}
where $k\le n$.
From Eqs. (\ref{eqn:prec3}) and (\ref{eqn:pkinn}), 
the term $p_{n-1}$ in the right hand side of Eq. (\ref{eqn:pkinn}) gives 
all the coefficients $P^{n}_{l,k}$ which contain
the factor $f(rs^{2n-1})^{-1}f(rs^{2n})^{-1}$. 
We have to show that the powers of $r$ and $s$ in $p_n$ coincide with 
those coming from $p_{n-1}$. 
In Eq. (\ref{eqn:pkinn}), if we take the coefficients for $(u,v)$ and $k=n-1$, 
we have 
$f(rs^{2n-1})^{-1}f(rs^{2n})^{-1}P^{k}_{v,u}s^{2n+v-1}r^{u+1}$.
We have 
\begin{align*}
&I^{n-1}(v,u)\cup \{i_{2u+1}=2n-1,i_{2(u+1)}=2n\} \\
&\quad=\left\{
(i_1,\ldots,i_{2(u+1)})\Bigg| 
\begin{array}{c}
1\le i_1<\ldots<i_{2(u+1)}\le 2n \\
i_{2p}=i_{2p-1}+1, 1\le p\le u+1 \\
i_{2u+1}=2n-1,i_{2(u+1)}=2n \\	
i_1+\cdots+i_{2(u+1)}=u+1+2(2n+v-1)
\end{array}
\right\}, \\
&\quad\subset I^{n}(2n+v-1,u+1). 
\end{align*}
Note that the values $2n+v-1$ and $u+1$ coincide with the powers of $s$ and $r$.

In a similar manner, we will compute 
the term $s^{-(n+k)}r^{-1}f(rs^{2k+1})f(rs^{2n})p_{k}$.
Then, we need to evaluate 
$\left(\prod_{l=2k+2}^{2n-1}f(rs^{l})^{-1}\right)P^{k}_{v,u}s^{(n+k)(n-k-1)+v}t^{n-k+u-1}$.
We have 
\begin{align*}
&I^{k}(v,u)\cup\{i_{2u+l}=2k+l, 1\le l\le 2(n-k) \} \\
&\quad=\left\{
(i_1,\ldots,i_{2(n+u-k)})\Bigg| 
\begin{array}{c}
1\le i_1<\ldots<i_{2(n+u-k)}\le 2n \\
i_{2p}=i_{2p-1}+1,\  1\le p\le n+u-k \\
i_{2u+l}=2k+l,\  1\le l\le 2(n-k) \\
i_1+\cdots+i_{2(n+u+1)}=(n-k-1)(2n+2k+1)+u+2v
\end{array}
\right\}, \\
&\quad\subset I^{n}((n+k)(n-k-1)+v,n-k+u-1). 
\end{align*}
Note that $(n+k)(n-k-1)+v$ and $(n-k+u-1)$ coincide with 
the exponents of $s$ and $r$.
Therefore, we prove Eq. (\ref{eqn:prec3}).
This completes the proof.
\end{proof}

We solve the recurrence relation for $\beta(r)$ by 
using $\alpha(r)$ and $f(r)$.
\begin{defn}
We define 
$\alpha_1(r):=f(rs)-\alpha(r)$.
\end{defn}
Note that $\gamma$ is expressed in terms of $\alpha(r)$ and $f(r)$ 
by using Eq. (\ref{eqn:alpha1}).
Then, from Eq. (\ref{eqn:beta}), 
the formal power series $\beta(r)$ satisfies 
\begin{align}
\label{eqn:beta1}
\beta(r)=-1-str+t\alpha_1(r)\beta(rs).
\end{align}

\begin{prop}
\label{prop:beta}
We have 
\begin{align}
\label{eqn:beta2}
\beta(r)=-1-\sum_{1\le k}\left(\prod_{j=0}^{k-2}\alpha_{1}(rs^{j})\right)\left(rs^{k}+\alpha_1(rs^{k-1})\right)t^{k}.
\end{align}
\end{prop}
\begin{proof}
By using Eq. (\ref{eqn:beta1}) recursively, it is easy to show that 
$\beta(r)$ can be expanded in terms of $\alpha_1(rs^{n})$ as in Eq. (\ref{eqn:beta2}).
\end{proof}

\begin{remark}
Note that $\alpha(r)$ is a series with variables $r$ and $s$ only, however, the 
series $\beta(r)$ consists of variables $r,s$ and $t$.
\end{remark}

From Eq. (\ref{eqn:alpha1}), $\gamma$ is expressed as $\gamma=t(f(r)-\alpha(r))$.
Note that $f(r)-\alpha(r)$ is independent of $t$.
From Eq. (\ref{eqn:gfg}), the formal power series $g(r)=0$.
We have the following proposition by the definition of $g(r)$.	
\begin{prop}
The generating function $G^{(s)}(t,m,r,s)$ satisfies 
\begin{align}
\label{eqn:Gab}
G^{(s)}(t,m,r,s)-t\alpha(r)G^{(s)}(t,m,rs,s)+\beta(r)=0.
\end{align}
\end{prop}

\begin{prop}
\label{prop:Gexp}
Let $\alpha(r)$ and $\beta(r)$ be formal power series defined as above.
We have 
\begin{align}
\label{eqn:Gab2}
G^{(s)}(t,m,r,s)=\sum_{n\ge0}\left(\prod_{j=1}^{n}\alpha(rs^{j-1})\right)
\left(-\beta(rs^{n})\right)t^{n}.
\end{align}
\end{prop}
\begin{proof}
By using Eq. (\ref{eqn:Gab}) recursively, one can easily show that 
$G^{(s)}(t,m,r,s)$ satisfies Eq. (\ref{eqn:Gab2}).
\end{proof}

\begin{remark}
As already mentioned above, the power series $\beta(r)$ contains 
the variable $t$ (see Proposition \ref{prop:beta}). 
Thus, to obtain the polynomial $G^{(s)}_{n}(m,r,s)$ from the expression (\ref{eqn:Gab2}),
we need to expand $\beta(r)$ by use of Eq. (\ref{eqn:beta2}). 
\end{remark}

\begin{cor}
\label{cor:G}
If we expand $G^{(s)}(t,m,r,s)=:\sum_{0\le n}G_{n}^{(s)}(m,r,s)t^{n}$, then 
we have 
\begin{align}
\label{eqn:Gnexp}
G_{n}^{(s)}(m,r,s)=\prod_{j=0}^{n-1}\alpha(rs^{j})
+\sum_{l=1}^{n}\left(\prod_{j=0}^{n-l-1}\alpha(rs^{j})\right)
\left(\prod_{j=0}^{l-2}\alpha_1(rs^{n-l+j})\right)(rs^{n}+\alpha_{1}(rs^{n-1})).
\end{align}
\end{cor}
\begin{proof}
It is straightforward from Propositions \ref{prop:beta} and \ref{prop:Gexp} 
that the coefficient of $t^{n}$ in $G^{(s)}(t,m,r,s)$ has the above expression.
\end{proof}

Recall that $\alpha(x)$ and $\alpha_1(x)$ are the formal power series with infinite 
number of terms.
On the other hand, the generating function $G_{n}^{(s)}(m,r,s)$ is a polynomial 
with respect to $m, r$ and $s$.
Thus, in the expression (\ref{eqn:Gnexp}) in Corollary \ref{cor:G}, the right 
hand side is reduced from a formal power series to a polynomial with finite terms.

The polynomial $G_{n}^{(s)}(m,r,s)$ satisfies the recurrence relation of Fibonacci type.
\begin{theorem}
\label{thrm:FibGfin}
We have 
\begin{align*}
G_{n}^{(s)}(m,r,s)=f(rs^{n})G_{n-1}^{(s)}(m,r,s)+rs^{n}G_{n-2}^{(s)}(m,r,s),
\end{align*}
where $f(x)$ is defined in Definition \ref{def:f}.
\end{theorem}
\begin{proof}
By  definition, recall that  
$\alpha_{1}(r)=f(rs)-\alpha(r)$, and $\alpha_{1}(r)\alpha(rs)=-rs$.
We have 
\begin{align}
\label{eqn:a1f}
\alpha_{1}(rs^{p-2})\alpha_1(rs^{p-1})=\alpha_1(rs^{p-2})f(rs^{p})+rs^{p-1}.
\end{align}
We abbreviate $G_n^{(s)}(m,r,s)$ as $G_{n}(r)$.
From Corollary \ref{cor:G}, we compute the difference between the $l$-th term in the sum 
of $G_{n}(r)$ and the $l-1$-th term in the sum of $f(rs^{n})G_{n-1}(r)$ for $2\le l$.
We obtain 
\begin{align*}
\begin{split}
&\left(\prod_{j=0}^{n-1-l}\alpha(rs^{j})\right)\left(\prod_{j=0}^{l-2}(\alpha_{1}(rs^{n-l+j})\right)(rs^{n}+\alpha_1(rs^{n-1})) \\
&\qquad\qquad-f(rs^{n})\left(\prod_{j=0}^{n-1-l}\alpha(rs^{j})\right)\left(\prod_{j=0}^{l-3}(\alpha_{1}(rs^{n-l+j})\right)(rs^{n-1}+\alpha_1(rs^{n-2})) 
\end{split}\\
&=
\prod_{j=0}^{n-1-l}\alpha(rs^{j})\prod_{j=0}^{l-3}\alpha_1(rs^{n-l+j}) \\
&\quad\times\left\{\alpha_1(rs^{n-2})rs^{n}+\alpha_1(rs^{n-2})\alpha_{1}(rs^{n-1})-rs^{n-1}f(rs^{n})-\alpha_1(rs^{n-2})f(rs^{n}))\right\}, \\
&=\prod_{j=0}^{n-1-l}\alpha(rs^{j})\prod_{j=0}^{l-3}\alpha_1(rs^{n-l+j})rs^{n}(1-\alpha(rs^{n-2})),	 \\
&=rs^{n}\left(\prod_{j=0}^{n-1-l}\alpha(rs^{j})\right)\left(\prod_{j=0}^{l-4}\alpha_1(rs^{n-l+j})\right)(rs^{n-2}+\alpha_1(rs^{n-3})),	
\end{align*}		
where we have used Eq. (\ref{eqn:a1f}) and the expression $f(x)=1+x(m-1)$.

By a similar calculation, the remaining terms in $G_{n}(r)-f(rs^{n})G_{n-1}(r)$
give $rs^{p}\cdot\prod_{j=0}^{n-3}\alpha(rs^{j})$.
The sum of the contributions yields $rs^{p}G_{n-1}(r)$.
This completes the proof.
\end{proof}

First few terms of $G_{n}^{(s)}(m,r,s)$ are 
\begin{align*}
G_{0}^{(s)}(m,r,s)&=1,\qquad G_{1}^{(s)}(m,r,s)=1+mrs, \\
G_{2}^{(s)}(m,r,s)&=1+r(ms+ms^{2})+r^2s^{3}m(m-1), \\
G_{3}^{(s)}(m,r,s)&=1+r(ms+ms^{2}+ms^{3})+r^{2}(m(m-1)s^3+m^2s^4+m(m-1)s^5)+r^{3}s^{6}m(m-1)^{2}, \\
G_{4}^{(s)}(m,r,s)&=1+r(ms+ms^2+ms^3+ms^4)\\
&\quad+r^2((-m+m^2)s^3+m^2s^4+(-m+2m^2)s^5+m^2s^6+(-m+m^2)s^7) \\
&\quad+r^3((m-2m^2+m^3)s^6+(-m^2+m^3)s^7+(-m^2+m^3)s^8+(m-2m^2+m^3)s^9) \\
&\quad+r^4(-m+3m^2-3m^3+m^4)s^{10}.
\end{align*}

We solve the recurrence relation in Theorem \ref{thrm:FibGfin} in terms 
of $f(x)$.

\begin{theorem}
\label{thrm:Gsinf}
Let $\widehat{I}^{n}(a,b)$ be the set of indices
\begin{align*}
\widehat{I}^{n}(a,b):=\left\{
(i_1,\ldots,i_{2a})\Big|
\begin{array}{c}
0\le i_1<\ldots<i_{2a}\le n \\
i_{2p}=i_{2p-1}+1,\quad 1\le p\le a \\
i_1+\cdots+i_{2a}=2b-a
\end{array}
\right\}.
\end{align*}
The generating function $G_{n}(m,r,s)$ can be expressed as 
\begin{align}
\label{eqn:Gsinf}
G_{n}^{(s)}(m,r,s)
=\left(\prod_{i=1}^{n}f(rs^{i})\right)
\sum_{a=0}^{\lfloor (n+1)/2\rfloor}\sum_{b=a^2}^{a(n-a+1)}\widehat{G}^{n}_{a,b}r^{a}s^{b},
\end{align}
where 
\begin{align*}
\widehat{G}^{n}_{a,b}
=\sum_{(i_1,\ldots,i_{2a})\in \widehat{I}^{n}(a,b)}\prodd_{j=1}^{2a}f(rs^{i_j})^{-1}.	
\end{align*}
Here, $\prodd$ means that if $i_1=0$, then we do not include the term $f(rs^{i_1})^{-1}$.
\end{theorem}
\begin{proof}
We calculate $G_{n+1}^{(s)}-f(rs^{n+1})G_{n}^{(s)}-rs^{n+1}G_{n-1}^{(s)}$ by using the expression (\ref{eqn:Gsinf}).
We first consider the case $n=2p$, $p\ge1$. 
The term $f(rs^{n+1})G_{n}$ gives the terms which does not contain $f(rs^{2p+1})^{-1}$ in 
$\widehat{G}^{2p+1}_{a,b}$.
On the other hand, the term $G_{n-1}$ gives the terms which contain $f(rs^{2p+1})^{-1}$ in 
$\widehat{G}^{2p+1}_{a,b}$.
By definition of $\widehat{I}^{2p+1}(a,b)$, to include the term $f(rs^{2p+1})^{-1}$ means 
that the terms also contain $f(rs^{2p})^{-1}$.
Note that the terms in $\widehat{G}^{2p-1}_{a,b}$ never contain $f(rs^{2p})^{-1}$.
Since $(p+1)^{2}-p^{2}=2p+1=n+1$, we need the factor $rs^{n+1}$ for $G_{n-2}^{(s)}$ to 
cancel the terms in $G_{n}^{(s)}$.
From these observations, it is easy to show that the expression (\ref{eqn:Gsinf}) 
solves the recurrence relation in Theorem \ref{thrm:FibGfin}.

One can show the statement in the case of $n=2p+1$ by a similar manner. 
This completes the proof.
\end{proof}

\subsubsection{Generating function with \texorpdfstring{$m=1$}{m 1}.}
When $m=1$, we have $f(r)=1$. 
We give explicit expressions of $p_n$, $q_n$ and $\alpha(r)$.
The polynomials $p_n$ and $q_n$ satisfy the same recurrence relation obtained 
from Eq. (\ref{eqn:prec3}) by specializing $f(x)=1$: 
\begin{align*}
x_{n}&=(1+s^{-1}+s^{-(2n-1)}r^{-1})x_{n-1}-s^{-1}x_{n-2},
\end{align*} 
where $x_{n}=p_{n}$ or $q_{n}$.
The solutions of the above recurrence relation are easily obtained.
The formal power series $p_{n}$ and $q_{n}$ are of the form 
\begin{align*}
p_{n}(m=1)
&=r^{-n}s^{-n^{2}}\sum_{k=0}^{n}\genfrac{[}{]}{0pt}{}{2n-k}{k}_{s}r^{k}s^{k^2}, \\
q_{n}(m=1)
&=r^{-n}s^{-n^{2}}\sum_{k=0}^{n-1}\genfrac{[}{]}{0pt}{}{2n-k-1}{k}_{s}r^{k}s^{k(k+1)}.
\end{align*} 

Then, in this case, by taking the limit $n\rightarrow \infty$, $\alpha(r)$ can be written as follows:
\begin{align*}
\alpha(r;m=1)&=\genfrac{}{}{1pt}{}{\sum_{n\ge0}s^{n^2}r^{n}\prod_{k=1}^{n}(1-s^k)^{-1}}
{\sum_{n\ge0}s^{n(n+1)}r^{n}\prod_{k=1}^{n}(1-s^k)^{-1}}, \\
&=1+ sr-s^3r^2 + (s^5+s^6)r^3 -(s^7+2s^8+s^9+s^{10})r^4 \\
&\quad + (1+s) (s^9+2s^{10}+s^{11}+2s^{12}+s^{14})r^5 +\cdots.
\end{align*}
In fact, since $f(x)=1$ for $m=1$, $\alpha(r;m=1)$ has the following simple expression 
in terms of Ramanujan's continued fraction:
\begin{align*}
\alpha(r;m=1)=
1+\cfrac{s}{r^{-1}+\cfrac{s^2}{1+\cfrac{s^3}{r^{-1}+\cfrac{s^{4}}{1+\cdots}}}}.
\end{align*}

\begin{remark}
Some remarks are in order.
\begin{enumerate}
\item 
The coefficients of the formal power series $\alpha(rs^{-1};m=1)$
are a sequence A138158 (see also A227543) in OEIS \cite{Slo}.
Further, the coefficients of $r^{n}$ in $\alpha(r;m=1,s=1)$ are Catalan 
numbers.
Actually, $\alpha(r,s):=\alpha(r,s;m=1)$ satisfies 
\begin{align}
\label{eqn:recalpha1}
\alpha(r,s)=-r+\alpha(r,s)\alpha(rs^{-1},s),
\end{align}
from Eqs. (\ref{eqn:alpha1}) and (\ref{eqn:alpha2}).
The recurrence relation (\ref{eqn:recalpha1}) is roughly 
equivalent to the recurrence relation for the generating 
function of Dyck paths, whose total number of size $n$ is 
the $n$-th Catalan number.
\item
Let $\widetilde{p}(r,s)$ and $\widetilde{q}(r,s)$ be the numerator
and the denominator of $\alpha(r,s;m=1)$.
Then, $\widetilde{p}(r,s)$ has the expansion with respect to $s$:
\begin{align*}
\widetilde{p}(r,s)
&=1 +rs+rs^2+rs^3+(r+r^2)s^4+(r+r^2)s^5+(r+2 r^2)s^6  \\
&\quad+(r+2 r^2)s^7+(r+3r^2)s^8+(r+3r^2+r^3)s^9+\cdots,
\end{align*}
and the coefficients are the sequence A268187 in OEIS \cite{Slo}.

Similarly, $\widetilde{q}(r,s)$ has the expansion with respect 
to $s$:
\begin{align*}
\widetilde{q}(r,s)
&=1+rs^2+rs^3+rs^4+rs^5+(r+r^2)s^6+(r+r^2)s^7+(r+2r^2)s^8 \\
&\quad+(r+2r^2)s^9+(r+3r^2)s^{10}+(r+3r^2)s^{11}+(r+4r^2+r^3)s^{12}+\cdots.
\end{align*}
\item Under the specialization $(m,r)=(1,1)$, $\alpha(r=1,s)$ is 
expressed as 
\begin{align*}
\alpha(r=1,s)=\prod_{0\le n}\genfrac{}{}{}{}{(1-s^{5n+2})(1-s^{5n+3})}{(1-s^{5n+1})(1-s^{5n+4})},
\end{align*}
which is the ratio of the summations appearing in Rogers--Ramanujan identities.
\end{enumerate}
\end{remark}

By applying Theorem \ref{thrm:Gsinf} for $m=1$, equivalently $f(x)=1$, 
we have an explicit expression of $G^{(s)}_{n}(m=1,r,s)$:
\begin{prop}
\label{prop:Gnm1}
The generating function $G^{(s)}_{n}(m=1,r,s)$ is written as 
\begin{align*}
G^{(s)}_{n}(m=1,r,s)=1+\sum_{1\le j\le\lfloor(n+1)/2\rfloor}r^{j}s^{j^2}\genfrac{[}{]}{0pt}{}{n-j+1}{j}_{s}.
\end{align*}
\end{prop} 
\begin{proof}
From Theorem \ref{thrm:FibGfin}, we have the recurrence relation of Fibonacci type 
$G^{(s)}_{n}(m=1,r,s)=G^{(s)}_{n-1}(m=1,r,s)+rs^{n}G^{(s)}_{n-2}(m=1,r,s)$.
By a simple calculation, it is easy to show that the expression satisfies the recurrence relation.
\end{proof}

\subsubsection{Generating function with the general parameters \texorpdfstring{$(m,r,s)$}{(m,r,s)}}
To obtain a closed expression of the generating function, we first 
consider the specialization $m\rightarrow 1+mr$.
In this case, we have $f(x)=1+mrx$ and the generating function 
has a simple formula.
\begin{prop}
\label{prop:Gmrsgen}
We have 
\begin{align*}
G^{(s)}_{n}(1+mr,r,s)
=\sum_{k=0}^{n}m^{k}\sum_{j=0}^{\lfloor(n-k+1)/2\rfloor}
r^{2k+j}s^{k(k+1)/2+(k+j)j}\genfrac{[}{]}{0pt}{}{n-j}{k}_{s}\genfrac{[}{]}{0pt}{}{n-k-j+1}{j}_{s}.
\end{align*}

\end{prop}
\begin{proof}
We will show that the expression above is the solution of the recurrence relation 
in Theorem \ref{thrm:FibGfin}.
The recurrence relation is $G_{n}=(1+mr^{2}s^{n})G_{n-1}+rs^{n}G_{n-2}$, where 
$G_{n}:=G_{n}(1+mr,r,s)$.
The coefficients of $m^{0}$ in the recurrence relation should satisfy
\begin{align*}
&\sum_{j=0}^{\lfloor(n+1)/2\rfloor}r^{j}s^{j^2}\genfrac{[}{]}{0pt}{}{n-j}{0}_{s}\genfrac{[}{]}{0pt}{}{n-j+1}{j}_{s} \\
&\quad=\sum_{j=0}^{\lfloor n/2\rfloor}r^{j}s^{j^2}\genfrac{[}{]}{0pt}{}{n-j-1}{0}_{s}\genfrac{[}{]}{0pt}{}{n-j}{j}_{s}
+\sum_{j=0}^{\lfloor(n-1)/2\rfloor}r^{j+1}s^{j^2+n}\genfrac{[}{]}{0pt}{}{n-j-2}{0}_{s}\genfrac{[}{]}{0pt}{}{n-j-1}{j}_{s},
\end{align*}
and this is verified by a simple calculation.
The coefficients of $m^{k}$, $1\le k\le n$, in the recurrence relation should satisfy
\begin{align}
\label{eqn:Gsgenmrelr}
\begin{split}
&\sum_{j=0}^{\lfloor(n-k+1)/2\rfloor}
r^{2k+j}s^{k(k+1)/2+(k+j)j}\genfrac{[}{]}{0pt}{}{n-j}{k}_{s}\genfrac{[}{]}{0pt}{}{n-k-j-1}{j}_{s} \\
&\qquad=
\sum_{j=0}^{\lfloor(n-k)/2\rfloor}
r^{2k+j}s^{k(k+1)/2+(k+j)j}\genfrac{[}{]}{0pt}{}{n-j-1}{k}_{s}\genfrac{[}{]}{0pt}{}{n-k-j}{j}_{s} \\
&\qquad+
\sum_{j=0}^{\lfloor(n-k+1)/2\rfloor}
r^{2k+j}s^{k(k-1)/2+(k+j-1)j+n}\genfrac{[}{]}{0pt}{}{n-j-1}{k-1}_{s}\genfrac{[}{]}{0pt}{}{n-k-j+1}{j}_{s} \\
&\qquad+
\sum_{j=0}^{\lfloor(n-k-1)/2\rfloor}
r^{2k+j+1}s^{k(k+1)/2+(k+j)j+n}\genfrac{[}{]}{0pt}{}{n-j-2}{k}_{s}\genfrac{[}{]}{0pt}{}{n-k-j-1}{j}_{s}.
\end{split}
\end{align} 
Then, the coefficients of $r^{2k+j}$, $1\le j$, should satisfy 
\begin{align*}
\genfrac{[}{]}{0pt}{}{n-j}{k}_{s}\genfrac{[}{]}{0pt}{}{n-k-j+1}{j}_{s}
&=\genfrac{[}{]}{0pt}{}{n-j-1}{k}_{s}
\genfrac{[}{]}{0pt}{}{n-k-j}{j}_{s} \\
&\quad+s^{n-j-k}\genfrac{[}{]}{0pt}{}{n-j-1}{k-1}_{s}
\genfrac{[}{]}{0pt}{}{n-k-j+1}{j}_{s} \\
&\quad+s^{n-k-2j+1}\genfrac{[}{]}{0pt}{}{n-j+1}{k}_{s}
\genfrac{[}{]}{0pt}{}{n-k-j}{j-1}_{s}.
\end{align*}
This equation is easily verified by use of 
$\genfrac{[}{]}{0pt}{}{n}{k}_{s}=\genfrac{[}{]}{0pt}{}{n-1}{k}_{s}+s^{n-k}\genfrac{[}{]}{0pt}{}{n-1}{k-1}_{s}$.
Similarly, the coefficients of $r^{2k+j}$ with $j$ holds the equality in Eq. (\ref{eqn:Gsgenmrelr}).
This completes the proof.
\end{proof}

Then, by the change of variables $m\rightarrow (m-1)/r$ in Proposition \ref{prop:Gmrsgen},
we obtain the simple formula for the generating function $G^{(s)}_{n}(m,r,s)$.
\begin{theorem}
\label{thrm:Gsgeneric}
We have 
\begin{align*}
G_{n}^{(s)}(m,r,s)=\sum_{k=0}^{n}(m-1)^{k}\sum_{j=0}^{\lfloor(n-k+1)/2\rfloor}
r^{k+j}s^{k(k+1)/2+(k+j)j}\genfrac{[}{]}{0pt}{}{n-j}{k}_{s}\genfrac{[}{]}{0pt}{}{n-k-j+1}{j}_{s}.
\end{align*}
\end{theorem}

\begin{remark}
Proposition \ref{prop:Gnm1} can be obtained from Theorem \ref{thrm:Gsgeneric} by setting 
$m=1$.
Proposition \ref{prop:Gss1} below can also be obtained 
from Theorem \ref{thrm:Gsgeneric} simply by setting $s=1$ and arranging the terms 
involving the powers of $m$.
\end{remark}

\subsubsection{Generating function at \texorpdfstring{$m=1+m'/r$}{m=1+m/r}}
In this case, we have $f(x)=1+xm/r$ and can expect that the generating function 
has a simple formula.

\begin{prop}
\label{prop:Gsm2rgen}
We have 
\begin{align*}
G^{(s)}_{n}(1+m/r,r,s)
&=\prod_{j=1}^{n}(1+ms^j) \\
&\quad+\sum_{k=1}^{\lfloor(n+1)/2\rfloor} 
\sum_{j=0}^{n+1-2k}r^{k}s^{k^2+(2k+j+1)j/2}m^{j}\genfrac{[}{]}{0pt}{}{n-k}{j}_{s}\genfrac{[}{]}{0pt}{}{n-k-j+1}{k}_{s}.
\end{align*}
\end{prop}
\begin{proof}
From Theorem \ref{thrm:FibGfin}, the generating function satisfies the recurrence relation 
$G_{n}=(1+ms^{n})G_{n-1}+rs^{n}G_{n-2}$ where $G_{n}:=G_{n}^{(s)}(1+m/r,r,s)$.	

The coefficients of $r^{1}m^{j}$, $0\le j$, in the recurrence relation should satisfy 
\begin{align}
\label{eqn:Gs1mr1}
\begin{split}
&\sum_{j=0}^{n-1}s^{1+j(j+3)/2}\genfrac{[}{]}{0pt}{}{n-1}{j}_{s}\genfrac{[}{]}{0pt}{}{n-j}{1}_{s}  \\
&=\sum_{j=0}^{n-2}s^{1+j(j+3)/2}\genfrac{[}{]}{0pt}{}{n-2}{j}_{s}\genfrac{[}{]}{0pt}{}{n-j-1}{1}_{s} \\
&\quad+\sum_{j=1}^{n-1}s^{1+(j-1)(j+2)/2+n}\genfrac{[}{]}{0pt}{}{n-2}{j-1}_{s}\genfrac{[}{]}{0pt}{}{n-j}{1}_{s}
+\sum_{j=0}^{n-2}s^{j(j+1)/2+n}\genfrac{[}{]}{0pt}{}{n-2}{j}_{s}.
\end{split}
\end{align}
Similarly, the coefficients of $r^{k}m^{k}$, $2\le k$, $0\le j\le n+1-2k$, in the recurrence relation 
should satisfy  
\begin{align}
\label{eqn:Gs1mr2}
\begin{split}
&\sum_{j=0}^{n+1-2k}s^{k^2+(2k+j+1)j/2}\genfrac{[}{]}{0pt}{}{n-k}{j}_{s}\genfrac{[}{]}{0pt}{}{n-k-j+1}{k}_{s} \\
&\quad=\sum_{j=0}^{n-2k}s^{k^2+(2k+j+1)j/2}\genfrac{[}{]}{0pt}{}{n-k-1}{j}_{s}\genfrac{[}{]}{0pt}{}{n-k-j}{k}_{s} \\
&\quad+\sum_{j=1}^{n+1-2k}s^{k^2+(2k+j)(j-1)/2+n}\genfrac{[}{]}{0pt}{}{n-k-1}{j-1}_{s}\genfrac{[}{]}{0pt}{}{n-k-j+1}{k}_{s} \\
&\qquad+\sum_{j=0}^{n+1-2k}s^{(k-1)^2+(2k+j-1)j/2+n}\genfrac{[}{]}{0pt}{}{n-k-1}{j}_{s}\genfrac{[}{]}{0pt}{}{n-k-j}{k-1}_{s}.
\end{split}
\end{align}
The equations (\ref{eqn:Gs1mr1}) and (\ref{eqn:Gs1mr2}) follow from 
$\genfrac{[}{]}{0pt}{}{n}{k}_{s}=\genfrac{[}{]}{0pt}{}{n-1}{k}_{s}+s^{n-k}\genfrac{[}{]}{0pt}{}{n-1}{k-1}_{s}$.
This completes the proof.
\end{proof}

\begin{remark}
The expression in Proposition \ref{prop:Gsm2rgen} may be also derived from 
Theorem \ref{thrm:Gsgeneric} by specializing $m\rightarrow 1+mr$.
However, we have to rearrange the terms involving $r$.
\end{remark}

\begin{cor}
\label{cor:Gsgenetype2}
We have the following expression for the generating function:
\begin{align*}
G_{n}^{(s)}(m,r,s)
&=\prod_{j=1}^{n}(1+r(m-1)s^{j}) \\
&\quad+\sum_{k=1}^{\lfloor(n+1)/2\rfloor}\sum_{j=0}^{n+1-2k}r^{k+j}s^{k^2+(2k+j+1)j/2}(m-1)^{j}
\genfrac{[}{]}{0pt}{}{n-k}{j}_{s}\genfrac{[}{]}{0pt}{}{n-k-j+1}{k}_{s}.
\end{align*}
\end{cor}
\begin{proof}
We substitute $m\rightarrow r(m-1)$ in Proposition \ref{prop:Gsm2rgen}.
\end{proof}

If we further specialize $m\rightarrow 1+1/r$, the generating function 
has a simple formula.
\begin{cor}
\label{prop:Gsmr2fin}
We have
\begin{align}
\label{eqn:Gsmr2fin2}
\begin{split}
G^{(s)}_{n}(1+1/r,r,s)&=\prod_{j=1}^{n}(1+s^j)
+\sum_{j=1}^{\lfloor n/2\rfloor}r^{j}s^{j^2}
\genfrac{}{}{}{}{\displaystyle{[n+1]_s\prod_{k=1}^{n-2j}[2k+2j]_{s}}}{[n-2j+1]_{s}!} \\
&\quad+\begin{cases}
0, & n: \text{even}, \\
r^{\lfloor (n+1)/2\rfloor}s^{\lfloor (n+1)/2\rfloor^{2}}, & n: \text{odd}. 
\end{cases}
\end{split}
\end{align}

\end{cor}
\begin{proof}
From Theorem \ref{thrm:FibGfin}, the generating function $G_{n}:=G^{(s)}_{n}(1+1/r,r,s)$ 
satisfies the recurrence relation $G_{n}=(1+s^n)G_{n-1}+rs^{n}G_{n-2}$.
Let $A(n,j)$ be the coefficient of $r^{j}$ in the sum of Eq. (\ref{eqn:Gsmr2fin2}).
The coefficients of $r^{j}$ ,$1\le j$, in the recurrence relation should satisfy 
\begin{align*}
A(n,j)=A(n-1,j)(1+s^{n})+A(n-2,j-1)s^{n},
\end{align*}
which can be verified by a straightforward calculation.

Similarly, the coefficients of $r^{0}$ and $r^{n}$ in the both sides of recurrence 
relation are equal. This completes the proof.
\end{proof}

\subsubsection{Generating function with \texorpdfstring{$m=1-1/r$}{m=1-1/r}.}
We consider the specialization $m=1-1/r$.
In this case, we have $f(x)=1-s$.

\begin{prop}
\label{prop:GfinA}
Let $A(n,k)$ be the following polynomial in $s$:
\begin{align}
\label{eqn:expAnk}
A(n,k):=\prod_{l=1}^{k}(1-s^{n-k+l})\genfrac{}{}{}{}{\prod_{l=1}^{n}(1-s^{l})}{\prod_{l=1}^{k}(1-s^{l})^2}.
\end{align}
Then, the generating function $G^{(s)}_{n}(1-1/r,r,s)$ is given by 
\begin{align*}
G^{(s)}_{n}(1-1/r,r,s)=\prod_{k=1}^{n}(1-s^{k})
+\sum_{k=1}^{\lfloor(n+1)/2\rfloor}r^{k}s^{k^2}\left(A(n-k+1,k)+s^{n+1}A(n-k,k)\right).
\end{align*}
\end{prop}
\begin{proof}
From Theorem \ref{thrm:FibGfin}, the generating function 
satisfies the recurrence relation $G_{n}=(1-s^{n})G_{n-1}+rs^{n}G_{n-2}$.
This recurrence relation is equivalent to 
\begin{align*}
A(n-k+1,k)+s^{n+1}A(n-k,k)&=(A(n-k,k)+s^{n}A(n-k-1,k))(1-s^{n}) \\
&\quad+s^{n-2k+1}(A(n-k,k-1)+s^{n-1}A(n-k-1,k-1)).
\end{align*}
This equation can be verified by use of the expression (\ref{eqn:expAnk}).
This completes the proof.
\end{proof}

\subsection{Derivation of the generating function via the transfer matrix method}
\label{sec:tmmforGs}
We give another way of deriving the generating function by use of the transfer 
matrix method.

A dimer configuration is classified into two classes. 
The first class is a dimer configuration such that there is no dimer 
at position $1$.
The second class is a configuration such that there is a dimer 
at the position $1$.

Then, we define the transfer matrix $T$ by 
\begin{align*}
T(m,r,s):=
\begin{pmatrix}
1 & 1 \\
mrs & rs(m-1) 
\end{pmatrix}.
\end{align*}
Here, the $(i,j)$-th entry of $T(m,r,s)$ corresponds to 
the transition weight from the $i$-th class to the $j$-th class.
For example, $(2,2)$-entry is $rs(m-1)$ since we are not allowed 
to place two dimers with the same color next to each other.

Then we define a matrix $\mathtt{MFib}(n;m,r,s)$ as an ordered 
product of the matrices $T(m,r,s)$:  
\begin{align}
\label{eqn:MFibdef}
\mathtt{MFib}(n;m,r,s):=T(m,r,s)T(m,rs,s)\cdots T(m,rs^{n-1},s), \qquad n\ge1,
\end{align}
with the initial condition $\mathtt{Mfib}(0;m,r,s)=I_{2}$ where $I_{2}$ is the 
identity matrix of rank $2$.

Given a matrix $M$, we denote by $M[i,j]$ the $(i,j)$-th entry of $M$.
The matrix $\mathtt{MFib}(n):=\mathtt{MFib}(n;m,r,s)$ has four entries 
and each entry can be interpreted as follows.
The polynomial $\mathtt{MFib}(n)[i,1]$, $i=1$ or $2$, is the generating 
function of dimer configurations of the $i$-th class and size $n$.
Then, the polynomial $\mathtt{MFib}[i,2]$, $i=1$ or $2$ is the 
generating function of dimer configurations of size $n+1$ which 
are the $i$-th class and have a dimer at position $n+1$.

From observations above, it is obvious that the generating function 
$G^{(s)}_{n}(m,r,s)$ is obtained by the matrix $\mathtt{M}(n;m,r,s)$ as follows.
\begin{prop}
\label{prop:tmmGs}
We have 
\begin{align*}
G^{(s)}_{n}(m,r,s)=\sum_{j=1,2}\mathtt{MFib}(n;m,r,s)[j,1].
\end{align*}
\end{prop}

If we rewrite the transfer matrix by use of the function $f(x)$ defined 
in Definition \ref{def:f} in the following form:
\begin{align*}
T(m,r,s)
=\begin{pmatrix}
1 & 1 \\
f(rs)+rs-1 & f(rs)-1
\end{pmatrix}.
\end{align*} 
Then, the generating functions $\mathtt{MFib}(n;m,r,s)[i,j]$, which are the $(i,j)$-th entry of the matrix 
$\mathtt{MFib}(n;m,r,s)_{i,j}$ with $1\le i,j\le 2$,	 
are expressed in terms of $f(x)$.

\begin{prop}
\label{prop:MFib}
Let $I(n;\alpha,\beta)$ be the set of integer sequences defined by 
\begin{align*}
I(n;\alpha,\beta)
:=\left\{
\mathbf{i}:=(i_1,i_2,\ldots,i_{2m})\Bigg|
\begin{array}{c}
0\le m\le \lfloor(\beta-\alpha+1)/2 \rfloor \\
\alpha\le i_1<i_2<\ldots<i_{2m}\le\beta \\
i_{2p}=i_{2p-1}+1, \quad 1\le p\le m
\end{array}
\right\}.
\end{align*}
Then, given $\mathbf{i}\in I(n;\alpha,\beta)$, we define the weight of $\mathbf{i}$ by 
\begin{align*}
\mathrm{wt}(\mathbf{i})
:=\begin{cases}
\displaystyle r^{m}s^{L(\mathbf{i})}\prod_{j=1}^{2m-1}f(rs^{i_j})^{-1} & \text{if } \beta=n+1 \text{ and } i_{2m}=n+1, \\
\displaystyle r^{m}s^{L(\mathbf{i})}\prod_{j=1}^{2m}f(rs^{i_j})^{-1} & \text{otherwise},
\end{cases}
\end{align*}
where 
\begin{align*}
L(\mathbf{i}):=\sum_{p=1}^{m}i_{2p-1}.
\end{align*}
Then, the generating functions $\mathtt{MFib}(n;m,r,s)[i,j]$, $1\le i,j\le 2$, is given by
\begin{align*}
\mathtt{MFib}(n;m,r,s)[1,1]&=\prod_{j=2}^{n}f(rs^{j})\sum_{\mathbf{i}\in I(n;2,n+1)}\mathrm{wt}(\mathbf{i}), \\
\mathtt{MFib}(n;m,r,s)[1,2]&=\prod_{j=2}^{n}f(rs^{j})\sum_{\mathbf{i}\in I(n;2,n)}\mathrm{wt}(\mathbf{i}), \\
\mathtt{MFib}(n;m,r,s)[2,1]&=\prod_{j=1}^{n}f(rs^{j})\sum_{\mathbf{i}\in I(n;1,n+1)}\mathrm{wt}(\mathbf{i})-\mathtt{MFib}(n,m,r,s)[1,1] \\
\mathtt{MFib}(n;m,r,s)[2,2]&=\prod_{j=1}^{n}f(rs^{j})\sum_{\mathbf{i}\in I(n;1,n)}\mathrm{wt}(\mathbf{i})-\mathtt{MFib}(n,m,r,s)[1,2]. \\
\end{align*}
\end{prop}
\begin{proof}
From the definition (\ref{eqn:MFibdef}), we have a recurrence relation for $\mathtt{MFib}(n;m,r,s)$:
\begin{align*}
\mathtt{MFib}(n+1;m,r,s)=\mathtt{MFib}(n;m,r,s)T(m,rs^{n},s).
\end{align*}
Then, it is straightforward to check that the expressions for $\mathtt{MFib}(n;m,r,s)[i,j]$, $1\le i,j\le 2$, 
satisfy the recurrence relation.
This completes the proof.
\end{proof}

\begin{defn}
We define 
\begin{align*}
G'_{n}(m,r,s):=\sum_{j=1,2}\mathtt{MFib}(n;m,r,s)[j,2].
\end{align*}
\end{defn}

\begin{prop}
We define the set of integer sequences $I'_{n}(a,b)$ by
\begin{align*}
I'_{n}(a,b):=\left\{
\mathbf{i}=(i_1,\ldots,i_{2(a-1)})\Bigg|
\begin{array}{c}
1\le i_1<\ldots<i_{2(a-1)}\le n-1 \\
i_{2p}=i_{2p-1}+1, \quad 1\le p\le a-1 \\
i_1+i_3+\cdots+i_{2(a-1)-1}=(a-1)^{2}+b
\end{array}
\right\}.
\end{align*}
Then, the difference of $G_{n}:=G_{n}(m,r,s)$ and $G'_{n}:=G'_{n}(m,r,s)$ 
is expressed in terms of $f(x)$ as 
\begin{align*}
G_{n}-G'_{n}
=\left(\prod_{j=1}^{n-1}f(rs^{j})\right)
\left(
\sum_{a=1}^{\lfloor (n+1)/2\rfloor}
r^{a}s^{n+(a-1)^2}
\sum_{b=0}^{(a-1)(n-2a+1)}
\sum_{\mathbf{i}\in I'_{n}(a,b)}
s^{b}\prod_{j=1}^{2(a-1)}f(rs^{i_{j}})^{-1}\right).
\end{align*}
\end{prop}
\begin{proof}
From Proposition \ref{prop:MFib}, the difference $G_{n}-G'_{n}$
is equal to 
\begin{align}
\label{eqn:difGGdash}
\prod_{j=1}^{n}f(rs^{j})\left(\sum_{\mathbf{i}\in I(n;1,n+1)}\mathrm{wt}(\mathbf{i})-\sum_{\mathbf{i}\in I(n;1,n)}
\mathrm{wt}(\mathbf{i})\right).
\end{align}
The difference $I(n;1,n+1)\setminus I(n;1,n)$ is given by 
\begin{align*}
I(n;1,n+1)\setminus I(n;1,n)&=I(n;1,n-1), \\
&=\bigcup_{a=1}^{\lfloor(n+1)/2\rfloor}I'_{n}(a,b).
\end{align*}
Here, we have a bijection between $\mathbf{i}:=(i_1,\ldots,i_{2m+2})\in I(n;1,n+1)\setminus I(n;1,n)$ and 
$\mathbf{i}':=(i_1,\ldots,i_{2m})\in I(n;1,n-1)$ by the correspondence 
\begin{align*}
\mathbf{i}=\mathbf{i}'\cup\{i_{2m+1}=n, i_{2m}=n+1\}.
\end{align*}
Then, we obtain the desired expression by rearranging the terms in Eq. (\ref{eqn:difGGdash}).
\end{proof}

\subsection{Determinant expression of the generating function}
The recurrence relation in Theorem \ref{thrm:FibGfin} gives a determinant 
expression of the generating function.

\begin{prop}
\label{prop:Gsdet1}
Let $M(n):=(M_{i,j})_{1\le i,j\le n+1}$ be a tridiagonal matrix of rank $n+1$ such 
that 
\begin{align*}
M_{i,j}:=
\begin{cases}
-1, & \text{for } j=i+1, \\
f(rs^{n+1-i}), & \text{for } 1\le i=j\le n, \\
1, & \text{for } i=j=n+1, \\
rs^{n+1-j}, & \text{for } j=i-1, \\
0, & \text{otherwise}.
\end{cases}
\end{align*}
Then, the generating function $G_{n}^{(s)}(m,r,s)$, $0\le n$, can be expressed as 
\begin{align}
\label{eqn:GsdetM}
G^{(s)}_{n}(m,r,s)=\det(M(n)).
\end{align}
\end{prop}
\begin{proof}
We expand $\det(M(n))$ with respect to  the first row, and 
obtain the recurrence relation of Fibonacci type.
Further, when $n=1$, we have 
\begin{align*}
\det
\begin{bmatrix}
f(rs) & -1 \\
rs & 1
\end{bmatrix}&=f(rs)+rs\\ 
&=G_{1}^{(s)}(m,r,s),
\end{align*}
and $G_{0}^{(s)}(m,r,s)=1$. 
From these observations, we have Eq. (\ref{eqn:GsdetM}).
\end{proof}

The generating function possesses the following symmetry with respect 
to $r$ and $s$.
\begin{prop}
\label{prop:Gsflip}
We have
\begin{align*}
G_{n}^{(s)}(m,r,s)=G_{n}^{(s)}(m,rs^{n+1},s^{-1}).
\end{align*}
\end{prop}
\begin{proof}
We consider the mirror image of a dimer configuration on a segment.
Then, a dimer at position $i$ corresponds to a dimer at position $n+1-i$.
Since $r$ counts the number of dimers, the substitution $(r,s)\rightarrow(rs^{n+1},s^{-1})$
gives the weight $s^{n+1-i}$ to the dimer. This completes the proof.
\end{proof}

The following theorem is a direct consequence of Proposition \ref{prop:Gsgfrec}.
However, we deduce it from Theorem \ref{thrm:FibGfin} by use of 
Proposition \ref{prop:Gsflip}, which reflects the symmetry of the generating function.
\begin{theorem}
\label{thrm:Gsrrflip}
The generating function $G_{n}^{(s)}(m,r,s)$ satisfies the recurrence relation 
\begin{align*}
G_{n}^{(s)}(m,r,s)=f(rs)G_{n-1}^{(s)}(m,rs,s)+rsG_{n-2}^{(s)}(m,rs^2,s).
\end{align*}
\end{theorem}
\begin{proof}
By applying Proposition \ref{prop:Gsflip} to the determinant expression in 
Proposition \ref{prop:Gsdet1}, we obtain the recurrence relation.
\end{proof}

\subsection{Specializations}
As already shown in Theorem \ref{thrm:FibGfin}, 
the generating function satisfies the recurrence relation of Fibonacci type.
A specialization of parameters reveals the combinatorial structures behind 
the recurrence relation.

We consider the several specializations of the parameter $m$, $r$ and $s$.
Once we obtain a generating function with variables $m$, $r$ and $s$, 
we treat them as a formal variable.
We consider five  cases: A) $m=1$, B) $m\neq1$ and $s=1$, C) $m\rightarrow-mr^{-1}$ and $s=1$, 
D) $m=-1$, and E) $m=r$ and $s=1$. 

\subsubsection{Case A: \texorpdfstring{$m=1$}{m=1}}	
We start from the following simple observation.
\begin{prop}
\label{prop:FibF}
If we write the generating function $G(t,1,s,r):=1+\sum_{1\le n}F_{n}(r,s)t^{n}$, 
the polynomial $F_{n}(r,s)$ satisfies the following Fibonacci recurrence relation with two variables:
\begin{align*}
F_{n}(r,s)=F_{n-1}(r,s)+rs^{n}F_{n-2}(r,s),
\end{align*}
with initial conditions $F_{-1}(r,s)=F_{0}(r,s)=1$.
\end{prop}
\begin{proof}
Since $m=1$, we have $f(r)=1$.
The series $F_{n}(r,s)$ satisfies the recurrence relation by Theorem \ref{thrm:FibGfin}.
\end{proof}

First few values of $F_{n}(r,s)$ are 
\begin{align*}
F_{0}(r,s)&=1, \\
F_{1}(r,s)&=1+rs, \\
F_{2}(r,s)&=1+rs+rs^2,\\
F_{3}(r,s)&=1+rs+rs^2+rs^3+r^2s^4, \\
F_{4}(r,s)&=1+rs+rs^2+rs^3+(r+r^2)s^4+r^2s^5+r^2s^6,\\
F_{5}(r,s)&=1+rs+rs^2+rs^3+(r+r^2)s^4+(r+r^2)s^5+2r^2s^6+r^2s^7+r^2s^8+r^3s^9.
\end{align*}

\begin{prop}
The series $F_{n}(r):=F_{n}(r,s)$ satisfies the following functional equation:
\begin{align}
\label{eqn:rrF}
F_{n+1}(r)=F_{n}(rs)+rsF_{n-1}(rs^{2}).
\end{align}
\end{prop}
\begin{proof}
Since $m=1$, we have $f(x)=1$. 
The recurrence relation (\ref{eqn:rrF}) directly 
follows from Theorem \ref{thrm:Gsrrflip}.
\end{proof}

\paragraph{Specialization $m=1$ and $r=1$}
Let $w\in\{0,1\}^{\ast}$ be a binary word.
We call $w$ a {\it Fibonacci word} if there is no consecutive $0$'s in $w$.
Then, a dimer configuration on a segment $[0,n]$ is bijective to a Fibonacci word $w=w_1\ldots w_{n}$ of length $n$.
If there is a dimer at the position $i$, we define $w_{i}=0$.
Otherwise, we define $w_{i}=1$.
Thus, the specialization $F_{n}(1,s)$ gives the number of Fibonacci words of 
length $n$ such that the sum of the positions of zero's is equal to $k$
(see also the sequence A127836 in OEIS \cite{Slo}).

\paragraph{Specialization $m=1$ and $s=1$}
We denote the coefficient $f_{n,k}$ of $r^{k}$ in $F_{n}(r,1)$ by 
$f_{n,k}:=[r^k]F_{n}(r,1)$.
\begin{prop}
We have
\begin{align}
\label{eqn:fnkBino}
f_{n,k}=\genfrac{(}{)}{0pt}{}{n+1-k}{k},
\end{align}
with $0\le n$ and $0\le k\le \lfloor n/2\rfloor$.
\end{prop}
\begin{proof}
From Proposition \ref{prop:FibF}, the coefficients $f_{n,k}$ satisfy 
the recurrence relation $f_{n,k}=f_{n-1,k}+f_{n-2,k-1}$.
On the other hand, the binomials  in the right hand side of Eq. (\ref{eqn:fnkBino})
also satisfy the same recurrence relation and the initial conditions.
Thus we obtain Eq. (\ref{eqn:fnkBino}).
\end{proof}
The coefficients $f_{n,k}$ appear in the sequence A011973 in OEIS \cite{Slo}.

\subsubsection{Case B: \texorpdfstring{$n\neq1$}{n\neq1} and \texorpdfstring{$s=1$}{s=1}}
We consider the case $m\neq1$ and $s=1$.
The generating function is easily computed as 
\begin{align}
\label{eqn:GFs1}
G(t,m,r,s=1)=\genfrac{}{}{1pt}{}{1+rt}{1-(1-r+mr)t-rt^2}.
\end{align}
By a straightforward computation, one can verify the following functional 
relations for $G(t,m,r):=G(t,m,r,s=1)$.
\begin{prop}
\label{prop:Grel}
The generating function $G(t,m,r)$ satisfies
\begin{align}
\label{eqn:Grel1}
\left(G(t,m,r)\right)^2&=\partial_{t}G(t,m,r)-\genfrac{}{}{1pt}{}{r}{t}\partial_{r}G(t,m,r), \\
\label{eqn:Grel2}
\left(G(t,m,r)\right)^2&=\left(1+(rt)^{-1}\right)\partial_{m}G(t,m,r),\\
\label{eqn:Grel3}
G(t,m,r)&=(1-t)\partial_{t}G(t,m,r)-(r^2t+rt^{-1})\partial_{r}G(t,m,r), \\
\label{eqn:Grel4}
rtG(t,m,r)&=(1-(1-r+mr	)t-rt^2)\partial_{m}G(t,m,r),\\
\label{eqn:Grel5}
G(t,m,r)&=-\genfrac{}{}{1pt}{}{1}{t}G(t^{-1},m,-tr(1+r+tr)^{-1}), \\
\label{eqn:Grel6}
G(t,m,r)&=-\genfrac{}{}{1pt}{}{1}{t}G(t^{-1},-rm,r^{-1}).
\end{align}
\end{prop}

\begin{defn}
\label{defn:Fnkm}
We denote by $F^{n}_{k}(m):=[t^{n}r^{k}]G(t,m,1,r)$ the polynomial 
coefficient of $t^{n}r^{k}$ in $G(t,m,1,r)$.
\end{defn}

\begin{prop}
\label{prop:Gs1exp}
The polynomials $F^{n}_{k}(m)$ with $0\le n$ and $0\le k\le n$ satisfy
the following recurrence relation:
\begin{align}
\label{eqn:Frr1}
F^{n}_{k}(m)=\genfrac{}{}{1pt}{}{n}{n-k}F^{n-1}_{k}(m)
+\genfrac{}{}{1pt}{}{k-1}{n-k}F^{n-2}_{k-1}(m).
\end{align}
\end{prop}
\begin{proof}
By definition of $F^{n}_{k}$, we have $G(t,m,r)=\sum_{n,k}F^{n}_{k}t^{n}r^{k}$.
We substitute this expression into Eq. (\ref{eqn:Grel3}).
Since we have 
\begin{align*}
(1-t)\partial_{t}G(t,m,r)&=\sum_{n,k}(1-t)nF^{n}_{k}t^{n-1}r^{k}, \\
(r^2t+rt^{-1})\partial_{r}G(t,m,r)
&=\sum_{n,k}\left(kF^{n}_{k}t^{n+1}r^{k+1}+kF^{n}_{k}t^{n-1}r^{k}\right),
\end{align*} 
the coefficients of $t^{n}r^{k}$ in Eq. (\ref{eqn:Grel3}) give
\begin{align*}
F^{n}_{k}=(n+1)F^{n+1}_{k}-nF^{n}_{k}-(k-1)F^{n-1}_{k-1}-kF^{n+1}_{k}.
\end{align*}
By rearranging the terms, we obtain the recurrence relation (\ref{eqn:Frr1}).
\end{proof}

We solve the recurrence relation (\ref{eqn:Frr1}) to obtain the analytic 
expression of a polynomial $F^{n}_{k}$. 
\begin{prop}
\label{prop:defF}
We have 
\begin{align}
\label{eqn:defF}
\begin{split}
F^{n}_{n}(m)&=m(m-1)^{n-1}, \\
F^{n}_{n-k}(m)&=
\frac{\partial_{m}^{k}}{k!}\left(m^{k}F_{n-k}^{n-k}(m)\right),
\end{split}
\end{align}
where $0\le n$ and $0\le k\le n$ with the initial condition 
$F^{n}_{0}(m)=1$.
\end{prop}
\begin{proof}
It is enough to show that the expressions (\ref{eqn:defF}) of $F^{n}_{k}$
satisfy the recurrence relation (\ref{eqn:Frr1}).
We have 
\begin{align*}
\partial_{m}\left(m^{r}m(m-1)^{n-r-1}\right)=(1+r)m^{r}(m-1)^{n-r-1}
+(n-r-1)m^{r+1}(m-1)^{n-r-2},
\end{align*}
and $m^{r+1}(m-1)^{n-r-2}=m^r(m-1)^{n-r-1}+m^{r}(m-1)^{n-r-2}$.
Substituting these equations into Eq. (\ref{eqn:defF}), 
we obtain the desired recurrence relation.
\end{proof}

First few values of $F^{n}_{r}(m)$ are 
\begin{align*}
\begin{array}{c|ccccc}
n\backslash r & 0 & 1 & 2 & 3 & 4\\ \hline
0 & 1  \\
1 & 1 & m \\
2 & 1 & 2m & m(m-1)\\
3 & 1 & 3m & 3m^2-2m & m(m-1)^{2} \\
4 & 1 & 4m & 3m(2m-1) & 2m(m-1)(2m-1) & m(m-1)^{3}
\end{array}
\end{align*}

Further specialization $r=1$ gives a summation formula for $F^{n}_{k}(m)$. 
\begin{prop}
\label{prop:Frel1}
We have 
\begin{align}
\label{eqn:Frel1}
[m^{k}]\left(\sum_{0\le k\le n}F^{n}_{n-k}(m)\right)=\genfrac{(}{)}{0pt}{}{\lfloor \frac{n+k}{2}\rfloor}{k}.
\end{align}
\end{prop}
\begin{proof}
We define the coefficient $\widetilde{F}^{n}_{p}:=[t^nm^p]G(t,m,1)$.
From Eq. (\ref{eqn:Grel4}) with $r=1$, we have 
\begin{align}
\label{eqn:Grel21}
tG(t,m,1)=(1-mt-t^2)\partial_{m}G(t,m,1).
\end{align}
Since $\partial_{m}G(t,m,1)=\sum_{n,p}p\widetilde{F}^{n}_{p}t^nm^{p-1}$, 
the coefficients of $t^{n}m^{p}$ in Eq. (\ref{eqn:Grel21}) satisfy 
\begin{align}
\label{eqn:rrforF}
\widetilde{F}^{n}_{p+1}=\widetilde{F}^{n-1}_{p}+\widetilde{F}^{n-2}_{p+1}.
\end{align}
From Eq. (\ref{eqn:defF}), one can easily show that $\widetilde{F}^{m}_{m}=1$. 
The right hand side of Eq. (\ref{eqn:Frel1})
also satisfies the recurrence relation (\ref{eqn:rrforF}) by use of the 
recurrence relation for the binomial coefficients. 
This completes the proof.
\end{proof}

These coefficients $\widetilde{F}^{n}_{p}$ correspond to the sequence A046854 in OEIS \cite{Slo}.

Proposition \ref{prop:Grel} gives several recurrence relations for 
$F^{n}_{k}(m)$ and $\widetilde{F}^{n}_{p}:=[t^nm^p]G(t,m,1)$.
For example, from Eq. (\ref{eqn:Grel2}), we have 
\begin{align*}
\sum_{\genfrac{}{}{0pt}{}{m+m'=n}{q+q'=p}}\widetilde{F}^{m}_{q}\widetilde{F}^{m'}_{q'}
=(p+1)\left(\widetilde{F}^{n}_{p+1}+\widetilde{F}^{n+1}_{p+1}\right).
\end{align*}

In Proposition \ref{prop:Gs1exp}, we consider the coefficient of $t^{n}r^{k}$.
Below, we consider the coefficient of $t^{n}m^{k}$ in the generating function.
\begin{prop}
\label{prop:Gss1}
We have
\begin{align}
\label{eqn:Gss1exp}
G_{n}^{(s)}(m,r,s=1)
=\sum_{l=0}^{n}m^{l}\sum_{k=l}^{n}
(-1)^{k-l}r^{k}\genfrac{(}{)}{0pt}{}{n-k+l}{l}\genfrac{(}{)}{0pt}{}{k-1}{l-1}.
\end{align}
\end{prop}
\begin{proof}
We show that Eq. (\ref{eqn:Gss1exp}) satisfies the recurrence relation 
in Theorem \ref{thrm:FibGfin} with $s=1$.
We compare the coefficients of $m^{l}$, $1\le n\le n$, between the right hand side and the left hand 
side of $G_{n}^{(s)}=(1+r(m-1))G_{n-1}^{(s)}+rG_{n-2}^{(s)}$. 
The recurrence relation is equivalent to 
\begin{align*}
&\sum_{k=l}^{n}(-1)^{k-l}r^{k}\genfrac{(}{)}{0pt}{}{n-k+l}{l}\genfrac{(}{)}{0pt}{}{k-1}{l-1} \\
&\quad=\sum_{k=l}^{n-1}(-1)^{k-l}r^{k}\genfrac{(}{)}{0pt}{}{n-k+l-1}{l}\genfrac{(}{)}{0pt}{}{k-1}{l-1}
+\sum_{k=l-1}^{n-1}(-1)^{k-l-1}r^{k+1}\genfrac{(}{)}{0pt}{}{n-k+l-2}{l-1}\genfrac{(}{)}{0pt}{}{k-1}{l-2} \\
&\quad-\sum_{k=l}^{n-1}(-1)^{k-l}r^{k+1}\genfrac{(}{)}{0pt}{}{n-k+l-1}{l}\genfrac{(}{)}{0pt}{}{k-1}{l-1}
+\sum_{k=l}^{n-2}(-1)^{k-l}r^{k+1}\genfrac{(}{)}{0pt}{}{n-k+l-2}{l}\genfrac{(}{)}{0pt}{}{k-1}{l-1}.
\end{align*}
Further, the coefficients of $r^{k}$, $l\le k\le n$, should satisfy 
\begin{align*}
\genfrac{(}{)}{0pt}{}{n-k+l}{l}\genfrac{(}{)}{0pt}{}{k-1}{l-1}
=\genfrac{(}{)}{0pt}{}{n-k+l-1}{l}\genfrac{(}{)}{0pt}{}{k-1}{l-1}
+\genfrac{(}{)}{0pt}{}{n-k+l-1}{l-1}\genfrac{(}{)}{0pt}{}{k-2}{l-2} \\
+\genfrac{(}{)}{0pt}{}{n-k+l}{l}\genfrac{(}{)}{0pt}{}{k-2}{l-1}
-\genfrac{(}{)}{0pt}{}{n-k+l-1}{l}\genfrac{(}{)}{0pt}{}{k-2}{l-1}.
\end{align*}
One can easily verify this relation by use of 
$\genfrac{(}{)}{0pt}{}{N}{M}=\genfrac{(}{)}{0pt}{}{N-1}{M}+\genfrac{(}{)}{0pt}{}{N-1}{M-1}$.
This completes the proof.
\end{proof}

\subsubsection{Case C: \texorpdfstring{$G(t,m',r,s)$}{G} with \texorpdfstring{$m'=-mr^{-1}$}{m'=mr^-1} and 
\texorpdfstring{$s=1$}{s=1}}
The expression (\ref{eqn:GFs1}) satisfies
\begin{align*}
G(t,-mr^{-1},r,s=1)=G(t,m,-(1+t+tr)^{-1},s=1)=\genfrac{}{}{1pt}{}{1+rt}{1+(-1+m+r)t-rt^{2}}.
\end{align*}
We denote the coefficient of $t^{n}r^{k}$ in $G(t,-mr^{-1},r,s=1)$ by 
\begin{align*}
\widetilde{G}^{n}_{k}(m):=[t^{n}r^{k}]G(t,-mr^{-1},r,s=1).
\end{align*}

To compute $\widetilde{G}^{n}_{k}(m)$, we first consider the expansion of the following
generating function: 
\begin{align*}
H(t,m,r):=(1+(-1+m+r)t-rt^{2})^{-1}=:\sum_{n,k}H^{n}_{k}(m)t^{n}r^{k}.
\end{align*}

The polynomial $H^{n}_{k}(m)$ can be expressed as derivatives of 
a simple polynomial of $m$.
\begin{prop}
we have 
\begin{align}
\label{eqn:defGcH}
\begin{split}
H^{n}_{0}(m)&=(1-m)^{n}, \\
H^{n}_{k}(m)&=\genfrac{}{}{1pt}{}{\partial_{m}^{k}}{k!}
\left((-m)^{k}H^{n-k}_{0}(m)\right),
\end{split}
\end{align}
where $0\le n$ and $0\le k\le n$ with the initial condition $H^{0}_{0}(m)=1$.
\end{prop}
\begin{proof}
By a straightforward calculation, one can show that the generating function $H(t,m,r)$ 
satisfies
\begin{align*}
((t-t^2)+t^2(1-t)\partial_{t}+(1+t^2r)\partial_{r})H(t,m,r)=0.
\end{align*}
Thus, the coefficients $H^{n}_{k}(m)$ satisfy the recurrence relation
\begin{align}
\label{eqn:recH}
H^{n}_{k}(m)
=-\genfrac{}{}{1pt}{}{n}{k}H^{n-1}_{k-1}(m)+\genfrac{}{}{1pt}{}{n-k}{k}H^{n-2}_{k-1}(m).
\end{align}
Note that we have 
\begin{align*}
\genfrac{}{}{1pt}{}{1}{k}\partial_{m}\left((-m)^{k}(1-m)^{n-k}\right)
&=-(-m)^{k-1}(1-m)^{n-k}-\genfrac{}{}{1pt}{}{n-k}{k}(-m)^{k}(1-m)^{n-k-1}, \\
&=-(-m)^{k-1}H^{n-k}_{0}(m) \\
&\quad-\genfrac{}{}{1pt}{}{n-k}{k}\left(-(-m)^{k-1}H^{n-k-1}_{0}(m)+(-m)^{k-1}H^{n-k}_{0}(m)\right),
\end{align*}
where we have used $-m=(-1)+(1-m)$.
From Eq. (\ref{eqn:defGcH}), the above equation is equivalent to 
the recurrence relation (\ref{eqn:recH}) with the initial conditions $H^{n}_{0}=(1-m)^{n}$.
This completes the proof.
\end{proof}

By use of the expression (\ref{eqn:defGcH}) of $H^{n}_{k}(m)$, 
we will show that the polynomials $\widetilde{G}^{n}_{k}(m)$ have simple explicit formulae.
\begin{prop}
\label{prop:Gtild}
We have 
\begin{align*}
\widetilde{G}^{n}_{k}(m):=\sum_{l=0}^{n-k}(-1)^{l+k}\genfrac{}{}{1pt}{}{\widetilde{g}^{l}_{k}}{k!}
\genfrac{(}{)}{0pt}{}{n-k}{l}m^{l},
\end{align*}
where 
\begin{align*}
\widetilde{g}^{l}_{k}:=
\begin{cases}
1, & k=0, \\
\prod_{j=0}^{k-1}(l+j), & k\neq0.
\end{cases}
\end{align*}
\end{prop}
\begin{proof}
Since we have $G(t,-mr^{-1},r)=(1+rt)H(t,m,r)$, a polynomial $\widetilde{G}^{n}_{k}(m)$ 
can be expressed in terms of $H^{n}_{k}(m)$ by 
\begin{align*}
\widetilde{G}^{n}_{k}(m)=
\begin{cases}
H^{n}_{k}(m)+H^{n-1}_{k-1}(m), & k\neq0, \\
H^{n}_{0}, & k=0.	
\end{cases}
\end{align*}
For $k\neq0$, we have 
\begin{align*}
\widetilde{G}^{n}_{k}(m)
&=H^{n}_{k}(m)+H^{n-1}_{k-1}(m), \\
&=\genfrac{}{}{1pt}{}{\partial_{m}^{k}}{k!}\left((-m)^{k}(1-m)^{n-k}\right)
+\genfrac{}{}{1pt}{}{\partial_{m}^{k-1}}{(k-1)!}\left((-m)^{k-1}(1-m)^{n-k}\right), \\
&=\genfrac{}{}{1pt}{}{n-k}{k}\genfrac{}{}{1pt}{}{\partial_{m}^{k-1}}{(k-1)!}
\left((-m)^{k}(1-m)^{n-k-1}\right), \\
&=\genfrac{}{}{1pt}{}{n-k}{k}\genfrac{}{}{1pt}{}{\partial_{m}^{k-1}}{(k-1)!}
\left(\sum_{l=0}^{n-k-1}(-1)^{l+k}\genfrac{(}{)}{0pt}{}{n-k-1}{l}m^{k+l}\right), \\
&=\sum_{l=0}^{n+k-1}(-1)^{l+k}\genfrac{}{}{1pt}{}{n-k}{k}\left(\prod_{j=1}^{k-1}\genfrac{}{}{1pt}{}{k+l-j+1}{k-j}\right)
\genfrac{(}{)}{0pt}{}{n-k-1}{l}m^{l+1}, \\
&=\sum_{l=0}^{n-k}(-1)^{l+k}\genfrac{}{}{1pt}{}{\prod_{j=0}^{k-1}(l+j)}{k!}\genfrac{(}{)}{0pt}{}{n-k}{l}m^{l}.
\end{align*}
For $k=0$, from Eq. (\ref{eqn:defH}), we have 
\begin{align*}
\widetilde{G}^{n}_{0}(m)=\sum_{l=0}^{n}(-1)^{l}\genfrac{(}{)}{0pt}{}{n}{l}m^{l}.
\end{align*}
This completes the proof.
\end{proof}

The sum of $\widetilde{G}^{n}_{k}(m)$ with respect to $k$ has also 
a simple expression in terms of binomial coefficients.
\begin{prop}
\label{prop:sumGtild}
We have 
\begin{align}
\label{eqn:sumGtildp}
\sum_{k=0}^{n}\widetilde{G}^{n}_{k}(m)
=\sum_{0\le p\le n}(-1)^{p}\genfrac{[}{]}{0pt}{}{n-\lfloor(n-p+1)/2\rfloor}{\lfloor (n-p)/2\rfloor}m^{p}.
\end{align}
If we denote by $\widetilde{F}_n(m)$ the above sum, the polynomials $\widetilde{F}_{n}(m)$ 
satisfy the recurrence relation of Fibonacci type:
\begin{align}
\label{eqn:recFtild}
\widetilde{F}_{n}(m)=(-m)\widetilde{F}_{n-1}(m)+\widetilde{F}_{n-2}(m),
\end{align}
with the initial conditions $\widetilde{F}_{-1}(m)=\widetilde{F}_{0}(m)=1$.
\end{prop}
\begin{proof}
We show that $\widetilde{F}_{n}(m):=\sum_{0\le k\le n}\widetilde{G}^{n}_{k}(m)$
satisfies the recurrence relation (\ref{eqn:recFtild})
From Proposition \ref{prop:Gtild}, the difference of coefficients of $(-1)^{r}m^{r}$ in 
Eq. (\ref{eqn:recFtild}) is computed as 
\begin{align*}
&\sum_{k=0}^{n-r}(-1)^{k}\left(\genfrac{(}{)}{0pt}{}{n-k}{r}\genfrac{}{}{}{}{(r)_{k}}{k!}
-\genfrac{(}{)}{0pt}{}{n-k-1}{r-1}\genfrac{}{}{}{}{(r-1)_{k}}{k!}
\right)
-\sum_{k=0}^{n-r-2}(-1)^{k}\genfrac{(}{)}{0pt}{}{n-k-2}{r}\genfrac{}{}{}{}{(r)_{k}}{k!} \\
&=\genfrac{(}{)}{0pt}{}{n}{r}-\genfrac{(}{)}{0pt}{}{n-1}{r-1}
+(-1)^{n-r}\left( \genfrac{}{}{}{}{(r)_{n-r}}{(n-r)!}-\genfrac{}{}{}{}{(r-1)_{n-r}}{(n-r)!}\right) \\
&\quad +\sum_{k=1}^{n-r-1}(-1)^{k}\left(\genfrac{(}{)}{0pt}{}{n-k}{r}\genfrac{}{}{}{}{(r)_{k}}{k!}
-\genfrac{(}{)}{0pt}{}{n-k-1}{r-1}\genfrac{}{}{}{}{(r-1)_{k}}{k!}
+\genfrac{(}{)}{0pt}{}{n-k-1}{r}\genfrac{}{}{}{}{(r)_{k-1}}{(k-1)!}\right), \\
&=\genfrac{(}{)}{0pt}{}{n-1}{r}+(-1)^{n-r}\genfrac{}{}{}{}{(n-2)!}{(n-r-1)!(r-1)!}
-((n-1)+(-1)^{n-r}r)\genfrac{}{}{}{}{(n-2)!}{r!(n-r-1)!}, \\
&=0.
\end{align*}
By definition, we have $\widetilde{F}_{0}(m)=1$ and $\widetilde{F}_{1}(m)=1-m$, which is 
equivalent to the initial conditions $\widetilde{F}_{-1}(m)=\widetilde{F}_{0}(m)=1$.

It is straightforward to show that the right hand side of Eq. (\ref{eqn:sumGtildp})
satisfies the recurrence relation Eq. (\ref{eqn:recFtild}).
This completes the proof.
\end{proof}

\begin{remark}
Some remarks are in order.
\begin{enumerate}
\item
The coefficients of Proposition \ref{prop:sumGtild} are expressed in terms 
of the hypergeometric function.
By the definition of $\widetilde{F}_{n}(m)$ and Proposition \ref{prop:Gtild}, we have
\begin{align*}
[m^{r}]\widetilde{F}_{n}(m)&=\sum_{k=0}^{n-r}(-1)^{r+k}\genfrac{(}{)}{0pt}{}{n-k}{r}\genfrac{}{}{}{}{(r)_{k}}{k!}, \\
&=(-1)^{r}\genfrac{(}{)}{0pt}{}{n}{r}\ _{2}\mathcal{F}_{1}\left[\{r,r-n\},\{-n\},-1\right].
\end{align*}
\item 
The generating function satisfies 
\begin{align*}
\sum_{0\le n}\widetilde{F}_{n}(m)t^{n}=\genfrac{}{}{}{}{1+t}{1+mt-t^2}.
\end{align*}
\item 
First few polynomials $\widetilde{F}_{n}(m)$ are 
\begin{align*}
\widetilde{F}_{0}(m)&=1, \qquad \widetilde{F}_{1}(m)=1-m, \\
\widetilde{F}_{2}(m)&=1-m+m^2, \\
\widetilde{F}_{3}(m)&=1-2m+m^2-m^3, \\
\widetilde{F}_{4}(m)&=1-2m+3m^2-m^3+m^4, \\
\widetilde{F}_{5}(m)&=1-3m+3m^2-4m^3+m^4-m^5. 
\end{align*}
These polynomials $\widetilde{F}_{n}(m)$ corresponds to the alternating sign of A065941 in OEIS \cite{Slo}.
\end{enumerate}
\end{remark}

In Proposition \ref{prop:sumGtild}, we compute the sum of $\widetilde{G}^{n}_{k}(m)$.
The next proposition shows that the sum of $\widetilde{G}^{n}_{k}(m)$ with alternating signs 
also has a simple formula.
\begin{prop}
\label{prop:CaseCsumGm}
We have 
\begin{align*}
\sum_{k=0}^{n}(-1)^{k}\widetilde{G}^{n}_{k}(m)=\sum_{k=0}^{n}(-1)^{k}\genfrac{(}{)}{0pt}{}{n+k}{2k}m^{k}.
\end{align*}
\end{prop}
Before proceeding to the proof, we introduce a lemma involving binomial coefficients.
\begin{lemma}
\label{lemma:CaseCbinomial}
Let $1\le n$, $1\le p\le n$, and $0\le x\le p$ be non-negative integers.
We have 
\begin{align}
\label{eqn:CaseCbinomial}
\sum_{k=0}^{n-p+x}\genfrac{(}{)}{0pt}{}{p+k-1}{k}\genfrac{(}{)}{0pt}{}{n-k}{p-x}=\genfrac{(}{)}{0pt}{}{n+p}{2p-x}.
\end{align}
\end{lemma}
\begin{proof}
Let $A(n,p,x)$ be the difference between the left hand side and the right hand side of Eq. (\ref{eqn:CaseCbinomial}).
By a straightforward calculation, we have 
\begin{align*}
A(n+1,p,x)-A(n,p,x)=A(n,p,x+1).
\end{align*}
By definition, when $p$ is sufficiently large, $A(n,p,x)=0$. 
Similarly, when $n=1$, it is easy to verify $A(1,p,x)=0$ for $p=1$ and $0\le x\le 1$.
Then, by induction on $x$, we have $A(n,p,x)=0$ for general $n$, $p$ and $x$.
\end{proof}

\begin{proof}[Proof of Proposition \ref{prop:CaseCsumGm}]
From Proposition \ref{prop:Gtild}, we have 
\begin{align*}
\sum_{k=0}^{n}(-1)^{k}\widetilde{G}^{n}_{k}(m)
=\sum_{k=0}^{n}\sum_{l=0}^{n-k}(-1)^{l}\genfrac{}{}{}{}{(l)_{k}}{k!}\genfrac{(}{)}{0pt}{}{n-k}{l}m^{l}.
\end{align*}
The coefficient of $m^{p}$ in the above expression is 
\begin{align*}
\sum_{k=0}^{n-p}(-1)^{p}\genfrac{}{}{}{}{(p)_{k}}{k!}\genfrac{(}{)}{0pt}{}{n-k}{p}
&=(-1)^{p}\sum_{k=0}^{n-p}\genfrac{(}{)}{0pt}{}{p+k-1}{k}\genfrac{(}{)}{0pt}{}{n-k}{p}, \\
&=(-1)^{p}\genfrac{(}{)}{0pt}{}{n+p}{2p},
\end{align*}
where we have used Lemma \ref{lemma:CaseCbinomial} with $x=0$.
This completes the proof.
\end{proof}

\subsubsection{Case D: \texorpdfstring{$m=-1$}{m=-1}}
By regarding $m$ as a formal variable, we consider the 
generating function $G(t,m=-1,-r,s=1):=\sum_{0\le n}R_{n}(r)t^{n}$.

\begin{prop}
The polynomials $R_{n}(r)$, $0\le n$, satisfy the following 
recurrence relation:
\begin{align}
\label{eqn:Rrec}
R_{n}(r)=(2r+1)R_{n-1}(r)-rR_{n-2}(r),
\end{align}
with initial conditions $R_{-1}(r)=R_{0}(r)=1$.
\end{prop}
\begin{proof}
From Theorem (\ref{thrm:FibGfin}), we have Eq. (\ref{eqn:Rrec}) with 
the specialization $(m,r,s)=(-1,-r,1)$.
\end{proof}

\begin{prop}
If we expand $R_{n}(r)=\sum_{k=0}^{n}c^{n}_{k}r^{k}$, we have 
\begin{align*}
c^{n}_k=[r^{k}]\left(\left(\genfrac{}{}{1pt}{}{1-r}{1-2r}\right)^{n-k+1}\right).
\end{align*}
\end{prop}
\begin{proof}
From Eq. (\ref{eqn:Rrec}), the coefficients $c^{n}_{k}$ should satisfy 
\begin{align}
\label{eqn:recc1}
c^{n}_{k}=2c^{n-1}_{k-1}+c^{n-1}_{k}-c^{n-2}_{k-1}.
\end{align}
On the other hand, 
\begin{align}
\label{eqn:recc2}
\begin{split}
c^{n}_{k}&=[r^{k}]\left(\left(\genfrac{}{}{1pt}{}{1-r}{1-2r}\right)^{n-k+1}\right), \\
&=\sum_{l=0}^{k}[r^{l}]\left(\genfrac{(}{)}{}{}{1-r}{1-2r}^{n-k}\right)\cdot[r^{k-l}]\genfrac{(}{)}{}{}{1-r}{1-2r}, \\
&=[r^{k}]\left(\genfrac{(}{)}{}{}{1-r}{1-2r}^{n-k}\right)
+\sum_{l=0}^{k-1}2^{k-l-1}[r^{l}]\genfrac{(}{)}{}{}{1-r}{1-2r}^{n-k}, \\
&=c^{n-1}_{k}+\sum_{l=0}^{k-1}2^{k-l-1}c^{n-k-1+l}_{l},
\end{split}
\end{align}
where we have used 
\begin{align*}
\genfrac{(}{)}{}{}{1-r}{1-2r}=1+\sum_{1\le k}2^{k-1}r^{k}.
\end{align*}
By using Eq. (\ref{eqn:recc2}) to compute $c^{n}_{k}-2c^{n-1}_{k-1}$, 
we obtain Eq. (\ref{eqn:recc1}) with initial conditions
$c^{n}_{0}=1$.
\end{proof}

First few expressions of $R_{n}(r)$ are 
\begin{align*}
\begin{array}{c|c}
n & R_{n}(r) \\ \hline
0 & 1 \\
1 & 1+r \\
2 & 1+2r+2r^2 \\
3 & 1+3r+5r^2+4r^3 \\
4 & 1+4r+9r^2+12r^3+8r^4 \\
5 & 1+5r+14r^2+25r^3+28r^4+16r^5
\end{array}
\end{align*}
The coefficients $c^{n}_{k}$ appears as the sequence A160232 in OEIS \cite{Slo}.

\begin{cor}
\label{cor:Fibodd}
The specialization $r=1$ of $R_{n}(r)$ is expressed in terms of 
a Fibonacci number: $R_{n}(1)=F_{2n-1}$ (see also A001519 in OEIS \cite{Slo}).
\end{cor}
\begin{proof}
If we set $r=1$ in Eq. (\ref{eqn:Rrec}), we have 
\begin{align*}
R_{n}(1)=3R_{n-1}(1)-R_{n-2}(1).
\end{align*}
The Fibonacci number $F_{k}$, $0\ge1$, satisfies 
\begin{align*}
F_{2n-1}&=F_{2n-2}+F_{2n-3}, \\
&=2F_{2n-3}+F_{2n-4}, \\
&=3F_{2n-3}-F_{2n-5}.
\end{align*} 
Since we have the same initial conditions $R_{0}(1)=1=F_{-1}$ and $R_{1}(1)=2=F_{1}$.
Thus we have $R_{n}(1)=F_{2n-1}$, which completes the proof.
\end{proof}

\subsubsection{Case E: \texorpdfstring{$m=r$}{m=r} and \texorpdfstring{$s=1$}{s=1}}
When $r=m$ and $s=1$, the generating function $G(t,m,m):=\sum_{n,k}M^{n}_{k}t^{n}m^{l}$ is easily computed 
as 
\begin{align}
\label{eqn:Gmm}
\begin{split}
G(t,m,1,m)&=\genfrac{}{}{1pt}{}{1+mt}{1-t+mt-m^2t-mt^2}, \\
&=\left(1-t-\genfrac{}{}{1pt}{}{m^2t}{1+mt}\right)^{-1}, \\
&=\left(1-(1+m^2)t+\sum_{2\le k}(-1)^{k}m^{k+1}t^{k}\right)^{-1}.
\end{split}
\end{align}

Given $n\ge1$, we denote by $\mathbf{m}:=(m_1,m_2,\ldots,m_{n})$ a sequence 
of non-negative integers. 
Let $\mathcal{P}(n)$ be the set of sequences $\mathbf{m}$ such that 
$\sum_{i=1}^{n}i\cdot m_{i}=n$.

Then, the coefficient $M^{n}_{k}$ is expressed as 
\begin{align*}
M^{n}_{k}&=[x^{k}]\left(
\sum_{\mathbf{m}\in\mathcal{P}(n)}(1+x^2)^{m_1}(-x^3)^{m_2}(x^4)^{m_3}\cdots((-1)^{n-1}x^{n+1})^{m_n}
\right), \\
&=(-1)^{k}
\sum_{\mathbf{m}\in\mathcal{P}(n)}
\sum_{0\le l\le m_1}\genfrac{(}{)}{0pt}{}{m_1}{l}
\cdot
\#\left\{\mathbf{m}\Big| 
2l+\sum_{j=2}^{n}(j+1)m_{j}=k
\right\}.
\end{align*}

Below, we give an expression of $M^{n}_{k}$ in terms of generalized hypergeometric functions.
Recall that $F^{n}_{k}(m)$ is defined through the generating function $G(t,m,r,s=1)$. 
We have 
\begin{align*}
M^{n}_{k}=[m^{k}]\left(\sum_{k=0}^{n}F^{n}_{k}(m)m^{k}\right).
\end{align*}
From Eq. (\ref{eqn:defF}), we have 
\begin{align}
\label{eqn:Mdef}
\begin{split}
M^{n}_{p}&=[m^{p}]\left( \sum_{k=0}^{n-1}m^{n-k}F^{n}_{n-k}(m)\right), \\
&=[m^{p}]\left(1+\sum_{k=0}^{n-1}\sum_{l=0}^{n-k-1}(-1)^{n-k-l-1}\genfrac{(}{)}{0pt}{}{n-k-1}{l}
\genfrac{}{}{}{}{(l+2)_{k}}{k!}m^{l+n-k+1}
\right).
\end{split}
\end{align}
We consider two cases: 1) $p$ is even, and 2) $p$ is odd.

\paragraph{\bf Case 1).}
We rewrite $p$ as $p=2(n-r)$ with $0\le n$ and $0\le r\le n$.
We have three cases: a) $0\le r\le \lfloor (n-1)/2\rfloor$, 
b) $r':=n-r$ and $0\le r'\le \lfloor (n+1)/2\rfloor$, 
and c) $r=n/2$ for $n$ even.

\paragraph{\bf Case 1a).}
We have $l+n-k+1=2(n-r)$, which gives $l=n-2r+k-1$.
The possible pairs of $(k,l)$ are 
\begin{align*}
(0,n-2r-1), (1, n-2r), \ldots, (r+1,n-r).
\end{align*}
Thus, from Eq. (\ref{eqn:Mdef}), we have 
\begin{align*}
M^{n}_{2(n-r)}
&=\sum_{k=0}^{r+1}(-1)^{2(r-k)}\genfrac{(}{)}{0pt}{}{n-k-1}{n-2r+k-1}
\genfrac{}{}{}{}{(n-2r+k+1)_{k}}{k!}, \\
&=(-1)^{2r}\genfrac{}{}{}{}{(n-1)!}{(n-2r-1)!(2r)!} 
\ _{4}\mathcal{F}_{3}\left[\mathbf{a}_{1},\mathbf{b}_{1},-16\right],
\end{align*}
where 
\begin{align*}
\mathbf{a}_{1}&=\left\{\genfrac{}{}{}{}{1-2r}{2},\genfrac{}{}{}{}{n-2r+1}{2},\genfrac{}{}{}{}{n+2-2r}{2},-r\right\}, \\
\mathbf{b}_{1}&=\{1-n,n-2r,n+1-2r\}.
\end{align*}

\paragraph{\bf Case 1b).}
We have $2r'=l+n-k+1$.
The possible pairs of $(k,l)$ are 
\begin{align*}
(n-2r'+1,0), (n-2r'+2,1),\ldots,(n-r',r'-1).
\end{align*}
Thus, we obtain
\begin{align*}
M^{n}_{2r'}&=\sum_{l=0}^{r'-1}(-1)^{2(r'-l-1)}\genfrac{(}{)}{0pt}{}{2r'-l-2}{l}
\genfrac{}{}{}{}{(l+1)_{n+l+1-2r'}}{(n+l+1-2r')!}, \\
&=(-1)^{2r'} \ _{4}\mathcal{F}_{3}\left[{\mathbf{a}_2,\mathbf{b}_{2}},-16\right],
\end{align*}
where 
\begin{align*}
\mathbf{a}_{2}&=\left\{1-r,\genfrac{}{}{}{}{3-2r}{2},\genfrac{}{}{}{}{n-2r+2}{2},\genfrac{}{}{}{}{3+n-2r}{2}\right\}, \\
\mathbf{b}_{2}&=\{1,2-2r,n+2-2r\}.
\end{align*}

\paragraph{\bf Case 1c).}
We consider the case $p=n$.
From Eq. (\ref{eqn:Mdef}), we have $l+n-k+1=n$.
The possible pairs of $(k,l)$ are $(k,l)=(k,k-1)$ with 
$0\le k\le n/2$. 
Note that $n$ is even.
\begin{align*}
M^{n}_{n}&=\sum_{1\le k\le n/2}(-1)^{n-2k}\genfrac{(}{)}{0pt}{}{n-k-1}{k-1}
\genfrac{}{}{}{}{(k+1)_{k}}{k!}, \\
&=2 \ _{3}\mathcal{F}_{2}\left[\left\{\genfrac{}{}{}{}{3}{2},\genfrac{}{}{}{}{2-n}{2},\genfrac{}{}{}{}{3-n}{2}\right\},
\{2,2-n\}, -16\right].
\end{align*}

\paragraph{\bf Case 2).}
We write $p=2(n-r)-1$.
We have three cases: a) $p=1$, b) $0\le r\le \lfloor (n-2)/2\rfloor$, and 
c) $r':=n-r$ and $2\le r'\le \lfloor (n+1)/2\rfloor$. 

\paragraph{\bf Case 2a).}
From Eq. (\ref{eqn:Mdef}), we have 
\begin{align*}
p=l+n-k+1\ge 2,
\end{align*}
since $0\le k \le n-1$ and $0\le l\le n-k-1$.
Thus we have $M^{n}_{1}=0$.

\paragraph{\bf Case 2b).}
We have $2(n-r)-1=l+n-k+1$, which gives $l=n-2r+k-2$.
The possible pairs of $(k,l)$ are 
\begin{align*}
(0,n-2r-2), (1,n-2r-1), \ldots, (r,n-r-1).
\end{align*}
Thus, we have 
\begin{align*}
M^{n}_{2(n-r)-1}&=\sum_{k=0}^{r}(-1)^{2(r-k)+1}\genfrac{(}{)}{0pt}{}{n-k-1}{n-2r+k-2}
\genfrac{}{}{}{}{(n-2r+k)_{k}}{k!}, \\
&=(-1)^{2r-1}\genfrac{}{}{}{}{(n-1)!}{(n-2r-2)!(2r+1)!}
\ _{4}\mathcal{F}_{3}\left[\mathbf{a}_{3},\mathbf{b}_{3},-16\right],
\end{align*}
where 
\begin{align*}
\mathbf{a}_{3}&=\left\{-\genfrac{}{}{}{}{2r+1}{2},\genfrac{}{}{}{}{n-r}{2},\genfrac{}{}{}{}{n-2r+1}{2},-r\right\}, \\
\mathbf{b}_{3}&=\left\{1-n,n-2r-1,n-2r\right\}.
\end{align*}

\paragraph{\bf Case 2c).}
We have $2r'-1=l+n-k+1$.
The possible pairs of $(k,l)$ are 
\begin{align*}
(n-2r'+2,0), (n-2r'+3,1),\ldots, (n-r',r'-2).
\end{align*}
Then, we have 
\begin{align*}
M^{n}_{2r'-1}&=\sum_{l=0}^{r'-2}(-1)^{2(r'-l)+1}
\genfrac{(}{)}{0pt}{}{2r'-l-3}{l}
\genfrac{}{}{}{}{(l+2)_{l-2r'+n+2}}{(l-2r'+n+2)!}, \\
&=(2r'-n-3)\ _{4}\mathcal{F}_{3}\left[\mathbf{a}_{4},\mathbf{b}_{4},-16\right],
\end{align*}
where 
\begin{align*}
\mathbf{a}_{4}
&=\left\{\genfrac{}{}{}{}{3-2r'}{2},2-r',\genfrac{}{}{}{}{n-2r'+4}{2},\genfrac{}{}{}{}{n-2r'+5}{2}\right\}, \\
\mathbf{b}_{4}
&=\left\{2, 3-2r', n-2r'+3\right\}.
\end{align*}

\subsubsection{Fibonacci structure}
We combine the two results, Proposition \ref{prop:FibF} and Corollary \ref{cor:Fibodd},
and obtain the next proposition which reveals the Fibonacci structure of $M^{n}_{k}$.
The observations above give the description of the coefficients $M^{r}_{k}$ in terms of 
hypergeometric functions. This expression may not be simplified further.
However, the sum of $M^{r}_{2k}$ or $M^{r}_{2k-1}$ gives a simple formula in terms of 
Fibonacci numbers.
 
Let $a^{\mathrm{Fib}}(n)$ and $b^{\mathrm{Fib}}(n)$ be 
sequences of positive integers defined by 
\begin{align*}
a^{\mathrm{Fib}}(n)&:=(F_{2n-1}+F_{n+1})/2, \\
b^{\mathrm{Fib}}(n)&:=(F_{2n-1}-F_{n+1})/2.
\end{align*}

\begin{prop}
We have 
\begin{align*}
\sum_{k: \mathrm{even}}M^{n}_{k}&=a^{\mathrm{Fib}}(n), \\
\sum_{k: \mathrm{odd}}M^{n}_{k}&=b^{\mathrm{Fib}}(n). \\
\end{align*}
\end{prop}
\begin{proof}
Recall that $M^{n}_{k}$ is a coefficient of $t^{n}m^{k}$
in $G(t,m,s=1,r=m)$, and $R_{n}(r)$ is a coefficient 
of $t^{n}$ in $G(t,m=-1,s=1,-r)$ by definition.
By the specialization $m=1$ for $M^{n}_{k}$ and 
the specialization $r=1$ for $R_{n}(r)$, 
the sum of $\sum_{k=1}^{n}M^{n}_{k}$ and $R_{n}(1)$ gives the 
sum of the coefficient of $t^{2j}$ for $0\le j$ in $\sum_{k=1}^{n}M^{n}_{k}$.
Therefore, we have 
\begin{align*}
\sum_{k: \mathrm{even}}M^{n}_{k}
&=[t^{n}]\left(\genfrac{}{}{}{}{1}{2}(G(t,m=1,s=1,r=1)+G(t,m=-1,s=1,r=-1))\right), \\
&=\genfrac{}{}{}{}{1}{2}\left(F_{2n-1}+F_{n}\right),
\end{align*}
where we have used Proposition \ref{prop:FibF} for $G(t,1,1,1)$ and 
Corollary \ref{cor:Fibodd} for $G(t,-1,1,-1)$.

Similarly, 
\begin{align*}
\sum_{k: \mathrm{odd}}M^{n}_{k}
&=[t^{n}]\left(\genfrac{}{}{}{}{1}{2}(G(t,m=1,s=1,r=1)-G(t,m=-1,s=1,r=-1))\right), \\
&=\genfrac{}{}{}{}{1}{2}\left(F_{2n-1}-F_{n}\right).
\end{align*}
This completes the proof.
\end{proof}

We define $M_{n}(m):=\sum_{0\le k\le 2n}M^{n}_{k}m^{k}$.
Some few values of $M_{n}(m)$ are as follows:
\begin{align*}
M_{0}(m)&=1, \\
M_{1}(m)&=1+m^2, \\
M_{2}(m)&=1+2m^2-m^3+m^4, \\
M_{3}(m)&=1+3m^2-2m^3+4m^4-2m^5+m^6, \\
M_{4}(m)&=1+4m^2-3m^3+8m^4-7m^5+7m^6-3m^7+m^8, \\
M_{5}(m)&=1+5m^2-4m^3+13m^4-14m^5+20m^6-16m^7+11m^8-4m^{9}+m^{10}.
\end{align*}

\begin{remark}
Two remarks are in order.
\begin{enumerate}
\item
The integer sequence $a^{\mathrm{Fib}}(n)$ is the sequence A005207 in OEIS \cite{Slo},
and $b^{\mathrm{Fib}}(n)$ is the sequence A056014 in OEIS \cite{Slo}.
\item
The polynomial $M_{n}(m)$, or equivalently $M^{n}_{k}$, can be regarded as a 
refinement of $a^{\mathrm{Fib}}(n)$ and $b^{\mathrm{Fib}}(n)$.
The sequence $M^{n}_{2k}$, $0\le k\le n$, is close to the sequence A092422 in OEIS \cite{Slo},
however, the fifth row is different.
Similarly, $M^{n}_{2k+1}$, $0\le k\le n-1$, is close to the sequence A229079 in OEIS \cite{Slo}, 
however, the fifth row is different.
\end{enumerate}
\end{remark}

\section{Dimers on a circle}
\label{sec:dimerscirc}
In this section, we study dimer configurations on a circle.
As in the case of dimers on a segment, we consider 
dimers with multi colors.

\subsection{Dimers on a circle: definition}
We introduce a dimer configuration on a circle of size $n$, where $n$ in the number 
of vertices. 
We label the vertices on a circle from $1$ to $n$ anticlockwise.
We consider a dimer configuration which satisfies the same conditions 
as in the case of dimers on a segment introduced in Section \ref{sec:dimersseg}.

A dimer at position $i$ is an edge connecting a vertex labeled by $i$ with a vertex
labeled by  $i+1$ for $1\le i\le n-1$.
In addition, we allow to have a dimer at position $n$, which is an edge 
connecting a vertex labeled by $n$ with a vertex labeled by $1$.
Recall that we do not have a dimer at position $n$ in the case of a segment.
Due to the existence of a dimer at position $n$, the generating functions 
on a segment and on a circle behave differently even in the large $n$ limit.

For example, we have eight dimer configurations of size $3$ 
as shown in Figure \ref{fig:dimescirc}.
The dimer configuration consisting of three dimers with different 
colors is not allowed in the case of a segment.

\begin{figure}[ht]
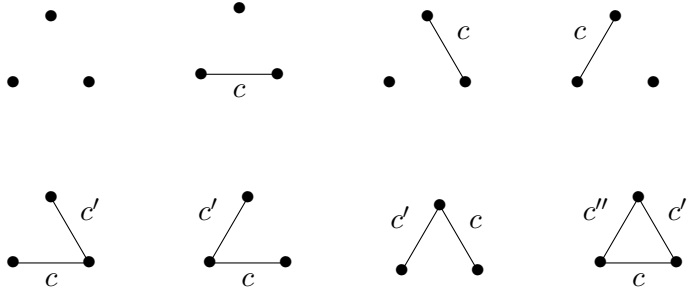

\begin{align*}
&\tikzpic{-0.5}{
\draw(0,0)node{$\bullet$}(1,0)node{$\bullet$}(0.5,0.866)node{$\bullet$};
}\qquad
\tikzpic{-0.5}{
\draw(0,0)node{$\bullet$}(1,0)node{$\bullet$}(0.5,0.866)node{$\bullet$};
\draw(0,0)--(1,0);\draw(0.5,0)node[anchor=north]{$c$};
}\qquad
\tikzpic{-0.5}{
\draw(0,0)node{$\bullet$}(1,0)node{$\bullet$}(0.5,0.866)node{$\bullet$};
\draw(1,0)--(0.5,0.866);\draw(0.75,0.433)node[anchor=south west]{$c$};
}\qquad
\tikzpic{-0.5}{
\draw(0,0)node{$\bullet$}(1,0)node{$\bullet$}(0.5,0.866)node{$\bullet$};
\draw(0,0)--(0.5,0.866);\draw(0.25,0.433)node[anchor=south east]{$c$};
} \\[24pt]
&\tikzpic{-0.5}{
\draw(0,0)node{$\bullet$}(1,0)node{$\bullet$}(0.5,0.866)node{$\bullet$};
\draw(0,0)--(1,0);\draw(0.5,0)node[anchor=north]{$c$};
\draw(1,0)--(0.5,0.866);\draw(0.75,0.433)node[anchor=south west]{$c'$};
}\qquad
\tikzpic{-0.5}{
\draw(0,0)node{$\bullet$}(1,0)node{$\bullet$}(0.5,0.866)node{$\bullet$};
\draw(0,0)--(1,0);\draw(0.5,0)node[anchor=north]{$c$};
\draw(0,0)--(0.5,0.866);\draw(0.25,0.433)node[anchor=south east]{$c'$};
}\qquad
\tikzpic{-0.5}{
\draw(0,0)node{$\bullet$}(1,0)node{$\bullet$}(0.5,0.866)node{$\bullet$};
\draw(1,0)--(0.5,0.866);\draw(0.75,0.433)node[anchor=south west]{$c$};
\draw(0,0)--(0.5,0.866);\draw(0.25,0.433)node[anchor=south east]{$c'$};
}\qquad
\tikzpic{-0.5}{
\draw(0,0)node{$\bullet$}(1,0)node{$\bullet$}(0.5,0.866)node{$\bullet$};
\draw(0,0)--(1,0);\draw(0.5,0)node[anchor=north]{$c$};
\draw(1,0)--(0.5,0.866);\draw(0.75,0.433)node[anchor=south west]{$c'$};
\draw(0,0)--(0.5,0.866);\draw(0.25,0.433)node[anchor=south east]{$c''$};
}
\end{align*}
\caption{Dimer configurations of size $3$. The colors $c,c'$ and $c''$ are 
all different.}
\label{fig:dimescirc}
\end{figure}

We define the weight of a dimer configuration $\mathcal{D}$ of size $n$ by 
\begin{align*}
\mathrm{wt}(\mathcal{D}):=t^{n}r^{a}s^{b}\cdot M(\mathcal{D}),
\end{align*}
where $a$ is the number of dimers, $b$ is the sum of the positions of dimers, 
and $M(\mathcal{D})$ is the total number of ways to put colors on dimers in $\mathcal{D}$.

For example, the weight of the dimer configurations with three dimers 
in Figure \ref{fig:dimescirc} is given by $m(m-1)(m-2)r^{3}s^{6}$.

\subsection{Generating function}
Let $\mathcal{P}(n)$ be the set of dimer configurations on a circle of size $n$.
We expand the generating functions on a circle as follows.
\begin{defn}
\label{defn:defGcn}
The generating function for the dimer model on a circle is defined as
\begin{align}
\label{eqn:degGastmrs}
\begin{split}
G^{(c)}(t,m,r,s)&:=\sum_{0\le n}G_{n}^{(c)}(m,r,s)t^{n}, \\
&:=\sum_{0\le n}\sum_{\mathcal{D}\in\mathcal{P}(n)}\mathrm{wt}(\mathcal{D})t^{n},
\end{split}
\end{align}
Here, we have the initial conditions 
$G^{(c)}_{n}(m,r,s)=G_{n}^{(s)}(m,r,s)$ for $n\le 2$.
\end{defn}

In Definition \ref{defn:defGcn}, the initial condition $G_{2}^{(c)}(m,r,s)=G_{2}^{(s)}(m,r,s)$ comes from the 
facts that we have two edges connecting two vertices and the colors of the two edges are 
different.
Other two initial conditions $G_{n}^{(c)}(m,r,s)=G_{n}^{(s)}(m,r,s)$ for $n\le1$
come from the similar reasons.
Note that the exponent of $t$ corresponds to the size of the system, and 
the numbers of vertices differ by one in the cases of the dimers on a segment 
and those on a circle.

First few values of $G^{(c)}_{n}(m,r,s)$ are 
\begin{align*}
G^{(c)}_{0}(m,r,s)&=1, \qquad G^{(c)}_{1}(m,r,s)=1+mrs, \\
G^{(c)}_{2}(m,r,s)&=1+r(ms+ms^{2})+r^{2}s^{3}m(m-1), \\
G^{(c)}_{3}(m,r,s)&=1+r(ms+ms^2+ms^3)+r^2(m(m-1)s^3+m(m-1)s^4+m(m-1)s^5)\\
&\quad+r^3 s^6m(m-1)(m-2), \\
G^{(c)}_{4}(m,r,s)&=1+r(ms+ms^2+ms^3+ms^4) \\
&+r^2(m(m-1)s^3+m^2s^4+2m(m-1)s^5+m^2 s^6+m(m-1)s^7) \\
&+r^3(m(m-1)^2s^6+m(m-1)^2s^7+m(m-1)^2s^8+m(m-1)^2s^9) \\
&+(-3m+6m^2-4 m^3+m^4)r^4s^{10}, \\
G^{(c)}_{5}(m,r,s)&=1+r(ms+ms^2+ms^3+ms^4+ms^5)  \\
&\quad+r^2\left(m(m-1)s^3+m^2s^4+m(2m-1)s^5+m(2m-1)s^6+m(2m-1)s^7\right.\\
&\qquad\left.+m^2 s^8+m(m-1)s^9\right) \\
&\quad+r^3(m(m-1)^{2}s^6 +m^{2}(m-1)s^7+m(2m-1)(m-1)s^8+m(2m-1)(m-1)s^9 \\
&\qquad+m(2m-1)(m-1)s^{10}+m^{2}(m-1)s^{11}+m(m-1)^{2}s^{12}) \\
&\quad+r^4 (m(m-1)^{3}s^{10} +m(m-1)^{3}s^{11}+m(m-1)^{3}s^{12} \\
&\qquad+m(m-1)^{3}s^{13}+m(m-1)^{3}s^{14})\\
&\quad+r^{5}s^{15}m(4-10m+10m^2-5m^3+m^4).
\end{align*}

The generating function $G^{(c)}_{n}(m,r,s)$, $n\ge3$, satisfies 
the following simple recurrence relation. 
\begin{theorem}
\label{thrm:Gcrr1}
For $n\ge3$, we have 
\begin{align*}
G^{(c)}_{n}(m,r,s)=G^{(s)}_{n}(m,r,s)+rs\left(G^{(s)}_{n-2}(m,rs,s)-G^{(c)}_{n-1}(m,rs,s)\right).
\end{align*}
\end{theorem}
\begin{proof}
Recall that the generating functions $G^{(s)}_{n}(m,r,s)$ is defined for dimers on 
the segment $[0,n]$, which means we have $n+1$ vertices. 
On the other hand, we have $n$ vertices for a dimer configuration on a circle
of size $n$.
To obtain a dimer configuration on a circle of size $n$,
we consider dimer configurations on a segment of size $n$ and $n-2$, and 
those on a circle of size $n-1$.
First, we obtain a dimer configuration on a circle of size $n$ from 
a dimer configuration on a segment of size $n$ by identifying the 
vertices labeled by $0$ and $n$.
In this process, we have some dimer configurations which are not 
admissible.
They are the dimer configurations such that the dimers at position $1$ 
and at position $n$ have the same color $c$.
We denote by $\mathcal{D}(1,n)$ be a such configuration.
The total number of $\mathcal{D}(1,n)$ is the same as the number of 
dimer configurations of size $n-1$ such that a dimer at position $1$ 
has the color $c$.
This correspondence is bijective since we delete the dimer at 
position $n$ in $\mathcal{D}(1,n)$ and reconnect the vertices at both ends 
to obtain a dimer configuration on a circle of size $n-1$.
Next task is to count the total number of dimer configurations 
such that it is on a circle of size $n-1$ and the dimer at position 
$1$ has a color $c$.
By attaching an empty edge to a dimer configuration on a segment of size $n-2$ and 
connect the vertices at both ends, we obtain a dimer configuration on a circle of 
size $n-1$ such that there is no dimer at position $1$.
Thus, the total number of these dimer configurations $\mathcal{D}(1,n)$ are 
calculated as $G^{(c)}(m,r,s)-G^{(s)}(m,r,s)$.
We construct a dimer configuration on a circle of size $n$ by inserting a 
edge at position $1$ into a dimer configuration on a circle of size $n-1$.
Further, since the dimer at position $1$ has a color $c$, we have the factor 
$rs$ which corresponds to the existence of a dimer and the position.
The shift $r\rightarrow rs$ comes from the fact that if we insert a dimer 
at position $1$, the positions of the remaining dimers are increased by one. 
Combining these observations, we obtain the recurrence relation.
\end{proof}

From Theorem \ref{thrm:Gcrr1}, we obtain the following corollary by a 
straightforward calculation.
\begin{cor}
\label{cor:GsGcrec}
The generating functions $G^{(s)}(t,m,r,s)$ and $G^{(c)}(t,m,r,s)$ satisfy 
\begin{align*}
G^{(c)}(t,m,r,s)=G^{(s)}(t,m,r,s)-trsG^{(c)}(t,m,rs,s)+t^2rsG^{(s)}(t,m,rs,s)+rst+mr^2s^4t^2.
\end{align*}
\end{cor}

The generating function $G^{(c)}_{n}(m,r,s)$ satisfies the following 
recurrence relation of Fibonacci type.
\begin{theorem}
\label{thrm:GcFib}
For $n\ge3$, $G^{(c)}_{n}(m,r,s)$ satisfies
\begin{align}
\label{eqn:GcFib}
\begin{split}
G^{(c)}_{n}(m,r,s)&=f(rs^{n})G^{(c)}_{n-1}(m,r,s)+rs^{n}G^{(c)}_{n-2}(m,r,s) \\
&\quad-rs(s-1)G^{(c)}_{n-2}(m,rs,s)+rs(s-1)G_{n-3}^{(s)}(m,rs,s) \\
&\quad+(-1)^{n}r^{n}s^{n(n+1)/2}m(m-1).
\end{split}
\end{align}
Initial conditions are $G^{(c)}_{n}(m,r,s)=G^{(s)}_{n}(m,r,s)$ for $n\le 2$.
\end{theorem}

To prove Theorem \ref{thrm:GcFib}, we rewrite the recurrence relation
for $G_{n}^{(s)}(m,r,s)$.

\begin{lemma}
\label{lemma:FibGsmock}
We have 
\begin{align*}
\begin{split}
G_{n}^{(s)}(m,r,s)&=f(rs^{n+1})G_{n-1}^{(s)}(m,r,s)+rs^{n+1}G_{n-2}^{(s)}(m,r,s) \\
&\quad-(s-1)G_{n}^{(s)}(m,r,s)+(s-1)G_{n-1}^{(s)}(m,r,s).
\end{split}
\end{align*}
\end{lemma}
\begin{proof}
Since $f(rs^{n+1})=sf(rs^{n})+1-s$, we have
\begin{align*}
&G_{n}^{(s)}(m,r,s)-f(rs^{n+1})G_{n-1}^{(s)}(m,r,s)-rs^{n+1}G_{n-2}^{(s)}(m,r,s) \\
&\quad=G_{n}^{(s)}(m,r,s)-sf(rs^{n})G_{n-1}^{(s)}(m,r,s)+(s-1)G_{n-1}^{(s)}(m,r,s)-rs^{n+1}G_{n-2}^{(s)}(m,r,s)\\
&\quad=-(s-1)G_{n}^{(s)}(m,r,s)+(s-1)G_{n-1}^{(s)}(m,r,s),
\end{align*}
where we have used Theorem \ref{thrm:FibGfin}.
\end{proof}

\begin{proof}[Proof of Theorem \ref{thrm:GcFib}]
We prove the statement by induction on $n$.
It is straightforward to show $G_{n}^{(c)}(m,r,s)$ and $G_{n}^{(s)}(m,r,s)$ with $n\le3$
satisfy the recurrence relation (\ref{eqn:GcFib}).
We assume that (\ref{eqn:GcFib}) is true up to $n-1$.

We define $A^{(\ast)}(m,r,s)$ as
\begin{align*}
A^{(\ast)}_{n}(m,r,s):=G_{n}^{(\ast)}(m,r,s)-f(rs^{n})G_{n-1}^{(\ast)}(m,r,s)-rs^{n}G_{n-2}^{(\ast)}(m,r,s),
\end{align*}
where $\ast$ is either $s$ or $c$.
From Theorem \ref{thrm:FibGfin}, we have $A^{(s)}_{n}(m,r,s)=0$ for $2\le n$.
We calculate $A^{(c)}_{n}(m,r,s)$ by use of the recurrence relation in Theorem \ref{thrm:Gcrr1}
and $A^{(s)}_{n}(m,r,s)$.
Then, we have 
\begin{align}
\label{eqn:recAc}
\begin{split}
A^{(c)}_{n}(m,r,s)&=-rsA_{n-1}^{(c)}(m,rs,s)  \\
&\quad+rs\left( G_{n-2}^{(s)}(m,rs,s)-f(rs^{n})G_{n-3}^{(s)}(m,rs,s)-rs^{n}G_{n-4}^{(s)}(m,rs,s) \right).
\end{split}
\end{align}
We substitute  the induction assumption for $A_{n-1}^{(c)}(m,r,s)$ and Lemma \ref{lemma:FibGsmock}
into the first and second terms in Eq. (\ref{eqn:recAc}).
Rearranging the terms by use of Theorem \ref{thrm:Gcrr1}, 
we have Eq. (\ref{eqn:GcFib}). 
This completes the proof.
\end{proof}

By applying Theorem \ref{thrm:Gcrr1} to Theorem \ref{thrm:GcFib}, 
we obtain the following Corollary.
\begin{cor}
\label{cor:Gcrr2}
We have 
\begin{align*}
G^{(c)}_{n}(m,r,s)&=f(rs^{n})G^{(c)}_{n-1}(m,r,s)+rs^nG^{(c)}_{n-2}(m,r,s) \\
&\quad+(s-1)\left(G_{n-1}^{(c)}(m,r,s)-G_{n-1}^{(s)}(m,r,s)\right)+(-1)^{n}r^{n}s^{n(n+1)/2}m(m-1).
\end{align*}
\end{cor}

We give an expression of $G^{(c)}(m,r,s)$ in 
terms of the function $f(x)$.
Recall that the power series $G^{(s)}(m,r,s)$ has the expression as in 
Theorem \ref{thrm:Gsinf}.
Thus, by combining this expression with Theorem \ref{thrm:Gcrr1}, 
we obtain the following proposition.
\begin{prop}
Let $\widetilde{I}^{(i)}(a,b;c,d)$, $i\in \{1,2\}$ be the sets of indices 
\begin{align*}
\widetilde{I}(a,b;c,d)
&:=\left\{ (i_1,\ldots,i_{2a}) \Big|
\begin{array}{c}
c\le i_1< \ldots< i_{2a}\le d \\
i_{2j}=i_{2j-1}+1, 1\le j\le a \\
i_1+i_2+\cdots+i_{2a}=2b-a
\end{array}
\right\}.
\end{align*}
We  define two functions $\widetilde{G}^{n}_{a,b}$ and $\widetilde{H}^{n}_{a,b}$ by 
\begin{align*}
\widetilde{G}^{n}_{a,b}&:=f(rs)^{-1}f(rs^{n})^{-1}\sum_{(i_{1},\ldots,i_{2(a-1)})\in \widetilde{I}(a-1,b-1;2,n-1)}
\prod_{j=1}^{2(a-1)}f(rs^{i_j})^{-1}, \\
\widetilde{H}^{n}_{a,b}&:=\sum_{(i_{1},\ldots,i_{2a})\in \widetilde{I}(a,b;1,n)}
\prod_{j=1}^{2a}f(rs^{i_j})^{-1}.
\end{align*}

We define four functions:
\begin{align*}
a_{1}(p,a)&:=a^2+2(a-1)(p-a), \\
a_{2}(p,a)&:=a(a+1)+2a(p-a), \\
a_{3}(p,a)&:=a^2+a-1+2(a-1)(p-a), \\
a_{4}(p,a)&:=a(a+2)+2a(p-a).
\end{align*}
Then, the generating function $G^{(c)}_{2p}:=G^{(c)}_{2p}(m,r,s)$ is given by 
\begin{align}
\label{eqn:Gceven1}
\begin{split}
G^{(c)}_{2p}
&=\left(\prod_{j=1}^{2p}f(rs^{j})\right)
\left(1+
\sum_{1\le a\le p-1}\left(
\sum_{a^2\le b\le a_{1}(p,a)} \widetilde{G}^{2p}_{a,b}r^{a}s^{b}
+\sum_{a(a+1)\le b\le a_{2}(p,a)} \widetilde{H}^{2p}_{a,b}r^{a}s^{b}\right) \right)\\
&\quad+r^{p}(s^{p^{2}}+s^{p(p+1)})
+r^{2p-1}s^{p(2p-1)}(-1+f(rs^{2p})).
\end{split}
\end{align}

Similarly, the generating function $G^{(c)}_{2p+1}:=G^{(c)}_{2p+1}(m,r,s)$ 
is given by 
\begin{align}
\label{eqn:Gcodd1}
\begin{split}
G^{(c)}_{2p+1}&=
\left(\prod_{j=1}^{2p+1}f(rs^{j})\right)\left(
1+\sum_{1\le a\le p}
\left(
\sum_{a^2\le b\le a_{3}(p,a)} \widetilde{G}^{2p+1}_{a,b}r^{a}s^{b}
+\sum_{a(a+1)\le b\le a_{4}(p,a)} \widetilde{H}^{2p+1}_{a,b}r^{a}s^{b}\right)\right) \\
&\quad+r^{2p}s^{p(2p+1)}(1-f(rs^{2p+1})).
\end{split}
\end{align}
\end{prop}
\begin{proof}
We show that the expressions (\ref{eqn:Gceven1}) and (\ref{eqn:Gcodd1}) hold 
for general $n$ by induction.

For $n\le 2$, the explicit expressions for $G_{n}^{(c)}(m,r,s)$ are given by 
\begin{align*}
G_{1}^{(c)}(m,r,s)&=f(rs)+rs, \\
G_{2}^{(c)}(m,r,s)&=f(rs)f(rs^2)+r(sf(rs^2)+s^{2}),
\end{align*}
which are the initial conditions.
For $n=3$, we have 
\begin{align*}
G_{3}^{(c)}(m,r,s)&=f(rs)f(rs^2)f(rs^3)+r(sf(rs^2)+s^2f(rs^3)+s^3f(rs))+r^2s^3(1-f(rs^3)),
\end{align*}
which coincides with the expression (\ref{eqn:Gcodd1}) for $p=1$.

We assume that (\ref{eqn:Gceven1}) and (\ref{eqn:Gcodd1}) hold up to $n-1$.
Recall that we have an expression for $G_{n}^{(s)}(m,r,s)$ by Theorem \ref{thrm:Gsinf}.
To obtain $G_{n}^{(c)}(m,r,s)$, we calculate the right hand side of the recurrence 
relation in Theorem \ref{thrm:Gcrr1}.
We first consider the case $n=2p$.

From the expression (\ref{eqn:Gsinf}) in Theorem \ref{thrm:Gsinf}, we rewrite 
$\widehat{G}_{a,b}^{2p}$
as 
\begin{align*}
\widehat{G}_{a,b}^{2p}
=f(rs)^{-1}\sum_{\mathbf{i}_{a-1}\in I(a-1,b-1;2,2p)}\prod_{j=1}^{2(a-1)}f(rs^{i_{j}})^{-1}
+\sum_{\mathbf{i}_{a}\in I(a,b;1,2p)}\prod_{j=1}^{2a}f(rs^{i_j})^{-1},
\end{align*}
where $\mathbf{i}_{a}:=(i_1,i_2,\ldots,i_{2a})$.
Then, $(\prod_{j=1}^{2p}f(rs^{j}))^{-1}G_{2p}^{(s)}(m,r,s)$ is rewritten 
as 
\begin{align}
\label{eqn:Gs2pinf}
\begin{split}
1+\sum_{a=1}^{p}r^{a}\left(\sum_{b=a^2}^{a_{3}(p,a)}s^{b}f(rs)^{-1}\sum_{\mathbf{i}_{a-1}\in I(a-1,b-1;2,2p)}
\prod_{j=1}^{2(a-1)}f(rs^{i_j})^{-1} \right.\\
\left.+\sum_{b=a(a+1)}^{a_{2}(p,a)}s^{b}\sum_{\mathbf{i}_{a}\in I(a,b;1,2p)}\prod_{j=1}^{2a}f(rs^{i_j})^{-1}\right).
\end{split}
\end{align}
We also rewrite $(\prod_{j=1}^{2p}f(rs^{j}))^{-1}rsG_{2p-2}^{(s)}(m,rs,s)$ as 
\begin{align}
\label{eqn:Gs2p1inf}
\begin{split}
rsf(rs)^{-1}f(rs^{2p})^{-1}+\sum_{a=2}^{p}r^{a}
\left(
\sum_{b=a^2}^{a_{1}(p,a)}s^{b}f(rs)^{-1}f(rs^{2p})^{-1}\sum_{\mathbf{i}_{a-1}\in I(a-1,b-1;2,2p-1)}
\prod_{j=1}^{2(a-1)}f(rs^{i_j})^{-1}    \right.\\
+\left.\sum_{b=a^{2}-a+1}^{a_{1}(p,a-1)+2}s^{b}f(rs)^{-1}f(rs^{2})^{-1}f(rs^{2p})^{-1}\sum_{\mathbf{i}_{a-1}\in I(a-2,b-3;3,2p-1)}
\prod_{j=1}^{2(a-2)}f(rs^{i_j})^{-1}
\right).
\end{split}
\end{align}
Similarly, we rewrite $(\prod_{j=1}^{2p}f(rs^{j}))^{-1}rsG_{2p-1}^{(c)}(m,rs,s)$
as 
\begin{align}
\label{eqn:Gc2p1inf}
\begin{split}
&rsf(rs)^{-1}+r^{2p-1}s^{p(2p-1)}(1-f(rs^{2p})) \\
&\qquad+\sum_{a=2}^{p}r^{a}\left(
\sum_{b=(a-1)^{2}+a}^{a_{1}(p,a-1)+2}s^{b}f(rs)^{-1}f(rs^2)^{-1}f(rs^{2p})^{-1}\sum_{\mathbf{i}_{a-1}\in I(a-2,b-3;3,2p-1)}
\prod_{j=1}^{2(a-2)}f(rs^{i_j})^{-1} \right. \\
&\qquad+\left.
\sum_{b=a^{2}}^{a_{3}(p,a)}s^{b}
f(rs)^{-1}\sum_{\mathbf{i}_{a-1}\in I(a-1,b-1;2,2p)}
\prod_{j=1}^{2(a-1)}f(rs^{i_j})^{-1}
\right) .
\end{split}
\end{align}
We substitute these expressions from (\ref{eqn:Gs2pinf}) to (\ref{eqn:Gc2p1inf}) 
into the right hand side of the recurrence relation in Theorem \ref{thrm:Gcrr1}.
The second term in the sum in Eq. (\ref{eqn:Gs2pinf}) gives $\widetilde{H}^{2p}_{a,b}$,
and the first term in the sum in Eq. (\ref{eqn:Gs2p1inf}) gives $\widetilde{G}^{2p}_{a,b}$.
Other terms cancel with each other.
Then, we obtain the expression (\ref{eqn:Gceven1}) for $G_{2p}^{(c)}(m,r,s)$.

One can similarly show that $G_{2p+1}^{(c)}(m,r,s)$ has the expression (\ref{eqn:Gcodd1}).
This completes the proof.
\end{proof}

In Theorem \ref{thrm:Gcrr1}, we obtain the recurrence relation of $G_{n}^{(c)}(m,r,s)$
involving $G_{n}^{(s)}(m,r,s)$.
Below, we give another system of recurrence relations for $G_{n}^{(c)}(m,r,s)$.

\begin{defn}
We define five generating functions for dimer configurations on a segment of size $n$
by $Z^{\alpha\beta}_{n}(m,r,s)$ with 
$(\alpha,\beta)=(\Phi,\Phi), (\Phi,c), (c,\Phi), (c,c')$ or $(c,c)$.
Then, $\alpha=\Phi$ (resp. $\beta=\Phi$) means that the first (resp. last) site 
has no dimer. On the other hand, $\alpha=c$ (resp. $\beta=c'$) means that the first (resp. last)
site has a dimer with color $c$ (resp. $c'$).
A color $c'$ is different from a color $c$.
\end{defn}

From the definitions of $Z^{\alpha\beta}_{n}(r):=Z^{\alpha\beta}_{n}(m,r,s)$, it is straightforward to show 
that they satisfy the following system of recurrence relations.
\begin{lemma}
\label{lemma:defrelZ}
For $2\le n$, we have 
\begin{align*}
Z^{\Phi\Phi}_{n}(r)&=Z^{\Phi\Phi}_{n-1}(rs)+mZ^{c\Phi}_{n-1}(rs), \\
Z^{c\Phi}_{n}(r)&=rs(Z^{\Phi\Phi}_{n-1}(rs)+(m-1)Z^{c\Phi}_{n-1}(rs)), \\
Z^{\Phi c}_{n}(r)&=Z^{\Phi c}_{n-1}(rs)+(m-1)Z^{c'c}_{n-1}(rs)+Z^{cc}_{n-1}(rs), \\
Z^{c'c}_{n}(r)&=rs(Z^{\Phi c}_{n-1}(rs)+(m-2)Z^{c'c}_{n-1}(rs)+Z^{cc}_{n-1}(rs)), \\
Z^{cc}_{n}(r)&=rs(Z^{\Phi c}_{n-1}(rs)+(m-1)Z^{c'c}_{n-1}(rs)).
\end{align*}
\end{lemma}

To obtain a dimer configuration on a circle, we connect the first site with the last site on a segment.
If we have dimers at the first and last sites on a circle, they are with different colors.
Thus, the generating function $G_{n}^{(c)}$ is expressed in terms of $Z^{\ast\ast}_{n}(r)$ as follows:
\begin{align*}
G_{n}^{(c)}(m,r,s)=Z^{\Phi\Phi}_{n}(r)+m(Z^{\Phi c}_{n}(r)+Z^{c\Phi}_{n}(r))+m(m-1)Z^{cc'}_{n}(r).
\end{align*} 
Note that we have no term involving $Z^{cc}_{n}(r)$.

\begin{defn}
\label{defn:G1G2}
We define two generating functions:
\begin{align*}
G^{(1)}_{n}(r)&:=Z^{\Phi\Phi}_{n}(r)+mZ^{c\Phi}_{n}(r), \\
G^{(2)}_{n}(r)&:=Z^{\Phi c}_{n}(r)+(m-1)Z_{n}^{c'c}(r).
\end{align*}
\end{defn}

\begin{prop}
\label{prop:GcinG1G2}
The generating function $G_{n}^{(c)}(m,r,s)$ is written as the sum of 
two generating functions:
\begin{align*}
G_{n}^{(c)}(m,r,s)=G^{(1)}_{n}(r)+m G^{(2)}_{n}(r),
\end{align*}
where $G^{(i)}_{n}(r)$, $i\in\{1,2\}$ satisfy the following functional relations.
The generating function $G^{(1)}_{n}(r)$ satisfies the recurrence relation of Fibonacci
type:
\begin{align*}
G^{(1)}_{n}(r)=(1+rs(m-1))G^{(1)}_{n-1}(rs)+rsG_{n-2}^{(1)}(rs^{2}),
\end{align*}
with the initial conditions $G_{1}^{(1)}(r):=1$ and $G_{2}^{(1)}(r):=1+mrs$.
Similarly, the generating function $G_{n}^{(2)}(r)$ satisfies 
\begin{align}
\label{eqn:G2recrel}
G_{n}^{(2)}(r)=(1+rs(m-2))G^{(2)}_{n-1}(rs)+rs(1+s+rs^2(m-1))G^{(2)}_{n-2}(rs^{2})
+r^2s^4G_{n-3}^{(2)}(rs^{3}),
\end{align}
with the initial conditions
\begin{align*}
&G_{1}^{(2)}(r):=0, \qquad G_{2}^{(2)}(r):=rs^2+(m-1)r^2s^3, \\
&G_{3}^{(2)}(r):=rs^3+r^{2}(m-1)(s^4+s^5)+r^3s^{6}(m-1)(m-2).
\end{align*}
\end{prop}
\begin{proof}
From the defining relations in Lemma \ref{lemma:defrelZ}, one can 
obtain the recurrence relations for $G_{n}^{(i)}(r)$, $i\in\{1,2\}$
by simple calculations.
\end{proof}

\begin{remark}
The generating function $G_{n}^{(1)}(r)$ is nothing but the generating function 
$G_{n-1}^{(s)}(m,r,s)$.
The difference from Theorem \ref{thrm:Gcrr1} is that we have an explicit recurrence 
relation for $G_{n}^{(2)}(r)$.
In the recurrence relation for $G_{n}^{(2)}(r)$, the existence of the term $G_{n-3}^{(2)}(rs^{3})$ 
gives the difference from a recurrence relation of Fibonacci type.
\end{remark}

\subsubsection{Generating function at \texorpdfstring{$m=1$}{m=1}.}
When $m=1$, the generating function $G^{(c)}_{n}(m=1,r,s)$ has 
the following form.
\begin{prop}
\label{prop:Gcm1}
We have
\begin{align}
\label{eqn:Gcm1rec1}
G^{(c)}_{n}(1,r,s)=1+\sum_{k=1}^{\lfloor n/2\rfloor}r^{k}s^{k^2}
[n]_{s}\genfrac{}{}{}{}{\prod_{j=1}^{k-1}[n-2k+j]_s}{[k]_s!}.
\end{align}
\end{prop}
\begin{proof}
Since $f(x)=1$, we have an expression for $G_{n}^{(s)}(m=1,r,s)$ by 
Proposition \ref{prop:Gnm1}.
The two generating functions $G_{n}^{(c)}(m,r,s)$ and $G_{n}^{(s)}(m,r,s)$
satisfy the recurrence relation as in Theorem \ref{thrm:Gcrr1}.
We prove the proposition by induction on $n$.
For $n=2$, we have $G_{2}^{(c)}(m=1,r,s)=1+r(s+s^{2})$ by a simple calculation.

Suppose that the expression is valid up to $n-1$. 
Then, by Proposition \ref{prop:Gnm1}, Theorem \ref{thrm:Gcrr1} and the induction 
assumption, we have the left hand side of the recurrence relation in 
Theorem \ref{thrm:Gcrr1} coincides with the expression (\ref{eqn:Gcm1rec1})
by a straightforward calculation.
\end{proof}

We calculate the difference between the generating function $G^{(c)}_{n}(m=1,r,s)$ and
the generating function $G^{(s)}_{n}(m=1,r,s)$.

\begin{prop}
\label{prop:Gscdif}
We have 
\begin{align}
\label{eqn:Gscdif1}
G^{(s)}_{n}(1,r,s)-G^{(c)}_{n}(1,r,s)
=r^{2}s^{n+1}
\sum_{k=0}^{\lfloor (n-3)/2\rfloor}r^{k}s^{k(k+2)}
\genfrac{[}{]}{0pt}{}{n-k-3}{k}_{s},
\end{align}
\end{prop}
\begin{proof}
From Theorem \ref{thrm:Gcrr1}, we have 
\begin{align*}
G_{n}^{(s)}(1,r,s)-G^{(c)}_{n}(1,r,s)=rs(G_{n-1}^{(c)}(1,rs,s)-G_{n-2}^{(s)}(1,rs,s)).
\end{align*}
We arrive the expression (\ref{eqn:Gscdif1}) by Propositions \ref{prop:Gnm1} and 
\ref{prop:Gcm1}. This completes the proof.
\end{proof}

\subsubsection{Generating function with general \texorpdfstring{$(m,r,s)$}{(m,r,s)}}
To obtain an explicit expression for the generating function $G^{(c)}_{n}(m,r,s)$, 
we study first the generating function with the specialization $m\rightarrow 1+mr$.
\begin{prop}
\label{prop:Gcmrsgen}
We have 
\begin{align*}
G^{(c)}_{n}(1+mr,r,s)&=(-1)^{n}mr^{n+1}s^{n(n+1)/2} \\
&\quad+\sum_{k=0}^{n}m^{k}\sum_{j=0}^{\lfloor(n-k)/2\rfloor}
r^{2k+j}s^{k(k+1)/2+(k+j)j}\genfrac{}{}{}{}{[n]_s}{[n-j]_{s}}
\genfrac{[}{]}{0pt}{}{n-j}{k+j}_{s}\genfrac{[}{]}{0pt}{}{k+j}{j}_{s}.
\end{align*}
\end{prop}
\begin{proof}
We have the recurrence relation for $G_{n}^{(s)}(m,r,s)$ and $G_{n}^{(c)}(m,r,s)$
by Theorem \ref{thrm:Gcrr1}.
With the specialization $m\rightarrow 1+mr$, we have the explicit expression for 
$G_{n}^{(s)}(1+mr,r,s)$ by Proposition \ref{prop:Gmrsgen}.

We prove the statement by induction on $n$.
When $n=2$, the statement holds by a simple calculation.
We assume that the expression is valid up to $n-1$.
We calculate $G_{n}^{(c)}(1+mr,r,s)$ by use of Theorem \ref{thrm:Gcrr1}.
We first compute the difference between $G_{n}^{(s)}$ and $G_{n}^{(c)}$, 
{\it i.e.}, 
\begin{align*}
G^{(\mathrm{dif})}_{n}(1+mr,r,s):=rs(G_{n-2}^{(s)}(1+mr,rs,s)-G_{n-1}^{(c)}(1+mr,rs,s)).
\end{align*}
By use of 
\begin{align*}
\begin{split}
\genfrac{[}{]}{0pt}{}{n-2-j}{k}_{s}\genfrac{[}{]}{0pt}{}{n-k-j-1}{j}_{s}
-\genfrac{}{}{}{}{[n-1]_{s}}{[n-j-1]_{s}}\genfrac{[}{]}{0pt}{}{n-j-1}{k+j}_{s}\genfrac{[}{]}{0pt}{}{k+j}{j}_{s} \\[11pt]
=
-s^{n-j-k-1}\genfrac{}{}{}{}{[j+k]_{s}[n-j-2]_{s}!}{[k]_{s}![j]_{s}![n-2j-k-1]_{s}!},
\end{split}
\end{align*}
we have 
\begin{align*}
G_{n}^{(\mathrm{dif})}(1+mr,r,s)
&=-\sum_{k=0}^{n-2}m^{k}\sum_{j=1}^{\lfloor(n-k+1)/2\rfloor}
r^{2k+j}s^{d(k,j)}\genfrac{}{}{}{}{[n-j-1]_{s}![k+j-1]_{s}}{[k]_{s}![j-1]_{s}![n-2j-k+1]_{s}!} \\
&\quad-m^{n-1}r^{2(n-1)}s^{(n+2)(n-1)/2}+(-1)^{n}mr^{n}s^{n(n+1)/2},
\end{align*}
where 
\begin{align*}
d(k,j)=k(k+1)/2+(k+j)j+n.
\end{align*}
We compute $G_{n}^{(s)}(1+mr,r,s)+G_{n}^{(\mathrm{dif})}(1+mr,r,s)$ by using  
\begin{align*}
\genfrac{[}{]}{0pt}{}{n-j}{k}_{s}\genfrac{[}{]}{0pt}{}{n-k-j+1}{j}_{s}
-\genfrac{}{}{}{}{s^{n-2j-k+1}[n-j-1]_{s}![k+j-1]_{s}}{[k]_{s}![j-1]_{s}![n-2j-k+1]_{s}!} 
=\genfrac{}{}{}{}{[n]_{s}}{[n-j]_{s}}\genfrac{[}{]}{0pt}{}{n-j}{k+j}_{s}\genfrac{[}{]}{0pt}{}{k+j}{k}_{s},
\end{align*}
which can be verified by a straightforward calculation.
Then, we obtain the expression for $G_{n}^{(c)}(1+mr,r,s)$, 
which completes the proof.
\end{proof}

By the change of variable $m\rightarrow (m-1)/r$ in Proposition \ref{prop:Gcmrsgen},
we obtain the following expression.
\begin{theorem}
\label{thrm:Gcmrs1}
We have
\begin{align*}
G^{(c)}_{n}(m,r,s)&=(-1)^{n}(m-1)r^{n}s^{n(n+1)/2} \\
&\quad+\sum_{k=0}^{n}(m-1)^{k}\sum_{j=0}^{\lfloor(n-k)/2\rfloor}r^{k+j}s^{k(k+1)/2+(k+j)j}
\genfrac{}{}{}{}{[n]_s}{[n-j]_{s}}
\genfrac{[}{]}{0pt}{}{n-j}{k+j}_{s}\genfrac{[}{]}{0pt}{}{k+j}{j}_{s}.
\end{align*}
\end{theorem}

\subsubsection{Generating function at \texorpdfstring{$m=1+m/r$}{m=1+m/r}}
We consider the specialization $m\rightarrow 1+m/r$.
As in the case of dimers on a segment, one can expect a simple formula 
for the generating function $G^{(c)}_{n}(m,r,s)$.

\begin{prop}
\label{prop:G1mrr1}
We have
\begin{align}
\label{eqn:Gc1mrr1}
\begin{split}
G_{n}^{(c)}(1+m/r,r,s)
&=(-1)^{n}mr^{n-1}s^{n(n+1)/2}
+\prod_{j=1}^{n}(1+ms^{j})    \\
&\quad+
r\sum_{k=0}^{n-2}m^{k}s^{d(1,k)}[n]_{s}\genfrac{[}{]}{0pt}{}{n-2}{k}_{s} \\
&\quad+\sum_{j=2}^{\lfloor(n-1)/2\rfloor}
r^{j}
\sum_{k=0}^{n-2j}m^{k}s^{d(j,k)}
[n]_{s}\genfrac{}{}{}{}{\prod_{i=1}^{j-1}[n-2j+i]_{s}}{[j]_{s}!}
\genfrac{[}{]}{0pt}{}{n-2j}{k}_{s}\\
&\quad
+\begin{cases}
\displaystyle r^{n/2}s^{n^2/4}\genfrac{}{}{}{}{[n]_{s}}{[n/2]_{s}}, & n: even, \\
0, & n: odd.
\end{cases}
\end{split}
\end{align}
where 
\begin{align*}
d(j,k):=j^{2}+k(k+2j+1)/2.
\end{align*}
\end{prop}
\begin{proof}
Recall that $G_{n}^{(s)}(m,r,s)$ and $G_{n}^{(c)}(m,r,s)$ satisfy 
the recurrence relation in Theorem \ref{thrm:Gcrr1}.
We have an expression of $G_{n}^{(s)}(1+m/r,r,s)$ 
by Proposition \ref{prop:Gsm2rgen}.
We prove the statement by induction on $n$.
First note that we have a relation for $q$-binomial coefficients:
\begin{align*}
\prod_{k=1}^{n}(1+ms^{k})=\sum_{k=0}^{n}m^{k}s^{d(1,k-1)}\genfrac{[}{]}{0pt}{}{n}{k}_{s}.
\end{align*}
We calculate the right hand side of the recurrence relation in Theorem \ref{thrm:Gcrr1}.

We have 
\begin{align*}
G_{n-2}^{(s)}(1+m/r,rs,s)&=\prod_{j=1}^{n-2}(1+ms^{j+1})  \\
&\quad+\sum_{j=1}^{\lfloor(n-1)/2\rfloor}
\sum_{k=0}^{n-1-2j}r^{j}s^{d(j,k)+j+k}m^{k}\genfrac{[}{]}{0pt}{}{n-2-j}{k}_{s}
\genfrac{[}{]}{0pt}{}{n-k-j-1}{j}_{s}.
\end{align*}
Similarly, $G_{n-1}^{(c)}(1+m/r,rs,s)$ has the following expression:
\begin{align*}
G_{n-1}^{(c)}(1+m/r,rs,s)&=(-1)^{n-1}mr^{n-2}s^{(n-1)(n+2)/2}
+\prod_{j=1}^{n-1}(1+ms^{j+1}) \\
&\quad+r\sum_{k=0}^{n-3}m^{k}s^{d(1,k)+k+1}[n-1]_{s}\genfrac{[}{]}{0pt}{}{n-3}{k}_{s} \\
&\quad+\sum_{j=2}^{\lfloor(n-2)/2\rfloor}r^{j}
\sum_{k=0}^{n-2j-1}m^{k}s^{d(j,k)+j+k}[n-1]_{s}
\genfrac{}{}{}{}{\prod_{i=1}^{j-1}[n-2j+i-1]_{s}}{[j]_{s}!}
\genfrac{[}{]}{0pt}{}{n-2j-1}{k}_{s} \\
&\quad+
\begin{cases}
\displaystyle r^{(n-1)/2}s^{(n-1)(n+1)/4}\genfrac{}{}{}{}{[n-1]_s}{[(n-1)/2]_{s}}, & n-1\equiv0\pmod2, \\
0, & otherwise.
\end{cases}
\end{align*}
We substitute the expressions above into the recurrence relation 
in Theorem \ref{thrm:Gcrr1}, and rearrange the terms to obtain 
the coefficients of $r^{j}m^{k}$.
Then, we obtain the expression (\ref{eqn:Gc1mrr1}) by a straightforward
calculation.
This completes the proof.
\end{proof}

If we specialize $m\rightarrow r(m-1)$ in Proposition \ref{prop:G1mrr1}, 
we can obtain the expression for the generating function $G_{n}^{(c)}(m,r,s)$
with general $(m,r,s)$.

\begin{cor}
\label{cor:Gcmrs1}
We have 
\begin{align*}
G_{n}^{(c)}(m,r,s)&=(-1)^{n}(m-1)r^{n}s^{n(n+1)/2}+\prod_{j=1}^{n}(1+r(m-1)s^{j}) \\
&\quad+\sum_{k=0}^{n-2}r^{k+1}(m-1)^{k}s^{d(1,k)}[n]_{s}\genfrac{[}{]}{0pt}{}{n-2}{k}_{s} \\
&\quad+\sum_{j=2}^{\lfloor(n-1)/2\rfloor}\sum_{k=0}^{n-2j}r^{k+j}(m-1)^{k}s^{d(j,k)}
[n]_{s}\genfrac{}{}{}{}{\prod_{i=1}^{j-1}[n-2j+i]_{s}}{[j]_{s}!}\genfrac{[}{]}{0pt}{}{n-2j}{k}_{s} \\
&\quad+
\begin{cases}
r^{n/2}s^{n^2/4}\genfrac{}{}{}{}{[n]_s}{[n/2]_s}, & n: even, \\
0, & n: odd.
\end{cases}
\end{align*}
\end{cor}
Note that the expression in Corollary \ref{cor:Gcmrs1} is different from the 
expression in Theorem \ref{thrm:Gcmrs1}.

The following proposition can be obtained by a similar method 
to the proof of Proposition \ref{prop:G1mrr1}.
Note that we have to rearrange the terms involving $q$-binomial 
coefficients if we specialize $m\rightarrow1$ in Proposition \ref{prop:G1mrr1}.
\begin{prop}
We have
\begin{align*}
G^{(c)}_{n}(1+1/r,r,s)&=(-1)^{n}r^{n-1}s^{n(n+1)/2}+\prod_{k=1}^{n}(1+s^k) \\
&\quad+\sum_{k=1}^{\lfloor n/2\rfloor-1}
r^{k}s^{k^2}\genfrac{}{}{}{}{\displaystyle [n]_{s}\prod_{j=k+1}^{n-k}[2j]_{s}}{\displaystyle \prod_{j=k+1}^{n-k}[j]_{s}}
\genfrac{}{}{}{}{\displaystyle \prod_{j=n-2k+1}^{n-1-k}[j]_s}{\displaystyle \prod_{j=2}^{k}[j]_s} \\
&\quad+
\begin{cases}
r^{n/2}s^{n^{2}/4}{\displaystyle \genfrac{}{}{}{}{[n]_s}{[n/2]_s}}, & n \text{: even}, \\[12pt]
r^{\lfloor n/2\rfloor}s^{\lfloor n/2\rfloor^{2}}{\displaystyle\genfrac{}{}{}{}{[n]_{s}[n+1]_{s}}{[(n+1)/2]_{s}}}, & n \text{: odd}.
\end{cases}
\end{align*}
\end{prop}

One can calculate the difference between $G_{n}^{(s)}$ and $G_{n}^{(c)}$ 
by recurrence relation in Theorem \ref{thrm:Gcrr1} and 
expressions in Propositions \ref{prop:Gsm2rgen} and \ref{prop:G1mrr1}.
A straightforward calculation yields the following proposition.
\begin{prop}
Set $(m,r,s)\rightarrow(1+m/r,r,s)$. 
We have 
\begin{align*}
[r^{1}]\left(G^{(s)}_{n}(1+m/r,r,s)-G^{(c)}_{n}(1+m/r,r,s)\right)
=ms^{n+1}\prod_{j=2}^{n-1}(1+ms^{j}),
\end{align*}
and 
\begin{align*}
\begin{split}
&[r^{2}]\left(G^{(s)}_{n}(1+m/r,r,s)-G^{(c)}_{n}(1+m/r,r,s)\right) \\
&\qquad=s^{n+1}\prod_{j=2}^{n-2}(1+ms^{j})+ms^{n+4}[n-3]_{s}\prod_{j=3}^{n-2}(1+ms^{j}).
\end{split}
\end{align*}
\end{prop}

We also consider the specialization $m\rightarrow 1+m/r$ for 
the generating function $G_{n}^{(2)}(m,r,s)$ defined in Definition \ref{defn:G1G2}.

\begin{prop}
The specialization $m\rightarrow 1+m/r$ gives the expression 
\begin{align}
\label{eqn:G2sp}
\begin{split}
\genfrac{}{}{}{}{G_{n}^{(2)}(m,r,s)}{rs^{n}}\Big|_{m\rightarrow 1+m/r}
&=\prod_{j=1}^{n-1}(1+ms^{j}) \\
&\quad+
\sum_{j=1}^{\lfloor(n-2)/2\rfloor}r^{j}\sum_{k=0}^{n-2-2j}m^{k}s^{d(j,k)}
\genfrac{[}{]}{0pt}{}{n-k-j-2}{j}_{s}\genfrac{[}{]}{0pt}{}{n-j-1}{k}_{s}  \\
&\quad+\sum_{j=1}^{n-2}(-1)^{j}r^{j}m^{n-1-j}s^{n(n-1)/2},
\end{split}
\end{align}
where 
\begin{align}
\label{eqn:defdjk}
d(j,k):=j(j+1)+k(k+2j+1)/2.
\end{align}
\end{prop}
\begin{proof}
One can prove the proposition by induction on $n$.
Recall that $G_{n}^{(2)}(m,r,s)$ satisfies the recurrence relation 
(\ref{eqn:G2recrel}) in Proposition \ref{prop:GcinG1G2}.
By substituting $m\rightarrow 1+m/r$ and rearranging the terms 
involving $m^{k}$, we can show that the expression (\ref{eqn:G2sp})
holds for general $n$.
\end{proof}

For general $(m,r,s)$, we have the following expression for $G_{n}^{(2)}(m,r,s)$.
\begin{cor}
\label{cor:G2mrs}
The generating function $G_{n}^{(2)}(m,r,s)$ is expressed as 
\begin{align*}
\genfrac{}{}{}{}{G_{n}^{(2)}(m,r,s)}{rs^{n}}
&=\prod_{j=1}^{n-1}(1+(m-1)rs^{j}) \\
&\quad+\sum_{j=1}^{\lfloor(n-2)/2 \rfloor} \sum_{k=0}^{n-2-2j}
r^{j+k}(m-1)^{k}s^{d(j,k)}\genfrac{[}{]}{0pt}{}{n-k-j-2}{j}_{s}\genfrac{[}{]}{0pt}{}{n-j-1}{k}_{s}  \\
&\quad+\sum_{j=1}^{n-2}(-1)^{j}r^{n-1}(m-1)^{n-1-j}s^{n(n-1)/2}.
\end{align*}
\end{cor}

\subsubsection{Generating function at \texorpdfstring{$m=1-1/r$}{m=1-1/r}.}
We have an expression for the specialization $m\rightarrow 1-m/r$ by 
putting $m\rightarrow -m$ in the expression (\ref{eqn:Gc1mrr1}) in 
Proposition \ref{prop:G1mrr1}.
If we further specialize $m\rightarrow 1$, {\it i.e.}, $m\rightarrow 1-1/r$,
one expects that the generating function has a simple closed formula.
The next proposition gives an explicit expression for the generating function 
$G_{n}^{(c)}(1-1/r,r,s)$. 
\begin{prop}
\label{prop:GfinB}
Let $B(n,k)$ be the following polynomial in $s$:
\begin{align*}
B(n,k):=(1-s^{n})\prod_{l=1}^{k-1}(1-s^{n-2k+l})\genfrac{}{}{}{}{\prod_{j=1}^{n-k}(1-s^{j})}{\prod_{j=1}^{k}(1-s^j)^2}.
\end{align*}
Then, the generating function $G^{(c)}_{n}(1-1/r,r,s)$ is given by 
\begin{align}
\label{eqn:Gc1minusr}
\begin{split}
G^{(c)}_{n}(1-1/r,r,s)&=\prod_{k=1}^{n}(1-s^{k})+(-1)^{n-1}r^{n-1}s^{n(n+1)/2} \\
&\quad+\sum_{k=1}^{\lfloor n/2\rfloor}r^{k}s^{k^2}B(n,k).
\end{split}
\end{align}
\end{prop}
\begin{proof}
We show that the expression (\ref{eqn:Gc1minusr}) by specializing 
$m\rightarrow -1$ in Proposition \ref{prop:G1mrr1}.

It is straightforward that $B(n,1)$ gives the coefficient of $r^{1}$ 
in $G_{n}^{(c)}(1-1/r,r,s)$.

We consider $2\le n$.
Since $d(j,k)=d(1,k-1)+j^2+kj$, we have 
\begin{align*}
\sum_{k=0}^{n-2j}(-1)^{k}s^{d(j,k)}\genfrac{[}{]}{0pt}{}{n-2j}{k}_{s}
&=s^{j^2}\prod_{i=1}^{n-2j}(1-s^{i+j}), \\
&=s^{j^2}\genfrac{}{}{}{}{\prod_{i=1}^{n-j}(1-s^{i})}{\prod_{i=1}^{j}(1-s^{i})}.
\end{align*}
Then, we have the expression (\ref{eqn:Gc1minusr}).
This completes the proof.
\end{proof}

\subsection{Generating function via the transfer matrix method}
\label{sec:tmmforGc}
Recall that the generating function $G_{n}^{(s)}$ can be computed 
by the transfer matrix method.
We have defined the matrix $\mathtt{MFib}(n;m,r,s)$ 
in Eq. (\ref{eqn:MFibdef}).
As in the case of $G_{n}^{(s)}(m,r,s)$, 
the generating function $G_{n}^{(c)}(m,r,s)$ can also be expressed 
in the entries of $\mathtt{MFib}(n;m,r,s)$.

We denote by $\mathrm{Tr}(M)$ the trace of a matrix $M$.
Then, the trace of $\mathtt{MFib}(n;m,r,s)$ gives the generating 
function $G_{n}^{(c)}(m,r,s)$.
\begin{prop}
\label{prop:tmmGc}
We have 
\begin{align*}
G_{n}^{(c)}(m,r,s)
=\mathrm{Tr}\left(\mathtt{MFib}(n;m,r,s)\right)
+(-1)^{n}r^{n}s^{n(n+1)/2}(m-1).
\end{align*}
\end{prop}
\begin{proof}
A dimer configuration on a circle can be obtained from a dimer 
configuration on a segment by gluing vertices at position $0$ and 
$n$.
Therefore, the generating function $G_{n}^{(c)}(m,r,s)$ can 
be roughly obtained by the trace of the matrix $\mathtt{MFib}(n;m,r,s)$.
We need to be careful for the weight of the dimer configuration with 
$n$ dimers.
In $\mathtt{MFib}(n;m,r,s)[2,2]$, the weight of the dimer configuration 
with $n$ dimers is given by $(m-1)^{n}$.
On the other hand, the weight of the dimer configuration with $n$ dimers
on a circle has $(m-1)^{n}+(-1)^{n}(m-1)$.
This is easily checked by use of Theorem \ref{thrm:Gcrr1} and 
the fact that the weight of the dimer configuration with $n$ dimers 
in $G_{n}^{(s)}(m,r,s)$ is $m(m-1)^{n-1}$.
Therefore, the difference of the weight between $G_{n}^{(c)}(m,r,s)$
and the trace of $\mathtt{MFib}(n;m,r,s)$ is 
given by $(-1)^{n}r^{n}s^{n(n+1)/2}(m-1)$. 
This completes the proof.
\end{proof}

\subsection{Specializations}
In this subsection, we consider several specializations of 
the formal parameters.

\paragraph{\bf Specialization $s=1$.}
In the case of $s=1$, one can easily obtain the generating 
function by solving  recurrence relations for the generating 
function.

Recall that $G^{(s)}(t,m,r,1)$ satisfies the recurrence relation (\ref{eqn:1dgf}) 
in Proposition \ref{prop:Gsgfrec}. 
By solving the recurrence relation for $s=1$, we obtain 
\begin{align*}
G^{(s)}(t,m,r,1)=\genfrac{}{}{}{}{1+rt}{1-(1+(m-1)r)t-rt^2}.
\end{align*}
Then, $G^{(s)}(t,m,r,1)$ and $G^{(c)}(t,m,r,1)$ satisfy the
recurrence relation in Corollary \ref{cor:GsGcrec} with $s=1$. 
The solution of the recurrence relation is given by
\begin{align*}
G^{(c)}(t,m,r,1)=\genfrac{}{}{}{}{1+2rt-r^2(-1+mt)t^2-mr^3(-1+m+t)t^3}{(1+rt)(1-(1+(m-1)r)t-rt^2)}.
\end{align*}
Let $\widehat{F}^{n}_{k}(m)$ be the polynomial of $m$ which appears as the coefficient 
of $t^{n}r^{k}$ in the formal series $G^{(c)}(t,m,r,1)$.
Namely, we define $\widehat{F}^{n}_{k}(m)=[t^{n}x^{k}]G^{(c)}(t,m,r,1)$.

Recall that $F^{n}_{k}(m)$ in Definition \ref{defn:Fnkm} satisfies 
the recurrence relation (\ref{eqn:Frr1}) in Proposition \ref{prop:Gs1exp}.
Similarly, $\widehat{F}^{n}_{k}(m)$ and $F^{n}_{k}(m)$ satisfy the 
following recurrence relation.
\begin{prop}
\label{prop:FwhFrec}
Let $\widehat{F}^{n}_{k}(m)$ and $F^{n}_{k}(m)$ be polynomials of $m$ defined as above.
Then, we have 
\begin{align*}
\widehat{F}^{n}_{k}(m)=F^{n}_{k}(m)+F^{n-2}_{k-1}(m)-\widehat{F}^{n-1}_{k-1}(m),
\end{align*}
where $n\ge3$.
\end{prop}
\begin{proof}
We substitute $s=1$ in the recurrence relation in Theorem \ref{thrm:Gcrr1}, and 
expand $G_{n}^{(s)}$ and $G_{n}^{(c)}$ in terms of $F^{n}_{k}(m)$ and 
$\widehat{F}^{n}_{k}(m)$.
\end{proof}

\begin{prop}
\label{prop:Fwhnk}
We have 
\begin{align}
\label{eqn:Fwhnk}
\begin{split}
\widehat{F}^{n}_{n}(m)&=(m-1)^{n}+(-1)^{n}(m-1), \\
\widehat{F}^{n}_{k}(m)
&=\genfrac{}{}{}{}{\partial_{m}^{n-k}}{(n-k)!}\left(m^{n-k}\widehat{F}^{k}_{k}(m)\right)
+\genfrac{}{}{}{}{\partial_{m}^{n-k-1}}{(n-k-1)!}\left( m^{n-k-1}\widehat{F}^{k-1}_{k-1}(m)\right)
+(-1)^{k-1}m.
\end{split}
\end{align}
\end{prop}
\begin{proof}
We prove the statement by induction on $n$.
When $n\le 3$, one can easily show that the expression (\ref{eqn:Fwhnk}) 
holds by a straightforward calculation.

We assume that $n\ge4$ and Eq. (\ref{eqn:Fwhnk}) holds up to $n-1$.
Then, since Proposition \ref{prop:defF} and $m(m-1)^{k-1}=(m-1)^{k}+(m-1)^{k-1}$,  we have 
\begin{align*}
F^{k}_{k}(m)=\widehat{F}^{k}_{k}(m)+\widehat{F}^{k-1}_{k-1}(m).
\end{align*}
From the expression (\ref{eqn:defF}) in Proposition \ref{prop:defF} 
and the recurrence relation in Proposition \ref{prop:FwhFrec}, 
we have 
\begin{align*}
\widehat{F}^{n}_{k}(m)&=F^{n}_{k}(m)+F^{n-2}_{k-1}(m)-\widehat{F}^{n-1}_{k-1}(m), \\
&=\genfrac{}{}{}{}{\partial_{m}^{n-k}}{(n-k)!}\left(m^{n-k}\widehat{F}^{k}_{k}+m^{n-k}\widehat{F}^{k-1}_{k-1}\right)
+\genfrac{}{}{}{}{\partial_{m}^{n-k-1}}{(n-k-1)!}\left(m^{n-k-1}\widehat{F}^{k-1}_{k-1}+m^{n-k-1}\widehat{F}^{k-2}_{k-2}\right) \\
&\quad-\genfrac{}{}{}{}{\partial_{m}^{n-k}}{(n-k)!}\left(m^{n-k}\widehat{F}^{k-1}_{k-1}\right)
-\genfrac{}{}{}{}{\partial_{m}^{n-k-1}}{(n-k-1)!}\left(m^{n-k-1}\widehat{F}^{k-2}_{k-2}\right)-(-1)^{k-2}m, \\
&=\genfrac{}{}{}{}{\partial_{m}^{n-k}}{(n-k)!}\left(m^{n-k}\widehat{F}^{k}_{k}(m)\right)
+\genfrac{}{}{}{}{\partial_{m}^{n-k-1}}{(n-k-1)!}\left( m^{n-k-1}\widehat{F}^{k-1}_{k-1}(m)\right)
+(-1)^{k-1}m,
\end{align*}
which completes the proof.
\end{proof}

\paragraph{\bf Specialization $m=1$ and $s=1$.}
The generating functions $G_{n}^{(c)}(r):=G^{(c)}_{n}(m=1,r,s=1)$  with $2\le n$ satisfy 
the recurrence relation of Fibonacci type (see Theorem \ref{thrm:GcFib}) 
\begin{align*}
G_{n}^{(c)}(r)=G_{n-1}^{(c)}(r)+rG_{n-2}^{(c)}(r),
\end{align*}
with the initial condition $G^{(c)}_{0}(r)=2$ and $G^{(c)}_{1}(r)=1$.
Therefore, the polynomial sequence $G^{(c)}_{n}(r)$, $1\le n$, gives 
the Lucas polynomials (see the sequence A034807 in OEIS \cite{Slo}).

\paragraph{\bf Specialization $m=1$ and $r=1$.}
We substitute $(m,r)=(1,1)$ in the expression in Theorem \ref{thrm:Gcmrs1},
We obtain 
\begin{align*}
G_{n}^{(c)}(1,1,s)
=\sum_{j=0}^{\lfloor n/2\rfloor}s^{j^2}
\genfrac{}{}{}{}{[n]_s}{[n-j]_s}\genfrac{[}{]}{0pt}{}{n-j}{j}_{s}.
\end{align*}

\paragraph{\bf Specialization $r=1$ and $s=1$.}
We consider the specialization $(r,s)=(1,1)$ for the generating 
functions $G^{(\ast)}_{n}(m,1,1)$ with $\ast$ is either $s$ or $c$.

\begin{prop}
\label{prop:Gcrs1}
Suppose $(m,r,s)=(m,1,1)$.
We have 
\begin{align}
\label{eqn:Gcm11}
\begin{split}
G_{2p}^{(c)}(m,1,1)&=1+m+\sum_{k=1}^{p}\genfrac{}{}{}{}{p}{k}\genfrac{(}{)}{0pt}{}{p+k-1}{2k-1}m^{2k}, \\
G_{2p+1}^{(c)}(m,1,1)&=1+2pm+\sum_{k=1}^{p}\genfrac{}{}{}{}{2p+1}{2k+1}\genfrac{(}{)}{0pt}{}{p+k}{k}m^{2k+1}.
\end{split}
\end{align}
\end{prop}
\begin{proof}
We prove the statement by induction on $p$.
For $G_{n}^{(c)}(m,1,1)$ with $n\le3$, a simple calculation gives the expression (\ref{eqn:Gcm11}).

Since $(r,s)=(1,1)$, we have the recurrence relation for $4\le n$ by Theorem \ref{thrm:GcFib}:
\begin{align*}
G_{n}^{(c)}(m,1,1)=mG_{n-1}^{(c)}(m,1,1)+G_{n-2}^{(c)}(m,1,1)+(-1)^{n}m(m-1).
\end{align*}
The induction assumption gives the expression (\ref{eqn:Gcm11}), which completes the proof.
\end{proof}

\section{Generating functions in the large size limit}
\label{sec:gfinfty}
In this section, we study the large $n$ limit of generating functions of dimers 
in one-dimension introduced in the previous two sections.
Even in the large $n$ limit ($n$ is the size of the system), we are able to observe 
the difference of the boundary conditions (defined on a segment or on a circle).
Nevertheless, the generating functions for two models give similar behaviors such 
as the appearance of Rogers--Ramanujan type sum in the large $n$ limit. 

Let $G^{(\ast)}_{n}(m,r,s)$ be the generating functions of dimers
where $\ast$ is either $s$, $c$ or $2$ ($s$ stands for segment and $c$ for circle). 
The generating functions are defined in Definition \ref{defn:Gtmrs}, Eq. (\ref{eqn:degGastmrs}) 
in Definition \ref{defn:defGcn} and Definition \ref{defn:G1G2}.

\subsection{Specialization \texorpdfstring{$r=1$}{r=1}}
In this subsection, we consider the specialization $r=1$.
When $m=1$, the generating function has only one formal variable $s$ and 
we obtain the Rogers--Ramanujan identities. 
Apart from $m=1$, we have simple explicit formulae only when $m=2$.
For $m\ge3$, the generating functions do not seem to have a simple formula and 
they can be expressed in terms of the formal power series $f^{RR}_{i}$, $1\le i$, 
defined in Section \ref{sec:notation}.

We start from the following simple observation.
\begin{prop}
\label{prop:Ginfm1}
Let $G^{(\ast)}_{n}(m,r,s)$ be the generating function for dimers on a segment or a circle.
Then, we have 
\begin{align}
\label{eqn:Ginfm1}
G_{\infty}^{(s)}(m=1,r,s)=1+\sum_{1\le n}\genfrac{}{}{}{}{r^{n}s^{n^2}}{\prod_{k=1}^{n}(1-s^{k})}.
\end{align}
The generating function $G^{(c)}_{\infty}(m=1,r,s)$ is given by 
\begin{align}
\label{eqn:Gcinfm1}
\lim_{n\rightarrow\infty}\genfrac{}{}{}{}{G^{(s)}_{n}(1,r,s)-G^{(c)}_{n}(1,r,s)}{r^{2}s^{n+1}}
&=1+\sum_{1\le n}\genfrac{}{}{}{}{r^{n}s^{n(n+2)}}{\prod_{k=1}^{n}(1-s^{k})}, 
\end{align}
\end{prop}
\begin{proof} 
From Proposition \ref{prop:Gnm1} and 
\begin{align*}
\lim_{n\rightarrow\infty}\genfrac{[}{]}{0pt}{}{n-j+1}{j}_{s}
&=\lim_{n\rightarrow\infty}\prod_{k=1}^{j}\genfrac{}{}{}{}{1-s^{n-j-k+2}}{1-s^{j-k+1}}, \\
&=\prod_{k=1}^{j}(1-s^{k})^{-1},
\end{align*}
we obtain Eq. (\ref{eqn:Ginfm1}) for dimers on a segment.

Similarly, we obtain Eq. (\ref{eqn:Gcinfm1}) from Proposition \ref{prop:Gscdif}.
\end{proof}

We take a specialization $r=1$ and consider the 
generating function $G^{(s)}_{\infty}(m,1,s)$.
The generating function for $m=1$ and $r=1$ has 
a nice formula by use of the Rogers--Ramanujan identity.
Further, when $m=2$, the generating function for $r=1$ also 
has a nice formula closely related to the Rogers--Ramanujan 
identity.
\begin{prop}
\label{prop:Gsm1}
We have 
\begin{align*}
G^{(s)}_{\infty}(m=1,1,s)&=f^{RR}_{1}(s),
\end{align*}
and 
\begin{align*}
\lim_{n\rightarrow\infty}\genfrac{}{}{}{}{G^{(s)}_{n}(1,1,s)-G^{(c)}_{n}(1,1,s)}{s^{n+1}}
&=s^{-1}(f^{RR}_{1}(s)-f^{RR}_{2}(s)), \\
&=f^{RR}_{3}(s), \\
&=1+s^3+s^4+s^5+s^6+s^7+2s^8+2s^9+3s^{10}+3s^{11}+\cdots.
\end{align*}
where the coefficients form the sequence A006141 in OEIS \cite{Slo}.
\end{prop}
\begin{proof}
From Proposition \ref{prop:Ginfm1} and setting $r=1$, the first identity is nothing 
but the Rogers--Ramanujan identity.

From Proposition \ref{prop:Ginfm1} and setting $r=1$, we have
\begin{align*}
1+\sum_{1\le n}\genfrac{}{}{}{}{s^{n(n+2)}}{\prod_{j=1}^{n}(1-s^j)}
&=s^{-1}\left(s+\sum_{1\le n}\genfrac{}{}{}{}{s^{(n+1)^2}}{\prod_{j=1}^{n}(1-s^j)}\right), \\
&=s^{-1}\left(s+\sum_{1\le n}(1-s^{n+1})\genfrac{}{}{}{}{s^{(n+1)^2}}{\prod_{j=1}^{n+1}(1-s^j)}\right), \\
&=s^{-1}\left(s+\sum_{2\le n}\genfrac{}{}{}{}{s^{n^2}}{\prod_{j=1}^{n}(1-s^{j})}
-\sum_{2\le n}\genfrac{}{}{}{}{s^{n^2+n}}{\prod_{j=1}^{n}(1-s^j)}  \right), \\
&=s^{-1}(f^{RR}_1(s)-f^{RR}_{2}(s)),
\end{align*} 
which completes the proof.
\end{proof}

Recall that the generating function $G_{n}^{(c)}(m,r,s)$ is the sum of two 
generating functions $G_{n}^{(1)}(m,r,s)$ and $G_{n}^{(2)}(m,r,s)$ as in 
Proposition \ref{prop:GcinG1G2}.
The generating function $G_{n}^{(1)}(m,r,s)$ is the same as $G_{n-1}^{(s)}(m,r,s)$,
so the large $n$ limit obeys Proposition \ref{prop:Gsm1}.
The next proposition shows the large $n$ behavior of the generating 
function $G_{n}^{(2)}(m,r,s)$.
\begin{prop}
\label{prop:G2m1}
The generating function $G_{n}^{(2)}(m,r,s)$ behaves in the large $n$ limit 
as follows:
\begin{align*}
\lim_{n\rightarrow\infty}\genfrac{}{}{}{}{G_{n}^{(2)}(1,1,s)}{s^{n}}
=f_2^{RR}(s).
\end{align*}
\end{prop}
\begin{proof}
We substitute $(m,r,s)=(1,1,s)$ in the expression in Corollary \ref{cor:G2mrs},
and take the large $n$ limit.
\end{proof}

The specialization $(m,r,s)=(2,1,s)$ in the generating function $G_{n}^{(s)}(m,r,s)$
gives the Rogers--Ramanujan type identity of modulus $6$.
\begin{prop}
\label{prop:Gsm2}
We have 
\begin{align}
\label{eqn:Gsmr2RR}
\begin{split}
G_{\infty}^{(s)}(m=2,1,s)&=\genfrac{}{}{}{}{f^{R}(-s^2,-s^4)}{f^{R}(-s)}\chi^{R}(s), \\
&=\left(\prod_{1\le n}\genfrac{}{}{}{}{1}{1+(-s)^{n}}\right)\prod_{0\le n}\prod_{j\in\{1,3,5\}}(1-s^{6n+j})^{-1}, \\
&=\prod_{0\le n}\genfrac{}{}{}{}{1+s^{2n+1}}{1-s^{2n+1}}, \\
&=1+2s+2s^{2}+4s^{3}+6s^{4}+8s^{5}+12s^{6} \\
&\quad+16s^{7}+22s^{8}+30s^{9}+40s^{10}+52s^{11}+\cdots,
\end{split}
\end{align}
where $\chi^{R}(s)$ is Ramanujan's theta function defined in Section \ref{sec:notation}.
The coefficients form the sequence A080054 in OEIS \cite{Slo}.
\end{prop}

To prove Proposition \ref{prop:Gsm2}, we first evaluate the generating function
$G^{(s)}_{n}(1+1/r,r,s)$ in the large $n$ limit.
\begin{lemma}
\label{lemma:Gsmr2}
We have 
\begin{align}
\label{eqn:Gsmr2rgen}
G^{(s)}_{\infty}(1+1/r,r,s)=\prod_{k=1}^{\infty}(1-s^{2k-1})^{-1}\cdot
\left(1+\sum_{k=1}^{\infty}r^{k}s^{k^2}\prod_{j=1}^{k}(1-s^{2j})^{-1}\right).
\end{align}
\end{lemma}
\begin{proof}
From Proposition \ref{prop:Gsmr2fin}, we consider the coefficient of $r^{j}$ in 
the large $n$ limit:
\begin{align*}
[2p+1]_s\genfrac{}{}{}{}{\prod_{k=1}^{2p-2j}[2k+2j]_s}{[2p-2j+1]_s!}&=
\genfrac{}{}{}{}{[2p+1]_s}{[2p-2j+1]_s}\genfrac{[}{]}{0pt}{}{2p-j}{j}_s\prod_{k=j+1}^{2p-j}(1+s^{k}), \\
&=\genfrac{}{}{}{}{[2p+1]_s}{[2p-2j+1]_s}\genfrac{}{}{}{}{\prod_{k=2p-2j+1}^{2p-j}[k]_s}{\prod_{k=1}^{j}[2k]_{s}}
\prod_{k=1}^{2p-j}(1+s^{k}), \\
&=\genfrac{}{}{}{}{[2p+1]_s}{[2p-2j+1]_s}\prod_{k=2p-2j+1}^{2p-j}[k]_s
\genfrac{}{}{}{}{\prod_{k=1}^{2p-j}(1+s^{k})}{\prod_{k=1}^{j}(1-s^{2k})}, \\
&\rightarrow\genfrac{}{}{}{}{\prod_{k=1}^{\infty}(1+s^{k})}{\prod_{k=1}^{j}(1-s^{2k})}, \qquad (n\rightarrow\infty).
\end{align*} 
Since we have $\prod_{n=1}^{\infty}(1+s^{n})\prod_{n=1}^{\infty}(1-s^{2n-1})=1$, 
we obtain Eq. (\ref{eqn:Gsmr2rgen}). 
\end{proof}

To simplify the right hand side of Eq. (\ref{eqn:Gsmr2rgen}) in the large 
$n$ limit with $r=1$,
we introduce the following lemma.
\begin{lemma}
\label{lemma:Gsmr2inf}
We have 
\begin{align}
\label{eqn:relGm2r1}
1+\sum_{k=1}^{n+1}
\genfrac{}{}{}{}{s^{k^2}\left(1-s^{2k(n-k+2)}\prod_{j=1}^{n-k+1}(1+s^{2j-1})\right)}
{\prod_{j=1}^{k}(1-s^{2j})}
=\prod_{j=0}^{n}(1+s^{2j+1}).
\end{align}
\end{lemma}
\begin{proof}
We will prove Eq. (\ref{eqn:relGm2r1}) by induction. The case $n=0$ is obvious 
and assume that Eq. (\ref{eqn:relGm2r1}) is valid up to $n-1$. 
Let $A(n)$ be the left hand side of Eq. (\ref{eqn:relGm2r1}). 
We show that $A(n+1)-A(n)$ is equal to $\prod_{j=1}^{n}(1+s^{2j+1})s^{2n+3}$.
The difference $A(n+1)-A(n)$ is equal to 
\begin{align*}
&s^{(n+2)^2}\prod_{j=1}^{n+2}(1-s^{2j})^{-1}-s^{(n+2)(n+4)}\prod_{j=1}^{n+2}(1-s^{2j})^{-1} \\
&\qquad\qquad+\sum_{k=1}^{n+1}\genfrac{}{}{}{}{s^{k(2n-k+4)}}{\prod_{j=1}^{k}(1-s^{2j})}
(1-s^{2k}-s^{2n+3})\prod_{j=1}^{n-k+1}(1+s^{2j-1}) \\
&=s^{(n+2)^2}\prod_{j=1}^{n+1}(1-s^{2j})^{-1}+\sum_{k=1}^{n+1}s^{k(2n-k+4)}\prod_{j=1}^{k-1}(1-s^{2j})^{-1}
\prod_{j=1}^{n-k+1}(1+s^{2j-1}) \\
&\qquad\qquad -s^{2n+3}\sum_{k=1}^{n+1}s^{k(2n-k+2)}\prod_{j=1}^{k}(1-s^{2j})^{-1}\prod_{j=1}^{n-k+1}(1+s^{2j-1}).
\end{align*}
Since we want to show that $A(n+1)-A(n)=s^{2n+3}A(n)$, from the equation above, it is enough to show that 
\begin{align}
\label{eqn:decomA1}
\begin{split}
&s^{2n+3}+s^{2n+3}\sum_{k=1}^{n+1}s^{k^2}\prod_{j=1}^{k}(1-s^{2j})^{-1} \\
&\qquad=s^{(n+2)^2}\prod_{j=1}^{n+1}(1-s^{2j})^{-1}
+\sum_{k=1}^{n+1}s^{k(2n-k+4)}\prod_{j=1}^{k-1}(1-s^{2j})^{-1}\prod_{j=1}^{n-k+1}(1+s^{2j-1}).
\end{split}
\end{align}
Note that we have
\begin{align}
\label{eqn:decomA2}
\begin{split}
&\sum_{k=1}^{n+1}s^{(k-1)(2n-k+3)}\prod_{j=1}^{k-1}(1-s^{2j})^{-1}\prod_{j=1}^{n-k+1}(1+s^{2j-1}) \\
&\qquad=\prod_{j=0}^{n-1}(1+s^{2j+1})+\sum_{k=1}^{n}s^{k(2n-k+2)}\prod_{j=1}^{k}(1-s^{2j})^{-1}\prod_{j=1}^{n-k}(1+s^{2j-1}).
\end{split}
\end{align}
By substituting Eq. (\ref{eqn:decomA2}) into Eq. (\ref{eqn:decomA1}) and rearranging terms, 
$A(n-1)=\prod_{j=0}^{n-1}(1+s^{2j+1})$.
This completes the proof.
\end{proof}

\begin{proof}[Proof of Proposition \ref{prop:Gsm2}]
The specialization $r=1$ in Lemma \ref{lemma:Gsmr2} is equivalent 
to the specialization $(m,r)=(2,1)$.
From Lemma \ref{lemma:Gsmr2}, we have 
\begin{align}
\label{eqn:Gsmr2fac}
G^{(s)}_{\infty}(2,1,s)=\prod_{k=1}^{\infty}(1-s^{2k-1})^{-1}\cdot
\left(1+\sum_{k=1}^{\infty}s^{k^2}\prod_{j=1}^{k}(1-s^{2j})^{-1}\right).
\end{align}
By taking the large $n$ limit in Lemma \ref{lemma:Gsmr2inf}, 
the second factor in the right hand side of Eq. (\ref{eqn:Gsmr2fac}) 
is equal to the product $\prod_{k=1}^{\infty}(1+s^{2k-1})$.
Thus, the third equality in Eq. (\ref{eqn:Gsmr2RR}) follows.

We show the third expression is equivalent to the second expression.
Since we have
\begin{align*}
\prod_{0\le n}\genfrac{}{}{}{}{1+s^{2n+1}}{1-s^{2n+1}}
=\prod_{0\le n}\genfrac{}{}{}{}{(1+s^{2n+1})}{(1-s^{6n+1})(1-s^{6n+3})(1-s^{6n+5})},
\end{align*}
it is enough to show that $\prod_{0\le n}(1+s^{2n+1})=\prod_{1\le n}(1+(-s)^{n})^{-1}$.
By multiplying $\prod_{0\le n}(1-s^{2n+1})$ on the both sides, we should have 
\begin{align*}
\prod_{0\le n}(1-s^{4n+2})=\prod_{1\le n}(1+s^{2n})^{-1}.
\end{align*}
By replacing $s^2$ by $s$, we should have $\prod_{0\le n}(1-s^{2n+1})=\prod_{1\le n}(1+s^n)^{-1}$,
which can be easily shown by use of $(1+s^{n})=(1-s^{2n})(1-s^{n})^{-1}$.

The first equality is obtained by the following observations:
\begin{align*}
\genfrac{}{}{}{}{f^{R}(-s^2,-s^4)}{f^{R}(-s)}
&=\genfrac{}{}{}{}{(q^2;q^6)_{\infty}(q^4;q^6)_{\infty}(q^6;q^6)_{\infty}}{
(q^1;q^3)_{\infty}(q^2;q^3)_{\infty}(q^3;q^3)_{\infty}}, \\
&=\genfrac{}{}{}{}{(q^2;q^2)_{\infty}}{(q;q)_{\infty}}, \\
&=\prod_{0\le n}(1-s^{2n+1})^{-1},
\end{align*}
and $\chi^{R}(s)=(-s;s^{2})_{\infty}=\prod_{0\le n}(1+s^{2n+1})$.
This completes the proof.
\end{proof}

For general $3\le m$ and $r=1$, we can write the generating function $G_{n}^{(s)}(m,r,s)$
in terms of $f^{RR}_{i}(s)$ defined in Eq. (\ref{eqn:deffRRi}).
\begin{prop}
We have 
\begin{align*}
G_{\infty}^{(s)}(m,1,s)
=
\sum_{n=0}^{\infty}(m-1)^{n}s^{n(n+1)/2}f_{n+1}^{RR}(s) \prod_{j=1}^{n}\genfrac{}{}{}{}{1}{(1-s^{j})}.
\end{align*}
\end{prop}
\begin{proof}
By putting $(m,r,s)\rightarrow(m-1,1,s)$ in Proposition \ref{prop:Gmrsgen} and 
taking the large $n$ limit, we obtain the desired equation.
\end{proof}

We consider the specialization $(m,r)=(2,1)$ of the 
generating function $G_{n}^{(2)}(m,r,s)$.
As we will see below, this specialization yields 
the Rogers--Ramanujan type identity of modulus $12$.
\begin{prop}
\label{prop:G2m2}
\begin{align*}
\lim_{n\rightarrow\infty}\genfrac{}{}{}{}{1}{s^{n}}G_{n}^{(2)}(m=2,r=1,s)
&=\genfrac{}{}{}{}{f^{R}(-s^{4},-s^{8})}{f^{R}(-s)}, \\
&=\prod_{1\le n}\genfrac{}{}{}{}{1+s^{2n}}{1-s^{2n-1}}, \\
&=1+s+2s^2+3s^3+4s^4+6s^5+9s^6+12s^7+\cdots.
\end{align*}
The coefficients appear in the integer sequence A001935 in OEIS \cite{Slo}.
\end{prop}
\begin{proof}
From Corollary \ref{cor:G2mrs},
we have 
\begin{align*}
\lim_{n\rightarrow\infty}\genfrac{}{}{}{}{1}{s^{n}}G_{n}^{(2)}(m=2,r=1,s)
=\prod_{j=1}^{\infty}(1+s^{j})+
\sum_{1\le j}\sum_{0\le k}s^{d(j,k)}\prod_{p=1}^{j}(1-s^{p})^{-1}
\prod_{q=1}^{k}(1-s^{q})^{-1}, 
\end{align*}
where $d(j,k)$ is defined in Eq. (\ref{eqn:defdjk})
By the definition of $q$-binomial coefficients, we have 
\begin{align*}
\sum_{k=0}^{\infty}s^{k(k+2j+1)/2}\prod_{p=1}^{k}(1-s^{k})^{-1}
=\prod_{p=1}^{\infty}(1+s^{j+p}).
\end{align*}
Thus, we have 
\begin{align*}
\lim_{n\rightarrow\infty}\genfrac{}{}{}{}{1}{s^{n}}G_{n}^{(2)}(m=2,r=1,s)
&=\sum_{j=0}^{\infty}s^{j(j+1)}\prod_{k=1}^{j}(1-s^{k})^{-1}\prod_{p=1}^{\infty}(1+s^{j+p}), \\
&=\prod_{p=1}^{\infty}(1+s^{p})
\left(\sum_{j=0}^{\infty}s^{j(j+1)}\prod_{k=1}^{j}(1-s^{k})^{-1}\prod_{p=1}^{j}(1+s^{p})^{-1}\right), \\
&=\prod_{0\le n}(1-s^{2n+1})^{-1}
\left(\sum_{j=0}^{\infty}s^{j(j+1)}\prod_{p=1}^{j}(1-s^{2p})^{-1}\right), \\
&=\prod_{1\le n}\genfrac{}{}{}{}{1+s^{2n}}{1-s^{2n-1}}.
\end{align*}
The first equality follows from the definition of Ramanujan's general theta functions.
This completes the proof.
\end{proof}

\subsection{Generating functions with various specializations}
In this subsection, we consider the generating functions in the large 
$n$ limit under various specializations of the formal variables $(m,r,s)$.

In the previous section, we observe the appearances of the Rogers--Ramanujan 
identities when $(m,r)=(1,1)$.
By considering the reversed generating functions (obtained by $s\rightarrow s^{-1}$), 
we obtain other types of Rogers--Ramanujan identities.

Other specializations give various formal power series with a simple 
combinatorial formula.

We start with the specialization $(m,r,s)\rightarrow(m/r,r,s)$.
Since the exponent of $r$ counts the number of connected component of dimers 
and $m$ is the ways of coloring of a dimer,
the generating functions $G_{n}^{(s)}(m,r,s)$ and $G_{n}^{(c)}(m,r,s)$ 
are polynomials with respect to $m,r$ and $s$ under the specialization. 
\begin{prop}
We have 
\begin{align}
\label{eqn:Gsinfmr}
\begin{split}
[r^{0}]G^{(s)}_{\infty}(m/r,r,s)&=\prod_{1\le k}(1+ms^{k}), \\
&=1+ms+ms^2+(m+m^2)s^3+(m+m^2)s^4+(m+2m^2)s^5 \\
&\quad+(m+2m^2+m^3)s^6+(m+3m^2+m^3)s^7+\cdots.
\end{split}
\end{align}
and 
\begin{align}
\label{eqn:Gcsmr1}
\begin{split}
\lim_{n\rightarrow\infty}[r^{1}]\left(\genfrac{}{}{}{}{G^{(s)}_{n}(m/r,r,s)-G^{(c)}_{n}(m/r,r,s)}{ms^{n+1}}\right)
&=\prod_{2\le k}(1+ms^{k}) \\
&=1+ms^2+ms^3+ms^4+(m+m^2)s^5 \\
&\quad+(m+m^2)s^6+(m+2m^2)s^7+(m+2m^2)s^8 \\
&\quad+(m+3m^2+m^3)s^9+\cdots.
\end{split}
\end{align}
\end{prop}
\begin{proof}
We first prove Eq. (\ref{eqn:Gsinfmr}). 
Since we replace $m$ by $m/r$ in $G_{\infty}^{(s)}(m,r,s)$ and $r$ is the total number of dimers, 
the left hand side of Eq. (\ref{eqn:Gsinfmr}) is equivalent to giving the weight of a dimer $m$ for any color. 
One can place at most one dimer of position $i$.
Recall that the exponent of $s$ is the sum of the positions of dimers. Therefore, a dimer at position $j$
has the weight $ms^{j}$, which means we have the weight $1$ if there is no dimer at position $j$.
From these observations, we have Eq. (\ref{eqn:Gsinfmr}) in the large $n$ limit.
	
The difference between a segment and a circle is that if we place dimers at position $1$ and $n$,
the colors of the dimers may be the same in the case of a segment, however, the colors should be 
different in the case of a circle.
We specialize $m$ to $m/r$, which means that the exponent of $m$ counts the number of connected dimers. 
The difference $G_{n}^{(s)}-G_{n}^{(c)}$ should contain the two dimers at position $1$ and $n$.
Thus, we have 
\begin{align*}
[r^{1}]\left(\genfrac{}{}{}{}{G^{(s)}_{n}(m/r,r,s)-G^{(c)}_{n}(m/r,r,s)}{ms^{n+1}}\right)
&=\prod_{2\le k\le n-1}(1+ms^{k}). 
\end{align*}
By taking the large $n$ limit, we obtain Eq. (\ref{eqn:Gcsmr1}).
\end{proof}

\begin{prop}
We have 
\begin{align}
\label{eqn:Gastinf1r}
G^{(\ast)}_{\infty}(1-1/r,r,s)=\prod_{j=1}^{\infty}(1-s^{j})\cdot\left(
1+\sum_{k=1}^{\infty}\genfrac{}{}{}{}{r^{k}s^{k^2}}{\prod_{j=1}^{k}(1-s^j)^2}\right),
\end{align}
where $\ast$ is either $s$ or $c$.
\end{prop}
\begin{proof}
By taking the large $n$ limit of $A(n,k)$ and $B(n,k)$ in Propositions \ref{prop:GfinA} and \ref{prop:GfinB}, 
we obtain Eq. (\ref{eqn:Gastinf1r}).
\end{proof}

\begin{remark}
The second factor in the right hand side of Eq. (\ref{eqn:Gastinf1r}) with $r=1$ is 
nothing but the generating function of the partition numbers. 
The coefficients appear 
in A000041 in OEIS \cite{Slo}, and the generating function is $\prod_{1\le n}(1-s^{n})^{-1}$.	
\end{remark}

Below, we consider the reversed polynomials with respect to $s$ and with the specialization 
$(m,r)=(1,1)$.
It is interesting to observe that the generating functions behave differently 
according to the parity of the size. 
We have two types of the systems, on a segment or on a circle, and 
two types of the parity of the size, even or odd.
Therefore, we have four types of generating functions and each of them 
gives distinct generalized Rogers--Ramanujan identity.

We first observe the large $n$ limit of the reversed polynomials for the generating 
function on a segment.
Let $d(n):=\lfloor (n+1)^{2}/4\rfloor$ be the top degree with respect to $s$ in $G^{(s)}_{\infty}(m,r,s)$.

\begin{prop}
\label{prop:Gsinvs}
Let $p\ge1$ be a positive integer.
We have 
\begin{align*}
\lim_{p\rightarrow\infty}s^{d(n)}G^{(s)}_{2p-1}(m=1,r=1,1/s)
&=\sum_{0\le k}\genfrac{}{}{}{}{s^{k^2}}{\prod_{j=1}^{2k}(1-s^{j})}, \\
&=\prod_{k=0}^{\infty}\genfrac{}{}{}{}{1}{(1-s^{2k+1})(1-s^{20k+4})(1-s^{20k+16})}, \\
&=1+s+s^2+2s^3+3s^4+4s^5+5s^6+7s^7+9s^8  \\
&\quad+12s^9+15s^{10}+19s^{11}+\cdots,
\end{align*}
and 
\begin{align*}
\lim_{p\rightarrow\infty}s^{d(n)}G^{(s)}_{2p}(m=1,r=1,1/s)
&=\sum_{0\le k}\genfrac{}{}{}{}{s^{k^2+k}}{\prod_{j=1}^{2k+1}(1-s^{j})}, \\
&=\prod_{k=0}^{\infty}\prod_{j\neq 0,\pm3,\pm4,\pm7,10(\bmod{20})}(1-s^{20k+j})^{-1}, \\
&=1+s+2s^2+2s^3+3s^4+4s^5+6s^6+7s^7+10s^{8}\\
&\quad+12s^{9}+16s^{10}+20s^{11}+\cdots.
\end{align*}
The coefficients form the sequence A122129 and A122135 respectively in OEIS \cite{Slo}.
\end{prop}
\begin{proof}
Suppose $n=2p$ and $r=m=1$. 
From Proposition \ref{prop:Gnm1} and change of the variable $j\rightarrow \lfloor(2p+1)/2\rfloor-j$, 
we obtain 
\begin{align*}
s^{d(n)}G^{(s)}_{2p}(m=1,r=1,1/s)&=\sum_{k=0}^{p}s^{k^2+k}\genfrac{[}{]}{0pt}{}{p+k+1}{p-k}_{s}, \\
&=\sum_{k=0}^{p}s^{k^2+k}\genfrac{[}{]}{0pt}{}{p+k+1}{2k+1}_{s}, \\
&\rightarrow\sum_{k=0}^{\infty}s^{k^{2}+k}\prod_{j=1}^{2k+1}(1-s^{j})^{-1},\quad (p\rightarrow\infty).	
\end{align*}
The second equality holds by Rogers--Ramanujan identities (see, e.g., \cite{And84a}).

For $n=2p+1$, one can show the desired equation in a similar manner and obtain the 
expression by Rogers--Ramanujan identities.
\end{proof}

In the following two propositions, we will see that the generalized Rogers--Ramanujan 
identities naturally come from the generating function for a circle.
Note that the identities are different from the case of a segment.

\begin{prop}
\label{prop:GcinftyRR}
Let $d'(2p):=p(p+1)$ and $d'(2p+1):=p(p+2)$ be the top degree with respect to $s$ in $G^{(c)}_{n}(m,r,s)$
for $n=2p$ or $2p+1$.
For $(m,r)=(1,1)$, we have 
\begin{align*}
\lim_{n\rightarrow\infty}s^{d'(2n+1)}G^{(c)}_{2n+1}(1,1,1/s)
&=\sum_{k=1}^{\infty}\genfrac{}{}{}{}{s^{k^2-1}}{\prod_{j=1}^{2k-1}(1-s^j)}, \\
&=\prod_{0\le n}(1-s^{2n+1})^{-1}(1-s^{20n+8})^{-1}(1-s^{20n+12})^{-1}, \\
&=\genfrac{}{}{}{}{f^{R}(-s^{4},-s^{16})}{\psi^{R}(-s)}, \\
&=1+s+s^2+2s^3+2s^4+3s^5+4s^6+5s^7+ 7s^8  \\
&\quad+ 9s^9+11s^{10}+\cdots.
\end{align*}
The coefficients correspond to the sequence A122130 in OEIS \cite{Slo}. 

Similarly, we have 
\begin{align*}
\lim_{n\rightarrow\infty}s^{d'(2n)}G^{(c)}_{2n}(1,1,1/s)&=\sum_{0\le k}\genfrac{}{}{}{}{s^{k^2+k}}{\prod_{j=1}^{2k}(1-s^{j})}, \\
&=\prod_{0\le n}\prod_{j\neq0,\pm1,\pm8,\pm9,10(\bmod{20})}(1-q^{20n+j})^{-1}, \\
&=1+s^2+s^3+2s^4+2s^5+4s^{6}+4s^{7}+6s^{8}+7s^{9}+10s^{10}+\cdots.
\end{align*}
The coefficients correspond to the sequence A122134 in OEIS \cite{Slo}.
\end{prop}
\begin{proof}
We first prove in the case of $n=2p+1$.
From Proposition \ref{prop:Gcm1}, putting $r=1$ and the change 
of variable $k\rightarrow p+1-k$, we have 
\begin{align*}
s^{d'(n)}G^{(c)}_{n}(1,1,1/s)
&=s^{d'(2p+1)}+\sum_{k=1}^{p}s^{k^2-1}[2p+1]_s\genfrac{[}{]}{0pt}{}{p-1+k}{p-k}_{s}[p+1-k]_{s}^{-1}, \\
&\rightarrow\sum_{k=1}^{\infty}\genfrac{}{}{}{}{s^{k^2-1}}{\prod_{j=1}^{2k-1}(1-s^{j})}.
\end{align*}
Other expressions follows as an example of multiple series 
Rogers--Ramanujan type identities (see e.g. \cite{And84a}).

One can similarly prove the case of $n=2p$ in a similar manner.
This completes the proof.
\end{proof}

The generating function $G_{n}^{(2)}(m,r,s)$ in the large $n$ limit 
behaves similar to the generating function $G_{n}^{(c)}(m,r,s)$.
Let $d(n)$ be an integer 
\begin{align*}
d(n):=
\begin{cases}
p(p-1), & n=2p, \\
(p-1)(p+1), & n=2p+1.
\end{cases}
\end{align*}
The value $d(n)$ is the top degree of $G_{n}^{(2)}(m=1,r=1,s)$ with respect to 
$s$.

\begin{prop}
\label{prop:G2invs}
Let $d(n)$ be defined as above. 
\begin{align*}
\lim_{p\rightarrow\infty}\genfrac{}{}{}{}{s^{d(n)}G_{2p}^{(2)}(m=1,r=1,1/s)}{s^{-2p}}
&=\sum_{0\le j}\genfrac{}{}{}{}{s^{j^2+j}}{\prod_{k=1}^{2j}(1-s^{k})}, \\
\lim_{p\rightarrow\infty}\genfrac{}{}{}{}{s^{d(n)}G_{2p+1}^{(2)}(m=1,r=1,1/s)}{s^{-(2p+1)}}
&=\sum_{1\le j}\genfrac{}{}{}{}{s^{j^2-1}}{\prod_{k=1}^{2j-1}(1-s^{k})}.
\end{align*}
\end{prop}
\begin{proof}
One can prove the proposition in a similar manner to the proof of Proposition 
\ref{prop:GcinftyRR}.
\end{proof}

We consider the reversed generating function of $G_{n}^{(s)}(m,r,s)$
with the specialization $(m,r,s)=(2,1,1/s)$.
This specialization gives the Rogers--Ramanujan type identities of 
modulus $4$.
\begin{theorem}
\label{thrm:revGnsm2}
The reversed generating function $G_{n}^{(s)}(2,1,1/s)$ in the large $n$ limit
is given by 
\begin{align*}
\lim_{n\rightarrow\infty}s^{n(n+1)/2}G_{n}^{(s)}(2,1,1/s)
&=2\genfrac{}{}{}{}{\psi^{R}(-s)}{\varphi^{R}(-s)}, \\
&=2(1+s+2s^2+3s^3+4s^4+6s^5+9s^6+12s^7+\cdots),
\end{align*}
where $\psi^{R}(x)$ and $\varphi^{R}(x)$ are Ramanujan theta functions defined 
in Section \ref{sec:notation}.
The coefficients appear in the sequence A001935 in OEIS \cite{Slo}.
\end{theorem}

Before proceeding to the proof of Theorem \ref{thrm:revGnsm2}, 
we introduce a lemma and a proposition as follows.

\begin{prop}
\label{prop:revGnsm2}
The reversed generating function $G_{n}^{(s)}(2,1,1/s)$ 
is expressed as 
\begin{align}
\label{eqn:revGnsm21}
\lim_{n\rightarrow\infty}s^{n(n+1)/2}G_{n}^{(s)}(2,1,1/s)
=\prod_{0\le j}(1+s^{j})
+\sum_{1\le k}\sum_{0\le j}
\genfrac{}{}{}{}{s^{k(k-1)+j(j-1)/2+kj}(1-s^{k+j})}{(s;s)_{k}(s;s)_{j}}.
\end{align}
\end{prop}
\begin{proof}
We substitute $(m,r,s)=(2,1,1/s)$ into the expression in Corollary \ref{cor:Gsgenetype2},
and change the variable $j$ to $n+1-2k-j$.
Then, we obtain the expression (\ref{eqn:revGnsm21}) by a simple calculation.
\end{proof}

\begin{lemma}
\label{lemma:revGnsm22}
We have
\begin{align}
\label{eqn:revGnsm221}
\sum_{1\le k}\sum_{0\le j}
\genfrac{}{}{}{}{s^{k(k-1)+j(j-1)/2+kj}(1-s^{k+j})}{(s;s)_{k}(s;s)_{j}}
=\prod_{j=1}^{\infty}(1+s^{j})\left(\sum_{1\le k}\genfrac{}{}{}{}{s^{k(k-1)}}{(s^2;s^2)_{k}}\right).
\end{align}
\end{lemma}
\begin{proof}
Let $A(r)$ be the left hand side of Eq. (\ref{eqn:revGnsm221}).
Then, we have 
\begin{align*}
A(r)&=\sum_{1\le k}\sum_{0\le j}\genfrac{}{}{}{}{s^{k(k-1)+j(j-1)/2+kj}}{(s;s)_{k}(s;s)_{j}}
-\sum_{1\le k}\sum_{0\le j}\genfrac{}{}{}{}{s^{k^2+j(j+1)/2+kj}}{(s;s)_{k}(s;s)_{j}}, \\
&=\prod_{j=1}^{\infty}(1+s^{j})\left(\sum_{1\le k}\genfrac{}{}{}{}{s^{k(k-1)}}{(s;s)_{k}(-s;s)_{k-1}}
-\sum_{1\le k}\genfrac{}{}{}{}{s^{k^2}}{(s;s)_{k}(-s;s)_{k}}\right), \\
&=\prod_{j=1}^{\infty}(1+s^{j})\left(\sum_{1\le k}\genfrac{}{}{}{}{s^{k(k-1)}}{(s^2;s^2)_{k}}\right),
\end{align*}
where we have used 
\begin{align}
\label{eqn:sbinf}
\sum_{j=0}^{\infty}\genfrac{}{}{}{}{(-1)^{j}z^{j}s^{j(j-1)/2}}{(s;s)_{j}}=(z;s)_{\infty},
\end{align}
with the specialization $z=-s^{k}$ and $z=-s^{k+1}$.
\end{proof}

\begin{proof}[Proof of Theorem \ref{thrm:revGnsm2}]
From Proposition \ref{prop:revGnsm2} and Lemma \ref{lemma:revGnsm22}, we obtain
\begin{align*}
\lim_{n\rightarrow\infty}s^{n(n+1)/2}G_{n}^{(s)}(2,1,1/s)
&=\prod_{1\le j}(1+s^j)\left(1+\sum_{k=1}^{\infty}\genfrac{}{}{}{}{s^{k(k-1)}}{(s^2;s^2)_{k}}\right), \\
&=\prod_{1\le j}(1+s^j)\sum_{0\le k}\genfrac{}{}{}{}{s^{k(k-1)}}{(s^2;s^2)_{k}}, \\
&=2\prod_{1\le j}(1+s^j)\prod_{1\le j}(1+s^{2j}), \\
&=2\genfrac{}{}{}{}{(s^2;s^2)_{\infty}(-s^{2};s^{2})_{\infty}}{(s;s)_{\infty}}, \\
&=2\genfrac{}{}{}{}{(s^{4};s^{4})_{\infty}}{(s;s)_{\infty}},
\end{align*}
where we have used Eq. (\ref{eqn:sbinf}) with $(z,s)=(-1,s^{2})$ for the third equality.

On the other hand, we have
\begin{align*}
\genfrac{}{}{}{}{\psi(-s)}{\varphi(-s)}&=\genfrac{}{}{}{}{f^{R}(-s,-s^3)}{f^{R}(-s,-s)}, \\
&=\genfrac{}{}{}{}{(s;s^4)_{\infty}(s^3;s^4)_{\infty}(s^4;s^4)_{\infty}}{(s;s^2)_{\infty}(s;s^2)_{\infty}(s^2;s^2)_{\infty}}, \\
&=\genfrac{}{}{}{}{(s^4;s^4)_{\infty}}{(s;s)_{\infty}},
\end{align*}
which completes the proof.
\end{proof}

Below, we consider the reversed generating function of $G_{n}^{(2)}(m,r,s)$ 
with the specialization $(m,r,s)=(2,1,1/s)$.
This specialization gives the Rogers--Ramanujan type identity 
of modulus $16$.
\begin{theorem}
\label{thrm:revGnm2}
The reversed generating function $G_{n}^{(2)}(m=2,r=1,1/s)$ satisfies 
\begin{align}
\label{eqn:revGnm2}
\lim_{n\rightarrow\infty}s^{n(n+1)/2}G_{n}^{(2)}(m=2,r=1,1/s)
=\genfrac{}{}{}{}{f^{R}(s^{6},s^{10})}{f^{R}(-s,-s^{3})}
-\begin{cases}
0, & n: even, \\
1, & n: odd.	
\end{cases}
\end{align}
\end{theorem}

Before proceeding to the proof of Theorem \ref{thrm:revGnm2},
we introduce some lemmas and a proposition.

\begin{prop}
We specialize the formal variables as $(m,r,s)=(2,1,1/s)$.
\label{prop:revG2m2}
\begin{align}
\label{eqn:revG2m2}
s^{n(n+1)/2}G_{n}^{(2)}(2,1,1/s)=
\sum_{0\le j}\sum_{0\le k}\genfrac{}{}{}{}{s^{\widetilde{d}(j,k)}}{(s;s)_{j}(s;s)_{k}(1-s^{j+k+1})}
+
\begin{cases}
1, & n : even, \\
0, & n : odd.
\end{cases}
\end{align}
where
\begin{align*}
\widetilde{d}(j,k):=(j+1)^{2}+\genfrac{}{}{}{}{1}{2}k(k+1)+k(j+1).
\end{align*}
\end{prop}
\begin{proof}
We substitute $(m,r,s)=(2,1,1/s)$ into the expression in Corollary \ref{cor:G2mrs} 
and perform a change of the variable $k\rightarrow n-2-2j-k$.
Then, it is straightforward to show that the reversed generating function 
satisfies Eq. (\ref{eqn:revG2m2}).
\end{proof}

To simplify the expression (\ref{eqn:revG2m2}), we introduce 
three formal power series as below.
\begin{defn}
\label{defn:revG2m2S}
We define three formal power series $S_{i}(r)$, $1\le i\le 3$.
\begin{align*}
S_{1}(r)&:=\sum_{0\le j}\sum_{0\le k}\genfrac{}{}{}{}{r^{j+k}s^{\widetilde{d}(j,k)}}{(s;s)_{j}(s;s)_{k}(1-s^{j+k+1})}, \\
S_{2}(r)&:=\sum_{0\le j}\sum_{0\le k}\genfrac{}{}{}{}{r^{j+k}s^{\widetilde{d}(j,k)}}{(s;s)_{j}(s;s)_{k}}, \\
S_{3}(r)&:=\sum_{0\le j}r^{j}s^{j(j+1)/2}\prod_{k=1}^{j}\genfrac{}{}{}{}{1+s^{k}}{1-s^{k}}.
\end{align*}
\end{defn}

\begin{lemma}
\label{lemma:Srecrel}
The power series $S_{1}(r)$ and $S_{2}(r)$ satisfy the relation:
\begin{align}
\label{eqn:S1S2}
S_{1}(r)-sS_{1}(rs)=S_{2}(r).
\end{align}
The power series $S_{2}(r)$ and $S_{3}(r)$ satisfy the difference 
equations:
\begin{align}
\label{eqn:S2}
S_{2}(r/s)-S_{2}(r)=rs S_{2}(r)+rs^{2} S_{2}(rs), \\
\label{eqn:S3}
S_{3}(r)-S_{3}(rs)=rsS_{3}(rs)+rs^{2}S_{3}(rs^{2}).
\end{align}
\end{lemma}
\begin{proof}
It is a routine to check that the power series $S_{i}(r)$, $1\le i\le 3$, satisfy
the relations.
For example, the coefficients of $r^{n}$ in the both sides of Eq. (\ref{eqn:S2}) 
are given by 
\begin{align*}
\genfrac{}{}{}{}{s^{n(n+1)/2+1}}{(s;s)_{n}}\prod_{j=1}^{n}(1+s^j).
\end{align*}
\end{proof}

By simple calculations, we have 
\begin{align*}
S_{2}(r)&=s+rs^3+2rs^4+2rs^5+(2r+r^2)s^{6}+\cdots, \\
S_{3}(rs)&=1+rs^2+2rs^3+2rs^4+(2r+r^2)s^5+\cdots.
\end{align*}
The expansions above and the recurrence relations (\ref{eqn:S2}) and (\ref{eqn:S3}) 
leads to the following relations between $S_{2}(r)$ and $S_{3}(r)$.
\begin{cor}
\label{cor:S2S3}
We have $S_{2}(r)=s S_{3}(rs)$.
Especially, 
\begin{align*}
S_{2}(1)=s S_{3}(s), \qquad S_{2}(1/s)=s S_{3}(1).
\end{align*}
\end{cor}

We give a different expression for $S_{1}(r)$ by solving the recurrence relation 
(\ref{eqn:S1S2}) in Lemma \ref{lemma:Srecrel}.
\begin{lemma}
We have 
\begin{align*}
S_{1}(r)=\sum_{0\le n}\genfrac{}{}{}{}{r^{n}s^{(n+1)(n+2)/2}(s^2;s^2)_{n}}{(s;s)_{n}(s;s)_{n+1}}.
\end{align*}
\end{lemma}
\begin{proof}
From the recurrence relation (\ref{eqn:S1S2}) and the relation in Corollary \ref{cor:S2S3}, we have 
\begin{align*}
S_{1}(r)&=s S_{1}(rs)+s S_{3}(rs), \\
&=\sum_{p=1}^{\infty}s^{p}S_{3}(rs^{p}), \\
&=\sum_{p=1}^{\infty}\sum_{0\le n}s^{p(n+1)}r^{n}s^{n(n+1)/2}
\prod_{k=1}^{n}\genfrac{}{}{}{}{1+s^{k}}{1-s^{k}}, \\
&=\sum_{0\le n}\genfrac{}{}{}{}{r^{n}s^{(n+1)(n+2)/2}(s^2;s^2)_{n}}{(s;s)_{n}(s;s)_{n+1}}.
\end{align*}
\end{proof}

\begin{cor}
\label{cor:revG2m2RR}
The power series $S_{1}(r)$ satisfies the Rogers--Ramanujan type identity of modulus $16$:
\begin{align*}
1+S_{1}(1)=\genfrac{}{}{}{}{f^{R}(s^{6},s^{10})}{f^{R}(-s,-s^3)}.
\end{align*}
\end{cor}
\begin{proof}
We have 
\begin{align*}
1+S_{1}(1)&=1+\sum_{0\le n}\genfrac{}{}{}{}{s^{(n+1)(n+2)/2}(s^2;s^2)_{n}}{(s;s)_{n}(s;s)_{n+1}}, \\
&=1+\sum_{1\le n}\genfrac{}{}{}{}{s^{n(n+1)/2}(s^2;s^2)_{n-1}}{(s;s)_{n-1}(s;s)_{n}}, \\
&=\genfrac{}{}{}{}{f^{R}(s^{6},s^{10})}{f^{R}(-s,-s^3)},
\end{align*}
where we have used the result in \cite[Eq. (7.13)]{GesStanton86} for the last equality.
\end{proof}


\begin{proof}[Proof of Theorem \ref{thrm:revGnm2}]
From Definition \ref{defn:revG2m2S}, Proposition \ref{prop:revG2m2} and Corollary \ref{cor:revG2m2RR}, 
the reversed generating function satisfies the expression (\ref{eqn:revGnm2}).
\end{proof}

We consider the specialization $(m,r,s)=(m,1,s)$ of the generating function
$G_{n}^{(s)}(m,r,s)$.
The coefficient of $m^{1}$ in $G_{n}^{(s)}(m,1,s)$ has a combinatorial 
interpretation as follows.
\begin{prop}
For $r=1$, we have 
\begin{align}
\label{eqn:Ginfint}
\begin{split}
[m^{1}]\left(G^{(s)}_{\infty}(m,1,s)\right)&=\sum_{1\le k}(-1)^{k-1}\genfrac{}{}{}{}{s^{k(k+1)/2}}{1-s^k}, \\
&=s+ s^2+ s^4+ 2s^6 +s^8+s^9+2s^{12}+2s^{15}+s^{16}+\cdots,
\end{split}
\end{align}
where the coefficient of $s^{n}$ is the number of divisors in the half-open 
interval $[\sqrt{n/2},\sqrt{2n})$.
The coefficients correspond to the sequence A067742 in OEIS \cite{Slo}.
\end{prop}
\begin{proof}
Recall that the weight for a dimer configuration is given by Eq. (\ref{eqn:wtdimerconf}).
A dimer configuration consists of several components of connected (or equivalently consecutive) 
dimers.
Since we have $r=1$, the weight of a component of $k$ consecutive dimers is given by 
$s^{a}m(m-1)^{k-1}$ with some integer $a$.
If we expand the weight, the term with the lowest degree with respect to $m$ is $(-1)^{k-1}s^{a}m$. 
Let $\sum_{1\le k}a_{k}s^{k}$ be the left hand side of Eq. (\ref{eqn:Ginfint}).
Then, by the observations above, $a_{n}$ is the number of partitions of $n$ into 
an odd number of consecutive parts minus the number of partitions of $n$ into an even 
number of consecutive parts. 

By \cite{HirHir05}, the number of partitions of $n$ into an odd (resp. even) number of 
consecutive parts is equal to the number of odd divisors of $n$ less than (resp. greater than)
$\sqrt{2n}$.

We will show that the number of partitions of $N$ into an even number of consecutive 
parts is equal to the number of odd divisors of $N$ less than $\sqrt{N/2}$.
Suppose that $N$ is expressed as a sum of $l=2p-1$, $1\le p$, consecutive integers: 
\begin{align*}
N=n+(n+1)+\cdots+(n+l-1).
\end{align*}
We construct a partition of $N$ into $l+1$ consecutive integers whose smallest part 
is $k$, {\it i.e.}, we have 
\begin{align*}
N=k+(k+1)+\cdots+(k+l).
\end{align*} 
Therefore we have 
\begin{align}
\label{eqn:Nby2p}
\begin{split}
k&=\genfrac{}{}{}{}{N-l(l+1)/2}{l+1}, \\
&=-p+\genfrac{}{}{}{}{1}{2}+\genfrac{}{}{}{}{N}{2p}.
\end{split}
\end{align}
Since $k$ is a positive integer and the expression (\ref{eqn:Nby2p}), 
if $N=2^{a}N'$ with $N'$ odd, then $p=2^ap'$ with $p'$ odd.
Thus, without loss of generality, $p$ is an odd divisor of $N$.
We also have the condition $k\ge1$, which is equivalent to 
\begin{align*}
p\le \genfrac{}{}{}{}{1}{4}\left(-1+\sqrt{8N+1}\right).
\end{align*}
Furthermore, by a straightforward calculation, we have 
\begin{align*}
\sqrt{N/2}-1<\genfrac{}{}{}{}{1}{4}\left(-1+\sqrt{8N+1}\right)<\sqrt{N/2}.
\end{align*}
From these observations, $p$ is an odd divisor of $N$ and less than $\sqrt{N/2}$.

The left hand side of Eq. (\ref{eqn:Ginfint}) is equal to the generating function of 
the numbers of odd divisors of $N$ in the interval $[\sqrt{N/2},\sqrt{2N})$.
By \cite{ChaEriStaMar02}, we obtain the right hand side of Eq. (\ref{eqn:Ginfint}),
which completes the proof.
\end{proof}	

We consider the specialization $(m,r,s)=(m,r,1)$.
Under this specialization, the coefficients of $m^{1}$ in the generating 
functions have simple formulae.
\begin{prop}
We have 
\begin{align}
\label{eqn:Gsm1}
[m^{1}]\left(G^{(s)}_{n}(m,r,s=1)\right)=\sum_{1\le p\le n}(-1)^{p+1}(n+1-p)r^{p},
\end{align}
and 
\begin{align}
\label{eqn:Gcmrs1m1}
[m^{1}]\left(G^{(c)}_{n}(m,r,s=1)\right)=(-1)^{n+1}(n-1)r^{n}+\sum_{1\le p\le n-1}(-1)^{p+1}n\cdot r^{p}.
\end{align}
\end{prop}
\begin{proof}
From Proposition \ref{prop:defF}, we have $F^{n}_{n}(m)=(-1)^{n-1}m+O(m^{2})$.
From Eq. (\ref{eqn:defF}), we $F^{n}_{n-k}(m)=(-1)^{n-k-1}(k+1)+O(m^{2})$.
Thus Eq. (\ref{eqn:Gsm1}) follows.

From the expression (\ref{eqn:Fwhnk}) in Proposition \ref{prop:Fwhnk},
we have 
\begin{align*}
\widehat{F}^{n}_{n}=(-1)^{n-1}(n-1)m+O(m^{2}).
\end{align*}
Again from the expression (\ref{eqn:Fwhnk}) in Proposition \ref{prop:Fwhnk}, 
we have 
\begin{align*}
\widehat{F}^{n}_{k}(m)&=(-1)^{k-1}(n-k+1)(k-1)m+(-1)^{k-2}(n-k)(k-2)m+(-1)^{k-1}m+O(m^2), \\
&=(-1)^{k-1}nm+O(m^2).
\end{align*}
From these observations, we have Eq. (\ref{eqn:Gcmrs1m1}).
\end{proof}

We first introduce the following lemma regarding to the $q$-binomial coefficients.
\begin{lemma}
\label{lemma:Gsmmsbinom}
\begin{align}
\label{eqn:Gsmmsbinom}
\sum_{j=0}^{n-1}s^{(j+1)(j+2)/2}\genfrac{[}{]}{0pt}{}{n-1}{j}_{s}\genfrac{[}{]}{0pt}{}{n-j}{1}_{s}
=s\genfrac{}{}{}{}{[n+1]_{s}}{[2]_{s}}
\prod_{j=0}^{\lfloor n/2\rfloor-1}\genfrac{}{}{}{}{[2\lfloor(n+1)/2\rfloor+2j]_{s}}{[1+2j]_{s}}.
\end{align}
\end{lemma}
\begin{proof}
Let $S(n)$ be the left hand side of Eq. (\ref{eqn:Gsmmsbinom}).
Since $[n+1]_{s}=[n-j]_{s}+s^{n-j}[j+1]_{s}$, we have 
\begin{align}
\label{eqn:Gsmmsbinom2}
S(n)=\sum_{j=0}^{n-1}[n+1]_{s}s^{(j+1)(j+2)/2}\genfrac{[}{]}{0pt}{}{n-1}{j}_{s}
-\sum_{j=0}^{n-1}s^{(j+1)(j+2)/2+(n-j)}\genfrac{[}{]}{0pt}{}{n-1}{j}_{s}[j+1]_{s}. 
\end{align}
By change of a variable $j\rightarrow n-1-j$ for the second sum of the right hand side of Eq. (\ref{eqn:Gsmmsbinom2}),
we have 
\begin{align}
\label{eqn:Gsmmsbinom3}
\begin{split}
&\sum_{j=0}^{n-1}s^{(j+1)(j+2)/2+(n-j)}\genfrac{[}{]}{0pt}{}{n-1}{j}_{s}[j+1]_{s} \\	
&\quad=\sum_{j=0}^{n-1}s^{(n-j)(n-j+1)/2+(j+1)}\genfrac{[}{]}{0pt}{}{n-1}{j}_{s}[n+1]_{s}
-\sum_{j=0}^{n-1}s^{(n-j)(n-j+1)/2+(n+1)}\genfrac{[}{]}{0pt}{}{n-1}{j}_{s}[j+1]_{s},\\
&\quad=\sum_{j=0}^{n-1}s^{(n-j)(n-j+1)/2+(j+1)}\genfrac{[}{]}{0pt}{}{n-1}{j}_{s}[n+1]_{s}
-\sum_{j=0}^{n-1}s^{(j+1)(j+2)/2+(n+1)}\genfrac{[}{]}{0pt}{}{n-1}{j}_{s}[n+j]_{s}.
\end{split}
\end{align}
Since the $q$-binomial coefficients satisfy 
\begin{align}
\label{eqn:qbinomdef}
\sum_{k=0}^{n}r^{k}s^{k(k+1)/2}\genfrac{[}{]}{0pt}{}{n}{k}_{s}=\prod_{j=1}^{n}(1+rs^{j}),
\end{align}
we have 
\begin{align*}
\sum_{j=0}^{n-1}[n+1]_{s}s^{(j+1)(j+2)/2}\genfrac{[}{]}{0pt}{}{n-1}{j}_{s}&=s[n+1]_{s}\prod_{j=1}^{n-1}(1+s^{j+1}), \\
\sum_{j=0}^{n-1}s^{(n-j)(n-j+1)/2+(j+1)}\genfrac{[}{]}{0pt}{}{n-1}{j}_{s}[n+1]_{s}&=s^{n+1}[n+1]_{s}\prod_{j=1}^{n-1}(1+s^{j}),
\end{align*}
where we have used $r=s$ and $r=s^{-n}$ respectively in Eq. (\ref{eqn:qbinomdef}).
By combining these observations together, we have 
\begin{align*}
S(n)&=s[n+1]_{s}\prod_{j=2}^{n-1}(1+s^{j})=s\genfrac{}{}{}{}{[n+1]_{s}}{[2]_{s}}\prod_{j=1}^{n-1}(1+s^{j}), \\
&=s\genfrac{}{}{}{}{[n+1]_{s}}{[2]_{s}}\prod_{j=1}^{n-1}\genfrac{}{}{}{}{(1-s^{2j})}{(1-s^{j})},
\end{align*}
and we obtain Eq. (\ref{eqn:Gsmmsbinom}) by rearranging terms. This completes the proof.
\end{proof}

\begin{prop}
We have 
\begin{align}
\label{eqn:Gs1mmsinf}
\begin{split}
[m]\left(G_{\infty}^{(s)}(1+1/m,m,s)\right)&=\genfrac{}{}{}{}{s}{1-s^{2}}\prod_{0\le k}(1-s^{2k+1})^{-1}, \\
&=s+s^2+2s^3+3s^4+4s^5+6s^6+8s^7+11s^8+14s^9+\cdots.	
\end{split}
\end{align}
The coefficients appear in the sequence A038348 in OEIS \cite{Slo}.
\end{prop}
\begin{proof}
By substituting $(m,r)\rightarrow(1+1/m,m)$ into the expression in Corollary \ref{cor:Gsgenetype2}, 
the coefficient of $m^{1}$ is given by 
\begin{align*}
\sum_{j=0}^{n-1}s^{(j+1)/(j+2)2}\genfrac{[}{]}{0pt}{}{n-1}{j}_{s}\genfrac{[}{]}{0pt}{}{n-j}{1}_{s},
\end{align*}
which can be expressed in terms of $q$-binomial coefficients by Lemma \ref{lemma:Gsmmsbinom}.
By taking the large $n$ limit, we obtain Eq. (\ref{eqn:Gs1mmsinf}).
\end{proof}

\begin{prop}
We have 
\begin{align}
\label{eqn:Gsmmsinf3}
\begin{split}
[m]\left(G_{\infty}^{(s)}(1/m,m,s)\right)
&=(-1)\left(\sum_{1\le k}\genfrac{}{}{}{}{x^{3k}}{1+s^{k}}\right)
\prod_{1\le k}(1+s^{k}), \\
&=-s^3-s^5-2s^6-2s^7-2s^8-4s^9-6s^{10}-6s^{11}-9s^{12}-\cdots.
\end{split}
\end{align}
The coefficients appear in the sequence A265251 in OEIS \cite{Slo}.
\end{prop}
\begin{proof}
From Theorem \ref{thrm:Gsgeneric}, the coefficient of $m^{1}$ in $G_{n}^{(s)}(1/m,m,s)$ is given by 
\begin{align*}
\sum_{k=0}^{n}s^{(k+1)(k+2)/2}\genfrac{[}{]}{0pt}{}{n-1}{k}_{s}[n-k]_{s}+
\sum_{k=0}^{n}(-k)s^{k(k+1)/2}\genfrac{[}{]}{0pt}{}{n}{k}_{s}.
\end{align*}
Note that the second sum in the above equation is simplified by the derivative of Eq. (\ref{eqn:qbinomdef})
with respect to $r$. Namely, we have 
\begin{align*}
\sum_{k=0}^{n}ks^{k(k+1)/2}\genfrac{[}{]}{0pt}{}{n}{k}_{s}
=\prod_{j=1}^{n}(1+s^{j})\left(\sum_{j=1}^{n}\genfrac{}{}{}{}{s^{j}}{1+s^{j}}\right).
\end{align*}
From Lemma \ref{lemma:Gsmmsbinom} and the expression above, 
we obtain 
\begin{align*}
[m^{1}]\left(G_{n}^{(s)}(1/m,m,s)\right)
=(-1)\prod_{j=1}^{n}(1+s^{j})
\left(\sum_{j=1}^{n}\genfrac{}{}{}{}{s^{j}}{1+s^{j}}-s\genfrac{}{}{}{}{[n+1]_{s}}{[2]_{s}}\genfrac{}{}{}{}{1}{1+s^{n}}
\right).
\end{align*}
Note that 
\begin{align*}
\sum_{j=1}^{n}\genfrac{}{}{}{}{s^{j}}{1+s^{j}}=\sum_{j=1}^{n}\genfrac{}{}{}{}{s^{3j}}{1+s^{j}}
-\sum_{j=1}^{n}s^{j}(s^{j}-1),
\end{align*}
and 
\begin{align*}
\sum_{j=1}^{n}s^{j}(s^{j}-1)=s(s-1)\genfrac{[}{]}{0pt}{}{n+1}{2}_{s},
\end{align*}
which can be verified by induction on $n$.
From these observations, we have 
\begin{align*}
[m^{1}]\left(G_{n}^{(s)}(1/m,m,s)\right)
=-\prod_{j=1}^{n}(1+s^{j})
\left(\sum_{j=1}^{n}\genfrac{}{}{}{}{s^{3j}}{1+s^{j}}\right)
+s^{2n+1}\genfrac{}{}{}{}{[n+1]_s}{[2]_{s}}\prod_{j=1}^{n-1}(1+s^{j}).
\end{align*}
By taking the large $n$ limit, we obtain Eq. (\ref{eqn:Gsmmsinf3}).
\end{proof}

\subsection{Catalan structure}
We first consider the specialization $(m,r,s)\rightarrow(1/m,r,1)$.
Under this specialization with large $n$ limit, the ratio of the generating 
functions of size $n$ and $n-1$ can be expressed in terms of the generating 
function for Narayana numbers. 
Since Narayana numbers are a fine refinement of Catalan numbers, we have 
a Catalan structure behind it. 
Secondly, we consider the specialization $(m,r,s)\rightarrow(m+1,1,1)$.
The ratio of the generating functions again possesses the generating function 
of Catalan number behind it.
Thirdly, the ratio of $G_{n}^{(s)}(m,r,s)$ and $G_{n}^{(c)}(m,r,s)$ with the 
specialization $(r,s)=(1,1)$ gives also an explicit formal 
power series.

We first see that the generating function of Narayana numbers naturally appears as the ratio of the generating functions 
with the specialization $s=1$.
\begin{prop}
\label{prop:Gs11m}
We have
\begin{align*}
\lim_{n\rightarrow\infty}\genfrac{}{}{}{}{G^{(s)}_{n-1}(1/m,r,s=1)}{m\cdot G^{(s)}_{n}(1/m,r,s=1)}
=\mathtt{Cat}'_{\infty}(m,r).
\end{align*}
where $\mathtt{Cat}'_{\infty}(x)$ is the generating function of Narayana numbers
defined by 
\begin{align*}
\mathtt{Cat}'_{\infty}(m,r):=\sum_{k=0}^{\infty}r^{-(k+1)}m^{-k}\sum_{p=0}^{k}(-1)^{k-p}r^{p}\cdot\mathtt{Nar}(k+1,p+1).
\end{align*}
Here, $\mathtt{Nar}(k,p)$ is the Narayana numbers defined in Eq. (\ref{eqn:defNara}).
\end{prop}
Before proceeding to the proof of Proposition \ref{prop:Gs11m},
we introduce a lemma which is needed in the proof.

\begin{lemma}
\label{lemma:Binomialxy}
Let $(n,k,a,x,y)$ be a set of positive integers satisfying $1\le n$, $0\le x$, 
$0\le y \le x$, $0\le k\le n-1-x$ and $n-k-1+y\le a \le n-x+y-1$.
Then, we have 
\begin{align}
\label{eqn:Binomialxy}
\begin{split}
&\sum_{u=0}^{k}\sum_{p=0}^{u}(-1)^{u}\genfrac{(}{)}{0pt}{}{2n-a+p-k-x-1}{n-k+u-y}
\genfrac{(}{)}{0pt}{}{a-p+u-y}{n-k+u-1}\mathtt{Nar}(u+1,p+1)  \\
&\hspace{2cm}=\genfrac{(}{)}{0pt}{}{2n-a-k-x-2}{n-k-1-y}\genfrac{(}{)}{0pt}{}{a-y-1}{n-k-2}.
\end{split}
\end{align}
\end{lemma}
\begin{proof}
Let $S(n,k,a;x,y)$ be the right hand side of Eq. (\ref{eqn:Binomialxy}).
We prove the lemma by induction.
When $x$ and $y$ are sufficiently large, the both sides are equal to zero.
When $(x,y)=(k,a-n+k+1)$, the both sides of Eq. (\ref{eqn:Binomialxy}) are equal to $1$.
Since the binomial coefficients satisfy 
\begin{align*}
\genfrac{(}{)}{0pt}{}{\alpha}{\beta}\genfrac{(}{)}{0pt}{}{\gamma}{\delta}
-\genfrac{(}{)}{0pt}{}{\alpha-1}{\beta-1}\genfrac{(}{)}{0pt}{}{\gamma-1}{\delta-1}
=\genfrac{(}{)}{0pt}{}{\alpha-1}{\beta}\genfrac{(}{)}{0pt}{}{\gamma}{\delta}
+\genfrac{(}{)}{0pt}{}{\alpha-1}{\beta-1}\genfrac{(}{)}{0pt}{}{\gamma-1}{\delta},
\end{align*}
the left hand side of Eq. (\ref{eqn:Binomialxy}) is equal to 
\begin{align*}
S(n,k,a;x-1,y-1)-S(n,k+1,a-1;x-1,y-1)=S(n,k,a;x,y-1)+S(n,k,a;x,y).
\end{align*}
The right hand side of Eq. (\ref{eqn:Binomialxy}) also satisfies 
the same recurrence relation with respect to $(k,a,x,y)$.
Thus, Eq. (\ref{eqn:Binomialxy}) follows from the induction.
\end{proof}

\begin{proof}[Proof of Proposition \ref{prop:Gs11m}]
We will show that 
\begin{align}
\label{eqn:Gss1Nara}
m\cdot G_{n}^{(s)}(1/m,r,1)\cdot \mathtt{Cat}'_{\infty}(m,r)=G_{n-1}^{(s)}(1/m,r,1)+O(m^{n}).
\end{align}
We have an expression for $G_n^{(s)}(m,r,1)$ by Proposition \ref{prop:Gss1}, one can calculate 
both sides of Eq. (\ref{eqn:Gss1Nara}).
We compare the coefficient of $m^{-(n-1)+k}$, $0\le k$, in the both sides.
By a simple calculation, we obtain that Eq. (\ref{eqn:Gss1Nara})is equivalent to 
\begin{align}
\label{eqn:BinomiCoef1}
\begin{split}
&\sum_{u=0}^{k}\left(\sum_{v=n-k+u}^{n}(-1)^{v-n+k+u}r^{v}\genfrac{(}{)}{0pt}{}{2n-v-k+u}{n-k+u}
\genfrac{(}{)}{0pt}{}{v-1}{n-k+u-1}\right) \hspace{2cm} \\
&\qquad\qquad\times\left(\sum_{p=0}^{u}(-1)^{u-p}r^{-(u+1)}\genfrac{}{}{}{}{r^{p}}{p+1}
\genfrac{(}{)}{0pt}{}{u}{p}\genfrac{(}{)}{0pt}{}{u+1}{p}\right) \hspace{2cm}\\
&\quad=\sum_{v=n-k-1}^{n-1}(-1)^{v-n+k+1}r^{v}\genfrac{(}{)}{0pt}{}{2n-v-k-2}{n-k-1}\genfrac{(}{)}{0pt}{}{v-1}{n-k-2}.
\end{split}
\end{align}	
The coefficient of $r^a$, $n-k-1\le a\le n-1$, in both sides of Eq. (\ref{eqn:BinomiCoef1}) is 
\begin{align*}
&\sum_{u=0}^{k}\sum_{p=0}^{u}(-1)^{u}\genfrac{(}{)}{0pt}{}{2n-a+p-k-1}{n-k+u}
\genfrac{(}{)}{0pt}{}{a-p+u}{n-k+u-1}\mathtt{Nar}(u+1,p+1)  \\
&\hspace{2cm}=\genfrac{(}{)}{0pt}{}{2n-q-k-2}{n-k-1}\genfrac{(}{)}{0pt}{}{a-1}{n-k-2}.
\end{align*}
This equality is nothing but Lemma \ref{lemma:Binomialxy} with $(x,y)=(0,0)$.
This completes the proof.
\end{proof}

\begin{remark}
\label{remark:NaraCat}
If we further specialize $r=1$ in Proposition \ref{prop:Gs11m}, the generating function is 
specialized as $\mathtt{Cat}'_{\infty}(m,1)=\mathtt{Cat}(-m^{-2})$, where $\mathtt{Cat}(x)$ 
is the generating function of Catalan numbers defined in Section \ref{sec:Catalan}. 
\end{remark}

We consider the ratio of the generating functions with the 
specializations $(m+1,1,1)$ and $(m,1,1)$ in the next proposition. 
In Proposition \ref{prop:Gs11m}, we observe an appearance of 
the generating function of Catalan numbers.
The ratio considered here is also expressed by a double sequence 
$T(n,k)$ which satisfies the recurrence relation involving Catalan numbers.

\begin{prop}
We have 
\begin{align*}
\genfrac{}{}{}{}{G^{(\ast)}_{n}(m+1,1,1)}{G^{(\ast)}_{n}(m,1,1)}
=1+\sum_{1\le k}T^{(\ast)}(n,k)m^{-k},
\end{align*}
where $T^{(\ast)}(n,k)$ satisfies the recurrence relation 
\begin{align}
\label{eqn:Gmplus1Trel}
T^{(\ast)}(n,k)=\sum_{1\le p\le k+1}(-1)^{p+1}\genfrac{(}{)}{0pt}{}{k+1}{p}T^{(\ast)}(n-p,k).
\end{align}
Here, $\ast$ is either $s$ or $c$.
\end{prop}
\begin{proof}
We give a proof of the statement in the case of $G_{n}^{(s)}$ since one can apply 
the same method to the case of $G_{n}^{(c)}$. 
Note that, from Corollary \ref{cor:Gcrr2}, $G_{n}^{(c)}$ with the specialization $(r,s)=(1,1)$ 
satisfies the recurrence relation  which has one extra term compared to the case of $G_{n}^{(s)}$.
	
We simply denote $G_{n}(m):=G_{n}^{(s)}(m,1,1)$.
Since we have a recurrence relation $G_{n}(m)=mG_{n-1}(m)+G_{n-2}(m)$, 
we have 
\begin{align}
\label{eqn:Gmplus11}
\begin{split}
\genfrac{}{}{}{}{G_{n}(m+1)}{G_{n}(m)}&=\genfrac{}{}{}{}{(m+1)G_{n-1}(m+1)+G_{n-2}(m+1)}{mG_{n-1}(m)+G_{n-2}(m)}, \\
&=\genfrac{}{}{}{}{(1+m^{-1})G_{n-1}(m+1)+m^{-1}G_{n-2}(m+1)}{G_{n-1}(m)}
\left(1+m^{-1}\genfrac{}{}{}{}{G_{n-2}(m)}{G_{n-1}(m)}\right)^{-1}.
\end{split}
\end{align}
Note that in the large $n$ limit, we have
\begin{align}
\label{eqn:Gmplus12}
\lim_{n\rightarrow\infty}\left(1+m^{-1}\genfrac{}{}{}{}{G_{n-2}(m)}{G_{n-1}(m)}\right)^{-1}=\mathtt{Cat}(-m^{2}),
\end{align}
from Proposition \ref{prop:Gs11m} and Remark \ref{remark:NaraCat}.
We also have 
\begin{align}
\label{eqn:Gmplus13}
\begin{split}
\genfrac{}{}{}{}{G_{n-2}(m+1)}{G_{n-1}(m)}&=\genfrac{}{}{}{}{G_{n-2}(m+1)}{mG_{n-2}(m)+G_{n-3}(m)}, \\
&=\genfrac{}{}{}{}{1}{m}\genfrac{}{}{}{}{G_{n-2}(m+1)}{G_{n-2}(m)}\left(1+m^{-1}\genfrac{}{}{}{}{G_{n-3}(m)}{G_{n-2}(m)}\right)^{-1}, \\
&\rightarrow\genfrac{}{}{}{}{1}{m}\genfrac{}{}{}{}{G_{n-2}(m+1)}{G_{n-2}(m)}\mathtt{Cat}(-m^{2}), \qquad (n\rightarrow\infty).
\end{split}
\end{align}
Substituting the relations (\ref{eqn:Gmplus12}) and (\ref{eqn:Gmplus13}) into Eq. (\ref{eqn:Gmplus11}), 
we obtain, in the large $n$ limit,
\begin{align}
\label{eqn:Gmplus14}
\genfrac{}{}{}{}{G_{n}(m+1)}{G_{n}(m)}
=\left(1+\genfrac{}{}{}{}{1}{m}\right)\genfrac{}{}{}{}{G_{n-1}(m+1)}{G_{n-1}(m)}\mathtt{Cat}(-m^{2})
+\genfrac{}{}{}{}{1}{m^{2}}\genfrac{}{}{}{}{G_{n-2}(m+1)}{G_{n-2}(m)}\left(\mathtt{Cat}(-m^{2})\right)^{2},
\end{align}
where
\begin{align*}
\left(\mathtt{Cat}(-m^{2})\right)^{2}=\sum_{0\le k}(-1)^{k}C_{k+1}m^{-2k}.
\end{align*}
By taking the coefficient of $m^{-k}$ in the both sides of Eq. (\ref{eqn:Gmplus14}), 
we obtain 
\begin{align}
\label{eqn:Gmplus15}
\begin{split}
T(n,k)&=\sum_{r=0}^{\lfloor k/2\rfloor}(-1)^{r}T(n-1,k-2r)C_{r}
+\sum_{r=0}^{\lfloor(k-1)/2\rfloor}(-1)^{r}T(n-1,k-1-2r)C_{r} \\
&\quad+\sum_{r=0}^{\lfloor(k-2)/2\rfloor}(-1)^{r}T(n-2,k-2-2r)C_{r+1}, \\
&=T(n-1,k)+\sum_{r=1}^{\lfloor k/2\rfloor}(-1)^{r}C_{r}\left(T(n-1,k-2r)-T(n-2,k-2r)\right) \\
&\quad+\sum_{r=0}^{\lfloor(k-1)/2\rfloor}(-1)^{r}C_{r}T(n-1,k-1-2r),
\end{split}
\end{align}
where $C_{r}$, $0\le r$, is the $r$-th Catalan number.

Note that we have 
\begin{align}
\label{eqn:Gmplus16}
\sum_{r=0}^{k}(-1)^{r}\genfrac{(}{)}{0pt}{}{k}{r}\left(T(n-r,k)-T(n-1-r,k)\right)=\sum_{r=0}^{k+1}(-1)^{r}\genfrac{(}{)}{0pt}{}{k+1}{r}T(n-r,k).
\end{align}

We prove the relation (\ref{eqn:Gmplus1Trel}) by induction on $k$.
When $k=0$, the relation is obvious. When $k=1$, the relation (\ref{eqn:Gmplus1Trel}) follows 
since $T(n,1)=n$. 
For general $k$, we have the relation (\ref{eqn:Gmplus1Trel}) 
by use of the relation (\ref{eqn:Gmplus15}) and Eq. (\ref{eqn:Gmplus16}).
This completes the proof.
\end{proof}

\begin{remark}
From the recurrence relation, the coefficient $T^{(\ast)}(n,k)$ is a 
polynomial of $n$ of degree $k$. 
Especially, the first two explicit expressions of $T^{(\ast)}(n,k)$ are 
\begin{align*}
T^{(\ast)}(n,1)&=n, \\
T^{(s)}(n,2)&=(n-2)(n+1)/2, \\
T^{(c)}(2p,2)&=(p+1)(2p+1), \\
T^{(c)}(2p+1,2)&=(p+1)(2p+3),
\end{align*}
where $\ast$ is either $s$ or $c$ and $1\le p$. 
\end{remark}

The polynomials $T^{(s)}(n,k)$ also satisfy the following recurrence relation.
\begin{prop}
We have 
\begin{align}
\label{eqn:Tsnk}
T^{(s)}(n,k)=\sum_{r=0}^{k}\genfrac{(}{)}{0pt}{}{\lfloor (2n-r)/2\rfloor}{n-r}\genfrac{(}{)}{0pt}{}{n-r}{l-r}
-\sum_{r=1}^{k}\genfrac{(}{)}{0pt}{}{\lfloor (2n-r)/2\rfloor}{n-r}T^{(s)}(n,l-r).
\end{align}
\end{prop}
\begin{proof}
The relation (\ref{eqn:Tsnk}) follows from the definition of $T^{(s)}(n,k)$ and 
the expression given by Proposition \ref{prop:Frel1} for $G^{(s)}(m,1,1)$.
\end{proof}

Below we consider the ratio of the generating functions $G_{n}^{(s)}(m,r,s)$ 
and $G_{n}^{(c)}(m,r,s)$ with the specialization $(r,s)=(1,1)$.
We introduce two integer sequences to describe the coefficients of $m$ 
in the ratio.

Let $\mu(n)$ be the signs of the power series of $x$:
\begin{align*}
\sum_{0\le n}\mu(n)x^{n}&:=\genfrac{}{}{}{}{1+x}{1+x^2}, \\
&=1+x-x^2-x^3+x^4+x^5-x^6-x^7+x^8+x^9-x^{10}+\cdots,
\end{align*}
the signs form the sequence A057077 in OEIS \cite{Slo}. 
A sequence of length $4$, $(1,1,-1,-1)$, is repeated.

We also define coefficients $\widehat{M}_{n}$ of a power series:
\begin{align*}
\sum_{0\le n}\widehat{M}_{n}x^{n}
&:=\genfrac{}{}{}{}{1}{2}\left(1+\sqrt{\genfrac{}{}{}{}{1+2x}{1-2x}}\right), \\
&=1+x+x^2+2 x^3+3 x^4+6x^5+10x^6+20x^7+35x^8+70x^9+126x^{10}+\cdots,
\end{align*}
where the coefficients form the sequence A210736 in OEIS \cite{Slo}.

The following proposition indicates the effect of boundaries for the generating 
functions.
\begin{prop}
\label{prop:Gsr1ratio}
Let $(r,s)=(1,1)$. We consider the following power series with respect to $1/m$: 
\begin{align}
\label{eqn:Gsr1ratio}
\begin{split}
\lim_{n\rightarrow\infty}\genfrac{}{}{}{}{G^{(s)}_{n}(m,1,1)}{G^{(c)}_{n}(m,1,1)}
&:=\sum_{0\le n}\mu(n)\widehat{M}_{n}m^{-n}, \\ 
&=1+\sum_{1\le n}(-1)^{\lfloor n/2\rfloor}\genfrac{(}{)}{0pt}{}{n-1}{\lfloor n/2\rfloor}m^{-n},
\end{split}
\end{align}
where $\mu(n)$ and $\widehat{M}_{n}$ are defined as above.
\end{prop}
Before proceeding to the proof of Proposition \ref{prop:Gsr1ratio}, 
we introduce a lemma.
\begin{lemma}
We have 
\begin{align}
\label{eqn:plq}
\begin{split}
\genfrac{(}{)}{0pt}{}{2p-l}{l}
-\sum_{k=l-q}^{l}(-1)^{l-k}\genfrac{}{}{}{}{p}{p-k}\genfrac{(}{)}{0pt}{}{2p-k-1}{k}
\genfrac{(}{)}{0pt}{}{2(l-k)-1}{l-k} \hspace{2cm} \\
=(-1)^{q+1}\sum_{k=0}^{q}\genfrac{(}{)}{0pt}{}{2p-l-1+k}{l-q-1}\genfrac{(}{)}{0pt}{}{q+k}{q},
\end{split} \\
\label{eqn:plq3}
\begin{split}
\genfrac{(}{)}{0pt}{}{2p-l}{l-1}
-\sum_{k=l-q}^{l-1}(-1)^{l-k-1}\genfrac{}{}{}{}{p}{p-k}\genfrac{(}{)}{0pt}{}{2p-k-1}{k}
\genfrac{(}{)}{0pt}{}{2(l-k-1)}{l-k-1} \hspace{2cm} \\
=2(-1)^{q}\sum_{k=0}^{q-1}\genfrac{(}{)}{0pt}{}{2p-l+k}{l-q-1}\genfrac{(}{)}{0pt}{}{q+k-1}{q-1},
\end{split} \\
\label{eqn:plq4}
\begin{split}
\genfrac{(}{)}{0pt}{}{2p-l+1}{l}
-\sum_{k=l-q}^{l}(-1)^{l-k}\genfrac{}{}{}{}{2p+1}{2p-2k+1}\genfrac{(}{)}{0pt}{}{2p-k}{k}
\genfrac{(}{)}{0pt}{}{2(l-k)-1}{l-k} \hspace{2cm}  \\
=(-1)^{q+1}\sum_{k=0}^{q}\genfrac{(}{)}{0pt}{}{2p-l+k}{l-q-1}\genfrac{(}{)}{0pt}{}{q+k}{q},
\end{split} \\
\label{eqn:plq5}
\begin{split}
\genfrac{(}{)}{0pt}{}{2p-l}{l}
-\sum_{k=l-q}^{l}(-1)^{l-k}\genfrac{}{}{}{}{2p+1}{2p-2k+1}\genfrac{(}{)}{0pt}{}{2p-k}{k}
\genfrac{(}{)}{0pt}{}{2(l-k)}{l-k} \hspace{2cm}  \\
=2(-1)^{q+1}\sum_{k=0}^{q}\genfrac{(}{)}{0pt}{}{2p-l+k}{l-q-1}\genfrac{(}{)}{0pt}{}{q+k}{q},
\end{split}
\end{align}
\end{lemma}
\begin{proof}
We will prove only Eq. (\ref{eqn:plq}) since other relations can be shown in a similar manner.
We show Eq. (\ref{eqn:plq}) by induction.
For $q=0$, the both sides of Eq. (\ref{eqn:plq}) are equal by a simple calculation.
We assume that Eq. (\ref{eqn:plq}) holds for up to $q$.
Then, Eq. (\ref{eqn:plq}) is equivalent to the following equation:
\begin{align}
\label{eqn:plq2}
\begin{split}
\sum_{k=0}^{q}\genfrac{(}{)}{0pt}{}{2p-l-1+k}{l-q-1}\genfrac{(}{)}{0pt}{}{q+k}{q}+
\sum_{k=0}^{q+1}\genfrac{(}{)}{0pt}{}{2p-l-1+k}{l-q-2}\genfrac{(}{)}{0pt}{}{q+k+1}{q+1} \hspace{2cm} \\
=\genfrac{}{}{}{}{p}{p-l+q+1}\genfrac{(}{)}{0pt}{}{2p-l+q}{l-q-1}\genfrac{(}{)}{0pt}{}{2q+1}{q}.
\end{split}
\end{align}
The left hand side of Eq. (\ref{eqn:plq2}) is equal to 
\begin{align*}
\begin{split}
\sum_{k=0}^{q}\genfrac{(}{)}{0pt}{}{2p-l-1+k}{l-q-1}\genfrac{(}{)}{0pt}{}{q+k}{k}+
\sum_{k=0}^{q}\genfrac{(}{)}{0pt}{}{2p-l+k}{l-q-1}\genfrac{(}{)}{0pt}{}{q+k+1}{k}  \hspace{2cm} \\
-\sum_{k=0}^{q}\genfrac{(}{)}{0pt}{}{2p-l-1+k}{l-q-1}\genfrac{(}{)}{0pt}{}{q+k+1}{k}
+\genfrac{(}{)}{0pt}{}{2p-l+q}{l-q-2}\genfrac{(}{)}{0pt}{}{2q+2}{q+1}, \\
\end{split} \\
\begin{split}
=\sum_{k=0}^{q}\genfrac{(}{)}{0pt}{}{2p-l-1+k}{l-q-1}\genfrac{(}{)}{0pt}{}{q+k}{k}+
\sum_{k=1}^{q+1}\genfrac{(}{)}{0pt}{}{2p-l-1+k}{l-q-1}\genfrac{(}{)}{0pt}{}{q+k}{k-1}  \hspace{2cm} \\
-\sum_{k=0}^{q}\genfrac{(}{)}{0pt}{}{2p-l-1+k}{l-q-1}\genfrac{(}{)}{0pt}{}{q+k+1}{k}
+\genfrac{(}{)}{0pt}{}{2p-l+q}{l-q-2}\genfrac{(}{)}{0pt}{}{2q+2}{q+1}, \\
\end{split}\\
\begin{split}
=\sum_{k=1}^{q}\genfrac{(}{)}{0pt}{}{2p-l-1+k}{l-q-1}
\left(\genfrac{(}{)}{0pt}{}{q+k}{k}+\genfrac{(}{)}{0pt}{}{q+k}{k-1}-\genfrac{(}{)}{0pt}{}{q+k+1}{k}
\right) \\
+\genfrac{(}{)}{0pt}{}{2p-l+q}{l-q-1}\genfrac{(}{)}{0pt}{}{2q+1}{q}
+\genfrac{(}{)}{0pt}{}{2p-l+q}{l-q-2}\genfrac{(}{)}{0pt}{}{2q+2}{q+1},
\end{split} \\
\begin{split}
=\genfrac{}{}{}{}{p}{p-l+q+1}\genfrac{(}{)}{0pt}{}{2p-l+q}{l-q-1}\genfrac{(}{)}{0pt}{}{2q+1}{q},
\end{split}
\end{align*}
where we have used the relation 
$\genfrac{(}{)}{0pt}{}{N}{M}=\genfrac{(}{)}{0pt}{}{N-1}{M}+\genfrac{(}{)}{0pt}{}{N-1}{M-1}$.
This completes the proof.
\end{proof}

\begin{proof}[Proof of Proposition \ref{prop:Gsr1ratio}]
We will show that 
\begin{align}
\label{eqn:Gsr11}
G^{(c)}_{n}(m,1,1)(\sum_{0\le k}\mu(k)\widehat{M}_{k}m^{-k})
=G^{(s)}_{n}(m,1,1)+O((1/m)^{-1}).	
\end{align} 
We show Eq. (\ref{eqn:Gsr11}) is valid for $n=2p$, $1\le p$ since one can 
prove the statement for $n=2p+1$ in a similar manner.
Note that $G^{(s)}_{n}(m,1,1)$ is a polynomial of $m$ with the top degree $n$.

Since the top degree of $G^{(s)}_{2p}(m,1,1)$ is $2p$, one can show the proposition 
by taking the large $p$ limit of Eq. (\ref{eqn:Gsr11}).

We calculate the left hand side of Eq. (\ref{eqn:Gsr11}).
From Proposition \ref{prop:Gcrs1}, 
the coefficient of $m^{2p-2l}$ is given by 
\begin{align}
\label{eqn:Gsr112}
\sum_{k=0}^{l}\genfrac{}{}{}{}{p}{p-k}\genfrac{(}{)}{0pt}{}{2p-k-1}{k}(-1)^{l-k}\genfrac{(}{)}{0pt}{}{2(l-k)-1}{l-k}
=\genfrac{(}{)}{0pt}{}{2p-l}{l},
\end{align}
where we have used Eq. (\ref{eqn:plq}) for $q=l$.
Then, the right hand side of Eq. (\ref{eqn:Gsr112}) is equal to the coefficient of $m^{2p-2l}$ 
in $G^{(s)}_{2p}(m,1,1)$ given by Proposition \ref{prop:Frel1}.

We calculate the coefficient of $m^{2p-2l+1}$, $1\le l\le p-1$, in the left hand side of Eq. (\ref{eqn:Gsr112})
in a similar manner by use of Eq. (\ref{eqn:plq3}).
Note that the generating function $G_{n}^{(s)}(m,1,1)$ is a polynomial 
of $m$ of degree $n$.
Therefore, by taking the large $n$ limit in Eq. (\ref{eqn:Gsr11}), 
we obtain the expression (\ref{eqn:Gsr1ratio}).

In the case of $n=2p+1$, one can show that Eq. (\ref{eqn:Gsr11}) by use of the relations (\ref{eqn:plq4}) and 
(\ref{eqn:plq5}). This completes the proof.
\end{proof}

\section{Correlation functions}
\label{sec:cf}
In this section, we calculate correlation functions in the large $n$ limit.
Before that, we start with the definition of a formal power series $\chi(m,r,s)$
which plays a central role in the analysis of the correlation functions.
Then, we introduce the correlation functions, which we call 
the emptiness formation probabilities and the moments.
The expressions of the correlation functions lead to the study 
of the generating functions in terms of the formal power series $\chi(m,r,s)$.
We also show that the reversed generating functions are expressed in terms of 
formal power series characterized by Dyck or Schr\"oder paths.

\subsection{The formal power series \texorpdfstring{$\chi(m,r,s)$}{chi(m,r,s)}}
\label{sec:defchi}
We define and study the formal power series $\chi(m,r,s)$.
This power series can be viewed as a generating function of 
Dyck/Motzkin paths with appropriate statistics.

\begin{defn}
\label{defn:chi}
Let $\chi(m,r,s)$ be the formal power series satisfying 
\begin{align*}
f(rs)\chi(m,r,s)=1-rs \chi(m,rs,s)\chi(m,r,s).
\end{align*}
When $f(x)=1+x(m-1)$, first few terms are given by 
\begin{align*}
\chi(m,r,s)&=1-mrs+r^2(m^2s^2+ms^3)+r^3(-m^3 s^3-2m^2s^4-m^2s^5-ms^6)\\
&+r^4(m^4s^4+3m^3s^5+(m^2+2m^3)s^6+(2m^2+m^3)s^7+2m^2s^8+m^2s^9+ms^{10})+\cdots.
\end{align*}
\end{defn}

\begin{remark}
Two remarks are in order.
\begin{enumerate}
\item
By comparing the defining relation for $\alpha(r)$ in Eq. (\ref{eqn:alpha3}) and 
the one for $\chi(m,r,s)$ in Eq. (\ref{defn:chi}), we notice that the formal power 
series $\chi(m,r,s)$ is the inverse of $\alpha(r)$.
Since the appearance of $\chi(m,r,s)$ seems to be natural in the study of correlation 
functions, we call it $\chi$ rather than $\alpha^{-1}$.
\item 
The defining relation for $\chi(m,r,s)$ is a generalization of the 
Catalan identity due to Carlitz and Riordan \cite{CarRio64}.
\end{enumerate}
\end{remark}

The polynomial $f(r)$ and its differential $f'(r)$ (with respect to $r$) can be expressed in terms of 
$\chi(r,s)$ as follows.
\begin{align}
\label{eqn:finchi}
f(r)&=\genfrac{}{}{}{}{1}{\chi(rs^{-1},s)}-r\chi(r,s), \\
\label{eqn:fpinchi}
\partial_{r}f(r)&=-\chi(r,s)-r\chi'(r,s)-s^{-1}\genfrac{}{}{}{}{\chi'(rs^{-1},s)}{\chi(rs^{-1},s)^{2}},
\end{align}
where $\chi'(r,s)$ is the differential of $\chi(r,s)$ with respect to $r$.

Before giving an expression of $\chi(m,r,s)$ in terms of the formal power 
series $f(x)$, we develop the notion of Dyck paths.

Let $Y(\mu)$ be a Young diagram consisting of boxes which 
are below $\mu_{0}$ and above $\mu$, where $\mu_{0}$ is the 
highest Dyck path.
Let $d\in Y(\mu)$ be a box in the Young diagram $Y(\mu)$.
We define the height of the box $d$ as the height of the center of $d$,
and denote it by $h(d)$.
For example, when $n=3$, the lowest path has three boxes since 
we have $Y=(2,1)$.
We have two boxes with height one and one box with height two.

\begin{prop}
\label{prop:chif}
If we expand $\chi(m,r,s)$ as $\chi(m,r,s)=:\sum_{0\le n}A_{n}(m,r,s)r^{n}$, 
the polynomial coefficient $A_{n}(m,r,s)$ is given by 
\begin{align}
\label{eqn:Anmrsinf}
A_{n}(m,r,s)=
(-1)^{n}\left(\prod_{j=1}^{n}f(rs^{j})^{-2}\cdot f(rs^{n+1})^{-1}\right)\cdot 
s^{n(n+1)/2}\cdot
\sum_{\mu\in\mathtt{Dyck}(n)}
\prod_{d\in Y(\mu)}
\genfrac{}{}{}{}{s^{-1}\cdot f(rs^{h(d)+2})}{f(rs^{h(d)})},
\end{align}
for $1\le n$ and $A_{0}(m,r,s)=f(rs)^{-1}$.
\end{prop}

\begin{remark}
Note that the coefficient $A_{n}(m,r,s)$ depends on the variables $r$ and $s$ 
through the power series $f(x)$.
First four expressions are 
\begin{align*}
A_{0}(m,r,s)&=\genfrac{}{}{}{}{1}{f(rs)}, \\
A_{1}(m,r,s)&=-\genfrac{}{}{}{}{s}{f(rs)^{2}f(rs^{2})}, \\
A_{2}(m,r,s)
&=\genfrac{}{}{}{}{1}{f(rs)^{3}f(rs^{2})^{2}f(rs^{3})}
\left(s^{3}f(rs)+s^{2}f(rs^{3})\right), \\
A_{3}(m,r,s)
&=-\genfrac{}{}{}{}{1}{f(rs)^{4}f(rs^{2})^{3}f(rs^{3})^{2}f(rs^{4})}
\left(s^{6}f(rs)^{2}f(rs^2)+s^5f(rs)^{2}f(rs^{4})\right. \\
&\quad\left.+2s^4f(rs)f(rs^{3})f(rs^4)+s^3f(rs^3)^2f(rs^{4})\right).
\end{align*} 
When $f(x)=1$, the coefficient $A_{n}(m,r,s)$ is independent 
of $r$ and it is a polynomial of $s$.
\end{remark}

\begin{proof}[Proof of Proposition \ref{prop:chif}]
From Definition \ref{defn:chi}, it is obvious that $A_{0}(m,r,s)=1/f(rs)$.
For $1\le n$, we show that the expression (\ref{eqn:Anmrsinf}) 
satisfies the recurrence relation in Definition \ref{defn:chi}.
The recurrence relation can be reduced to the following 
relation among $A_{j}(m,r,s)$:
\begin{align}
\label{eqn:ChiinfDyck}
\begin{split}
A_{n}(m,r,s)f(rs)+sf(rs^2)^{-1}A_{n-1}(m,r,s)+s^{n}f(rs)^{-1}A_{n-1}(m,rs,s) \\
+\sum_{j=1}^{n-2}s^{j+1}A_{j}(m,rs,s)A_{n-1-j}(m,r,s)=0.
\end{split}
\end{align}
A Dyck path $\lambda$ can be rewritten as a concatenation of two Dyck paths $\lambda_1$ and $\lambda_2$
such that $\lambda_2$ is a prime Dyck path.
Recall that we have a natural bijection between prime Dyck paths of size $j+1$ and Dyck paths 
of size $j$.
Thus, when $\lambda=\lambda_1\circ\lambda_2$, 
the Young diagram $Y(\lambda)$ is a union of $Y(\lambda_1)$, $Y(\lambda_{2})$ and 
a rectangle with edges of length $|\lambda_1|$ and $|\lambda_2|$.
We denote by $\square(i,j)$ the rectangle with edges of length $i$ and $j$.
We will compute $s^{j+1}A_{j}(m,rs,s)A_{n-1-j}(m,r,s)$.
Note that the shift $r\rightarrow rs$ in $A_{j}(m,rs,s)$ corresponds to 
the bijection between prime Dyck paths of size $j$ and Dyck paths of size $j+1$.
We define $F'(p,q)$ as 
\begin{align*}
F'(p,q)&:=\prod_{d\in Y(\square(p,q))}
\genfrac{}{}{}{}{s^{-1}f(rs^{h(d)+2})}{f(rs^{h(d)})}, \\
&=s^{-pq}\cdot\genfrac{}{}{}{}{f(rs^{q+1})f(rs^{p+q+1})\prod_{i=q+2}^{p+q}f(rs^{i})^{2}}{f(rs)f(rs^{p+1})\prod_{i=2}^{p}f(rs^{i})^{2}}.
\end{align*}
Let 
\begin{align*}
A'_{n}(r,s):=\left(\sum_{\mu\in\mathtt{Dyck}(n)}\prod_{d\in Y(\mu)}
\genfrac{}{}{}{}{s^{-1}f(rs^{h(d)+2})}{f(rs^{h(d)})}\right).
\end{align*}
Then, we have, by a simple calculation,  
\begin{align*}
A'_{j}(rs,s)A'_{n-j-1}(r,s)=F'(n-j-1,j+1)^{-1}\sum_{\lambda\in\mathtt{Dyck}'(n)}\prod_{d\in Y(\mu)}
\genfrac{}{}{}{}{s^{-1}f(rs^{h(d)+2})}{f(rs^{h(d)})},
\end{align*}
where $\mathtt{Dyck}'(n)$ is the set of Dyck paths $\lambda$ such that 
$\lambda=\lambda_1\circ\lambda_2$ and $\lambda_2$ is a prime Dyck path of size $j+1$.

One can compute the second and the third terms in Eq. (\ref{eqn:ChiinfDyck}) in a similar manner.
Then combining these observations together, 
we have Eq. (\ref{eqn:Anmrsinf}).
\end{proof}

Given a Dyck path $\mu$, we denote by $\mathrm{Peak}(\mu)$ the 
number of peaks in $\mu$, that is, $\mathrm{Peak}(\mu)$ is the 
number of consecutive steps $U$ and $D$.
We denote by $|Y(\mu_{1}/\mu)|$ the number of boxes in the 
skew shape $\mu_1/\mu$ where $\mu_{1}$ is the lowest path.
\begin{prop}
\label{prop:chiinDyck}
The formal power series $\chi(m,r,s)$ is expressed as
\begin{align}
\label{eqn:ChiinDyck}
\chi(m,r,s)=\sum_{0\le n}\sum_{\mu\in\mathtt{Dyck}(n)}(-r)^{n}m^{\mathrm{Peak}(\mu)}s^{n+|Y(\mu_{1}/\mu)|}.
\end{align}
\end{prop}
\begin{proof}
Let 
\begin{align}
\label{eqn:chinms}
\chi_{n}(m,s):=\sum_{\mu\in\mathtt{Dyck}(n)}m^{\mathrm{Peak}(\mu)}s^{n+|Y(\mu_{1}/\mu)|}.
\end{align}
By a simple calculation, we have $\chi_{0}(m,s)=1$ and $\chi_{1}(m,s)=ms$, which implies 
that the above expression holds for $n=0$ and $n=1$.
We prove Eq. (\ref{eqn:chinms}) by induction on $n$.
From the recurrence relation in Definition \ref{defn:chi} and $f(rs)=1+rs(m-1)$, $\chi_{n}(m,s)$ should satisfy
\begin{align}
\label{eqn:chinrec}
\chi_{n}(m,s)=ms\chi_{n-1}(m,s)+\sum_{j=0}^{n-2}s^{n-j}\chi_{j}(m,s)\chi_{n-j-1}(m,s).
\end{align}
A Dyck path $\lambda$ of size $n$ is rewritten as a concatenation of two Dyck paths 
$\lambda_1$ and $\lambda_2$, namely $\lambda=\lambda_1\circ\lambda_2$ where 
$\lambda_1$ is a Dyck path of size $j$ and $\lambda_2$ is a prime Dyck path of 
size $n-j$.
Then, we have the natural bijection between prime Dyck paths of size $n-j$ and 
Dyck paths of size $n-j-1$ for $n-j\ge2$.
Note that when $n-j=1$, we have a unique Dyck path $UD$.
From these observations, it is straightforward that Eq. (\ref{eqn:chinrec}) holds
for $n$ if Eq. (\ref{eqn:chinms}) holds for up to $n-1$.
This completes the proof.
\end{proof}

The formal power series $\chi(m,r,s)$ has also another 
expression in terms of Motzkin paths.

Given a Motzkin path $\lambda$, 
we denote the number of up steps by $N_{U}(\lambda)$, 
the number of horizontal steps of height zero by $N_{H}(\lambda)$.

Let $\lambda$ be a prime path.
Since $\lambda$ is prime, the first step and the last step 
are $U$ and $D$ respectively. 
If there exists a horizontal step, its height is at least one.
Let $h$ be a horizontal step in $\lambda$ such that 
its position is $\mathrm{pos}(h)$ from left and its height 
is one.
We denote by $H_{1}(\lambda)$ the set of horizontal steps at height one in $\lambda$.
We denote by $\widetilde{\lambda}$ a Motzkin path obtained from 
$\lambda$ by deleting the first step $U$ and the last step $D$.
Therefore, $\widetilde{\lambda}$ may not be prime.
We rewrite $\widetilde{\lambda}$ as a concatenation of 
prime Motzkin path: 
$\widetilde{\lambda}=\widetilde{\lambda}_{1}\circ\widetilde{\lambda}_{2}\circ\cdots\circ\widetilde{\lambda}_{p}$.
Here, each $\lambda_{i}$, $1\le i\le p$, is prime.
Then, for a prime path $\lambda$, we define recursively
\begin{align*}
F(\lambda)&:=\prod_{h\in H_{1}(\lambda)}(m+s^{\mathrm{pos(h)}-1})\prod_{i=1}^{p}F(\widetilde{\lambda}_{i}), \\ 
F(H)&:=m, \\
F(\emptyset)&:=1,
\end{align*}
where $\emptyset$ stands for the empty path (the path of size zero).
If $\lambda$ is not prime, $\lambda$ is written as a concatenation of prime paths: 
$\lambda=\lambda_{1}\circ\lambda_{2}\circ\cdots\circ\lambda_{p}$.
Then, we define 
\begin{align}
\label{eqn:wtMotzkin}
F(\lambda)&:=\prod_{i=1}^{p}F(\lambda_{i}).
\end{align}

The power series $\chi(m,r,s)$ is expressed in terms of Motzkin paths as follows.
\begin{prop}
\label{prop:chiinMotzkin}
We have 
\begin{align*}
\chi(m,r,s)=
\sum_{0\le n}
(-r)^{n}\sum_{\lambda\in\mathtt{Mot}(n)}
m^{N(\lambda)}
s^{n+|Y(\lambda_{0}^{n}/\lambda)|}F(\lambda),
\end{align*}
where 
\begin{align*}
N(\lambda):=N_{U}(\lambda)+N_{H}(\lambda).
\end{align*}
\end{prop} 

\begin{example}
For example, we consider the two Motzkin paths of size $n=5$: $UHHHD$ and $UUDHD$.
The areas $|Y(\lambda_{0}^{n}/\lambda)|$ are $4$ and $5$ for $UHHHD$ and $UUDHD$ respectively.
Thus, the contributions to $\chi(m,r,s)$ are given by
\begin{align*}
ms^{9}(m+s)(m+s^2)(m+s^3), 
\qquad m^2s^{10}(m+s^3).
\end{align*}
\end{example}

\begin{proof}[Proof of Proposition \ref{prop:chiinMotzkin}]
Let 
\begin{align}
\label{eqn:chininMot}
\chi_{n}(m,s):=\sum_{\lambda\in\mathtt{Mot}(n)}m^{N(\lambda)}s^{n+|Y(\lambda_{0}^{n}/\lambda)|}F(\lambda).
\end{align}
For $n=0$, we have $\chi_{0}(m,s)=1$ since we have a unique empty path.
For $n=1$, we have $\chi_{1}(m,s)=ms$ since we have a unique path $H$.
For $n=2$, we have $\chi_{2}(m,s)=m^2s^2+ms^3$ since we have two paths $HH$ and $UD$.

Then, $\chi_{n}(m,s)$ should satisfy the same equation as (\ref{eqn:chinrec}) for $n\ge3$.
As in the case  of a Dyck path, a Motzkin path $\lambda$ of length $n$ can be written 
as a concatenation of two Motzkin paths $\lambda_1$ and $\lambda_2$, {\it i.e.},
$\lambda=\lambda_1\circ\lambda_2$ where $\lambda_2$ is a prime Motzkin path.
By a similar argument to the proof of Proposition \ref{prop:chiinDyck}, 
we have Eq. (\ref{eqn:chininMot}), 
which completes the proof.
\end{proof}

Similarly, the inverse of the power series $\chi(m,r,s)^{-1}$ has 
the expression in terms of prime Dyck and Motzkin paths.

\begin{prop}
\label{prop:chiinvDyck}
The inverse of the power series $\chi(m,r,s)$ can be written
as 
\begin{align}
\label{eqn:chiinv}
\chi(m,r,s)^{-1}
=1+\sum_{1\le n}(-1)^{n-1}r^{n}
\sum_{\mu\in\mathtt{Dyck^{(\mathrm{pr})}}(n)}
m^{\mathrm{Peak}(\mu)}s^{n+|Y(\mu_{0}^{n}/\mu)|}.
\end{align}
\end{prop}
\begin{proof}
Let 
\begin{align*}
\overline{\chi_{n}(m,s)}:=\sum_{\mu\in\mathtt{Dyck^{(\mathrm{pr})}}(n)}
m^{\mathrm{Peak}(\mu)}s^{n+|Y(\mu_{0}^{n}/\mu)|}.
\end{align*}
Then, since $\chi(m,r,s)\chi(m,r,s)^{-1}=1$, 
$\chi_{n}(m,s)$ and $\overline{\chi_{n}(m,s)}$ for $1\le n$ should satisfy
\begin{align}
\label{eqn:chirecchiinv}
\chi_{n}(m,s)=\sum_{j=1}^{n}\overline{\chi_{j}(m,s)}\chi_{n-j}(m,s),
\end{align}
where $\chi_{n}(m,s)$ is defined in Eqn. (\ref{eqn:chinms}).
As in the proof of Proposition \ref{prop:chiinDyck}, 
a Dyck path of size $n$ is expressed as a concatenation of two Dyck 
paths one of which is a prime path.
For $n=0$, since we have a unique Dyck path $\emptyset$, 
$\chi_{0}(m,s)=\overline{\chi_{0}(m,s)}=1$.
Thus, by construction of $\chi_{n}(m,s)$, it is obvious that  
Eq. (\ref{eqn:chirecchiinv}) holds.
This implies that Eq. (\ref{eqn:chiinv}) holds.
\end{proof}

\begin{prop}
\label{prop:chiinvinMot}
The inverse of the power series $\chi(m,r,s)$ can be written as 
\begin{align*}
\chi(m,r,s)^{-1}
=1+\sum_{1\le n}(-1)^{n-1}r^{n}\sum_{\lambda\in\mathtt{Mot}^{(\mathrm{pr})}(n)}
m^{N(\lambda)}s^{n+|Y(\lambda_{0}^{n}/\lambda)|}F(\lambda).
\end{align*}
\end{prop}
\begin{proof}
By a similar argument to the proof of Proposition \ref{prop:chiinvDyck},
the formal power series $\chi(m,r,s)^{-1}$ is written as 
a generating function of prime Motzkin paths.
\end{proof}

When $m=1$, or equivalently $f(x)=1$, the formal power series 
$\chi(1,r,s)$ is expressed by use of the Ramanujan's continued 
fraction.
\begin{prop}
Let $\chi(m,r,s)$ be the formal power series defined as above.
Then, we have
\begin{align*}
\chi(1,r,s)
&=\genfrac{}{}{}{}{\sum_{0\le n}r^{n}s^{n(n+1)}\prod_{k=1}^{n}(1-s^{k})^{-1}}
{\sum_{0\le n}r^{n}s^{n^2}\prod_{k=1}^{n}(1-s^{k})^{-1}},  \\
\chi(1,r,1)&=\mathtt{Cat}(-r).
\end{align*}
Especially, when $(m,r)=(1,1)$, we have 
\begin{align*}
\chi(1,1,s)&=\prod_{1\le n}\genfrac{}{}{}{}{(1-s^{5n-1})(1-s^{5n-4})}{(1-s^{5n-2})(1-s^{5n-3})}, \\
&=f_{2}^{RR}/f_{1}^{RR}.
\end{align*}
\end{prop}

\subsection{Emptiness formation probability}
In this subsection, we study correlation functions of the dimer model on 
a segment, which we call emptiness formation probabilities.

\begin{defn}
Let $n$ be the size of the system.
We define $\mathcal{Z}_{n}(i)$ as the probability of dimer configurations 
such that there are no dimers at position up to at least $i$.
We call $\mathcal{Z}_{n}(i)$ an emptiness formation probability.
\end{defn}

\begin{lemma}
The probability $\mathcal{Z}_{n}(i)$ is expressed as 
\begin{align*}
\mathcal{Z}_{n}(i)=\genfrac{}{}{}{}{G^{(s)}_{n-i}(m,rs^{i},s)}{G^{(s)}_{n}(m,r,s)}.
\end{align*}
\end{lemma}
\begin{proof}
Since there are no dimers at position up to at least $i$, such a dimer configuration
of size $n$ is bijective to a dimer configuration of size $n-i$.
Recall $s$ enumerates the position of a dimer and $r$ enumerates the number of dimers.
Thus, by shifting the variable $r$ to $rs^{i}$, the position of a dimer is shifted from 
$j$ to $j+i$. 
From these observations, we obtain the expression.
\end{proof}

\begin{lemma}
\label{lemma:GrsG}
In the large $n$ limit, we have 
\begin{align*}
\lim_{n\rightarrow\infty}\genfrac{}{}{}{}{G^{(s)}_{n-1}(m,rs,s)}{G^{(s)}_{n}(m,r,s)}=\chi(m,r,s).
\end{align*}
\end{lemma}
\begin{proof}
From the recurrence relation in Theorem \ref{thrm:Gsrrflip}, we have 
\begin{align}
\label{eqn:defGschi}
1=f(rs)\genfrac{}{}{}{}{G^{(s)}_{n-1}(m,rs,s)}{G^{(s)}_{n}(m,r,s)}
+rs\genfrac{}{}{}{}{G^{(s)}_{n-1}(m,rs,s)}{G^{(s)}_{n}(m,r,s)}
\genfrac{}{}{}{}{G^{(s)}_{n-2}(m,rs^2,s)}{G^{(s)}_{n-1}(m,rs,s)}.
\end{align}
In the large $n$ limit, we define $\overline{\chi}(m,r,s)$ as 
\begin{align*}
\overline{\chi}(m,r,s):=\lim_{n\rightarrow\infty}\genfrac{}{}{}{}{G^{(s)}_{n-1}(m,rs,s)}{G^{(s)}_{n}(m,rs,s)},
\end{align*}
Similarly, in the large $n$ limit, we have 
\begin{align*}
\overline{\chi}(m,rs,s)=\lim_{n\rightarrow\infty}\genfrac{}{}{}{}{G^{(s)}_{n-2}(m,rs^2,s)}{G^{(s)}_{n-1}(m,rs,s)}.
\end{align*}
Then, by substituting $\overline{\chi}(m,r,s)$ into Eq. (\ref{eqn:defGschi}) and by taking 
the large $n$ limit, 
$\overline{\chi}(m,r,s)$ satisfies the same defining recurrence relation as $\chi(m,r,s)$.
The constant terms in $\chi(m,r,s)$ and $\overline{\chi}(m,r,s)$ are both one.
Thus, $\overline{\chi}(m,r,s)=\chi(m,r,s)$, which completes the proof.
\end{proof}

\begin{theorem}
\label{thrm:Zinf}
The probability $\mathcal{Z}_{\infty}(i)$ is given by
\begin{align*}
\mathcal{Z}_{\infty}(i)=\prod_{j=1}^{i}\chi(m,rs^{j-1},s).
\end{align*}
\end{theorem}
\begin{proof}
From the definition of $\mathcal{Z}_{\infty}(i)$, we have 
\begin{align*}
\mathcal{Z}_{\infty}(i)&=\lim_{n\rightarrow\infty}
\genfrac{}{}{}{}{G^{(s)}_{n-i}(m,rs^{i},s)}{G^{(s)}_{n}(m,r,s)}, \\
&=\lim_{n\rightarrow\infty}
\prod_{j=1}^{i}\genfrac{}{}{}{}{G^{(s)}_{n-j}(m,rs^{j},s)}{G^{(s)}_{n+1-j}(m,rs^{j-1},s)}, \\
&=\prod_{j=1}^{i}\chi(m,rs^{j-1},s),
\end{align*}
where we have used Lemma \ref{lemma:GrsG}. This completes the proof.
\end{proof}

The following corollary is a direct consequence of Definition \ref{defn:chi} and 
Theorem \ref{thrm:Zinf}.
\begin{cor}
Suppose $(m,s)=(1,1)$. 
The probability $\mathcal{Z}(i)$ is given by 
\begin{align*}
\mathcal{Z}_{\infty}(i)=\left(\genfrac{}{}{}{}{1-\sqrt{1+4r}}{-2r}\right)^{i}.
\end{align*}
Especially, the specialization $r=1$ gives 
\begin{align*}
\mathcal{Z}_{\infty}(i)\Big|_{(m,r,s)=(1,1,1)}=\left(\genfrac{}{}{}{}{-1+\sqrt{5}}{2}\right)^{i}.
\end{align*}
\end{cor}

To obtain the finite analogue of Theorem \ref{thrm:Zinf}, 
we define a formal power series $B_{n}(r):=B_{n}(m,r,s)$ 
as follows.
\begin{defn}
\label{defn:recB}
We denote $\chi(r):=\chi(m,r,s)$ for simplicity.
Let $B_{n}(r):=B_{n}(m,r,s)$ be a formal power series satisfying 
\begin{align*}
B_{n+1}(r)=B_{n}(rs)
+(-1)^{n}r^{n}s^{n(n+1)/2}\chi(r)\left(\prod_{j=1}^{n-1}\chi(rs^{j})^2\right)\chi(rs^{n})B_{1}(rs^{n}),
\end{align*}
with the initial condition 
\begin{align*}
B_{1}(r):=rs\chi(r)-rs\chi(r)\chi(rs).
\end{align*}
\end{defn}

We introduce a lemma which relates the power series $B_{n}(r)$ with the generating 
function $G_{n}^{(s)}(m,r,s)$.

\begin{lemma}
\label{lemma:GBrel}
The generating function $G_{n}^{(s)}(m,r,s)$ and the power series $B_{n}(r)$ satisfy  
\begin{align}
\label{eqn:GBrel}
-G^{(s)}_{n}(m,rs,s)+G^{(s)}_{n+1}(m,r,s)\chi(r,s)+(-1)^{n+1}r^{n}s^{n(n+1)/2}B_{1}(rs^{n})\prod_{j=0}^{n-1}\chi(rs^{j},s)
=0.
\end{align}
\end{lemma}
\begin{proof}
For $n=0$, we have 
\begin{align}
\label{eqn:GBtrivial}
\begin{split}
-G^{(s)}_{0}(m,rs,s)+G^{(s)}_{1}(m,r,s)\chi(r,s)-B_{1}(r)&=-1+(f(rs)+rs)\chi(r)-B_{1}(r), \\ 
&=0,
\end{split}
\end{align} 
where we have used the recurrence relation in Theorem \ref{thrm:Gsrrflip}, 
the initial conditions $G_{0}(m,r,s)=G_{-1}(m,r,s)=1$, and the definition 
of $B_{1}(r)$.  

For $1\le n$, by use of the recurrence relation in Theorem \ref{thrm:Gsrrflip} and 
Eq. (\ref{eqn:finchi}), the left hand side of Eq. (\ref{eqn:GBrel}) is proportional to
\begin{align*}
-G^{(s)}_{n-1}(m,rs^{2},s)+G^{(s)}_{n}(m,rs,s)\chi(rs,s)+(-1)^{n}r^{n-1}s^{n(n+1)/2-1}B_{1}(rs^{n})\prod_{j=1}^{n-1}\chi(rs^{j},s).
\end{align*}
By substituting $r\rightarrow rs^{-1}$ in the equation above, 
we obtain 
\begin{align*}
-G_{n-1}^{(s)}(m,rs,s)+G_{n}^{(s)}(m,r,s)\chi(r,s)+(-1)^{n}s^{n(n-1)/2}B_{1}(rs^{n-1})\prod_{j=0}^{n-2}\chi(rs^{j},s)=0,
\end{align*}
which is nothing but Eq. (\ref{eqn:GBrel}) with $n\rightarrow n-1$.
By repeating this procedure $n$ times, Eq. (\ref{eqn:GBrel}) is equivalent to  
Eq. (\ref{eqn:GBtrivial}).
Since Eq. (\ref{eqn:GBtrivial}) holds, this completes the proof.
\end{proof}

The next proposition is the finite analogue of Theorem \ref{thrm:Zinf}.
The value $\mathcal{Z}_{n}(i)$ can be expressed in terms of the 
formal power series $\chi(m,r,s)$ and $B_{n}(m,r,s)$.
\begin{prop}
\label{prop:Zni}
We have 
\begin{align}
\label{eqn:Zni}
\mathcal{Z}_{n}(i)=\prod_{j=0}^{i-1}\chi(m,rs^{j},s)
+\genfrac{}{}{}{}{(-1)^{n-i-1}}{G^{(s)}_{n}(m,r,s)}\left(\prod_{k=i-1}^{n-2}rs^{k+1}\chi(rs^{k})\right)B_{i}(r)\genfrac{}{}{}{}{B_{1}(rs^{n-1})}{B_{1}(rs^{i-1})},
\end{align}
for $1\le i$ and $\mathcal{Z}_{n}(0)=1$.
\end{prop}
\begin{proof}
We denote $\mathcal{Z}_{n}(i;r):=\mathcal{Z}_{n}(i)$ when we emphasize 
the dependence of $r$ in $\mathcal{Z}_{n}(i)$.
We prove the proposition by induction on $i$.

From the recurrence relation in Theorem \ref{thrm:Gsrrflip},
we have 
\begin{align*}
G_{n-i+1}^{(s)}(r)=f(rs)G_{n-i}^{(s)}(rs)+rsG_{n-i-1}^{(s)}(rs^{2}).
\end{align*}
By substituting $r\rightarrow rs^{i-1}$ and dividing the both sides by $G_{n}^{(s)}(r)$, 
we have a recurrence relation for $\mathcal{Z}_{n}(i)$:
\begin{align}
\label{eqn:Zrec1}
\mathcal{Z}_{n}(i-1)=f(rs^{i})\mathcal{Z}_{n}(i)+rs^{i}\mathcal{Z}_{n}(i+1).
\end{align}
Since we have 
\begin{align*}
\mathcal{Z}_n(2;r)&=\genfrac{}{}{}{}{G^{(s)}_{n-2}(rs^2)}{G^{(s)}_{n}(r)}
=\genfrac{}{}{}{}{G^{(s)}_{n-2}(rs^2)}{G^{(s)}_{n-1}(rs)}\genfrac{}{}{}{}{G^{(s)}_{n-1}(rs)}{G^{(s)}_{n}(r)}, \\
&=\mathcal{Z}_{n-1}(1;rs)\mathcal{Z}_{n}(1;r),
\end{align*}
the probability $\mathcal{Z}_{n}(1)$ satisfies 
\begin{align}
\label{eqn:Zn1rec}
f(rs)\mathcal{Z}_{n}(1;r)=1-rs\mathcal{Z}_{n-1}(1;rs)\mathcal{Z}_{n}(1;r),
\end{align}
where we have used the recurrence relation (\ref{eqn:Zrec1}) and $\mathcal{Z}_{n}(0)=1$ by definition.
By a straightforward calculation and from Eq. (\ref{eqn:GBtrivial}), we have 
\begin{align*}
\mathcal{Z}_{1}(1)=\chi(r)-\genfrac{}{}{}{}{B_{1}(r)}{G^{(s)}_{1}(m,r,s)},
\end{align*}
which coincides with the expression (\ref{eqn:Zni}).
One can show that the expression (\ref{eqn:Zni}) for $i=1$ satisfies 
the recurrence relation (\ref{eqn:Zn1rec}) by use of Eq. (\ref{eqn:finchi}).
The difference between the left hand side and right hand side of Eq. (\ref{eqn:Zn1rec})
is proportional to 
\begin{align*}
-G^{(s)}_{n}(m,rs,s)+G^{(s)}_{n+1}(m,r,s)\chi(r,s)
+(-1)^{n+1}r^{n}s^{n(n+1)/2}B_{1}(rs^{n})\prod_{j=0}^{n-1}\chi(rs^{j},s),
\end{align*}
which is zero by Lemma \ref{lemma:GBrel}. 
We establish Eq. (\ref{eqn:Zn1rec}).

For $2\le i$, we can easily show that the expression (\ref{eqn:Zni}) 
satisfies the recurrence relation (\ref{eqn:Zrec1}) if 
$B_{n}(r)$ satisfies the recurrence relation in Definition \ref{defn:recB}.
This completes the proof.
\end{proof}

\subsection{Average number of dimers}
\label{sec:Averd}
In this subsection, we show that the average number of dimers for 
general $(m,r,s)$ in the large $n$ limit is expressed 
in terms of the formal power series $\chi(m,r,s)$ 
defined in Section \ref{sec:defchi}.

\begin{defn}
We define the average number of dimers of size $n$ as 
\begin{align*}
\mathcal{N}_{n}(m,r,s):=\genfrac{}{}{}{}{\partial_{r}G^{(s)}_{n}(m,r,s)}{G^{(s)}_{n}(m,r,s)}.
\end{align*}
\end{defn}

We define the formal power series of the logarithm of $\chi(m,r,s)$ 
with respect to $r$ by 
\begin{align*}
\sum_{1\le n}\mathtt{N}(n;m,s)\genfrac{}{}{}{}{r^{n}}{n}
:=\log(\chi(m,r,s)).
\end{align*}
First few values of $\mathtt{N}(n; m,s)$ are 
\begin{align*}
\mathtt{N}(1; m,s)&=-ms, \\
\mathtt{N}(2; m,s)&=m^2s^2+2ms^3, \\
\mathtt{N}(3; m,s)&=-(m^3s^3+3m^2s^4+3m^2s^5+3ms^6), \\
\mathtt{N}(4; m,s)&=m^4s^4+4m^3s^5+2m^2(1+2m)s^6+4m^2(1+m)s^7+8m^2s^8+4m^2s^9+4ms^{10}.
\end{align*}

\begin{theorem}
\label{thrm:Ninf1}
The average number of dimers in the large $n$ limit is given by
\begin{align}
\label{eqn:Ninf11}
\mathcal{N}_{\infty}(m,r,s)=-\sum_{1\le n}\genfrac{}{}{}{}{\mathtt{N}(n;m,s)}{1-s^{n}}r^{n-1}.
\end{align}
for $s$ sufficiently small.
As an equivalent expression, we have
\begin{align}
\label{eqn:Ninf12}
\mathcal{N}_{\infty}(m,r,s)=-\partial_{r}\left(\log\left(\prod_{0\le n}\chi(m,rs^{n},s)\right)\right).
\end{align}
\end{theorem}

Before proceeding to the proof of Theorem \ref{thrm:Ninf1}, we introduce the following lemma.
\begin{defn}
\label{defn:g}
We define a formal power series $g(x):=\chi(x)+f'(x)$.
\end{defn}

We define a sequence of formal power series $\overline{\mathcal{N}}_{n}(r)$ as 
\begin{align*}
\overline{\mathcal{N}}_{n}(r)=sg(rs)\chi(r)+sf(rs)\chi(r)\overline{\mathcal{N}}_{n-1}(rs)
+rs^{3}\chi(r)\chi(rs)\overline{\mathcal{N}}_{n-2}(rs^2),
\end{align*}
with the initial conditions
\begin{align*}
\overline{\mathcal{N}}_{0}(r)&:=sg(rs)\chi(r)=-\genfrac{}{}{}{}{\chi'(r)}{\chi(r)}-rs^{2}\chi(r)\chi'(rs), \\
\overline{\mathcal{N}}_{-1}(r)&:=0.
\end{align*}

Let $\overline{\mathcal{N}}_{n}(r;k):=[r^{k}]\overline{\mathcal{N}}_{n}(r)$ be the polynomial coefficient of 
$r^{k}$ in $\overline{\mathcal{N}}_{n}(r)$.
Here, the polynomial coefficient means that 
we have $[r^{k}]F(r):=F_k(r)$ if $F(r):=\sum_{0\le k}r^{k}F_{k}(r)$ and 
$F_{k}(r)$ is a product of $\chi(rs^{j})$ and $\chi'(rs^{j})$ for some $j$.

\begin{lemma}
\label{lemma:Npsum}
We have 
\begin{align}
\label{eqn:Npsum}
\sum_{k=0}^{m}\overline{\mathcal{N}}_{m-k}(r;k)r^{k}
=-\sum_{k=0}^{m}\genfrac{}{}{}{}{s^{k}\chi'(rs^{k})}{\chi(rs^{k})}
+\sum_{k=1}^{m+1}(-1)^{k}r^{k}s^{d(m,k)}\genfrac{}{}{}{}{\chi'(rs^{m+1})}{\chi(rs^{m+1-k})}\prod_{j=m+1-k}^{m}\chi(rs^{j})^{2},
\end{align}
where 
\begin{align*}
d(m,k):=-\genfrac{}{}{}{}{1}{2}k(k-3)+m(k+1)+1.
\end{align*}
\end{lemma}
\begin{proof}
We prove the lemma by induction.
When $m=1$, the left hand side of Eq. (\ref{eqn:Npsum}) is equal to
\begin{align*}
s\chi(r)g(rs)+s^2f(rs)\chi(r)\chi(rs)g(rs^{2})
&=-\genfrac{}{}{}{}{\chi'(r)}{\chi(r)}-\genfrac{}{}{}{}{s\chi'(rs)}{\chi(rs)} \\
&\quad-rs^{4}\chi(rs)\chi'(rs^2)+r^2s^5\chi(r)\chi(rs)^{2}\chi'(rs^{2}).
\end{align*}
For $2\le m$, we consider the difference between $\mathcal{N}_{m+1-k}(r)$ and 
$\mathcal{N}_{m-k}(r)$.
From the definition of $\mathcal{N}_{n}(r)$, we have 
\begin{align*}
\overline{\mathcal{N}}_{n}(r;0)&=sg(rs)\chi(r)+sf(rs)\chi(r)\mathcal{N}_{n-1}(r;0), \\
\overline{\mathcal{N}}_{n}(r;k)
&=sf(rs)\chi(r)\overline{\mathcal{N}}_{n-1}(rs;k)
+s^3\chi(r)\chi(rs)\overline{\mathcal{N}}_{n-2}(rs^2;k-1).	
\end{align*}
From these recurrence relations, we have 
\begin{align*}
\overline{\mathcal{N}}_{m+1}(r;0)-\overline{\mathcal{N}}_{m}(r;0)
&=sf(rs)\chi(r)(\overline{\mathcal{N}}_{m}(r;0)-\overline{\mathcal{N}}_{m-1}(r;0)), \\
&=s^{m+2}g(rs^{m+2})\prod_{j=1}^{m+1}f(rs^{j})\chi(rs^{j-1}),
\end{align*}
and 
\begin{align*}
&\overline{\mathcal{N}}_{m+1-k}(r;k)-\overline{\mathcal{N}}_{m-k}(r;k) \\
&\qquad=s^{m+2}g(rs^{m+2})\prod_{j=1}^{m+1}f(rs^{j})\chi(rs^{j-1})
\sum_{\mathbf{i}\in I(m+1;k)}s^{L(\mathbf{i})}\prod_{p=1}^{k}\genfrac{}{}{}{}{1}{f(rs^{i_{p}})},
\end{align*}
where 
\begin{align*}
I(m;k)&:=\left\{\mathbf{i}:=(i_1,\ldots,i_{2k})\Big|
\begin{array}{c}
1\le i_1<i_2<\ldots<i_{2k}\le m \\
i_{2p}=i_{2p-1}+1, \quad 1\le p\le k \\
\end{array}
\right\}, \\
L(\mathbf{i})&:=\sum_{p=1}^{k}i_{2p-1}, \qquad \text{for }\mathbf{i}\in I(m;k).
\end{align*}
By substituting the relations (\ref{eqn:finchi}) and (\ref{eqn:fpinchi}),  
and the definition of $g(x)$ in Definition \ref{defn:g} into 
the above expressions and rearranging the terms,
we obtain the desired expression (\ref{eqn:Npsum}) by a straightforward calculation.
This completes the proof.
\end{proof}

\begin{proof}[Proof of Theorem \ref{thrm:Ninf1}]
We first show that Eq. (\ref{eqn:Ninf11}) is equivalent to Eq. (\ref{eqn:Ninf12}).
Since we have $(1-s^{n})^{-1}=\sum_{0\le p}s^{pn}$, Eq. (\ref{eqn:Ninf11}) is equal to 
\begin{align*}
-\sum_{1\le n}\genfrac{}{}{}{}{\mathtt{N}(n;m,s)}{1-s^{n}}r^{n-1}
&=-\sum_{1\le n}\sum_{0\le p}\mathtt{N}(n;m,s)s^{pn}r^{n-1}, \\
&=-\sum_{0 \le p}\partial_{r}\log(\chi(m,rs^{p},s)),  \\
&=-\partial_{r}\left(\log\left(\prod_{0\le n}\chi(m,rs^{n},s)\right)\right).
\end{align*}
We denote $G_{n}(r):=G_{n}^{(s)}(m,r,s)$.
From the recurrence relation in Theorem \ref{thrm:Gsrrflip}, 
we have
\begin{align*}
\partial_{r}G_{n}(r)=sf'(rs)G_{n-1}(rs)+sf(rs)G'_{n-1}(rs)+sG_{n-2}(rs^2)+rs^{3}G'_{n-2}(rs^{2}),
\end{align*} 
where $F'(x)$ is the derivative of $F(x)$ with respect to $r$.
By taking the large $n$ limit, we obtain 
\begin{align*}
\mathcal{N}_{\infty}(r)=sg(rs)\chi(r)+sf(rs)\mathcal{N}_{\infty}(rs)\chi(r)
+rs^{3}\mathcal{N}_{\infty}(rs^{2})\chi(rs)\chi(r).
\end{align*}	
If we take the large $m$ limit in Lemma \ref{lemma:Npsum},
we have 
\begin{align}
\label{eqn:Ninf3}
\mathcal{N}_{\infty}(r)
=\lim_{m\rightarrow\infty}\sum_{k=0}^{m}\overline{\mathcal{N}}_{m-k}(r;k)r^{k}
=-\sum_{0\le k}\genfrac{}{}{}{}{s^{k}\chi'(rs^{k})}{\chi(rs^{k})},
\end{align}
where we used $m<d(m,k)$.
It is easy to check that the right hand side of Eq. (\ref{eqn:Ninf3})
is equal to Eq. (\ref{eqn:Ninf12}).
This completes the proof.
\end{proof}

\subsection{Generating functions revisited}
\label{sec:Gfrevisit}
In the previous subsection, we show that the average number of dimers 
is expressed in terms of the formal power series $\chi(m,r,s)$.
This implies that the generating function can also be expressed in terms 
of the power series $\chi(m,r,s)$, since the average number of dimers 
is written as the ratio of $G_{n}^{(s)}(m,r,s)$ and $\partial_{r}G_{n}^{(s)}(m,r,s)$.
In fact, we first give an expression of $G_{\infty}^{(s)}(m,r,s)$ in terms 
of the formal power series $\chi(m,r,s)$.

Recalling that the formal power series $\chi(m,r,s)$ is expressed by use of 
Dyck or Motzkin paths with statistics, the expression of $G_{n}^{(s)}$ in terms 
of $\chi(m,r,s)$ leads to the expression of $G_{n}^{(s)}$ in terms of 
Dyck or Motzkin paths.
From the similarity between $G_{n}^{(s)}(m,r,s)$ and $G_{n}^{(c)}(m,r,s)$, 
one can also expect that $G_{n}^{(c)}(m,r,s)$  has an expression in terms 
of Dyck or Motzkin paths.

We start with the next theorem, which is a generalization of 
Propositions \ref{prop:Gsm1} and \ref{prop:Gsm2} to general $(m,r,s)$. 
\begin{theorem}
\label{thrm:Ginfchi}
We have 
\begin{align*}
G_{\infty}^{(s)}(m,r,s)=\prod_{1\le n}\genfrac{}{}{}{}{1}{\chi(m,rs^{n-1},s)}.
\end{align*}
\end{theorem}
\begin{proof}
By definition of $\mathcal{N}_{\infty}(m,r,s)$ and Theorem \ref{thrm:Ninf1}, we have 
\begin{align*}
\genfrac{}{}{}{}{\partial_{r}G^{(s)}_{\infty}(m,r,s)}{G^{(s)}_{\infty}(m,r,s)}
&= \partial_{r}\log(G^{(s)}_{\infty}(m,r,s)), \\
&=-\partial_{r}\left(\log\left(\prod_{0\le n}\chi(m,rs^{n},s)\right)\right).
\end{align*}
By integrating the above equation, 
we have 
\begin{align}
\label{eqn:intGs}
G_{\infty}^{(s)}(m,r,s)=\prod_{1\le n}\genfrac{}{}{}{}{1}{\chi(m,rs^{n-1},s)}+ C,
\end{align}
where $C$ is a constant. 
Since the constant term of the left hand side in Eq. (\ref{eqn:intGs}) is one and 
$\chi(m,rs^{n-1},s)^{-1}=1+mrs^{n}+\cdots$, 
we have $C=0$. This completes the proof.
\end{proof}

By taking the large $n$ limit in Theorem \ref{thrm:Gsrrflip}, 
the generating function $G_{\infty}^{(s)}(m,r,s)$ satisfies the recurrence relation:
\begin{align*}
G_{\infty}^{(s)}(m,r,s)=f(rs)G_{\infty}^{(s)}(m,rs,s)+rsG_{\infty}^{(s)}(m,rs^2,s).
\end{align*}

From Proposition \ref{prop:chiinDyck}, the power series $\chi(m,rs^{j},s)$, $0\le j$,
has an expression in terms of Dyck paths.
Thus, a product of inverses of $\chi(m,rs^{j},s)$, $0\le j$, has also an expression 
in terms of Dyck paths.
We define the finite product $\Theta_{n}(m,r,s)$ as
\begin{align*}
\Theta_{n}(m,r,s):=\prod_{j=0}^{n-1}\genfrac{}{}{}{}{1}{\chi(m,rs^{j},s)}.
\end{align*}

Given an integer $n$, we will define the weight $\mathrm{wt}(\mu;s)$ 
for a Dyck path $\mu\in\mathtt{Dyck}(k)$.
Suppose that a Dyck path $\mu$ of size $k$ can be written as 
a concatenation of $n$ Dyck paths $\mu_1,\ldots\mu_{n}$
satisfying the following two conditions:
\begin{enumerate}
\item $0\le l(\mu_{i})\le k$ for all $1\le i\le n$, 
\item $\sum_{i=1}^{n}l(\mu_{i})=k$.
\end{enumerate}
By definition, we may have several Dyck paths of size zero in $\{\mu_{i}:1\le i\le n\}$.
We denote by $\mathcal{D}_{n}(\mu)$  the set of $n$-tuple Dyck paths $\vec{\mu}:=(\mu_1,\ldots,\mu_{n})$.

Then, we define the weight $\mathrm{wt}_{n}(\mu;s)$ by 
\begin{align*}
\mathrm{wt}_{n}(\mu;s)
:=
s^{l(\mu)+|Y(\mu_{0}^{(k)}/\mu)|}\sum_{\vec{\mu}\in\mathcal{D}_{n}(\mu)}
s^{d(\vec{\mu})},
\end{align*}
where 
\begin{align*}
d(\vec{\mu}):=\sum_{i=1}^{n}(i-1)l(\mu_{i}).
\end{align*}

\begin{theorem}
\label{thrm:ThetainfinDyck}
Let $\mu_{0}^{(k)}$ be the lowest Dyck path of length $2k$.
Then, we have
\begin{align}
\label{eqn:ThetainDyck}
\Theta_{n}(m,r,s)
=\sum_{0\le k}\sum_{\mu\in\mathtt{Dyck}(k)}
(-1)^{k-\mathrm{Peak}(\mu)}r^{k}m^{\mathrm{Peak}(\mu)}
\mathrm{wt}_{n}(\mu;s),
\end{align}
where $\mathrm{Peak}(\mu)$ is the number of peaks in $\mu$, that is, 
the number of prime Dyck paths in $\mu$.
\end{theorem}
\begin{proof}
By definition, the formal power series $\Theta_{n}(m,r,s)$ can be 
expressed in a product of the power series $\chi(m,r,s)^{-1}$.
Then, $\chi(m,r,s)^{-1}$ has an expression in terms of prime 
Dyck paths (see Proposition \ref{prop:chiinvDyck}). 
The shift $r\rightarrow rs^{i-1}$ for some $1\le i\le n$ corresponds to 
the factor $(i-1)$ in $d(\vec{\mu})$.
From these observations together with Proposition \ref{prop:chiinvDyck},
we have Eq. (\ref{eqn:ThetainDyck}).
\end{proof}

\begin{example}
We consider the case $n=4$ and the coefficient of $r^{3}m^{2}$ in $\Theta_{4}(m,r,s)$.
We have three Dyck paths $\mu_i$, $1\le i\le 3$, with two peaks:
\begin{align*}
\mu_{1}=UDUUDD, \qquad \mu_{2}=UUDDUD, \qquad \mu_{3}=UUDUDD.
\end{align*}
For $\mu_1$, the heights of peaks are one and two from left to right and we have 
\begin{align*}
\mathcal{D}_{n}(\mu_1):=\{ \{\nu_1, \nu_{2},\emptyset,\emptyset\},\{\nu_1,\emptyset, \nu_2,\emptyset\},
\{\nu_1,\emptyset,\emptyset,\nu_2\}, \{\emptyset,\nu_1, \nu_2,\emptyset\},
\{\emptyset,\nu_1,\emptyset, \nu_2\},\{\emptyset,\emptyset,\nu_1, \nu_2\} \}, 
\end{align*}
where $\nu_1=UD$ and $\nu_2=UUDD$.
We have $s^{l(\mu_1)+|Y(\mu_{0}^{(k)}/\mu_{1})|}=s^{4}$.
Thus, the polynomial $\mathrm{wt}_{4}(\mu_1;s)$ is given by 
\begin{align*}
\mathrm{wt}_{4}(\mu_1;s)=s^6+s^8+s^9+s^{10}+s^{11}+s^{12}.
\end{align*}
Similarly, for $\mu_2$ and $\mu_3$, 
we have six elements in $\mathcal{D}_{n}(\mu_2)$ and 
four elements in $\mathcal{D}_{n}(\mu_3)$.
By a straightforward calculation, we have 
\begin{align*}
\mathrm{wt}_{4}(\mu_2;s)=s^{5}+s^{6}+s^{7}+s^{8}+s^{9}+s^{11},
\end{align*}
and 
\begin{align*}
\mathrm{wt}_{4}(\mu_3;s)=s^{5}+s^{8}+s^{11}+s^{14}.
\end{align*}
From these observations, the coefficient of $r^{3}m^{2}$ in $\Theta_{4}(m,r,s)$
is given by
\begin{align*}
-\mathrm{wt}_{4}(\mu_1;s)-\mathrm{wt}_{4}(\mu_2;s)+\mathrm{wt}_{4}(\mu_3;s)
=-2s^6-s^7-s^8-2s^{9}-s^{10}-s^{11}-s^{12}+s^{14}.
\end{align*}
\end{example}

By taking the large $n$ limit in Theorem \ref{thrm:ThetainfinDyck}, we obtain the following corollary 
by use of Theorem \ref{thrm:Ginfchi}.
\begin{cor}
\label{cor:GsinfinDyck}
The generating function is written as 
\begin{align*}
G_{\infty}^{(s)}(m,r,s)=\sum_{0\le k}\sum_{\mu\in\mathtt{Dyck}(k)}
(-1)^{k-\mathrm{Peak}(\mu)}r^{k}m^{\mathrm{Peak}(\mu)}
\mathrm{wt}_{\infty}(\mu;s).
\end{align*}
\end{cor}

From Theorem \ref{thrm:Ginfchi}, one expects that the generating function 
$G_{n}^{(s)}(m,r,s)$ can be written in terms of $\chi(m,r,s)$.
The next theorem tells us that such an expression exists.
\begin{theorem}
\label{thrm:Gsfinchi}
Let $I(n)$ be the set of indices 
\begin{align*}
I(n):=\left\{\mathbf{i}:=(i_1,i_2,\ldots,i_{2u})\Bigg|
\begin{array}{c}
0\le u\le\lfloor(n-1)/2\rfloor \\
1\le i_1<i_2<\ldots<i_{2u}\le n-1 \\
i_{2p}=i_{2p-1}+1, \quad 1\le p\le u
\end{array}
\right\}.
\end{align*}   
Given $\mathbf{i}\in I(n)$, we define the weight $\mathrm{wt}(\mathbf{i})$
as 
\begin{align*}
\mathrm{wt}(\mathbf{i}):=r^{u}s^{\mathrm{deg}(\mathbf{i})},
\end{align*}
where
\begin{align*}
\mathrm{deg}(\mathbf{i}):=\sum_{1\le p\le u}i_{2p-1}.
\end{align*}
Then, the generating function $G_{n}^{(s)}(m,r,s)$ is given by 
\begin{align*}
G_{n}^{(s)}(m,r,s)
&=\prod_{j=0}^{n-1}\genfrac{}{}{}{}{1}{\chi(m,rs^{j},s)} \\
&-rs^{n}\left(\chi(m,rs^{n},s)-1\right)
\left(\prod_{j=1}^{n-1}f(rs^{j})\right)
\left(\sum_{\mathbf{i}\in I(n)}\mathrm{wt}(\mathbf{i})
\prod_{p=1}^{2u}\genfrac{}{}{}{}{1}{f(rs^{i_{p}})}\right).
\end{align*}
\end{theorem}
\begin{proof}
We prove Theorem by induction.
When $n=1$, we have
\begin{align*}
G_{1}^{(s)}(m,r,s)&=\genfrac{}{}{}{}{1}{\chi(m,r,s)}-rs \chi(m,rs,s)+rs, \\
&=1+mrs,
\end{align*}
where we have used the defining relation in Definition \ref{defn:chi}.
For $2\le n$, one can verify the expression satisfies the 
recurrence relation in Theorem \ref{thrm:Gsrrflip} by a straightforward calculation 
using Definition \ref{defn:chi}.
\end{proof}

Let $p_{1},\ldots,p_{t}$ be the peaks in a Dyck path $\mu$ from left to right, and 
$h(p_{i})$, $1\le i\le t$, be the height of the peak $p_{i}$.
For example, when $\mu=UDUUDD$, we have two peaks whose heights 
are one and two from left to right.

Similarly, we denote by $v_{1},\ldots,v_{t-1}$ the valleys in 
the Dyck path $\mu$ from left to right. Here, a valley means a partial path $DU$ in 
$\mu$. Then, the height of a valley $v_{i}$ is denoted by $h(v_{i})$.
For example, when $\mu=UUDUDD$, we have a unique valley whose 
height is one, that is, $h(v_{1})=1$.

In the next theorem, we give a finite analogue of Corollary \ref{cor:GsinfinDyck}. 
Recall that in the case of $\Theta_{n}(m,r,s)$, we have a contribution of a Dyck path 
$\mu\notin\mathtt{Dyck^{(f)}}(k)$ to $\Theta_{n}(m,r,s)$.
From Theorem \ref{thrm:Gsfinchi}, the leading contribution of the generating function 
$G_{n}^{(s)}(m,r,s)$ is $\Theta_{n}(m,r,s)$.
The difference between $G_{n}^{(s)}$ and $\Theta_{n}(m,r,s)$ corresponds to 
Dyck paths $\mu\notin\mathtt{Dyck^{(f)}}(k)$. 
More precisely, the generating function $G_{n}^{(s)}(m,r,s)$ can be expressed 
in terms of Dyck paths in $\mathtt{Dyck^{(f)}}(k)$.
Note that a Dyck path in $\mathtt{Dyck^{(f)}}(k)$ has valleys only at height zero, 
that is, $h(v_{i})=0$ for $1\le i\le t-1$.
As in the case of $\Theta_{n}(m,r,s)$ (see Theorem \ref{thrm:ThetainfinDyck}), 
we assign an integer sequence $\mathbf{j}:=(j_1,\ldots,j_{t})$ to the peaks $p_{i}$, $1\le i\le t$, 
in $\mu\in\mathtt{Dyck^{(f)}}(k)$ such that 
\begin{align}
\label{eqn:jdef2}
\begin{split}
&1\le j_1\le n+1-h(p_{1}), \\
&j_i+h(p_i)\le j_{i+1}\le n+1-h(p_{i+1}), \qquad 1\le i\le n-1.
\end{split}
\end{align}
We denote by $J'(\mu;n)$ the set of sequences $(j_1,\ldots,j_{t})$ satisfying conditions (\ref{eqn:jdef2}).

For positive integers $1\le h$ and $1\le j$, we define a function $H(h,j)$ as
\begin{align}
\label{eqn:defH}
H(h,j):=hj+h(h-1)/2.
\end{align}
Then, we define the weight $\mathrm{wt}^{'}_{n}(\mu;s)$ for a Dyck path $\mu\in\mathtt{Dyck}(k)$
by
\begin{align*}
\mathrm{wt}^{'}_{n}(\mu;s)
:=\sum_{\mathbf{j}\in J'(\mu;n)}s^{d(\mathbf{j})},
\end{align*}
where 
\begin{align*}
d(\mathbf{j}):=\sum_{i=1}^{t}H(h(p_{i}),j_i).
\end{align*}
 
The generating function $G_{n}^{(s)}(m,r,s)$ can be expressed 
in terms of fundamental Dyck paths.
\begin{theorem}
\label{thrm:GsninDyck}
We have
\begin{align}
\label{eqn:GsninDyck}
G_{n}^{(s)}(m,r,s)
=\sum_{k=0}^{n}\sum_{\mu\in\mathtt{Dyck^{(f)}}(k)}
(-1)^{k-\mathrm{Peak}(\mu)}r^{k}m^{\mathrm{Peak}(\mu)}
\mathrm{wt}^{'}_{n}(\mu;s).
\end{align}
\end{theorem}
\begin{proof}
We prove the theorem by induction on $n$.
For $n=1$, we have two Dyck paths $\emptyset$ and $UD$ which are of length 
less than or equal to $1$.
For the path $UD$, we have $\mathbf{j}=(1)$ and obtain the expression (\ref{eqn:GsninDyck}).
For $n=2$, we have four Dyck paths $\emptyset$, $UD$, $UDUD$ and $UUDD$.
Then, we have 
\begin{align*}
J'(\emptyset;2)=\emptyset, \quad J'(UD;2)=\{\{1\},\{2\}\}, 
\quad J'(UDUD;2)=\{\{1,2\}\},\quad  J'(UUDD;2)=\{\{1\}\}. 
\end{align*}
Note that the heights of two peaks in $UDUD$ are both one and the height of 
the unique peak in $UUDD$ is two.
From these observations, the expression (\ref{eqn:GsninDyck}) holds for $n=2$.

We will show that Eq. (\ref{eqn:GsninDyck}) holds for $n$ by showing that 
it satisfies the recurrence relation in Theorem \ref{thrm:Gsrrflip}.
By expanding the recurrence relation, we have 
\begin{align}
\label{eqn:Gsnrec2}
G_{n}^{(s)}(m,r,s)=G_{n-1}^{(s)}(m,rs,s)+rsmG_{n-1}^{(s)}(m,rs,s)-rsG_{n-1}^{(s)}(m,rs,s)+rsG_{n-2}^{(s)}(m,rs^2,s).
\end{align}
By induction assumption, the term $\mathrm{wt}'_{n}(\mu;s)$ depends on 
the values in $J'(\mu;n)$.
It is easy to see that the shifts $r\rightarrow rs$ and $r\rightarrow rs^{2}$ 
increase the values in $J'(\mu;n)$ by one and two respectively.

Recall that prime Dyck paths in a fundamental Dyck path $\mu$ are 
of the form $U^{p}D^{p}$ with some positive integer $p$.
Let $\mu$ be a fundamental Dyck path of length $n'$.
Let $\mu_1$ be the first prime Dyck path in $\mu$.
We have two cases: 1) the size of $\mu_1$ is one, and 2) the size 
of $\mu_1$ is larger than one.  

\paragraph{Case 1)}
When $\mu=(01)^{n}$, only the second term $rsmG_{n-1}^{(s)}(m,rs,s)$ in Eq. (\ref{eqn:Gsnrec2})
gives a contribution since $\mu=(01)\circ(01)^{n-1}$ and the coefficient $rsm$ can be interpreted 
as the path $(01)$. Further, the shift $r\rightarrow rs$ gives the shift of elements in 
$J'((01)^{n-1};n-1)$ by one. Therefore, Eq. (\ref{eqn:Gsnrec2}) holds for $\mu=(01)^{n}$.
Below, we suppose that $\mu\neq(01)^{n}$.
The elements $(j_1,\ldots,j_{t})$ in $J'(\mu;n)$ are classified into two classes.
The first class is the elements with $j_1=1$ and the second one is with $j_1\ge2$.
Since $\mu\neq(01)^{n}$ and $\mu_1=(01)$, we have contributions from the first and 
the second terms in the right had side of Eq. (\ref{eqn:Gsnrec2}).
Recall that the entries in an element in $J'(\mu;n-1)$ are in $[1,n-1]$.
As observed above, the shift $r\rightarrow rs$ increases the entries by one. 
Therefore, the minimum entries is two.
Similarly, we have a contribution from the second term in Eq. (\ref{eqn:Gsnrec2}).
In this case, the factor $rsm$ corresponds to $\mu_1$ and it is equivalent that 
the minimum entry is one.
Thus, by taking account into the above two cases, we have the set $J'(\mu;n)$.

\paragraph{Case 2)}
Suppose that $\mu_1=0^{p}1^{p}$ with a positive integer $p\ge2$.
By a similar argument to Case 1), the first term in the right hand side of Eq. (\ref{eqn:Gsnrec2})
gives the contributions to elements $J'(\mu;n)$  whose minimum entry is greater than one.
Below, we consider the contributions from the third and fourth terms in the right hand side 
of Eq. (\ref{eqn:Gsnrec2}).
Let $\mu'$ be a fundamental Dyck path obtained from $\mu$ by replacing $\mu_1$ by $0^{p-1}1^{p-1}$.
The length of $\mu'$ is one less than that of $\mu$.
Since the shift $r\rightarrow rs^{q}$ with $q\in\{1,2\}$ increases the entries in the elements of $J'(\mu';n-1)$
and $J'(\mu';n-2)$.
Thus the difference of the third and fourth terms gives the elements in 
$J'(\mu';n-1)$ such that the minimum entries of them are two.
Recall that the first prime Dyck paths in $\mu$ and $\mu'$ are $0^{p}1^{p}$ and 
$0^{p-1}q^{p-1}$ respectively. 
By a straightforward calculation, the difference of $d(\mathbf{j})$ 
between $\mu_1$ with $j_1=1$ and $0^{p-1}1^{p-1}$ with $j_1=2$ is one.
Then this difference cancels by the factor $rs$ in the third and fourth terms.
From the observations above, the right hand side of Eq. (\ref{eqn:Gsnrec2}) 
gives exactly $J'(\mu;n)$.

By combining the two cases above, Eq. (\ref{eqn:GsninDyck}) holds for $n$ by 
induction assumption.
This completes the proof.
\end{proof}

Since $\chi(m,r,s)^{-1}$ can be expressed in terms of prime 
Motzkin paths as in Proposition \ref{prop:chiinvinMot}, one can expect that 
the generating function $G_{n}^{(s)}(m,r,s)$ can be expressed 
in terms of Motzkin paths. 
In fact, the next proposition gives such an expression.
As in the case of Dyck paths, $G_{n}^{(s)}(m,r,s)$ is written 
as the generating function of $\lambda\in\mathtt{Mot}^{(f)}(k)$.
Note that it requires fundamental Motzkin paths instead of prime paths.

Let $\lambda\in\mathtt{Mot}^{(f)}(k)$ be a Motzkin path of size $k$ which is written as 
a concatenation of fundamental Motzkin paths, $\lambda=\lambda_{1}\circ\lambda_2\circ\cdots\circ\lambda_{p}$
with some positive integer $1\le p\le n$.
We assign an integer sequence $\mathbf{j}:=(j_1,\ldots,j_{p})$ to the fundamental Motzkin paths $\lambda_{i}$, 
$1\le i\le p$, in $\mathtt{Mot}^{(f)}(k)$ such that 
\begin{align}
\label{eqn:condjMot}
\begin{split}
&1\le j_1\le n+1-l(\lambda_{1}), \\
&j_i+l(\lambda_{i})\le j_{i+1}\le n+1-l(\lambda_{i+1}), \quad 1\le i\le p-1.
\end{split}
\end{align}
We denote $J(\lambda;n)$ the set of sequences $\mathbf{j}$ satisfying the conditions (\ref{eqn:condjMot}).
We define the weight $\mathrm{wt}_{n}(\lambda;s)$ for 
a Motzkin path $\lambda\in\mathtt{Mot}^{(f)}(k)$ as 
\begin{align*}
\mathrm{wt}_{n}(\lambda;m,s):=(-1)^{k-p}m^{p}\sum_{\mathbf{j}\in J(\lambda;n)}s^{d'(\mathbf{j})},
\end{align*}
where $p$ is the number of fundamental paths in $\lambda$ and 
\begin{align*}
d'(\mathbf{j}):=\sum_{i=1}^{p}l(\lambda_{i})(l(\lambda_{i})-1)/2
+\sum_{i=1}^{p}j_{i} l(\lambda_{i}).
\end{align*}
Then, we have the following theorem.
\begin{theorem}
\label{thrm:GsninMot}
The generating function $G_{n}^{(s)}(m,r,s)$ is expressed as 
\begin{align}
\label{eqn:GsninMot}
G_{n}^{(s)}(m,r,s)
=\sum_{k=0}^{n}r^{k}\sum_{\lambda\in\mathtt{Mot}^{(f)}(k)}
\mathrm{wt}_{n}(\lambda;m,s).
\end{align}
\end{theorem}
\begin{proof}
One can prove the theorem by induction on $n$ as in the case of Dyck paths 
(see Theorem \ref{thrm:GsninDyck}).
The key is the following correspondence between fundamental Motzkin paths 
of size $l$ and $l-1$.
Since a fundamental Motzkin path is of the form $U^{N}D^{N}$ or $U^{N}HD^{N}$, 
we transform the Motzkin path by replacing $UD$ by $H$ or $UHD$ by $UD$.
This transformation gives a fundamental Motzkin path of size one less.
Then, by a similar argument to the proof of Theorem \ref{thrm:GsninDyck},
one can show that $G_{n}^{(s)}(m,r,s)$ satisfies Eq. (\ref{eqn:GsninMot}).
This completes the proof.
\end{proof}

Similarly, the generating function $G_{n}^{(c)}(m,r,s)$ can be expressed 
in terms of Dyck paths and Motzkin paths respectively.

We first consider the expression by Dyck paths.
Since the generating function $G_{n}^{(c)}(m,r,s)$ is defined on a circle,
we need to impose the cyclic condition instead of the condition (\ref{eqn:jdef2}).
Suppose that a Dyck path $\mu\in\mathtt{Dyck}^{(f)}(k)$ is written as a concatenation
of fundamental Dyck paths of size $l_{i}$, $1\le i\le t$.
As in the case of a segment, 
we assign an integer sequence $\widetilde{\mathbf{j}}:=(j_1,\ldots,j_t)$ to the 
peaks $p_{i}$, $1\le i\le t$, in $\mu$.
The sequence $\widetilde{\mathbf{j}}$ satisfies the following conditions:
\begin{align*}
&1\le j_1\le n, \\
&j_{i}+h(p_{i})\le j_{i+1}\le n, \quad 0\le i\le t-1,  
\end{align*}
and if $j_{n}+h(p_{n})\ge n$, then we have the condition 
\begin{align*}
j_{n}+h(p_{n})-n\le j_{1}.
\end{align*}
We denote by $\widetilde{J}(\mu;n)$ the set of sequences $\widetilde{\mathbf{j}}$ which 
satisfy the above conditions.
We define the weight $\widetilde{\mathrm{wt}_{n}}(\mu;s)$ 
by 
\begin{align*}
\widetilde{\mathrm{wt}_{n}}(\mu;s):=\sum_{\widetilde{\mathbf{j}}\in\widetilde{J}(\mu;n)}s^{d(\widetilde{\mathbf{j}})},
\end{align*}
where 
\begin{align*}
d(\widetilde{\mathbf{j}}):=
\sum_{i=1}^{t}H(h(p_i),j_i)+
\begin{cases}
0, & \text{if } j_{t}\le n+1-h(p_{t}), \\
-(h(p_t)+j_t-n-1)n, & \text{otherwise}.
\end{cases}
\end{align*}
Here, the function $H(h,j)$ is defined in Eq. (\ref{eqn:defH}).

\begin{theorem}
\label{thrm:GcninDyck}
The generating function $G_{n}^{(c)}(m,r,s)$ has an expression 
in terms of the Dyck paths:
\begin{align}
\label{eqn:GcninDyck}
G_{n}^{(c)}(m,r,s)
=\sum_{k=0}^{n}\sum_{\mu\in\mathtt{Dyck}^{(f)}(k)}
(-1)^{k-\mathrm{Peak}(\mu)}r^{k} m^{\mathrm{Peak}(\mu)}
\widetilde{\mathrm{wt}_{n}}(\mu;s)+(-1)^{n}r^{n}s^{n(n+1)/2}m.
\end{align}
\end{theorem}

\begin{example}
We calculate the coefficient of $r^{3}m^{2}$ in $G_{4}^{(c)}(m,r,s)$.
Since the exponents of $r$ and $m$ are $3$ and $2$, it is enough 
to consider the two Dyck paths in $\mathtt{Dyck}^{(f)}(4)$: $\mu_1:=UUDDUD$ and $\mu_2:=UDUUDD$.
Then, we have 
\begin{align*}
\widetilde{J}(\mu_1;4)&=\{\{1,3\}, \{1,4\}, \{2,4\}\}, \\
\widetilde{J}(\mu_2;4)&=\{\{1,2\}, \{1,3\}, \{2,3\}, \{2,4\}, \{3,4\}\}.
\end{align*}
From these we have 
\begin{align*}
\widetilde{\mathrm{wt}_{4}}(\mu_1;s)&=s^{6}+s^{7}+s^{9}, \\
\widetilde{\mathrm{wt}_{4}}(\mu_2;s)&=s^{6}+s^{7}+2s^{8}+s^{9}.
\end{align*}
The paths $\mu_1$ and $\mu_2$ have two peaks, which gives the factor $-1$.
Then, the coefficient of $r^{3}m^{2}$ is given by 
\begin{align*}
-2(s^{6}+s^{7}+s^{8}+s^{9}).
\end{align*}
\end{example}

\begin{proof}[Proof of Theorem \ref{thrm:GcninDyck}]
We prove the theorem by induction on $n$.
It is obvious that the expression (\ref{eqn:GcninDyck}) holds up to $n=3$.

We calculate $G_{n}^{(c)}(m,r,s)$ by making use of the recurrence 
relation in Theorem \ref{thrm:Gcrr1}.
Recall that we have the expression (\ref{eqn:GsninDyck}) for $G_{n}^{(s)}$ 
in Theorem \ref{thrm:GsninDyck}.
The difference $\widetilde{J}(\mu;n)$ and $J'(\mu;n)$ is given by 
\begin{align}
\label{eqn:JnJn}
\widetilde{J}(\mu;n)\setminus J'(\mu;n)
=\left\{\widetilde{\mathbf{j}}\in\widetilde{J}(\mu;n)\Big| 
n+1-h(p_{t})<j_{t}\le n\right\},
\end{align}
Note that if the last fundamental Dyck path in $\mu$ is of size one,
$\widetilde{J}(\mu;n)\setminus J'(\mu;n)=\emptyset$.
Therefore, we show that $rs(G_{n-2}^{(s)}(m,rs,s)-G_{n-1}^{(c)}(m,rs,s))$ 
gives the contributions to $G_{n}^{(c)}(m,r,s)$ which are expressed 
by the set $\widetilde{J}(\mu;n)\setminus J'(\mu;n)$. 

Suppose that the last fundamental Dyck path in $\mu$ is $\mu_{t}$ and its 
size is larger than one.
Let $\mu'$ be the Dyck path obtained from $\mu$ by replacing $\mu_{t}$ by 
a fundamental Dyck path of size $l(\mu_{t})-1$.
We calculate the difference $rs(G_{n-2}^{(s)}(m,rs,s)-G_{n-1}^{(c)}(m,rs,s))$ as 
follows.
We have 
\begin{align}
\label{eqn:Jn1Jn2}
\widetilde{J}(\mu';n-1)\setminus J'(\mu';n-2)
=\left\{\widetilde{\mathbf{j}}\in\widetilde{J}(\mu';n-1)\Big|
n-h(\mu_{t})\le j_{t}\le n-1
\right\}.
\end{align}
It is obvious that the cardinalities of the sets (\ref{eqn:JnJn}) and 
(\ref{eqn:Jn1Jn2}) are the same.
Then, we need to show that each elements in the sets (\ref{eqn:JnJn}) and 
(\ref{eqn:Jn1Jn2}) give the same monomial in $s$.
We show the statement above by use of definition of $\widetilde{\mathrm{wt}_{n}}(\mu;s)$.

We calculate the factor for $\mu'$ by using $H(h,j)$.
For the elements $\widetilde{\mathbf{j}}\in\widetilde{J}(\mu';n-1)\setminus J'(\mu';n-2)$,  
the exponent of $s$ is given by 
\begin{align}
\label{eqn:shiftH}
\begin{split}
&\sum_{i=1}^{t}H(l_i,j_i)+\sum_{i=1}^{t}l_{i}+1-(l_t+j_t-n)(n-1) \\ 
&\quad=\sum_{i=1}^{t-1}H(l_i,j_i+1)+H(l_{t}+1,j_{t}+1)-(l_t+j_t+1-n)n,
\end{split}
\end{align}
where $l_i:=l(\mu_{i})$, $1\le i\le t-1$, and $l_{t}=l(\mu_{t}')$.
The relation (\ref{eqn:shiftH}) implies that the exponent of $s$ for 
$\mu'$ is equal to that for $\mu$.
Further the element $\widetilde{\mathbf{j}}$ in Eq. (\ref{eqn:JnJn}) is obtained from 
the element in Eq. (\ref{eqn:Jn1Jn2}) by adding one to all entries..
From these observations, we have a natural bijection between the sets
in Eqs. (\ref{eqn:JnJn}) and (\ref{eqn:Jn1Jn2}) preserving the exponents 
of $s$.	

For the coefficient of $r^{n}m$ in $G_{n}^{(c)}(m,r,s)$, Eq. (\ref{eqn:GcninDyck})
holds by a straightforward calculation.
From these observations, the expression (\ref{eqn:GcninDyck}) holds for $n$, 
which completes the proof.
\end{proof}

As in the case of $G_{n}^{(s)}(m,r,s)$, the generating function $G_{n}^{(c)}(m,r,s)$
can be written as a generating function of Motzkin paths.
Let $\lambda\in\mathtt{Mot}^{(f)}(k)$ be a Motzkin path of size $k$ which is written as 
a concatenation of fundamental Motzkin paths $\lambda_{i}$, $1\le i\le t$.
We assign an integer sequence $\widetilde{j}:=(j_1,\ldots,j_t)$ satisfying
\begin{align*}
&1\le j_1\le n, \\
&j_{i}+l(\lambda_{i})\le j_{i+1}\le n, \quad 1\le i\le t-1,
\end{align*}
and if $j_{t}+l(\lambda_{t})\ge n$, we impose 
\begin{align*}
j_{t}+l(\lambda_{t})-n\le j_1.
\end{align*}
We denote by $\widetilde{J}(\lambda;n)$ the set of sequences $\widetilde{j}$ satisfying 
the conditions above.
Then, we define the weight $\widetilde{\mathrm{wt}_{n}}(\lambda;m,s)$ by 
\begin{align*}
\widetilde{\mathrm{wt}_{n}}(\lambda;m,s):=(-1)^{k-t}m^{t}
\sum_{\widetilde{\mathbf{j}}\in\widetilde{J}(\lambda;n)}s^{d'(\widetilde{\mathbf{j}})},
\end{align*}
where $t$ is the number of fundamental paths in $\lambda$ and 
\begin{align*}
d'(\widetilde{\mathbf{j}}):=\sum_{i=1}^{t}H(l(\lambda_i),j_{i})
+
\begin{cases}
0, & j_{t}+l(\lambda_{t})\le n+1, \\
-(j_{t}+l(\lambda_{t})-n-1)n, & \text{otherwise},
\end{cases}
\end{align*}
and $H(h,j)$ is defined in Eq. (\ref{eqn:defH}).

\begin{theorem}
\label{thrm:GcninMot}
The generating function $G_{n}^{(c)}(m,r,s)$ can be expressed in terms of Motzkin paths:
\begin{align*}
G_{n}^{(c)}(m,r,s)=\sum_{k=0}^{n}r^{k}
\sum_{\lambda\in\mathtt{Mor}^{(f)}(k)}
\widetilde{\mathrm{wt}_{n}}(\lambda;m,s)+(-1)^{n}r^{n}s^{n(n+1)/2}m.
\end{align*}
\end{theorem}
\begin{proof}
One can show the theorem by induction on $n$ in a similar manner to the 
proofs of Theorem \ref{thrm:GsninMot} and Theorem \ref{thrm:GcninDyck}.
\end{proof}

Below, we consider the large $n$ limit of the generating function $G_{n}^{(2)}(m,r,s)$.

\begin{defn}
Let $\xi(m,r,s)$ be a formal power series satisfying the recurrence relation:
\begin{align}
\label{eqn:defnxi}
\begin{split}
1&=(1+rs(m-2))\xi(m,r,s)+rs(1+s+rs^2(m-1))\xi(m,r,s)\xi(m,rs,s) \\
&\quad+r^2s^4\xi(m,r,s)\xi(m,rs,s)\xi(m,rs^2,s).
\end{split}
\end{align}
\end{defn}

The first few terms of $\xi(m,r,s)$ are given by 
\begin{align*}
\xi(m,r,s)&=1-(m's+s^2)r+(m'^2s^2+2m's^3+s^4+m's^4+s^5)r^2 \\
&\quad-(m'^3s^3+3m'^2s^4+3m's^5+2m'^2s^5+s^6+4m's^6+m'^2s^6+2s^7+2m's^7\\ 
&\quad+s^8+m's^8+s^9)r^3+\cdots,
\end{align*}
where $m':=m-1$.

The series $\xi(m,r,s)$ can be expressed as a generating function
of Schr\"oder paths with statistics.
Suppose that $\kappa$ is a Schr\"oder path of size $n$, and 
$\kappa$ has $N(\kappa;H)$ horizontal steps and $N(\kappa;D)$ down
steps.

Since a Schr\"oder path $\kappa$ consists of up, down and horizontal steps, 
each up step in $\kappa$ has a ``balanced" down step in $\kappa$ as in the case 
of Dyck paths.
Here, a balanced down step is defined as follows.
If an up step $u_{i}$ connects the two points $(i,j)$ and $(i+1,j+1)$ in the Cartesian 
coordinate, the balanced down step corresponding to the up step $u_{i}$ is the down step 
such that it connects $(k,j+1)$ and $(k+1,j)$, and $k$ is the minimum value 
satisfying $i<k$.
We call the pair of an up step and its corresponding balanced down step 
an arc.
We define the size $l(a)$ of an arc $a$ by $l(a)=(k-i+1)/2$.
We denote by $\mathrm{Arc}(\kappa)$ the set of arcs in $\kappa$.
Note that the cardinality of $\mathrm{Arc}(\kappa)$ is given by 
$N(\kappa;D)$ since if we delete horizontal steps from $\kappa$, 
we have a Dyck paths of size $N(\kappa;D)$.	

\begin{example}
When $\kappa=UHD$, we have $N(\kappa;H)=1$ and $N(\kappa;D)=1$.
We have one arc in $\kappa$ and the size of this arc is $2$.

Similarly, when $\kappa=UDUD$, we have $N(\kappa;H)=0$ and 
$N(\kappa;D)=2$. We have two arcs of size one.
\end{example}

\begin{defn}
\label{defn:wtSchr}
Let $N(\kappa;H)$ and $\mathrm{Pos}(\kappa)$ be the values defined as above.
Then, we define the weight for a Schr\"oder path $\kappa$ by
\begin{align*}
\mathrm{wt}(\kappa;m,s):=(m-1)^{N(\kappa;H)}s^{A(\kappa)},
\end{align*}
where 
\begin{align*}
A(\kappa):=\sum_{a\in\mathrm{Arc}(\kappa)}l(a).
\end{align*}
\end{defn}

Given a Schr\"oder path $\kappa$, the weight $\mathrm{wt}(\kappa;m,s)$ satisfies 
the following recurrence relation.
\begin{lemma}
\label{lemma:recwtkappa}
Let $\mathrm{wt}(\kappa;m,s)$ be the weight defined as above.
Then, it satisfies the recurrence relation:
\begin{align}
\label{eqn:etarecrel}
\begin{split}
\sum_{\kappa_{0}\in\mathtt{Sch}(n)}\mathrm{wt}(\kappa_{0};m,s)
&=((m-1)+s+s^{n})\sum_{\kappa_{1}\in\mathtt{Sch}(n-1)}\mathrm{wt}(\kappa_{0};m,s) \\
&\quad+\sum_{p=1}^{n-2}s^{n-p}\sum_{\kappa_{2}\in\mathtt{Sch}(p)}\sum_{\kappa_{3}\in\mathtt{Sch}(n-p-1)}
\mathrm{wt}(\kappa_{2};m,s)\mathrm{wt}(\kappa_{3};m,s).
\end{split}
\end{align}
\end{lemma}
\begin{proof}
We prove the identity (\ref{eqn:etarecrel}) by constructing a bijective map.

One can show that Eq. (\ref{eqn:etarecrel}) holds for $n=0,1$ and $2$ 
by a straightforward calculation. 
Suppose that $\kappa_{0}$ is a Schr\"oder path of size $n\ge3$. 
Since a path is written as a concatenation of prime paths, 
we the set of Schr\"oder paths of size $n$ is divided into four types.
The first type satisfies that the first step is the horizontal step.
The second type satisfies that the first and second steps are $U$ and $D$.
Third type satisfies that the first step is $U$, the last step is $D$ and 
it is a prime path of size $n$. 
The fourth type is the set of paths which are not classified in the three types 
as above. This means that the path $\kappa_{0}$ is a concatenation of two 
paths $\kappa_2$ and $\kappa_{3}$ such that $\kappa_3$ is a prime 
path of size greater than one and less than $n-1$.

Suppose $\kappa_{0}$ is the first type.
Then, $\kappa_{0}$ is written as $\kappa_{0}=H\kappa_{1}$ for 
some $\kappa_{1}\in\mathtt{Sch}(n-1)$.
The weights are related by
\begin{align*}
\mathrm{wt}(\kappa_{0};m,s)=(m-1)\cdot\mathrm{wt}(\kappa_1).
\end{align*}
Suppose that $\kappa_{0}$ is the second type, {\it i.e.},
$\kappa_{0}=UD\kappa_{1}$ with $\kappa_1\in\mathtt{Sch}(n-1)$.
The weight are related by 
\begin{align*}
\mathrm{wt}(\kappa_{0};m,s)=s\cdot\mathrm{wt}(\kappa_{1}).
\end{align*}
Suppose that $\kappa_{0}$ is the third type, {\it i.e.},
$\kappa_{0}=U\kappa_{1}D$ with $\kappa_1\in\mathtt{Sch}(n-1)$.
Then, the weights are related by 
\begin{align*}
\mathrm{wt}(\kappa_{0};m,s)=s^{n}\cdot\mathrm{wt}(\kappa_1;m,s).
\end{align*}
Suppose that $\kappa_{0}$ is the fourth type, {\it i.e.},
$\kappa_{0}=\kappa_{2}\circ\kappa'_{3}$ such that 
$\kappa'_{3}$ is a prime path of size $n-p$ with $1\le p\le n-2$.
By definition of prime paths, we have a bijection between 
prime Schr\"oder paths $\kappa'_{3}$ of size $n-p$ and Schr\"oder paths $\kappa_{3}$ 
of size $n-p-1$. 
This bijection is realized by deleting the first step $U$ and the last step $D$ 
from $\kappa'_{3}$
Then, the weights are related by 
\begin{align*}
\mathrm{wt}(\kappa_{0};m,s)=s^{n-p}\mathrm{wt}(\kappa_{2};m,s)\mathrm{wt}(\kappa_{3};m,s).
\end{align*}
From these observations, we have Eq. (\ref{eqn:etarecrel}).
This completes the proof.
\end{proof}

\begin{prop}
The power series $\xi(m,r,s)$ can be expressed in terms of Schr\"oder paths as 
\begin{align*}
\xi(m,r,s)=\sum_{0\le n}(-rs)^{n}
\sum_{\kappa\in\mathtt{Sch}(n)}\mathrm{wt}(\kappa;m,s).
\end{align*}
\end{prop}
\begin{proof}
Suppose that $\xi(m,r,s)$ is expanded as 
\begin{align*}
\xi(m,r,s)=\sum_{0\le n}(-rs)^{n}\xi_{n}.
\end{align*}
From the defining relation (\ref{eqn:defnxi}) for $\xi(m,r,s)$, 
we have $\xi_{0}=1$ and 
\begin{align}
\label{eqn:xirecrel}
\begin{split}
0&=\xi_{n}-(m-2)\xi_{n-1}-\sum_{u=0}^{n-1}\xi_{u}\xi_{n-1-u}s^{n-1-u}(1+s) \\
&\quad+(m-1)\sum_{n=0}^{n-2}\xi_{u}\xi_{n-2-u}s^{n-1-u} 
+\sum_{u+v+w=n-2}s^{v+2w+2}\xi_{u}\xi_{v}\xi_{w},
\end{split}
\end{align}
for $1\le n$.

Let $\overline{\xi}_{n}:=\sum_{\kappa\in\mathtt{Sch}(n)}\mathrm{wt}(\kappa;m,s)$.
Then, we have $\overline{\xi}_{0}=1$. 
We prove that $\overline{\xi}_{n}$ satisfies the recurrence relation (\ref{eqn:xirecrel}) by 
induction on $n$.
From Lemma \ref{lemma:recwtkappa}, $\overline{\xi_{n}}$ satisfies 
\begin{align}
\label{eqn:xirecrel2}
\overline{\xi}_{n}=(m-1)\overline{\xi}_{n-1}+\sum_{u=0}^{n-1}s^{n-u}\overline{\xi}_{u}\overline{\xi}_{n-1-u}.
\end{align}
We also have 
\begin{align}
\label{eqn:xirecrel3}
\begin{split}
\sum_{u+v+w=n-2}s^{v+2w+2}\overline{\xi}_{u}\overline{\xi}_{v}\overline{\xi}_{w}
&=\sum_{u=0}^{n-2}s^{n-1-u}\overline{\xi}_{u}
\sum_{v=0}^{n-2-u}s^{n-1-u-v}\overline{\xi}_{v}\overline{\xi}_{n-2-u-v}, \\
&=\sum_{n=0}^{n-2}s^{n-1-u}\overline{\xi}_{u}(\overline{\xi}_{n-1-u}-(m-1)\overline{\xi}_{n-2-u}) \\
&=-\overline{\xi}_{n-1}+\sum_{n=0}^{n-1}s^{n-1-u}\overline{\xi}_{u}\overline{\xi}_{n-1-u}
-(m-1)\sum_{n=0}^{n-2}s^{n-1-u}\overline{\xi}_{u}\overline{\xi}_{n-2-u}.
\end{split}
\end{align}
By substituting Eqs. (\ref{eqn:xirecrel2}) and (\ref{eqn:xirecrel3}) into 
Eq. (\ref{eqn:xirecrel}), one can easily show that $\overline{\xi}_{n}$ 
satisfies the recurrence relation (\ref{eqn:xirecrel}).
Further, since they have the same initial conditions $\xi_{0}=1$ and $\overline{\xi}_{0}=1$, 
we have $\xi_{n}=\overline{\xi}_{n}$. 
This completes the proof.
\end{proof}

\begin{prop}
\label{prop:G2infxi}
In the large $n$ limit, we have
\begin{align}
\label{eqn:G2infxi}
\lim_{n\rightarrow\infty}(rs^{n})^{-1}G_{n}^{(2)}(m,r,s)
=\prod_{0\le j}\genfrac{}{}{}{}{1}{\xi(m,rs^{j},s)}.
\end{align}
\end{prop}
\begin{proof}
The generating function $G_{n}^{(2)}(m,r,s)$ satisfies the 
recurrence relation (\ref{eqn:G2recrel}).
By dividing the both sides of Eq. (\ref{eqn:G2recrel}) by $G_{n}^{(2)}(m,r,s)$
and taking the large $n$ limit, we have a recurrence relation (\ref{eqn:defnxi}).
The constant term of $(rs^{n})^{-1}G_{n}^{(2)}(m,r,s)$ is one.
The constant term of the right hand side of Eq. (\ref{eqn:G2infxi})
is also one.
By applying the same method used in Sections \ref{sec:Averd} and 
\ref{sec:Gfrevisit}, we obtain Eq. (\ref{eqn:G2infxi}).
\end{proof}

\subsection{Reversed generating functions}
In Propositions \ref{prop:Gsinvs}, \ref{prop:GcinftyRR} and \ref{prop:G2invs},
we observe that the generating functions $G_{n}^{(s)}(m,r,1/s), G_{n}^{(c)}(m,r,1/s)$
and $G_{n}^{(2)}(m,r,1/s)$ are expressed by the Rogers--Ramanujan type identities
if $(m,r)=(1,1)$.
In this subsection, we consider the case $m\neq1$.

\begin{defn}
Let $H_{n}^{(*)}(m,r,s)$ be the generating function defined by 
\begin{align*}
H_{n}^{(*)}(m,r,s)
=r^{n}s^{n(n+1)/2}G_{n}^{(*)}(m,1/r,1/s),
\end{align*}
where $\ast$ is either $s$ or $2$.
We call $H_{n}^{(*)}(m,r,s)$ a reversed generating function.
\end{defn} 

Below, we first consider the reversed generating function 
$H_{n}^{(s)}(m,r,s)$, and then $H_{n}^{(2)}(m,r,s)$.

\subsubsection{The reversed generating function \texorpdfstring{$H_{n}^{(s)}(m,r,s)$}{H_n^2(m,r,s)}}
\begin{prop}
\label{prop:recrelH}
The reversed generating function $H_{n}^{(s)}(m,r,s)$ satisfies the 
recurrence relation:
\begin{align}
\label{eqn:recrelH}
H_{n}^{(s)}(m,r,s)=rsf(r^{-1}s^{-1})H_{n-1}^{(s)}(m,rs,s)
+rs^{2}H_{n-2}^{(s)}(m,rs^2,s),
\end{align}
where $f(x)$ is defined in Definition \ref{def:f}.
\end{prop}
\begin{proof}
Since $G_{n}^{(s)}(m,r,s)$ satisfies the recurrence relation of 
Fibonacci type (see Theorem \ref{thrm:FibGfin}),
it is obvious that $H_{n}^{(s)}(m,r,s)$ satisfies 
the recurrence relation \ref{eqn:recrelH}.
\end{proof}

\begin{defn}
\label{defn:eta}
Let $\eta(m,r,s)$ be a formal power series with respect to 
$r$ and $s$ satisfying the recurrence relation:
\begin{align*}
rsf(r^{-1}s^{-1})\eta(m,r,s)=1-rs^{2}\eta(m,r,s)\eta(m,rs,s).
\end{align*}
\end{defn}
The first few terms of $\eta(m,r,s)$ are given by 
\begin{align*}
\eta(m,r,s)
=\genfrac{}{}{}{}{1}{m'}-(m's+s^{2})\genfrac{}{}{}{}{r}{m'^{3}}
+(m'^2s^2+2m's^3+s^4+m's^4+s^5)\genfrac{}{}{}{}{r^2}{m'^{5}}
+\cdots,
\end{align*}
where $m'=m-1$.

\begin{prop}
The series $\eta(m,r,s)$ is expressed in terms of Schr\"oder paths:
\begin{align}
\label{eqn:eta}
\eta(m,r,s)=\sum_{0\le n}(-rs)^{n}(m-1)^{-(2n+1)}
\sum_{\kappa\in\mathtt{Sch}(n)}\mathrm{wt}(\kappa;m,s),
\end{align}
where $\mathrm{wt}(\kappa;m,s)$ is defined in Definition \ref{defn:wtSchr}.
\end{prop}	
\begin{proof}
We substitute the expression (\ref{eqn:eta}) into the defining relation 
of $\eta(m,r,s)$ in Definition \ref{defn:eta}, and obtain 
the recurrence relation for $\mathrm{wt}(\kappa;m,s)$:
\begin{align}
\label{eqn:etarecrel2}
\begin{split}
\sum_{\kappa_{0}\in\mathtt{Sch}(n)}\mathrm{wt}(\kappa_{0};m,s)
&=((m-1)+s+s^{n})\sum_{\kappa_{1}\in\mathtt{Sch}(n-1)}\mathrm{wt}(\kappa_{0};m,s) \\
&\quad+\sum_{p=1}^{n-2}s^{n-p}\sum_{\kappa_{2}\in\mathtt{Sch}(p)}\sum_{\kappa_{3}\in\mathtt{Sch}(n-p-1)}
\mathrm{wt}(\kappa_{2};m,s)\mathrm{wt}(\kappa_{3};m,s).
\end{split}
\end{align}
To prove the proposition, it is enough to show that $\mathrm{wt}(\kappa;m,s)$ satisfies 
the recurrence relation (\ref{eqn:etarecrel2}).
However, this is obvious from Lemma \ref{lemma:recwtkappa}.
This completes the proof.
\end{proof}

\begin{theorem}
\label{thrm:revGsineta}
The reversed generating function $H_{n}^{(s)}(m,r,s)$ is expressed as an infinite 
product in terms of $\eta(m,r,s)$.
\begin{align*}
H_{\infty}^{(s)}(m,r,s)=\genfrac{}{}{}{}{m}{m-1}\prod_{0\le n}\genfrac{}{}{}{}{1}{\eta(m,rs^{n},s)}.
\end{align*}
\end{theorem}
\begin{proof}

Since $H_{n}^{(s)}$ satisfies the recurrence relation (\ref{eqn:recrelH}) as in 
Proposition \ref{prop:recrelH}, one can apply the same method used in
Sections \ref{sec:Averd} and \ref{sec:Gfrevisit}.
Note that the coefficient of $r^{0}s^{0}$ in $H_{n}^{(s)}$ is $m(m-1)^{n-1}$ and
we have $[r^{0}s^{0}]\eta(m,r,s)^{-1}=(m-1)$.
Thus the factor $m/(m-1)$ comes from the correction of the constant term.
This completes the proof.
\end{proof}

\subsubsection{The reversed generating function \texorpdfstring{$H_{n}^{(2)}(m,r,s)$}{H_n^2(m,r,s)}}
In this section, we study the reversed generating function $H_{n}^{(2)}(m,r,s)$ for $G_{n}^{(2)}(m,r,s)$.
\begin{prop}
\label{prop:recrelH2}
The reversed generating function $H_{n}^{(2)}(m,r,s)$ satisfies 
the following recurrence relation:
\begin{align}
\label{eqn:recrelH2}
\begin{split}
H_{n}^{(2)}(m,r,s)
&=(rs+(m-2))H_{n-1}^{(2)}(m,rs,s)+(rs^2+rs+(m-1))H_{n-2}^{(2)}(m,rs^2,s) \\
&\quad+rs^2H_{n-3}^{(2)}(m,rs^3,s).
\end{split}
\end{align}
\end{prop}

This recurrence relation leads to the definition of power series 
$\nu(m,r,s)$ as follows.
\begin{defn}
\label{defn:defnu}
Let $\nu(m,r,s)$ be a power series satisfying the recurrence relation	
\begin{align}
\label{eqn:recrelnu}
\begin{split}
1&=(rs+(m-2))\nu(m,r,s)+(rs^2+rs+(m-1))\nu(m,r,s)\nu(m,rs,s) \\
&\quad+rs^{2}\nu(m,r,s)\nu(m,rs,s)\nu(m,rs^2,s).
\end{split}
\end{align}
\end{defn}
The first few terms of $\nu(m,r,s)$ are given by 
\begin{align*}
\begin{split}
\nu(m,r,s)&=\genfrac{}{}{}{}{1}{m-1}-\genfrac{}{}{}{}{mrs}{(m-1)^3}
+\genfrac{}{}{}{}{mr^2s^2(m+s)}{(m-1)^5}
-\genfrac{}{}{}{}{mr^3s^3(m^2+s^3+ms(2+s))}{(m-1)^{7}}  \\
&\quad+\genfrac{}{}{}{}{mr^4s^4(m^3+s^{6}+m^2s(3+2s+s^2)+ms^2(1+2s+2s^2+s^3))}{(m-1)^{9}}
-\cdots.
\end{split}
\end{align*}

\begin{remark}
\label{remark:nuchi}
The series $\nu(m,r,s)$ is the same to the power series 
$\chi(m,r,s)$ except that $\nu(m,r,s)$ has a factor 
$(m-1)^{2n+1}$ in each term with $r^{n}$. 
However, the recurrence relations for $\chi(m,r,s)$ and 
$\nu(m,r,s)$ have different appearances.
Note also that the series $\nu(m,r,s)$ has poles at $m=1$.
This is because we cannot have $n$ dimers on a segment 
of length $n$ when $m=1$. 
\end{remark}

We define the weight $\mathrm{wt}^{(2)}(\mu;m,s)$ for a Dyck path $\mu$ 
of size $n$ with the statistics described as below.
Given a Dyck path $\mu$, we denote by $\mathrm{Peak}(\mu)$ the number 
of peaks in $\mu$.
Let $\mu_{0}$ be the lowest Dyck path, that is, $\mu_{0}=(UD)^{n}$.
Then, we denote by $|Y(\mu_{0}/\mu)|$ the number of boxes above 
$\mu_{0}$ and below $\mu$.
Then, we define the weight by 
\begin{align*}
\mathrm{wt}^{(2)}(\mu;m,s):=m^{\mathrm{Peak}(\mu)}s^{|Y(\mu_{0}/\mu)|}.
\end{align*}

\begin{prop}
The power series $\nu(m,r,s)$ has an expression in terms of 
Dyck paths with weights.
\begin{align*}
\nu(m,r,s)=\sum_{0\le n}(-1)^{n}\genfrac{}{}{}{}{r^{n}s^{n}}{(m-1)^{2n+1}}
\sum_{\mu\in\mathtt{Dyck}(n)}\mathrm{wt}^{(2)}(\mu;m,s).
\end{align*}
\end{prop}
\begin{proof}
By a simple calculation, we can show that the constant term $[r^{0}]\mu(m,r,s)$ is 
equal to $(m-1)^{-1}$. 
We consider the coefficient of $r^{n}$ in the both sides of Eq. (\ref{eqn:recrelnu}) with $1\le n$.
Let $\nu_{n}:=\nu_{n}(m,s)$ be a polynomial 
\begin{align*}
\nu_{n}:=\sum_{\mu\in\mathtt{Dyck}(n)}\mathrm{wt}^{(2)}(\mu;m,s).
\end{align*}
From Definition \ref{defn:defnu}, $\nu_{n}(m,s)$ should satisfy 
\begin{align}
\label{eqn:recnu2}
\begin{split}
0&=(m-2)\nu_{n}-(m-1)^2\nu_{n-1}-(m-1)(1+s)\sum_{p=0}^{n-1}\nu_{p}\nu_{n-1-p}s^{n-1-p} \\
&\quad+\sum_{p=0}^{n}\nu_{p}\nu_{n-p}s^{n-p}
-s\sum_{p_1+p_2+p_3=n-1}\nu_{p_1}\nu_{p_2}\nu_{p_3}s^{p_2+2p_3}.
\end{split}
\end{align}
From Definition \ref{defn:chi}, Proposition \ref{prop:chiinDyck} 
and Remark \ref{remark:nuchi}, $\nu_{n}$ satisfies 
\begin{align}
\label{eqn:recrelnuchi}
\nu_{n}-(m-1)\nu_{n-1}=\sum_{p=0}^{n-1}\nu_{p}\nu_{n-1-p}s^{n-1-p}.
\end{align}
We also have 
\begin{align}
\label{eqn:triplenu}
\begin{split}
\sum_{p_1+p_2+p_3=n-1}\nu_{p_1}\nu_{p_2}\nu_{p_3}s^{p_2+2p_3}
&=\sum_{p_1=0}^{n-1}s^{n-1-p_1}\nu_{n-p_1}
\left(\sum_{p_2=0}^{n-1-p_1}\nu_{p_2}\nu_{n-1-p_{1}-p_2}s^{n-1-p_1-p_2}\right), \\
&=\sum_{p=0}^{n-1}s^{n-1-p}\nu_{p}(\nu_{n-p}-(m-1)\nu_{n-1-p}), \\
&=s^{-1}\nu_{n+1}-(1+s^{-1})(m-1)\nu_{n}+(m-1)^{2}\nu_{n-1}-s^{-1}\nu_{n},
\end{split}
\end{align}
where we have used Eq. (\ref{eqn:recrelnuchi}) twice.

By a straightforward calculation, we can show that 
$\nu_{n}$ satisfies the recurrence relation (\ref{eqn:recnu2})
by use of Eqs. (\ref{eqn:recrelnuchi}) and (\ref{eqn:triplenu}).
This completes the proof.
\end{proof}

\begin{theorem}
\label{thrm:revGsinnu}
The reversed generating function $H_{n}^{(2)}(m,r,s)$ in the large $n$ 
limit is given by
\begin{align}
\label{eqn:H2infnu}
\lim_{n\rightarrow\infty}\left(m H_{n}^{(2)}(m,r,s)+(-1)^{n-1}(m-1)\right)
=\prod_{0\le j}\genfrac{}{}{}{}{1}{\nu(m,rs^{j},s)}.
\end{align}
\end{theorem}
\begin{proof}
Define 
\begin{align*}
\widetilde{H}_{n}(m,r,s):=m H_{n}^{(2)}(m,r,s)+(-1)^{n-1}(m-1).
\end{align*}
The generating function $\widetilde{H}_{n}(m,r,s)$ satisfies 
the same recurrence relation (\ref{eqn:recrelH2}) in Proposition \ref{prop:recrelH2} 
as $H_{n}^{(2)}(m,r,s)$.
Note that the constant term of $\widetilde{H}_{n}(m,r,s)$ is $(m-1)^{n}$.	
One can apply the same method used in Section \ref{sec:Averd} and 
Section \ref{sec:Gfrevisit}, and obtain Eq. (\ref{eqn:H2infnu}).
\end{proof}

\subsection{Moments}
In this subsection, we study the correlation functions, which 
we call moments.
A moment is a generalization of the average number of dimers.
We show that the moments are expressed in terms of the 
formal power series $\chi(m,r,s)$ and another formal power series 
$B_{n}(m,r,s)$ defined in Definition \ref{defn:recB}.
Note that $B_{n}(m,r,s)$ is also defined through the power series $\chi(m,r,s)$.
  
Below, we first start with considering the case with $(m,r,s)=(1,1,1)$.
We denote $G_{n}:=G_{n}(1,1,1)$, $G'_{n}:=\partial_{r}G_{n}(1,1,1)$ 
and $\mathcal{N}_{n}:=\mathcal{N}_{n}(1,1,1)$.
Then, by definition of $\mathcal{N}_{n}(m,r,s)$, we have 
\begin{align}
\label{eqn:Nnlarge}
\begin{split}
\mathcal{N}_{n}&=\genfrac{}{}{}{}{G'_{n}}{G_{n}}, \\
&=\genfrac{}{}{}{}{G'_{n-1}+G'_{n-2}+G_{n-2}}{G_{n}}, \\
&=\mathcal{N}_{n-1}\genfrac{}{}{}{}{G_{n-1}}{G_{n}}
+\mathcal{N}_{n-2}\genfrac{}{}{}{}{G_{n-2}}{G_{n}}
+\genfrac{}{}{}{}{G_{n-2}}{G_{n}}.
\end{split}
\end{align}
From this recurrence relation, we have an asymptotic expression of 
$\mathcal{N}_{n}$ as follows.
\begin{theorem}
In the large $n$ limit, we have 
\begin{align*}
\mathcal{N}_{n}\approx
\genfrac{}{}{}{}{2^{-(n+1)}((45-19\sqrt{5})(-3+\sqrt{5})^{n}+2^{n+1}(140-63\sqrt{5}+30(-9+4\sqrt{5})n))}
{15(-25+11\sqrt{5})}.
\end{align*}
\end{theorem}
\begin{proof}
Since we consider the large $n$ limit, we have 
\begin{align*}
\lim_{n\rightarrow\infty}\genfrac{}{}{}{}{G_{n-1}}{G_{n}}
=\genfrac{}{}{}{}{-1+\sqrt{5}}{2}, 
\qquad
\lim_{n\rightarrow\infty}\genfrac{}{}{}{}{G_{n-2}}{G_{n}}
=\genfrac{(}{)}{}{}{-1+\sqrt{5}}{2}^{2}, 
\end{align*}
where we have used the recurrence relation for Fibonacci numbers.
The initial conditions are 
\begin{align*}
\mathcal{N}_{1}=\genfrac{}{}{}{}{1}{2},\qquad \mathcal{N}_{2}=\genfrac{}{}{}{}{2}{3}.
\end{align*}
By substituting these into Eq. (\ref{eqn:Nnlarge}) and solving 
the recurrence relation, we obtain the desired expression.
\end{proof}

\begin{cor}
By taking the large $n$ limit, we obtain 
\begin{align*}
\lim_{n\rightarrow\infty}\genfrac{}{}{}{}{\mathcal{N}_{n}}{n}&=\genfrac{}{}{}{}{-18+8\sqrt{5}}{-25+11\sqrt{5}}, \\[10pt]
&=0.2763932022\ldots.
\end{align*}
\end{cor}

We give another relation between $\mathcal{N}_{n}(m,r,s)$ and the formal power series $\chi(m,r,s)$.

\begin{defn}
We define a formal power series $\Delta_{n}(r):=\Delta_{n}(m,r,s)$ by 
\begin{align*}
\Delta_{n}(r):=\prod_{j=0}^{n-1}\chi(m,rs^{j},s).
\end{align*}
\end{defn}

The following proposition gives the finite analogue of 
Theorem \ref{thrm:Ninf1}.
\begin{prop}
\label{prop:Ninf2}
We have
\begin{align}
\label{eqn:Nfinchi}
\mathcal{N}_{n}(m,r,s)=-\partial_{r}\left(\log(\Delta_{n}(r))\right)
+\genfrac{}{}{}{}{1}{G_{n}(m,r,s)\Delta_{n}(r)}\partial_{r}B_{n}(r),
\end{align}
where $B_{n}(r)$ is defined in Definition \ref{defn:recB}.
\end{prop}
\begin{proof}
We prove the proposition by induction on $n$.
When $n=1$, we have an explicit expression of $G_{1}^{(s)}(m,r,s)$
in terms of $\chi(m,r,s)$ by Theorem \ref{thrm:Gsfinchi}.
We abbreviate $\partial_{r}F(r)$ as $F'(r)$.
By taking the derivative of $G_{1}(r):=G_{1}^{(s)}(m,r,s)$, we obtain 
\begin{align*}
G'_{1}(r)&=-\genfrac{}{}{}{}{\chi'(r)}{\chi(r)^{2}}+\left(rs-rs\chi(rs)\right)', \\
&=-G_{1}(r)	\genfrac{}{}{}{}{\chi'(r)}{\chi(r)}+\genfrac{}{}{}{}{\chi'(r)}{\chi(r)}(rs-rs\chi(rs))
+(rs-rs\chi(rs))', \\
&=-G_{1}(r)\genfrac{}{}{}{}{\chi'(r)}{\chi(r)}+\genfrac{}{}{}{}{1}{\chi(r)}(rs\chi(r)-rs\chi(r)\chi(rs))', \\
&=-G_{1}(r)\genfrac{}{}{}{}{\chi'(r)}{\chi(r)}+\genfrac{}{}{}{}{1}{\chi(r)}\partial_{r}B_1(r),
\end{align*}
where we have used $1/\chi(rs)=G_1(r)-(rs-rs\chi(rs))$ and the definition of $B_{1}(r)$.
By dividing the both side by $G_{1}(r)$, we obtain the desired expression for $\mathcal{N}_{1}(m,r,s)$.

For $2\le n$, we show that the expression (\ref{eqn:Nfinchi}) satisfies a recurrence relation discussed below.
Since $G_{n}^{(s)}(m,r,s)$ satisfies the recurrence relation in Theorem \ref{thrm:Gsrrflip},
we have 
\begin{align*}
G_{n}'(r)&=sf'(rs)G_{n-1}(rs)+sf(rs)G'_{n-1}(rs)+sG_{n-2}(rs^2)+rs^3G'_{n-2}(rs^2), \\
&=sf'(rs)\mathcal{Z}_{n}(1;r)G_{n}(r)+sf(rs)\mathcal{N}_{n-1}(rs)G_{n-1}(rs) \\
&\quad+s\mathcal{Z}_{n-1}(1;rs)G_{n-1}(rs)+rs^{3}\mathcal{N}_{n-2}(rs^2)G_{n-2}(rs^2).
\end{align*}
By substituting Proposition \ref{prop:Zni} and Eq. (\ref{eqn:Nfinchi}) for $n-1$ and $n-2$ into the right hand 
side of the above recurrence relation, one can verify that $G_{n}'(r)$ can be expressed 
as Eq. (\ref{eqn:Nfinchi}).
This completes the proof.
\end{proof}

By taking the large $n$ limit in Proposition \ref{prop:Ninf2} and from 
Theorems \ref{thrm:Ninf1} and \ref{thrm:Ginfchi}, 
we obtain the following corollary.
\begin{cor}
\label{cor:Ninf2}
We have 
\begin{align*}
\genfrac{}{}{}{}{\mathcal{N}_{\infty}(m,r,s)}{G_{\infty}^{(s)}(m,r,s)}
=-\partial_{r}\left( \prod_{0\le j}\chi(m,rs^{j},s) \right).
\end{align*}
\end{cor}

More generally, we define the following probabilities.
\begin{defn}
Let $1\le p,q$ be positive integers.
We define 
\begin{align*}
\mu_{n}^{(p,q)}(m,r,s):=
\genfrac{}{}{}{}{\partial_{r}^{p}\partial_{s}^{q}G^{(s)}_{n}(m,r,s)}{G^{(s)}_{n}(m,r,s)}.
\end{align*}
We call $\mu_{n}^{(p,q)}(m,r,s)$ the $(p,q)$-th moment.
\end{defn}

In the case of $(p,q)=(1,0)$, we have $\mu_{\infty}^{(1,0)}(m,r,s)=\mathcal{N}_{\infty}(m,r,s)$. 

In Proposition \ref{prop:Ninf2}, we obtain $\mathcal{N}_{n}(m,r,s)$ in terms of 
$\chi(m,r,s)$ and the formal power series $B_{n}(r)$.
This suggests that the $(p,q)$-th moment has a similar form.
As we will see below, the moments have a simple expression in terms of 
$\chi(m,r,s)$ and $B_{n}(r)$.

We denote $\partial^{(p,q)}:=\partial_{r}^{p}\partial_{s}^{q}$.
We rewrite $\partial^{(p,q)}G^{(s)}_{n}(n,m,r)$ as 
\begin{align}
\label{eqn:GprimePQ}
\partial^{(p,q)}G^{(s)}_{n}(m,r,s)=G^{(s)}_{n}(m,r,s)P_{n}^{(p,q)}(m,r,s)+Q_{n}^{(p,q)}(m,r,s),
\end{align}
with initial conditions 
\begin{align*}
&P_{n}^{(0,0)}(m,r,s)=1, \qquad P_{n}^{(1,0)}(m,r,s)=\partial^{(1,0)}\overline{\Delta}_{n}(r), 
\qquad P_{n}^{(0,1)}(m,r,s)=\partial^{(0,1)}\overline{\Delta}_{n}(r,s),\\
&Q_{n}^{(0,0)}(m,r,s)=0, \qquad Q_{n}^{(1,0)}(m,r,s)=\genfrac{}{}{}{}{\partial^{(1,0)}B_{n}(r)}{\Delta_{n}(r)},
\qquad Q_{n}^{(0,1)}(m,r,s)=\genfrac{}{}{}{}{\partial^{(0,1)}B_{n}(r)}{\Delta_{n}(r)},
\end{align*}
where 
\begin{align*}
\overline{\Delta}_{n}(r,s):=-\log(\Delta_{n}(r,s)).
\end{align*}

Then, the formal power series $P_{n}^{(p,q)}(m,r,s)$ and $Q_{n}^{(p,q)}(n,m,r)$
satisfy the following recurrence relation.
\begin{prop}
\label{prop:GprimePQrec}
The power series $P_{n}^{(p,q)}(r,s):=P_{n}^{(p,q)}(m,r,s)$ and 
$Q_{n}^{(p,q)}(r,s):=Q_{n}^{(p,q)}(m,r,s)$
satisfy 
\begin{align*}
P_{n}^{(p,q)}(r,s)&=\partial^{(1,0)}P_{n}^{(p-1,q)}(r,s)+P_{n}^{(1,0)}(r,s)P_{n}^{(p-1,q)}(r,s), \\
P_{n}^{(p,q)}(r,s)&=\partial^{(0,1)}P_{n}^{(p,q-1)}(r,s)+P_{n}^{(0,1)}(r,s)P_{n}^{(p,q-1)}(r,s),
\end{align*}
and 
\begin{align*}
Q_{n}^{(p,q)}(r,s)&=\partial^{(1,0)}Q_{n}^{(p-1,q)}(r,s)+Q_{n}^{(1,0)}(r,s)P_{n}^{(p-1,q)}(r,s), \\
Q_{n}^{(p,q)}(r,s)&=\partial^{(0,1)}Q_{n}^{(p,q-1)}(r,s)+Q_{n}^{(0,1)}(r,s)P_{n}^{(p,q-1)}(r,s).
\end{align*}
\end{prop}
\begin{proof}
By taking a derivative of Eq. (\ref{eqn:GprimePQ}) with respect to $r$ or $s$, 
we obtain the recurrence relations.
\end{proof}

One can solve the recurrence relations in Proposition \ref{prop:GprimePQrec}
in terms of $\chi(m,r,s)$ and $B_{n}(r)$ where $B_{n}(r)$ is defined in Definition \ref{defn:recB}.
\begin{prop}
\label{prop:PQinDelta}
The power series $P_{n}^{(p,q)}(r,s)$ and $Q_{n}^{(p,q)}(r,s)$ are expressed as 
\begin{align}
\label{eqn:PQinDelta}
\begin{split}
P_{n}^{(p,q)}(r,s)&=\Delta_{n}(r,s)\cdot\partial^{(p,q)}\genfrac{(}{)}{}{}{1}{\Delta_{n}(r,s)}, \\
Q_{n}^{(p,q)}(r,s)&
=\partial^{(p,q)}\genfrac{(}{)}{}{}{B_{n}(r,s)}{\Delta_{n}(r,s)}
-B_{n}(r,s)\partial^{(p,q)}\genfrac{(}{)}{}{}{1}{\Delta_{n}(r,s)}.
\end{split}
\end{align}
\end{prop}
\begin{proof}
From Proposition \ref{prop:Ninf2}, we have expressions (\ref{eqn:PQinDelta}) for $(p,q)=(1,0)$.
By a similar calculation, we have expressions (\ref{eqn:PQinDelta})  and $(p,q)=(0,1)$.
For $(p,q)$ with $2\le p$ or $2\le q$, 
we have 
\begin{align*}
P_{n}^{(p,q)}(r,s)&=\partial^{(1,0)}P_{n}^{(p-1,q)}(r,s)+P_{n}^{(1,0)}(r,s)P_{n}^{(p-1,q)}(r,s), \\
&=\partial_{(1,0)}\Delta_{n}(r,s)\cdot\partial^{(p-1,q)}\genfrac{(}{)}{}{}{1}{\Delta_{n}}
+\Delta_{n}\cdot\partial^{(p,q)}\genfrac{(}{)}{}{}{1}{\Delta_{n}} \\
&\quad+\Delta_{n}^2\partial^{(1,0)}\genfrac{(}{)}{}{}{1}{\Delta_{n}}\partial^{(p-1,q)}\genfrac{(}{)}{}{}{1}{\Delta_{n}}, \\
&=\Delta_{n}\cdot\partial^{(p,q)}\left(\genfrac{}{}{}{}{1}{\Delta_{n}}\right),
\end{align*}
and 
\begin{align*}
Q_{n}^{(p,q)}(r,s)&=\partial^{(1,0)}Q_{n}^{(p-1,q)}(r,s)+Q_{n}^{(1,0)}(r,s)P_{n}^{(p-1,q)}(r), \\
&=\partial^{(p,q)}\genfrac{(}{)}{}{}{B_{n}(r,s)}{\Delta_{n}}
-\partial^{(1,0)}B_{n}(r)\cdot\partial^{(p-1,q)}\genfrac{(}{)}{}{}{1}{\Delta_{n}} \\
&\quad-B_{n}(r,s)\partial^{(p,q)}\genfrac{(}{)}{}{}{1}{\Delta_{n}}
+\genfrac{}{}{}{}{\partial^{(1,0)}B_{n}(r,s)}{\Delta_{n}}\partial^{(p-1,q)}\genfrac{(}{)}{}{}{1}{\Delta_{n}}\cdot\Delta_{n}, \\
&=\partial^{(p,q)}\genfrac{(}{)}{}{}{B_{n}(r,s)}{\Delta_{n}}-B_{n}(r,s)\partial^{(p,q)}\genfrac{(}{)}{}{}{1}{\Delta_{n}}.
\end{align*}
We have a similar calculation with respect to $q$.
This completes the proof.
\end{proof}

From Eq. (\ref{eqn:GprimePQ}) and Propositions \ref{prop:GprimePQrec} and \ref{prop:PQinDelta},
we have the following corollary.  
\begin{cor}
\label{cor:moments}
The $(p,q)$-th moment $\mu_{n}^{(p,q)}(m,r,s)$ is expressed 
in terms of $P_{n}^{(p,q)}(m,r,s)$ and $Q_{n}^{(p,q)}(m,r,s)$:
\begin{align*}
\mu_{n}^{(p,q)}(m,r,s)=P_{n}^{(p,q)}(m,r,s)+\genfrac{}{}{}{}{Q_{n}^{(p,q)}(m,r,s)}{G^{(s)}_{n}(m,r,s)}.
\end{align*}
\end{cor}

Note that Corollary \ref{cor:moments} is a generalization of Proposition \ref{prop:Ninf2}.
By definition, the moment $\mu_{n}^{(p,q)}(m,r,s)$ is a rational function. 
The formal power series $\Delta_{n}(r)$ and $B_{n}(r)$ are defined by use of 
the formal power series $\chi(m,r,s)$. 
Though the expression in Corollary \ref{cor:moments} is essentially expressed in terms 
of $\chi(m,r,s)$, which is an infinite formal power series, the sum 
$G_{n}^{(s)}(m,r,s)P_{n}^{(p,q)}(m,r,s)+Q_{n}^{(p,q)}(m,r,s)$ reduces the infinite 
formal power series to a polynomial $\partial^{(p,q)}G_{n}^{(s)}(m,r,s)$.

\section{Applications}
\label{sec:app}
In this section, we consider the two applications which appear 
in the study of the dimer models in one-dimension.

We first introduce a one-to-many correspondence between
a Motzkin path and a set of Dyck paths preserving the weight.
This correspondence gives the reason why we have two expressions 
for the generating functions $G_{n}^{(s)}(m,r,s)$ or $G_{n}^{(c)}(m,r,s)$ 
in terms of Dyck and Motzkin paths.
Note that the correspondence is not one-to-one, however, the weight given 
to a Motzkin path corresponds to the sum of the weights given 
to the set of Dyck paths.

The second is the application of the generating functions $G_{n}^{(s)}(m,r,s=1)$ and 
$G_{n}^{(c)}(m,r,s=1)$ to the generating functions for independent sets of graphs.
The graphs are given by the Cartesian product of a segment (resp. a circle) and a complete graph
for $G_{n}^{(s)}(m,r,s=1)$ (resp. $G_{n}^{(c)}(m,r,s)$).
Then, we show that the two generating functions coincide with each other. 

\subsection{A correspondence between a Motzkin path and a set of Dyck paths}
Recall that we have two expressions of the formal power series $\chi(m,r,s)$
in terms of Dyck paths (Proposition \ref{prop:chiinDyck}) and Motzkin paths (Proposition \ref{prop:chiinMotzkin}).
Since we have two different expressions for $\chi(m,r,s)$, one can expect 
a correspondence between Dyck paths of size $n$ and Motzkin paths of size $n$.
Since the total number of Dyck paths of size $n$ and the one of Motzkin paths of size $n$ 
are different, we do not have a bijection between them.
However, we can construct a one-to-many correspondence between a Motzkin path
and a set of Dyck paths, which preserves the total weight. 

The weight for a Motzkin path $\lambda$ is given by $F(\lambda)$ as in Eq. (\ref{eqn:wtMotzkin}).
If we denote the number of horizontal steps at height greater than zero by $N$, 
the weight $F(\lambda)$ has $2^{N}$ terms.
By defining $2^{N}$ Dyck paths for a given Motzkin path $\lambda$, 
we will construct a one to $2^{N}$ map between the Motzkin path $\lambda$ and a set of Dyck paths.
Further, we show that an each term of the weight given to a Motzkin path $\lambda$ 
corresponds to the weight given to a Dyck path as in Eq. (\ref{eqn:ChiinDyck}), that is, the correspondence is 
weight preserving.

By definition, a Motzkin path of size $n$ has $n$ steps and a Dyck path of size $n$
has $2n$ steps.
The basic idea for the correspondence between a Motzkin path and a set of Dyck paths 
can be summarized as follow.
First, we enlarge a Motzkin path $\lambda$ by replacing $U\rightarrow UU$, $H\rightarrow HH$
and $D\rightarrow DD$.
By this operation, we obtain a Motzkin path $\lambda'$ of the size $2n$ from 
the Motzkin path $\lambda$.
Secondly, we replace $HH$ by $UD$ if the horizontal steps are at height zero, 
and $HH$ by $DU$ if the horizontal steps are at height greater than zero.
Then, we  delete several boxes from $\lambda'$ to preserve the weight with respect 
to $s$. 
Let $\mu$ be a Dyck path obtained from the Motzkin paths $\lambda'$.
Thirdly, since each horizontal step in $\lambda$ has a weight of the form $(m+s^{i})$
with some integer $i$, we add $i$ boxes to the newly obtained path $\mu$.
We also show that the addition of $i$ boxes to the Dyck path $\mu$ decreases 
the number of peaks in $\mu$ by one, and the addition of boxes preserves 
the weight with respect to both $m$ and $s$.	
	
Below, we will construct a one-to-many correspondence between a Motzkin path of size $n$
and a set of Dyck paths of size $n$ preserving the weights.
We first consider the case where a Motzkin path $\lambda$ of size $n$ is 
prime.

\paragraph{(Step 1)}
As mentioned above, we enlarge the path $\lambda$ by substituting 
$X\rightarrow X^2$, $X\in\{U,H,D\}$.
We denote by $\lambda'$ the newly obtained Motzkin path of size $2n$.
We construct a Dyck path of size $2n$ from $\lambda'$ by
replacing $HH$ by $UD$ if the two horizontal steps are at the ground 
level, and by $DU$ otherwise.
We denote by $\mu'$ the Dyck path obtained from $\lambda'$.

\paragraph{(Step 2)}
Recall that the exponent of $s$ in the weight reflects the area 
of a skew shape $\lambda_{0}/\lambda$.
Suppose that the maximum height of $\lambda$ is $h$.
Let $a_{i}$, $1\le i \le h$, be the area of the region 
which is above height $i-1$ and below height $i$ in the skew shape 
$\lambda_{0}/\lambda$.
We denote this $h$ positive integers by $\mathbf{a}:=(a_{1},\ldots,a_{h})$. 
By definition, we have $\sum_{i=1}^{h}a_{i}=|Y(\lambda_{0}/\lambda)|$.
We put $a_{i}$ dots on the boxes at height $2i-1$ in the skew shape 
$\mu_{0}/\mu$.
By construction, note that the number of boxes at height $2i-1$ in $\mu_{0}/\mu'$ 
is equal to the value $a_{i}$.
Since we have no dots in the boxes at even height, we move the dots in the southeast 
direction, or equivalently in the $(1,-1)$-direction, until they contact another 
dotted box or an up step in $\mu_{0}$.
We denote by $\mu$ the Dyck paths obtained from $\mu'$.

\paragraph{(Step 3)}
Recall that we replace two horizontal steps $HH$ at height greater than zero 
by $DU$.
In Step 2, we move the dotted boxes in the $(1,-1)$-direction.
In $\mu'$, we have a valley corresponding to the two successive steps $DU$ originated 
from $HH$.
As in the case of dotted boxes, we also move the valley in $\mu'$ in the south east 
direction until it contact the path $\mu$, and distinguish it from other valleys 
in $\mu$.	
The horizontal step in $\lambda$ has a factor $(m+s^{i})$ with some $i$ 
(see Eq. (\ref{eqn:wtMotzkin})).
We have two choices: one is to take the factor $m$ and the other is to take factor 
$s^{i}$. 
When we have $N$ horizontal steps in $\lambda$, we have $2^{N}$ ways to get a factor.
If we choose the factor $m$, we do nothing on $\mu$.
We consider the case where we take the factor $s^{i}$.
We put $i$ boxes in the $(-1,1)$-	direction from its valley.
If the south west edge of an added box does not contact to the northeast edge 
of a box in $\mu$, we move the added box in the $(-1,-1)$-direction until 
it contacts the box in $\mu$.
We perform this operation on all the valleys where we choose a factor $s^{i}$.
Then, we obtain $2^{N}$ Dyck paths of size $n$ from $\lambda$.

When a Motzkin path $\lambda$ is not prime, the path $\lambda$ is written 
as a concatenation of prime paths.
By the above algorithm, we have Dyck paths for each prime Motzkin path.
Therefore, Dyck paths corresponding to $\lambda$ are obtained by concatenating 
the Dyck paths corresponding to prime Motzkin paths.

\begin{prop}
The correspondence between a Motzkin path and a set of Dyck paths 
is weight preserving.
\end{prop}
\begin{proof}
By construction of the correspondence, it is obvious that it preserves the 
weight regarding $s$. 
Therefore, we show that the correspondence preserves the weight regarding $m$.
Let $b_{i}$, $1\le i\le 2n$, be the number of dotted boxes in the $i$-th 
column from right in (Step 2).
Let $M$ be the total number of $b_{i}$ such that 
$b_{i}\le b_{i-1}$ with $b_0:=0$. Since we move the dotted boxes in $(1,-1)$-direction, 
the value $M$ is equal to the exponent of $m$.
Further, by construction, $M$ is equal to the sum of the numbers of up steps and 
horizontal steps at the ground level in the Motzkin path.
Recall the operation in (Step 3) to add boxes to a Dyck path $\mu'$.  
By definition of $F(\lambda)$ in Eq. (\ref{eqn:wtMotzkin}), we add $i$ boxes 
in $\mu'$. 
The number of peaks in $\mu'$ is decreased by one by the addition of $i$ boxes.
This is because $i$ corresponds to the position of a horizontal step minus one.
From these observations, the correspondence is weight preserving regarding 
both $s$ and $m$.
\end{proof}

\begin{example}
We consider the Motzkin path $\lambda=UUUDHUDDD$.
The weight of this path is $F(\lambda)=m^{4}s^{16}(m+s^{3})$.
The correspondence gives two Dyck paths of weight $m^{5}s^{16}$ 
and $m^{4}s^{19}$.

By Step 1, we have $\mu'=U^{6}D^3U^3D^{6}$.
Then, we obtain an integer sequence $\mathrm{a}=(8,6,2)$.
We transform the paths $\mu'$ with dotted boxes as follows:
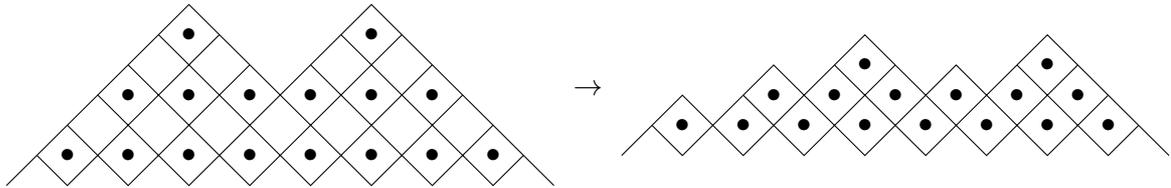
\begin{align*}
\tikzpic{-0.5}{[scale=0.4]
\draw(0,0)--(6,6)--(9,3)--(12,6)--(18,0);
\draw(1,1)--(2,0)--(7,5)(2,2)--(4,0)--(8,4)(3,3)--(6,0)--(9,3)
(4,4)--(8,0)--(13,5)(5,5)--(10,0)--(14,4)
(9,3)--(12,0)--(15,3)(10,4)--(14,0)--(16,2)(11,5)--(16,0)--(17,1);
\draw(2,1)node{$\bullet$}(4,1)node{$\bullet$}(6,1)node{$\bullet$}(8,1)node{$\bullet$}(10,1)node{$\bullet$}
(12,1)node{$\bullet$}(14,1)node{$\bullet$}(16,1)node{$\bullet$};
\draw(4,3)node{$\bullet$}(6,3)node{$\bullet$}(8,3)node{$\bullet$}(10,3)node{$\bullet$}(12,3)node{$\bullet$}
(14,3)node{$\bullet$};
\draw(6,5)node{$\bullet$}(12,5)node{$\bullet$};
}
\rightarrow
\tikzpic{-0.5}{[scale=0.4]
\draw(0,0)--(2,2)--(3,1)--(5,3)--(6,2)--(8,4)--(10,2)--(11,3)--(12,2)--(14,4)--(18,0);
\draw(1,1)--(2,0)--(3,1)--(4,0)--(6,2)(4,2)--(6,0)--(9,3)(6,2)--(8,0)--(10,2)(7,3)--(10,0)--(12,2)
(10,2)--(12,0)--(15,3)(12,2)--(14,0)--(16,2)(13,3)--(16,0)--(17,1);
\draw(2,1)node{$\bullet$}(4,1)node{$\bullet$}(6,1)node{$\bullet$}(8,1)node{$\bullet$}(10,1)node{$\bullet$}
(12,1)node{$\bullet$}(14,1)node{$\bullet$}(16,1)node{$\bullet$};
\draw(5,2)node{$\bullet$}(7,2)node{$\bullet$}(9,2)node{$\bullet$}(11,2)node{$\bullet$}(13,2)node{$\bullet$}
(15,2)node{$\bullet$};
\draw(8,3)node{$\bullet$}(14,3)node{$\bullet$};	
}
\end{align*}
The Dyck path in the right hand side of the above figure is $\mu_{1}=U^2DU^2DU^2D^2UDU^2D^4$.
We will construct another Dyck path from $\mu_{1}$. 
The factor of the horizontal edge in $\lambda$ is $(m+s^{3})$. 
By Step 3, we put three boxes at the valley corresponding to the horizontal
step.
Then, we have the following operation:
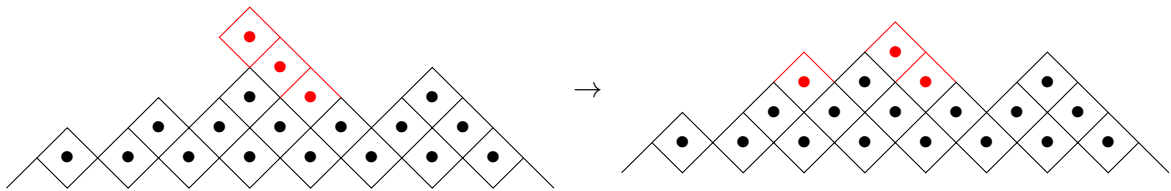
\begin{align*}
\tikzpic{-0.5}{[scale=0.4]
\draw(0,0)--(2,2)--(3,1)--(5,3)--(6,2)--(8,4)--(10,2)--(11,3)--(12,2)--(14,4)--(18,0);
\draw(1,1)--(2,0)--(3,1)--(4,0)--(6,2)(4,2)--(6,0)--(9,3)(6,2)--(8,0)--(10,2)(7,3)--(10,0)--(12,2)
(10,2)--(12,0)--(15,3)(12,2)--(14,0)--(16,2)(13,3)--(16,0)--(17,1);
\draw[red](8,4)--(7,5)--(8,6)--(11,3)(8,4)--(9,5)(9,3)--(10,4);
\draw(2,1)node{$\bullet$}(4,1)node{$\bullet$}(6,1)node{$\bullet$}(8,1)node{$\bullet$}(10,1)node{$\bullet$}
(12,1)node{$\bullet$}(14,1)node{$\bullet$}(16,1)node{$\bullet$};
\draw(5,2)node{$\bullet$}(7,2)node{$\bullet$}(9,2)node{$\bullet$}(11,2)node{$\bullet$}(13,2)node{$\bullet$}
(15,2)node{$\bullet$};
\draw(8,3)node{$\bullet$}(14,3)node{$\bullet$};	
\draw[red](10,3)node{$\bullet$}(9,4)node{$\bullet$}(8,5)node{$\bullet$};
}\rightarrow
\tikzpic{-0.5}{[scale=0.4]
\draw[red](8,4)--(9,5)--(11,3)(9,3)--(10,4);
\draw[red](5,3)--(6,4)--(7,3);
\draw(0,0)--(2,2)--(3,1)--(5,3)--(6,2)--(8,4)--(10,2)--(11,3)--(12,2)--(14,4)--(18,0);
\draw(1,1)--(2,0)--(3,1)--(4,0)--(6,2)(4,2)--(6,0)--(9,3)(6,2)--(8,0)--(10,2)(7,3)--(10,0)--(12,2)
(10,2)--(12,0)--(15,3)(12,2)--(14,0)--(16,2)(13,3)--(16,0)--(17,1);
\draw(2,1)node{$\bullet$}(4,1)node{$\bullet$}(6,1)node{$\bullet$}(8,1)node{$\bullet$}(10,1)node{$\bullet$}
(12,1)node{$\bullet$}(14,1)node{$\bullet$}(16,1)node{$\bullet$};
\draw(5,2)node{$\bullet$}(7,2)node{$\bullet$}(9,2)node{$\bullet$}(11,2)node{$\bullet$}(13,2)node{$\bullet$}
(15,2)node{$\bullet$};
\draw(8,3)node{$\bullet$}(14,3)node{$\bullet$};	
\draw[red](10,3)node{$\bullet$}(9,4)node{$\bullet$}(6,3)node{$\bullet$};
}
\end{align*}
We have a Dyck path $\mu_{2}=U^{2}DU^{3}DU^{2}D^{3}U^{2}D^{4}$.
As a summary, we have two Dyck paths $\mu_1$ and $\mu_{2}$ from 
the Motzkin path $\lambda$ preserving the weight.

Two Motzkin paths $UHHD$ and $UUDD$ of length four give the following sets 
of Dyck paths of length four:
\begin{align*}
UHHD &\rightarrow \{UUDUDUDD, UUUDDUDD, UUUDUDDD, UUUUDDDD \}, \\
UUDD &\rightarrow \{ UUDUUDDD \}.
\end{align*}
Note that the highest Motzkin paths $U^{n}D^{n}$ does not give the highest 
Dyck path $U^{2n}D^{2n}$.
\end{example}

\subsection{Independent sets in graph theory}

A {\it graph} $G$ is an pair $G=(V,E)$ where $V$ is a set of vertices 
and $E$ is a set of paired vertices.
An element of $E$ is called {\it edges}.
The {\it order} of a graph is its number of vertices $|V|$, and 
the {\it size} of a graph is its number of edges $|E|$.
Let $e=(x,y)\in E$ with $x,y\in V$. 
We ignore the order of $x$ and $y$ in $e$, namely, we identify 
$(x,y)$ with $(y,x)$.

In this section, we consider only three types of graphs: 
a segment of order $n$, $S_{n}$, a circle of order $n$, $C_{n}$, 
and a complete graph of order $n$, $K_{n}$.
Let $V_{n}$ be the set of $n$ vertices, that is, $V_{n}:=\{v_{i}:1\le i\le n\}$.
We define the set of vertices for the three types of graphs is $V_{n}$.
Then, the sets of edges for the graphs are given by 
\begin{enumerate}
\item $E=\{(v_i,v_{i+1}): 1\le i\le n-1, v_{i},v_{i+1}\in V_{n} \}$ for $S_{n}$,
\item $E=\{(v_i,v_{j}): 1\le i<j\le n, i-j\equiv1\pmod{n}, v_{i},v_{j}\in V_{n} \}$ for $C_{n}$,
\item $E=\{(v_i,v_j):1\le i<j\le n, v_{i},v_{j}\in V_{n}\}$ for $K_{n}$.
\end{enumerate} 

Given a graph $G=(V,E)$, an {\it independent set} $\mathtt{Ind}(G)$ is a set of 
vertices in $G$ such that they are not adjacent. 
In other words, every two vertices in $\mathtt{Ind}(G)$ are not connected 
by an edge.

In this subsection, we connect the generating functions $G_{n}^{(x)}(m,r,s=1)$, $x\in\{s,c\}$, 
with the generating function of independent sets in a graph. 
Let $G_{i}(n,m)$, $1\le i\le2$, be the graphs given by the Cartesian product 
of $S_{n}$ and $K_{m}$, which is  $G_{1}(n,m)=S_{n}\times K_{m}$, and the Cartesian 
product $C_{n}$ and $K_{m}$, which is $G_{2}(n,m)=C_{n}\times K_{m}$ respectively.

We consider the following generating functions for the graphs $G_{i}$, $1\le i\le 2$.
We denote by  $|\mathtt{Ind}(G)|$ be the number of vertices in $\mathtt{Ind}(G)$.
\begin{defn}
We define the generating functions by 
\begin{align*}
Y_{i}(n,m;r):=\sum_{\mathtt{Ind}(G_{i}(n,m))}r^{|\mathtt{Ind}(G_{i}(n,m))|}.
\end{align*}
\end{defn}
Note that $Y_{i}(n,m;r)$ is a polynomial with respect to $r$.

\begin{example}
We consider the graph $G_{2}(3,2)$.
We name the three vertices of two $C_{3}$ as $1,2,3$ and $1',2',3'$.
Since $G_{2}$ is a Cartesian product of $C_{3}$ with $K_{2}$, we have 
edges $(1,1')$, $(2,2')$  and $(3,3')$.
As a result $G_{2}(n,m)$ with $(n,m)=(3,2)$ has an expression 
$G_{2}=(V,E)$ such that 
\begin{align*}
V&=\{1,2,3,1',2',3',\}, \\
E&=\{(1,2),(1,3),(2,3),(1',2'),(1',3'),(2',3'),(1,1'),(2,2'),(3,3')\}.
\end{align*}
Then, the independent sets of order $2$ are given by
\begin{align*}
\{1,2'\}, \quad \{1,3'\}, \quad \{2,1'\}, \quad \{2,3'\},\quad \{3,1'\},\quad \{3,2'\}.
\end{align*}
Similarly, we have one and six independent sets of order $0$ and $1$ respectively.
Thus, the generating function is given by
\begin{align*}
Y_{2}(3,2;r)=1+6r+6r^{2}.
\end{align*}
\end{example}

The next proposition connects the two generating functions $Y_{1}(n,m;r)$ (resp. $Y_{2}(n,m;r)$) 
with $G_{n}^{(s)}(m,r,s=1)$ (resp. $G_{n}^{(c)}(m,r,s=1)$).
\begin{prop}
We have 
\begin{align}
\label{eqn:YinGns}
Y_{1}(n,m;r)=G_{n}^{(s)}(m,r,s=1), \\
\label{eqn:YinGnc}
Y_{2}(n,m;r)=G_{n}^{(c)}(m,r,s=1).
\end{align}
\end{prop}
\begin{proof}
We prove the proposition by constructing a bijection between 
an independent set of $G_{i}$ and a dimer configuration.
We consider only the case of $G_{1}(n,m)$ since one can similarly 
prove in the case of $G_{2}(n,m)$.

Let $u_{i}$, $1\le i\le n$, be the vertices of $S_{n}$ and 
$v_{j}$, $1\le j\le m$ be the vertices of $K_{m}$.
Since the graph $G_{1}(n,m)$ is a Cartesian product, a vertex 
in $G_{1}(n,m)$ is written as $(u_{i},v_{j})$ with $i\in[1,n]$ and 
$j\in[1,m]$.
It is obvious that we have one independent set of order $0$.
We will consider an independent set of order $k\ge1$.

Suppose that $(u_{i_1},v_{j_1})\in\mathtt{Ind}(G_{1})$.
Then, a vertex $(u_{i_1},v_{j'})$ with $j'\neq j_1$ is not 
in $\mathtt{Ind}(G_{1})$ since $v_{j_1}$ and $v_{j'}$ are connected 
by an edge by the definition of a complete graph.
This condition is equivalent to the following condition for a dimer configuration: 
one can not put two dimers at the same place. 

Since $u_{i_1}$ is a vertex in $S_{n}$, a vertex $u_{i'}$, $i'=i_1\pm1$, 
is connected with $u_{i_1}$ by an edge.
Therefore, $(u_{i'},v_{j_1})$ with $i'=i_1\pm1$ is not in $\mathtt{Ind}(G_{1})$.
However, if $j'\neq j_{1}$, a vertex $(u_{i'},v_{j'})$ with $i'=i_1\pm1$ can 
be in $\mathtt{Ind}(G_{1})$.
This condition is translated in terms of dimers as follows:
if we put two dimers which are next to each other, the colors of dimes 
are different.  
A color of a dimer corresponds to a choice of $j$ for $v_{j}$.

These two conditions above are nothing but the defining conditions for dimer configurations 
on a segment.
Thus, we have a one-to-one correspondence between an element of an independent set of size $k$
and a dimer configuration with $k$ dimers.
Note that the exponent of $r$ in $G_{n}^{(s)}(m,r,s=1)$ counts the number of dimers, which is equal to 
the order of an independent set.
Therefore, we have the identity (\ref{eqn:YinGns}).
One can show Eq. (\ref{eqn:YinGnc}) in the case of the dimer model on a circle in a similar 
manner, which completes the proof.
\end{proof}

\bibliographystyle{amsplainhyper} 
\bibliography{biblio}

\providecommand{\bysame}{\leavevmode\hbox to3em{\hrulefill}\thinspace}
\begin{thebibliography}{10}

\bibitem{Aig98}
M.~Aigner, \emph{{M}otzkin numbers}, Europ. J. Comb. \textbf{19} (1998),
  663--675.

\bibitem{AmdCheMolSag13}
T.~Amdeberhan, X.~Chen, V.~H. Moll, and B.~E. Sagan, \emph{{G}eneralized
  {F}ibonacci polynomials and {F}ibonomial coefficients}, Annals of
  Combinatorics \textbf{18} (2013), no.~4, 541--562,
  \href{http://arxiv.org/abs/1306.6511}{\path{arXiv:1306.6511}}.

\bibitem{And81}
G.~E. Andrews, \emph{{T}he hard-hexagon model and {R}ogers--{R}amanujan type
  identities}, Proc. Nat. Acad. Sci. U.S.A. \textbf{78} (1981), 5290--5292.

\bibitem{And84a}
\bysame, \emph{{M}ultiple {S}eries {R}ogers--{R}amanujan {T}ype {I}dentities},
  Pac. J. Math. \textbf{114} (1984), no.~2, 267--283.

\bibitem{And84}
\bysame, \emph{{T}he {T}heory of {P}artitions}, Cambridge University Press,
  1984.

\bibitem{And86}
\bysame, \emph{$q$-{S}eries: {T}heir {D}evelopment and {A}pplication in
  {A}nalysis, {N}umber {T}heory, {C}ombinatorics, {P}hysics, and {C}omputer
  {A}lgebra}, CBMS Reginal Conference Series in Mathematics, vol.~66, American
  Mathematical Society, 1986.

\bibitem{AndBaxFor84}
G.~E. Andrews, R.~J. Baxter, and P.~J. Forrester, \emph{{E}ight-vertex {SOS}
  model and generalized {R}ogers--{R}amanujan--type identities}, J. Statist.
  Phys. \textbf{35} (1984), 193--266.

\bibitem{BacBerFerCunOinWes14}
A.~Bacher, A.~Bernini, L.~Ferrari, B.~Gunby, R.~Pinzani, and J.~West,
  \emph{{T}he {D}yck pattern poset}, Discrete Math. \textbf{321} (2014),
  12--–23, \href{http://arxiv.org/abs/1303.3785}{\path{arXiv:1303.3785}}.

\bibitem{BarDLunFezzPin97}
E.~Barcucci, A.~Del Lungo, S.~Fezzi, and R.~Pinzani, \emph{{N}ondecreasing
  {D}yck paths and $q$-{F}ibonacci numbers}, Disc. Math. \textbf{170} (1997),
  211--217.

\bibitem{Bax82}
R.~J. Baxter, \emph{{E}xactly {S}olved {M}odels in {S}tatistical {M}echanics},
  Academic Press, London, 1982.

\bibitem{BirGilWei14}
D.~Birmajer, J.~B. Gil, and M.~D. Weiner, \emph{{C}onvolutions of tribonacci,
  {F}uss-{C}atalan, and {M}otzkin sequences}, Fibonacci Quart. \textbf{52}
  (2014), 54--60,
  \href{http://arxiv.org/abs/1412.0075}{\path{arXiv:1412.0075}}.

\bibitem{BlaPet14}
S.~A. Blanco and T.~K. Petersen, \emph{{C}ounting {D}yck {P}aths by {A}rea and
  {R}ank}, Ann. Comb. \textbf{18} (2014), no.~2, 171--197,
  \href{http://arxiv.org/abs/1206.0803}{\path{arXiv:1206.0803}}.

\bibitem{BonShaSim93}
J.~Bonin, L.~Shapiro, and R.~Simion, \emph{{S}ome $q$-analogues of the
  {S}chr{\"o}der numbers arising from combinatorial statistics on lattice
  paths}, J. Statist. Plann. Inference \textbf{34} (1993), 35--55.

\bibitem{BouMelRec02}
M.~Bousquet-M{\'{e}}lou and A.~Rechnitzer, \emph{{L}attice animals and heaps of
  dimers}, Discrete Math. \textbf{258} (2002), 235--274.

\bibitem{Bres79}
D.~M. Bressoud, \emph{{A} generalization of the {R}ogers--{R}amanujan
  identities for all moduli}, J. Comb. Theory \textbf{27} (1979), 64--68.

\bibitem{CarRio64}
L.~Carlitz and J.~Riordan, \emph{{T}wo element lattice permutation numbers and
  their $q$-generalization}, Duke Math. J. \textbf{31} (1964), 371--388.

\bibitem{ChaEriStaMar02}
R.~Chapman, K.~Ericksson, R.~P. Stanley, and R.~Martin, \emph{{O}n the {N}umber
  of {D}ivisors of $n$ in a {S}pecial {I}nterval: 10847}, The Americal
  Mathematical Monthly \textbf{109} (2002), no.~1, 80.

\bibitem{ChePan17}
Z.~Chen and H.~Pan, \emph{{I}dentities involving weighted {C}atalan,
  {S}chr{\"o}der and {M}otzkin paths}, Adv. Appl. Math. \textbf{86} (2017),
  81--98.

\bibitem{Cig04}
J.~Cigler, \emph{$q$-{F}ibonacci {P}olynomials and the {R}ogers--{R}amanujan
  {I}dentities}, Ann. Comb. \textbf{8} (2004), 269--285.

\bibitem{Cok03}
C.~Coker, \emph{{E}numerating a class of lattice paths}, Discrete Math.
  \textbf{271} (2003), 13--28.

\bibitem{DelVie84}
M.-P. Delest and G.~Viennot, \emph{{A}lgebraic languages and polyominoes
  enumeration}, Theoret. Comput. Sci. \textbf{34} (1984), no.~1, 169--206.

\bibitem{DenSim95}
A.~Denise and R.~Simion, \emph{{T}wo combinatorial statistics on {D}yck paths},
  Discrete Math. \textbf{137} (1995), 155--176.

\bibitem{Deu99}
E.~Deutsch, \emph{{D}yck path enumeration}, Discr. Math. \textbf{204} (1999),
  167--202.

\bibitem{DonSha77}
R.~Donaghey and L.~W. Shapiro, \emph{{M}otzkin numbers}, J. Comb. Theory A
  \textbf{23} (1977), 291--301.

\bibitem{Dra09}
B.~Drake, \emph{{L}imits of areas under lattice paths}, Discrete Math.
  \textbf{309} (2009), no.~12, 3936--3953.

\bibitem{Eli20}
S.~Elizalde, \emph{{S}ymmetric peaks and symmetric valleys in {D}yck paths},
  preprint (2020), 19 pages,
  \href{http://arxiv.org/abs/2008.05669}{\path{arXiv:2008.05669}}.

\bibitem{Fla80}
P.~Flajolet, \emph{{C}ombinatorial aspects of continued fractions}, Discrete
  Math. \textbf{32} (1980), no.~2, 125--161.

\bibitem{FlaSed09}
P.~Flajolet and R.~Sedgewick, \emph{Analytic combinatorics}, Cambridge
  University Press, 2009.

\bibitem{FloRam20}
R.~Fl{\'o}rez and J.~L. Ram{\'i}rez, \emph{{E}numerating symmetric and
  asymmetric peaks in {D}yck paths}, Discrete Math. \textbf{343} (2020),
  no.~12, 112118.

\bibitem{FurHof85}
J.~F{\"{u}}rlinger and J.~Hofbauer, \emph{$q$-{C}atalan numbers}, J. Comb. Th.
  A \textbf{40} (1985), 248--264.

\bibitem{GarGan20}
A.~M. Garsia and G.~Ganzberger, \emph{{F}ibonacci {P}olynomials}, preprint
  (2020), 18 pages,
  \href{http://arxiv.org/abs/2009.10213}{\path{arXiv:2009.10213}}.

\bibitem{GesStanton86}
I.~Gessel and D.~Stanton, \emph{{A}pplications of $q$-{L}agrange inversion to
  basic hypergeometric series}, Trans. Amer. Math. Soc. \textbf{277} (1986),
  no.~1, 173--201.

\bibitem{Gor61}
B.~Gordon, \emph{{A} combinatorial generalization of the {R}ogers--{R}amanujan
  identities}, Amer. J. Math. \textbf{83} (1961), 393--399.

\bibitem{HirHir05}
M.~D. Hirschhorn and P.~M. Hirschhorn, \emph{{P}artitions into {C}onsecutive
  {P}arts}, Mathematics Magazine \textbf{78} (2005), no.~5, 396--397.

\bibitem{HogBicJoh77}
V.~E. Hoggatt, Jr. and M.~Bicknell-Johnson, \emph{{F}ibonacci convolution
  sequences}, Fibonacci Quarterly \textbf{15} (1977), no.~2, 117--122.

\bibitem{HogLon74}
V.~E. Hoggatt, Jr. and C.~T. Long, \emph{{D}ivisibility properties of
  generalized {F}ibonacci polynomials}, Fibonacci Quarterly \textbf{12} (1974),
  no.~2, 113--120.

\bibitem{Kos01}
T.~Koshy, \emph{{F}ibonacci and {L}ucas numbers with applications}, Pure and
  Applied Mathematics (New York), Wiley-Interscience, 2001.

\bibitem{Kra01}
C.~Krattenthaler, \emph{{P}ermutations with restricted patterns and {D}yck
  paths}, Adv. Appl. Math. \textbf{27} (2001), 510--530,
  \href{http://arxiv.org/abs/math/0002200}{\path{arXiv:math/0002200}}.

\bibitem{Kre70}
G.~Kreweras, \emph{{S}ur les \'enentails de segments}, Cahiers du B.U.R.O.
  \textbf{15} (1970), 1--41.

\bibitem{Kre86}
\bysame, \emph{{J}oint distribution of three descriptive parameters of
  bridges}, Combinatoire {\'{E}}num{\'e}rative, Proc. (Montr{\'e}al, Que.,
  1985/Qu{\'e}bec, Que., 1985), Lec. Notes in Math., vol. 1234, Springer,
  Berlin, Heidelberg, 1986, pp.~177--191.

\bibitem{KreMos86}
G.~Kreweras and P.~Moszkowski, \emph{{A} new enumerative property of the
  {N}arayana numbers}, J. Statist. Plann. and Infer. \textbf{14} (1986),
  63--67.

\bibitem{KrePou86}
G.~Kreweras and Y.~Poupard, \emph{{S}ubdivision des nombres de {N}arayana
  suivant deux param{\'e}tres suppl{\'e}mentaires}, Europ. J. Combin.
  \textbf{7} (1986), 141--149.

\bibitem{Liu02}
G.~Liu, \emph{{F}ormulas for convolution {F}ibonacci numbers and polynomials},
  Fibonacci Quarterly \textbf{40} (2002), no.~4, 352--357.

\bibitem{Mac1918}
P.~A. MacMahon, \emph{Combinatory {A}nalysis}, vol.~2, Cambridge Univ. Press,
  1918.

\bibitem{Man02}
T.~Mansour, \emph{{C}ounting peaks at height $k$ in a {D}yck path}, J. Integer
  Seq. \textbf{5} (2002), 02.1.1,
  \href{http://arxiv.org/abs/math/0203222}{\path{arXiv:math/0203222}}.

\bibitem{ManSun08}
T.~Mansour and Y.~Sun, \emph{{B}ell polynomials and $k$-generalized {D}yck
  paths}, Discrete Appl. Math. \textbf{156} (2008), 2279--2292,
  \href{http://arxiv.org/abs/0805.1273}{\path{arXiv:0805.1273}}.

\bibitem{ManSun09}
\bysame, \emph{{I}dentities involving {N}arayana polynomials and {C}atalan
  numbers}, Discrete Math. \textbf{309} (2009), 4079--4088,
  \href{http://arxiv.org/abs/0805.1274}{\path{arXiv:0805.1274}}.

\bibitem{MerRogSprVer97}
D.~Merlini, R.~Sprugnoli D.~G.~Rogers, and M.~C. Verri, \emph{{O}n some
  alternative characterizations of {R}iordan arrays}, Canad. J. Math.
  \textbf{49} (1997), no.~2, 301--320.

\bibitem{MerSprVer02}
D.~Merlini, R.~Sprugnoli, and M.~C. Verri, \emph{{S}ome statistics on {D}yck
  paths}, J. Statist. Plann. and Infer. \textbf{101} (2002), 211--227.

\bibitem{Mol12}
V.~H. Moll, \emph{{N}umbers and functions: from a classical-experimental
  mathematician's point of view}, Student mathematical library, American
  Mathematcial Society, Providence, RI, 2012.

\bibitem{Mor04}
P.~Moree, \emph{{C}onvoluted {C}onvolved {F}ibonacci {N}umbers}, J. Integer
  Seq. \textbf{7} (2004), 04.2.2,
  \href{http://arxiv.org/abs/math/0311205}{\path{arXiv:math/0311205}}.

\bibitem{OstderJeu15}
R.~Oste and J.~Van der Jeugt, \emph{{M}otzkin paths, {M}otzkin polynomials and
  recurrence relations}, Electron. J. Comb. \textbf{22} (2015), no.~2,
  $\#$P2.8.

\bibitem{PerPin99}
E.~Pergola and R.~Rinzani, \emph{{A} {C}ombinatorial {I}nterpretation of the
  {A}rea of {S}chor{\"o}der {P}aths}, Electron. J. Combin. \textbf{6} (1999),
  \#R40.

\bibitem{Ric95}
P.~E. Ricci, \emph{{G}eneralised {L}ucas polynomials and {F}ibonacci
  polynomials}, Rivista di Matematica della Universt{\`a} di Parma, Series 5
  \textbf{4} (1995), 137--146.

\bibitem{Rio68}
J.~Riordan, \emph{{C}ombinatorial {I}dentities}, Wiley, 1968.

\bibitem{DGRog81a}
D.~G. Rogers, \emph{{R}hyming schemes: {C}rossings and coverings}, Discrete
  Math. \textbf{33} (1981), 67--77.

\bibitem{DGRog81b}
D.~G. Rogers and L.~W. Shapiro, \emph{{D}eques, trees and lattice paths}, Lect.
  Notes Math. \textbf{884} (1981), 293--303.

\bibitem{Rog1894}
L.~J. Rogers, \emph{{S}econd memoir on the expansion of certain infinite
  products}, Proc. London Math. Soc. \textbf{25} (1894), 318--343.

\bibitem{SapTasTsi07}
A.~Sapounakis, I.~Tasoulas, and P.~Tsikouras, \emph{{C}ounting strings in
  {D}yck paths}, Discrete Math. \textbf{307} (2007), no.~23, 2909--2924.

\bibitem{Sch70}
E.~Schr{\"o}der, \emph{{V}ier {K}ombinatorische {P}robleme}, Z. Math. Phys.
  \textbf{15} (1870), 361--376.

\bibitem{Sch1917}
I.~Schur, \emph{{E}in {B}eitrag zur additiven {Z}ahlentheorie und zur {T}heorie
  der {K}ettenbr{\"{u}}che}, S.-B. Preuss. Akad. Wiss.Phys. Math. Klasse
  (1917), 302--321.

\bibitem{Sil18}
A.~V. Sills, \emph{{A}n invitation to the {R}ogers--{R}amanujan identities},
  CRC Press, 2018.

\bibitem{Sla52}
L.~J. Slater, \emph{{F}urther {I}dentities of the {R}ogers--ramanujan {T}ype},
  Proc. London Math. Soc. \textbf{54} (1952), no.~2, 147--167.

\bibitem{Slo}
N.~J.~A. Sloane, \emph{{T}he {O}n-{L}ine {E}ncyclopedia of {I}nteger
  {S}equences}, \url{http://oeis.org}.

\bibitem{Sta97}
R.~P. Stanley, \emph{{E}numerative {C}ombinatorics}, vol.~1, Cambridge
  University Press, 1997.

\bibitem{Sta972}
\bysame, \emph{{E}numerative {C}ombinatorics}, vol.~2, Cambridge University
  Press, 1999.

\bibitem{Sta15}
\bysame, \emph{{C}atalan {N}umbers}, Cambridge University Press, 2015.

\bibitem{Sul98}
R.~A. Sulanke, \emph{{B}ijective {R}ecurrences concerning {S}chr{\"o}er paths},
  Electro. J. Combin. \textbf{47} (1998), \#R47.

\bibitem{Sul00}
\bysame, \emph{{C}ounting lattice paths by {N}arayana polynomials}, Electron.
  J. Combin. \textbf{7} (2000), \#R40.

\bibitem{Sul02}
\bysame, \emph{{T}he {N}arayana distribution}, J. Statist. Plann. Inference
  \textbf{101} (2002), 311--326.

\bibitem{Vie83}
G.~Viennot, \emph{{U}ne {T}h{\'e}orie {C}ombinatoire des {P}olynomes
  {O}rthogonaux {G}eneraux}, Notes de conf\'erences donn\'ees au D\'epartement
  de math\'ematiques et d'informatique, Universit\'e du Qu\'ebec \`a
  Montr\'eal, Montr\'eal, 1983.

\bibitem{Vie85}
\bysame, \emph{{A} combinatorial theory for general orthogonal polynomials with
  extensions and applications}, Polyn\^omes Orthogonaux et Applications
  (C.~Brezinski, A.~Draux, A.~P. Magnus, P.~Maroni, and A.~Ronveaux, eds.),
  Lecture Notes Math., vol. 1171, Springer, Berlin, Heidelberg, 1985,
  pp.~139--157.

\bibitem{Vie86}
\bysame, \emph{{H}eaps of {P}ieces. {I}. {B}asic definitions and combinatorial
  lemmas}, Lecture Notes in Math., vol. 1234, Springer, Berlin, 1986,
  pp.~321--350.

\end{thebibliography}

\end{document}